\documentclass[a4paper,10pt]{article}
\usepackage{geometry}                
\usepackage[active]{srcltx}
\usepackage{hyperref}
\usepackage[normalem]{ulem}

\usepackage{graphicx}
\usepackage{amssymb}
\usepackage{epstopdf}
\usepackage{amsmath}
\usepackage{amsthm}
\usepackage{amssymb}
\usepackage{latexsym}
\usepackage{mathtools}
\usepackage{mathrsfs}
\usepackage{graphics}
\usepackage{latexsym}
\usepackage{psfrag}
\usepackage{import}
\usepackage{verbatim}
\usepackage{enumerate}
\usepackage{enumitem}
\usepackage{graphicx}
\usepackage[usenames]{color}

\newcommand{\dom}{{\rm{Dom\,}}}
\newcommand{\R}{\mathbb{R}}
\newcommand{\N}{\mathbb{N}}

\newcommand{\cH}{\mathcal{H}}

\newcommand{\Q}{\mathbb{Q}}

\newcommand{\LL}{\text{\rm L}}

\newcommand{\supp}{\text{\rm supp}}

\newcommand{\gr}{\textrm{graph}}
\newcommand{\id}{\mathrm{Id}}

\newcommand{\diam}{{\rm{diam\,}}}

\newcommand{\ve}{\varepsilon}

\newcommand{\cI}{\mathcal{I}}

\newcommand{\T}{\mathcal{T}}

\renewcommand{\L}{\mathcal{L}}

\newcommand{\CD}{\mathsf{CD}}

\newcommand{\Geo}{{\rm Geo}}
\newcommand{\TGeo}{{\rm TGeo}}
\newcommand{\MCP}{\mathsf{MCP}}

\newcommand{\Ent}{{\rm Ent}}

\newcommand{\Dom}{{\rm Dom}}
\newcommand{\fs}{\mathfrak{s}}
\newcommand{\fc}{\mathfrak{c}}
\newcommand{\fI}{\mathfrak{I}}
\newcommand{\cK}{\mathcal{K}}
\newcommand{\fa}{\mathfrak{a}}
\newcommand{\fb}{\mathfrak{b}}
\newcommand{\dd}{\mathrm{d}}

\newcommand{\mui}{\mu_\infty}
\newcommand{\rmC}{{\mathrm C}}

\newcommand{\Cb}[1]{\rmC_b(#1)}

\newcommand{\Prob}{\mathcal P}

\newcommand{\BB}{\mathscr{B}}
\newcommand{\BorelSets}[1]{\BB(#1)}

\newcommand{\Probabilities}[1]{\mathcal P(#1)}  
\newcommand{\prob}{\Probabilities}

\newcommand{\Ric}{{\rm Ric}}

\newcommand{\eps}{\varepsilon}
\newcommand{\tr}{\textrm{tr}}

\renewcommand{\L}{\mathcal{L}}

\newcommand{\ggamma}{\boldsymbol\gamma}

\newcommand{\vol}{\mathrm{Vol}}

\renewcommand{\P}{\mathbb P}

\renewcommand{\P}{\mathcal{P}}

\newcommand{\TMCP}{\mathsf{TMCP}}
\newcommand{\TCD}{\mathsf{TCD}}
\newcommand{\wTCD}{\mathsf{wTCD}}

\renewcommand{\H}{\mathcal{H}}
\newcommand{\spt}{{\rm spt}}
\newcommand{\mm}{\mathfrak m}
\newcommand{\qq}{\mathfrak q}

\newcommand{\ee}{{\rm e}}
\newcommand{\QQ}{\mathfrak Q}

\newcommand{\sfd}{\mathsf d}

\renewcommand{\cH}{\mathcal{H}}
\newcommand{\forevery}{\text{for every }}

\theoremstyle{plain}
\newtheorem{lemma}{Lemma}[section]
\newtheorem{theorem}[lemma]{Theorem}

\newtheorem{proposition}[lemma]{Proposition}
\newtheorem{corollary}[lemma]{Corollary}

\newtheorem*{theorem*}{Theorem}
\newtheorem*{maintheorem*}{Main Theorem}

\theoremstyle{definition}

\newtheorem{definition}[lemma]{Definition}
\newtheorem*{definition*}{Definition}
\newtheorem*{remark*}{Remark}
\newtheorem{remark}[lemma]{Remark}

\numberwithin{equation}{section}

\title{Optimal transport in Lorentzian synthetic  spaces, synthetic timelike Ricci curvature lower bounds and applications}
\author{Fabio Cavalletti\thanks{F. Cavalletti: Mathematics Area, SISSA, Trieste (Italy), email:cavallet@sissa.it.}\,\,  and Andrea Mondino\thanks{A. Mondino: Mathematical Institute,  University of Oxford  (UK), email: Andrea.Mondino@maths.oxford.ac.uk.}
}

\date{}     

\begin{document}
\maketitle

\begin{abstract}
The goal of the present work is three-fold. The first goal is to set  foundational results on optimal transport in Lorentzian (pre-)length spaces, including cyclical monotonicity, stability of optimal couplings and Kantorovich duality (several results are new even for smooth Lorentzian manifolds). 
\\ The second one is to give a synthetic notion of ``timelike Ricci curvature bounded below and dimension bounded above''  for a  measured Lorentzian pre-length space using optimal transport. The key idea being to analyse convexity properties of Entropy functionals along future directed timelike geodesics of probability measures. 
This notion is proved to be stable under a suitable weak convergence of  measured Lorentzian pre-length spaces, giving a glimpse on the strength of the approach we propose.
The third goal is to draw applications, most notably extending volume comparisons and Hawking singularity Theorem (in sharp form) to the synthetic setting.
\\The  framework of  Lorentzian pre-length spaces includes as remarkable classes of examples: space-times endowed with a causally plain (or, more strongly, locally Lipschitz) continuous Lorentzian metric, closed cone structures, some approaches to quantum gravity.
\end{abstract}

\bibliographystyle{plain}

\tableofcontents

\section*{Introduction}
As the title suggests, the goal of the present work is three-fold. The first goal is to set  foundational results on optimal transport in Lorentzian synthetic spaces.  The second one is to use optimal transport to give a synthetic notion for a  Lorentzian  space of ``timelike Ricci curvature bounded below and dimension bounded above''  verifying suitable properties like compatibility with the classical case and stability under weak convergence.
The third goal is to draw applications, most notably extending volume comparisons and Hawking singularity Theorem (in sharp form) to the synthetic framework.

The Lorentzian synthetic framework adopted in the paper is the one of \emph{Lorentzian pre-legth (and geodesic) spaces} introduced by Kunzinger and S\"amann in \cite{KS} (see also an independent approach by Sormani and Vega \cite{SoVe}). The basic idea is that Lorentzian pre-length (resp.\;geodesic) spaces are the non-smooth analog of  Lorentzian manifolds, in the same spirit as  classical metric (resp.\;geodesic) spaces are the non-smooth analog of Riemannian manifolds (see Section \ref{Subsec:BasicsLorentzianSpaces} for the precise notions).

In the metric (measured) framework, the celebrated work of Sturm \cite{sturm:I, sturm:II} and Lott-Villani \cite{lottvillani:metric}  laid the foundations for a theory of metric measure spaces satisfying Ricci curvature lower bounds and dimension upper bounds in a synthetic sense via optimal transport, the so-called $\CD(K,N)$ spaces.  The theory of $\CD(K,N)$ spaces flourished in the last years with strong connections with analysis, geometry and probability.
The ambition of the present paper is to lay the foundations for a parallel theory in the Lorentzian setting, which is the natural geometric framework for general relativity. 

\subsection*{Motivations}
Before discussing the main results, let us motivate the questions that we address. A main motivation for this work is the need to consider Lorentzian metrics/spaces of low regularity. Such a necessity is clear both from the PDE point of view in general relativity (i.e.\;the Cauchy initial value problem for the Einstein equations) and from physically relevant models. 

From the PDE point of view, the standard local existence results for the vacuum Einstein equations assume the metric to be of Sobolev regularity $H^{s}_{{\rm loc}}$, with $s>\frac{5}{2}$ (see for instance \cite{Rendall}). The Sobolev regularity of the metric has been lowered even further (e.g. \cite{KRS}). Related to the initial value problem for the Einstein equations, one of the main open problems in the field is the so called (weak/strong) censorship conjecture (see e.g. \cite{Chri, Daf}). Such a conjecture (strong form) states roughly that the maximal globally hyperbolic development of generic initial data for the Einstein equations is inextendible as a suitably regular Lorentzian manifold.
Formulating a precise statement of the conjecture is itself non-trivial since one needs to give a precise meaning to ``generic initial data'' and ``suitably regular Lorentzian manifold''. Understanding the latter is where Lorentzian metrics of low regularity and related inextendibility results become significant.  The strongest form of the conjecture would prove inextendibility for a  $C^{0}$ metric.  As pointed out by Chrusciel-Grant \cite{CG}, causality theory for $C^{0}$ metrics departs significantly from classical theory (e.g. the lightlike curves emanating from a point may span a set with non-empty interior, a phenomenon called ``bubbling''). Nevertheless, Sbierski \cite{Sbi} gave a clever proof of $C^{0}$-inextendibility of Schwarzschild,  \cite{MiSu} showed $C^{0}$-inextendibility for timelike geodesically complete spacetimes, and \cite{GKS} pushed the inextendibility to Lorentzian length spaces.

From the point of view of physically relevant models, several types of matter in a spacetime may give a discontinuous energy-momentum tensor and thus, via the Einstein's equations, lead to a Lorentzian metric of regularity lower than $C^{2}$ (e.g. \cite{Lichn}). Examples of such a behaviour are spacetimes that model the inside and outside of a star,  matched spacetimes \cite{MaSe}, self-gravitating compressible fluids \cite{BuLe}, or shock waves. Some physically relevant models require even lower regularity, for instance: spacetimes with conical singularities \cite{Vick2},  cosmic strings \cite{Vick1} and (impulsive) gravitational waves (see for instance \cite{PenGW}, \cite[Chapter 20]{GrPo}).

Finally, a long term motivation for studying non-regular Lorentzian spaces is the desire of understanding the ultimate nature of spacetime. The rough picture is that at the quantum level (and thus in extreme physical conditions, e.g.\;gravitational collapse, origin of the universe), the spacetime may be very singular and possibly not approximable by smooth  structures (see Remark \ref{rem:OtherExamples}). 
\\

In case of a metric of low regularity, the approach to curvature used so far is distributional, taking advantage that the underline spacetime is a differentiable manifold. This permits  \cite{GeTr} (see also \cite{StVick}) to define distributional curvature tensors for  $W^{1,2}_{{\rm loc}}$-Lorentzian metrics  satisfying a suitable non-degeneracy condition (satisfied for instance when the metric is $C^{1}$, see  \cite{Graf}). One of the goals of the present work  is to address the question of (timelike Ricci) curvature when \emph{not only the the metric tensor, but the spacetime itself is singular}.
\\

A lower bound on the timelike Ricci curvature of a spacetime $(M^{n},g)$, i.e.
\begin{equation}\label{eq:Ricgeq-Kg}
\text{There exists $K\in \R$ such that } \Ric_{g}\geq- K g(v,v)  \text{ for all  timelike vectors $v\in TM$,}
 \end{equation}
 is quite a natural assumption in general relativity. Of course, for a $C^{2}$-metric $g$,  \eqref{eq:Ricgeq-Kg} is satisfied on compact subsets of the space-time.  
Recalling that the Einstein's equations postulate proportionality of $\Ric_{g}$ and $T-\frac{1}{n-2} \tr_{g}(T) g$ (where $T$ is the so-called energy-momentum tensor),  for a general cosmological constant $\Lambda \in \R$, \eqref{eq:Ricgeq-Kg} is equivalent to require that 
 $$T(v,v) \geq  - \frac{1}{n-2} \tr_{g}(T) + \frac{1}{8\pi} \left(  K-\frac{2\Lambda}{n-2}\right),  \text{ for all  $v\in TM$ with $g(v,v)=-1$}.$$
 In particular, if  $\inf_{M} \tr_{g}(T)>-\infty$ (or, equivalently,  $\inf_{M} {\rm R}_{g}>-\infty$ where  ${\rm R}_{g}$ is the scalar curvature of $g$), then the \emph{weak} energy condition $T(v,v)\geq 0$ for all timelike $v$ (which is believed to hold for most physically reasonable $T$, according to \cite[pag. 218]{Wald}) implies \eqref{eq:Ricgeq-Kg}.
\\ The case $K=0$ in \eqref{eq:Ricgeq-Kg} corresponds to the  \emph{strong energy condition} of Hawking and Penrose \cite{PenColl, HawkSing, HawkPen}. 
\\ 
Let us stress that the framework \eqref{eq:Ricgeq-Kg} includes any solution of the Einstein's equations \emph{in vacuum} (i.e.\;with null stress-energy tensor $T$) with \emph{possibly non-zero cosmological constant} $\Lambda$. Already such a framework is highly interesting as the standard black hole metrics (e.g.\;Schwartzshild, Kerr) are solutions of the Einstein's vacuum equations, and also the more recent literature on black holes typically focuses on vacuum solutions (see e.g. \cite{Chri, Daf, DHR, KRS}).  A key role in such  breakthroughs on black holes  is given by a deep analysis of the system of non-linear hyperbolic partial differential equations corresponding to the Einstein's vacuum equations (in a suitable Gauge). At least in the smooth setting, it was recently proved by the second author and Suhr \cite{MoSu} that the optimal transport point of view is compatible with the hyperbolic PDEs one (in the sense that it is possible to characterise solutions of the Einstein's equations in terms of optimal transport). We believe that the aforementioned facts suggest a promising future for the proposed optimal transport approach.

\subsection*{Outline of the content of the paper}
\subsubsection*{General synthetic setting}
We now pass to discuss the content of the paper. The synthetic framework is the one of \emph{measured  Lorentzian pre-length  spaces} $(X,\sfd, \mm, \ll, \leq, \tau)$ where $X$ is a set endowed with a proper metric $\sfd$   (i.e. closed and bounded subsets are compact) and the associated metric topology,  a preorder $\leq$ (playing the role of \emph{causal relation}) and a transitive relation $\ll$ contained in $\leq$ (playing the role of \emph{chronological/timelike  relation}), a lower semicontinuous function $\tau: X\times X\to [0,\infty]$  (called \emph{time-separation function})  with $\{\tau>0\}=\{(x,y)\in X^{2}: x\ll y\}$ and satisfying reverse triangle inequality \eqref{eq:deftau}, and a non-negative Radon measure  $\mm$  with $\supp \, \mm=X$.  
\\The subset of causal pairs is denoted with $X^{2}_{\leq}:=\{(x,y)\in X^{2}\,:\, x\leq y\}$.

A curve $\gamma:[0,1]\to X$ is \emph{causal} if  it is continuous and for every $t_{0}\leq t_{1}$ it holds $\gamma_{t_{0}}\leq \gamma_{t_{1}}$.  One can naturally associate a $\tau$-length to $\gamma$, denoted by $\LL_{\tau}(\gamma)$ (see Definition \ref{def:Ltau}). A  causal curve is a \emph{geodesic} if it \emph{maximises} the $\tau$-length  and is parametrized by $\tau$-arc length, i.e. if  $\LL_{\tau}(\gamma)=\tau(\gamma_{0}, \gamma_{1})$ and $\tau(\gamma_s, \gamma_t)=(t-s)\, \tau(\gamma_0, \gamma_1)$ for all $0\leq s\leq t\leq 1$. 
The space $X$ is said to be \emph{geodesic} if for all $(x,y)\in X^{2}_{\leq}$  there is a geodesic $\gamma$ from $x$ to $y$.

Important classes of examples entering the framework of measured  Lorentzian pre-length/geodesic  spaces are spacetimes with a causally plain (or, more strongly, locally Lipschitz) $C^{0}$-metric (see Remark \ref{rem:C0metrics}), closed cone structures, as well as  some approaches to quantum gravity (see Remark \ref{rem:OtherExamples}). 

\subsubsection*{Optimal transport in Lorentzian pre-length spaces}
Our approach to synthetic timelike Ricci curvature lower bounds is via optimal transport of causally related probability measures. To this aim, in Section \ref{sec:OTLorentz} we throughly analyse optimal transport in Lorentzian pre-length spaces.
A key object is the space of Borel probability measures $\Prob(X)$ on $X$, and the subspace $\Prob_{c}(X)$ of Borel probability measures with compact support. In order to lift the causal structure of $X$ to $\Prob(X)$, it is useful to consider the set of   \emph{causal couplings} between two probability measures $\mu,\nu\in \Prob(X)$:
$$\Pi_{\leq}(\mu,\nu):=\{\pi\in \Prob(X^{2}): \pi(X^{2}_{\leq})=1, \, (P_{1})_{\sharp} \pi=\mu, \, (P_{2})_{\sharp} \pi=\nu \},$$
 where $P_{i}:X\times X\to X$ is the projection on the $i^{th}$ factor, and $(P_{i})_{\sharp}: \Prob(X^{2})\to \Prob(X)$ is the associated push-forward map defined as $\left((P_{i})_{\sharp} \pi \right)(B):=\pi\left(P_{i}^{-1} (B) \right)$ for every Borel subset $B\subset X$.
 We say that $(\mu,\nu)$ are \emph{causally related} if $\Pi_{\leq}(\mu,\nu)\neq \emptyset$. The rough picture is that $\mu$ and $\nu$ represent some random distribution of events in the spacetime $X$, and the two are causally related if it is possible to causally match  events described by $\mu$ with events described by $\nu$ (possibly in a multi-valued way) via the causal coupling $\pi$.  We endow $\Prob(X)$ with the  \emph{$p$-Lorentz-Wasserstein distance}  defined by
\begin{equation}\label{eq:defWpIntro}
\ell_{p}(\mu,\nu):= \sup_{\pi \in \Pi_{\leq}(\mu,\nu)} \left(  \int_{X\times X}  \tau(x,y)^{p} \, \pi(\dd x\dd y)\right)^{1/p}, \quad p\in (0,1].
\end{equation}
When $\Pi_{\leq}(\mu,\nu)=\emptyset$ we set $\ell_{p}(\mu,\nu):=-\infty$.  The name $p$-Lorentz-Wasserstein distance is motivated by the fact that $\ell_{p}$ satisfies a reversed triangle inequality (see Proposition \ref{prop:RTIellp}).
\\ Note that  \eqref{eq:defWpIntro} extends to Lorentzian pre-length spaces the corresponding notion given in the smooth Lorentzian setting in \cite{EM17}  (see also \cite{McCann, MoSu}, and \cite{Suhr} for $p=1$).    A coupling $\pi\in  \Pi_{\leq}(\mu,\nu)$ maximising in \eqref{eq:defWpIntro} is said \emph{$\ell_{p}$-optimal}.  The set of \emph{$\ell_{p}$-optimal} couplings from $\mu$ to $\nu$ is denoted by $  \Pi_{\leq}^{p\text{-opt}}(\mu,\nu)$.

Maximising over \emph{causal} couplings $\Pi_{\leq}(\mu,\nu)$ instead of all the couplings can be modelled with an auxiliary cost (denoted with $\ell^{p}$) taking value $-\infty$ outside of $X^{2}_{\leq}$ (see Remark \ref{rem:eqellq}). The fact that the cost function takes value $-\infty$ makes the associated optimal transport problem more challenging: several fundamental results (see e.g. \cite{AGSBook, Vil:topics,Vil}) 
are not available in the present setting and classical concepts take a different flavour. 
These include: cyclical monotonicity, stability of optimal couplings, Kantorovich duality. 
It is indeed the goal of Section  \ref{sec:OTLorentz} 
to study such notions in this setting. It is beyond the scopes of the introduction to give a detailed account of the results (several are new even in the smooth Lorentzian setting), we only mention few notions (in a slightly simplified form) that will be useful for analysing timelike Ricci curvature bounds. 

We say that $(\mu,\nu)\in \Prob_{c}(X)^{2}$ is \emph{timelike $p$-dualisable (by $\pi\in \Pi_{\leq}(\mu,\nu)$)}  if   $\ell_{p}(\mu,\nu)\in (0,\infty)$,  $\pi\in  \Pi_{\leq}^{p\text{-opt}}(\mu,\nu)$ and $\supp \, \pi\subset \{\tau>0\}$. The pair  $(\mu,\nu)\in \Prob_{c}(X)^{2}$ is \emph{strongly timelike $p$-dualisable} if in addition there exists a subset $\Gamma\subset  \{\tau>0\}\subset  X^{2}$ such that \emph{every} $p$-optimal coupling $\pi'\in \Pi_{\leq}^{p\text{-opt}}(\mu,\nu)$ is concentrated on $\Gamma$, i.e. $\pi'(\Gamma)=1$ (see Definitions  \ref{def:Tpdual} and \ref{def:TStrongDual} for the precise notions).  

Let us also mention that if $X$ is geodesic and globally hyperbolic then $(\Prob_{c}(X),\ell_{p})$  is geodesic as well, for $p\in (0,1)$. More precisely (see Proposition \ref{prop:ellqgeodCS}), if $(\mu_{0},\mu_{1})\in \Prob_{c}(X)^{2}$ is timelike $p$-dualisable, then there exists an $\ell_{p}$-geodesic $(\mu_{t})_{t\in [0,1]}\subset \Prob_{c}(X)$ joining them.

\subsubsection*{Synthetic timelike Ricci curvature lower bounds via optimal transport}
The relation between optimal transport and timelike Ricci curvature bounds in the \emph{smooth} Lorentzian setting has been the object of recent works by McCann \cite{McCann} and Mondino-Suhr \cite{MoSu}. 
The key idea is that timelike Ricci curvature lower bounds can be equivalently characterised in terms of convexity properties of the Bolzmann-Shannon entropy functional $\Ent(\cdot|\mm)$ along $\ell_{p}$-geodesics of probability measures (where, for smooth Lorentzian manifolds, $\mm$ is the standard volume measure).
 Recall that, for a probability measure $\mu \in \Prob(X)$, the  entropy $\Ent(\mu|\mm)$ is defined by
 $$
\Ent(\mu|\mm) = \int_{X} \rho \log(\rho) \, \mm,
$$
if $\mu = \rho \, \mm$ is absolutely continuous with respect to $\mm$ and $(\rho\log(\rho))_{+}$ is  $\mm$-integrable;  otherwise we set $\Ent(\mu|\mm) = +\infty$.  We denote $\Dom(\Ent(\cdot|\mm)):=\{\mu\in \Prob(X): \Ent(\mu|\mm)<\infty\}$.
\\The following definition is thus natural.

\begin{definition*}[$\mathsf{TCD}^{e}_{p}(K,N)$ and  $\mathsf{wTCD}^{e}_{p}(K,N)$ conditions]\label{def:TCD(KN)Intro}
Fix $p\in (0,1)$, $K\in \R$, $N\in (0,\infty)$. We say that  a  measured Lorentzian  pre-length space $(X,\sfd,\mm, \ll, \leq, \tau)$ satisfies  $\mathsf{TCD}^{e}_{p}(K,N)$ 
(resp. $\mathsf{wTCD}^{e}_{p}(K,N)$)  if the following holds.
For any couple $(\mu_{0},\mu_{1})\in (\Dom(\Ent(\cdot|\mm))\cap  \Prob_{c}(X))^{2}$ which is (resp. strongly)   timelike $p$-dualisable   by some 
$\pi\in \Pi^{p\text{-opt}}_{\leq}(\mu_{0},\mu_{1})$,  
there exists an  $\ell_{p}$-geodesic $(\mu_{t})_{t\in [0,1]}$ such that  
the function $[0,1]\ni t\mapsto e(t) : = \Ent(\mu_{t}|\mm)$ is 
semi-convex and it satisfies
\begin{equation*}
e''(t) - \frac{1}{N} e'(t)^{2 } \geq K \int_{X\times X} \tau(x,y)^{2} \, \pi(\dd x\dd y),\quad  \text{in the distributional sense on $[0,1]$.}
\end{equation*}
\end{definition*}

\begin{remark*}[Notation]
The notation $\mathsf{TCD}^{e}_{p}(K,N)$ comes by  analogy with the
corresponding Lott-Sturm-Villani theory of curvature dimension conditions in metric-measure
spaces. Here the superscript $e$
refers to the so-called ``entropic" formulation of the $\CD$ condition by Erbar, Kuwada and Sturm \cite{EKS}; such a formulation is slightly simpler, but equivalent under suitable technical assumptions. The possibility $p\in (1,\infty), p\neq 2$ was investigated by Kell \cite{KellACV} in the
metric-measure setting.  The leading $\mathsf{T}$ stands for ``timelike'', following the notation of Woolgar and Wylie in their paper on $N$-Bakry-\'Emery spacetimes \cite{WW}, and of  McCann \cite{McCann}.
The symbol $\mathsf{w}$ in $\mathsf{wTCD}$ has to be read ``weak $\mathsf{TCD}$'' and it 
is justified by the comparison with $\mathsf{TCD}$ requiring 
convexity estimates for the entropy along a smaller family of  $\ell_{p}$-geodesics. 
\end{remark*}
\medskip

\noindent
The  $\mathsf{TCD}^{e}_{p}(K,N)$ (resp.  $\mathsf{wTCD}^{e}_{p}(K,N)$) condition satisfies the following natural compatibility properties:
\begin{itemize}
\item $\mathsf{TCD}^{e}_{p}(K,N)$ (resp. $\mathsf{wTCD}^{e}_{p}(K,N)$) implies  $\mathsf{TCD}^{e}_{p}(K',N')$ (resp. $\mathsf{wTCD}^{e}_{p}(K',N')$) for all $K'\leq K$, $N'\geq N$, see Lemma \ref{lem:TCDscaling};
\item A smooth globally hyperbolic Lorentzian manifold  $(M^{n}, g)$ has $\dim(M)=n\leq N$ and $\Ric_{g}(v,v)\geq -K g(v,v)$ for every timelike $v\in TM$  if and only if it satisfies   $\mathsf{TCD}^{e}_{p}(K,N)$, if and only if it satisfies  $\mathsf{wTCD}^{e}_{p}(K,N)$, see Theorem \ref{thm:CharLorRic} and Corollary \ref{Cor:CompMCPsmooth}.
\end{itemize}

\noindent
We show that  $\mathsf{wTCD}^{e}_{p}(K,N)$ spaces satisfy the following geometric properties:
\begin{itemize}
\item a timelike Brunn-Minkowski  inequality, see Proposition \ref{prop:BrunnMnk};
\item  a timelike Bishop-Gromov  inequality,  Proposition \ref{prop:BisGro};
\item  a timelike Bonnet-Myers  inequality, Proposition \ref{prop:BonMy}.
\end{itemize}

A weaker variant of the $\mathsf{TCD}^{e}_{p}(K,N)$ condition is obtained by considering $(K,N)$-convexity properties only for those $\ell_{p}$-geodesics $(\mu_{t})_{t\in [0,1]}$ where $\mu_{1}$ 
is a Dirac delta. 
In the metric measure setting, such a variant goes under the name of 
Measure Contraction Property ($\mathsf{MCP}$ for short) and was developed  independently by Sturm \cite{sturm:II} and Ohta \cite{Ohta1}. 
We call ``Timelike Measure Contraction Property'' ($\mathsf{TMCP}^{e}(K,N)$ for short) such a weaker variant of $\mathsf{TCD}^{e}_{p}(K,N)$, see Definition \ref{def:TMCP} for the precise notion.
The following holds:
\begin{itemize}
\item under mild conditions on the space $X$ (satisfied for instance for causally plain, globally hyperbolic spacetimes with a $C^{0}$ metric) $\mathsf{wTCD}^{e}_{p}(K,N)$ implies $\mathsf{TMCP}^{e}(K,N)$, see Proposition \ref{prop:CD->MCP};
\item $\mathsf{TMCP}^{e}(K,N)$ implies  $\mathsf{TMCP}^{e}(K',N')$ for all $K'\leq K$, $N'\geq N$, see Lemma \ref{lem:TCDscaling};
\item a smooth  globally hyperbolic Lorentzian manifold  $(M^{n}, g)$ with $\dim(M)=n\geq 2$ satisfies $\mathsf{TMCP}^{e}(K,n)$ if and only if $\Ric_{g}(v,v)\geq -K g(v,v)$ for every timelike $v\in TM$, see Theorem \ref{thm:TMCPsmooth};
\item if  a smooth globally hyperbolic Lorentzian manifold  $(M^{n}, g)$ satisfies $\mathsf{TMCP}^{e}(K,N)$, then   $\dim(M)=n\leq N$, see Corollary \ref{Cor:CompMCPsmooth};
\item  the aforementioned timelike Bishop-Gromov  inequality (Proposition \ref{prop:BisGro}) and timelike Bonnet-Myers  inequality (Proposition \ref{prop:BonMy}) remain valid
for  $\mathsf{TMCP}^{e}(K,N)$ spaces; 
\end{itemize}

In addition to the aforementioned geometric consequences of the synthetic curvature bounds, the main results 
of this part concern the fundamental property of  
\emph{stability}. 
We show a weak stability property for  $\mathsf{TCD}$ stating that if a sequence of  $\mathsf{TCD}^{e}_{p}(K,N)$ spaces converges weakly to a limit Lorentzian pre-length space, then the limit space satisfies $\mathsf{wTCD}^{e}_{p}(K,N)$ (see Theorem \ref{thm:StabTCD} for the precise statement).
Instead the $\mathsf{TMCP}^{e}(K,N)$ condition is 
stable in the usual sense: 
if a sequence of  $\mathsf{TMCP}^{e}(K,N)$ spaces
converges weakly to a limit Lorentzian pre-length space, 
then the limit space satisfies $\mathsf{TMCP}^{e}
(K,N)$,  see Theorem  \ref{thm:StabTMCP}.

\smallskip
It is worth stressing that 
stability of $\mathsf{TCD}^{e}_{p}(K,N)$ and  $\mathsf{TMCP}^{e}(K,N)$ has no counterpart in the classical versions of $\CD$ (and $\MCP$) for extended metrics. 
Indeed, while the theory of extended 
metric measure spaces verifying $\CD$ has been investigated \cite{AES}, so far no general results on their stability has been established. 
The fact that the convexity of the entropy is required to hold \emph{only} along $\ell_{p}$-geodesics connecting \emph{timelike $p$-dualisable} measures does not permit to follow the stream of ideas of the known stability arguments of \cite{lottvillani:metric, sturm:I, sturm:II, GMS} for $\CD$ and $\MCP$ 
and a new strategy has to be devised.

We mention another obstruction 
to the classical approach to  stability and Gromov (pre-)compactness: 
while the $\CD/\MCP$ conditions imply a control on the volume growth of metric balls and thus  compactness in  pointed-measured-Gromov-Hausdorff topology (which is thus the natural notion for weak convergence of spaces), in our setting the $\tau$-balls typically have infinite volume (for instance in Minkowski space, $\tau$-spheres are  hyperboloids); thus we cannot expect the same (pre)-compactness properties in the  pointed-measured-Gromov-Hausdorff topology (which is thus not anymore the clearly natural notion for weak convergence of spaces).

\subsubsection*{Timelike non-branching $\mathsf{TMCP}^{e}(K,N)$ and applications}
An important subclass of Lorentzian geodesic spaces is the one of \emph{timelike non-branching} structures: roughly the ones for which timelike geodesic do not branch (both forward and backward in time), see Definition \ref{def:TNB} for the precise notion. In the classical Lorentzian setting, this is satisfied for $C^{1,1}$ metrics and it is expected to fail for lower regularity.
The same phenomenon happens in the Riemannian/metric setting, where the non-branching assumption (or slightly weaker variants) is rather standard in the recent literature of $\CD/\MCP$ spaces.

For timelike non-branching $\mathsf{TMCP}^{e}(K,N)$ spaces we obtain:
\begin{itemize}
\item solution to the $\ell_{p}$-Monge problem:  if $(\mu_{0}, \mu_{1})$ are timelike $p$-dualisable with $\mu_{0}\in \Dom(\Ent(\cdot|\mm))$, then there exists a unique $\ell_{p}$-optimal coupling $\pi\in \Pi^{p\text{-opt}}_{\leq}(\mu_{0},\mu_{1})$ such that $ \pi \left( \{\tau>0\} \right)=1$ and it is induced by a map; see Theorem \ref{T:1}.
\\Under the same assumptions, there exists a unique  $\ell_{p}$-geodesic  from $\mu_{0}$ to $\mu_{1}$; see Theorem  \ref{T:2};

\item a synthetic notion of mean curvature bounds for achronal Borel sets having locally finite ``area", see Section \ref{SS:SyntMC};

\item a sharp version (holding for every $K\in \R, N\in [1,\infty)$) of the Hawking singularity Theorem, see Theorem  \ref{thm:HawkSingSint}. Let us mention that the statement of Theorem   \ref{thm:HawkSingSint} is sharp, as for $N\in \N$ the bounds are attained in the smooth model  spaces identified in \cite{GrTr} (see also \cite{Graf1});

\item sharp versions of timelike Bishop-Gromov, Poincar\'e, and  Bonnet-Myers  inequalities (see Propositions  \ref{thm:BGAchronal},   \ref{prop:Poincare}, \ref{prop:BisGro2}, \ref{prop:BonMy2}; for the sharpness see Remark \ref{rem:sharpness}).
\end{itemize}

In order to obtain the applications in the last three bullet points, in Section \ref{Sec:LocTMCP}, we study the $\ell_{1}$-optimal transport problem associated to the $\tau$-distance function $\tau_{V}$ from an achronal set $V$ (see \eqref{eq:deftauV} for the definition of $\tau_{V}$). The rough idea is that $\tau_{V}$ induces a partition of $I^{+}(V)$, namely ``the chronological future of $V$'', into timelike geodesics (also called ``rays''). In the smooth setting (outside the cut locus) such rays correspond to the gradient flow curves of $\tau_{V}$. Such a partition of $I^{+}(V)$ induces a disintegration of $\mm$ into  one-dimensional conditional measures, which satisfy $\MCP(K,N)$ (see Theorems \ref{P:nointialpoints} and \ref{P:localKant} for the precise statements). In Section \ref{SS:SyntMC}, the disintegration is used to construct an ``area measure'' as well as ``normal variations'' of $V$,  and thus define synthetic notions of mean curvature bounds.  At this point, the above applications will follow.

 The fact that the one-dimensional conditional measures satisfy $\MCP(K,N)$ is not trivial: recall indeed that the $\TMCP^{e}(K,N)$ and $\TCD^{e}_{p}(K,N)$ conditions are expressed in terms of $\ell_{p}$  (not $\ell_{1}$) optimal transport, while here we are dealing with an $\ell_{1}$-optimal transport problem. The key idea  to overcome this issue is to   include $\ell^{p}$-cyclically monotone sets inside $\ell$-cyclically monotone sets; this technique was  introduced in  \cite{cava:decomposition} and  pushed further in \cite{CM1,CM2} for  the metric setting.
In the present setting, since the cost $\ell^{p}$ may take the value $-\infty$, $\ell^{p}$-cyclical monotonicity does not 
directly imply optimality.  Nontheless using the work of Bianchini-Caravenna \cite{biacar:cmono} and its consequences included in Proposition \ref{P:OptiffMon},
we will use cyclically monotone sets to  construct \emph{locally optimal} couplings and to deduce local estimates 
on the disintegration that will be then globalized. Another useful idea is that there is a natural way to construct $\ell_{p}$-geodesics with $0<p<1$: 
translate along transport rays by a constant ``distance''. Notice that $0<p<1$ plays a crucial role here, as an 
analogous statement in the Riemannian setting does not hold true for $W_{2}$.

Let us conclude the introduction by pointing out that the reader interested in space-times with continuous metrics can find the main applications specialised to such a framework  in Section \ref{SS:C0metric}.
\\In analogy with the huge impact that the synthetic theory of Ricci curvature lower bounds had in the geometric analysis of metric measure spaces, it is natural to expect several other geometric and analytic applications of the tools developed here; for instance, in a forthcoming paper \cite{CMTCD2}, we will obtain timelike isoperimetric inequalities and other applications.
\medskip

\textbf{Acknowledgements.}
A.\,M. \;acknowledges support by the European Research Council (ERC), under the European's Union Horizon 2020 research and innovation programme, via the ERC Starting Grant  “CURVATURE”, grant agreement No. 802689. In the first steps of the project he was also supported by  the EPSRC First Grant EP/R004730/1  ``Optimal transport and Geometric Analysis''. \\
The authors are grateful to the anonymous reviewer whose careful reading and comments improved the exposition of the paper.

\section{Preliminaries} 
\subsection{Basics on Lorentzian synthetic spaces} \label{Subsec:BasicsLorentzianSpaces}

In this section we briefly recall some basic notions and results from the theory of Lorentzian length (resp. geodesic) spaces. We follow the approach of Kunzinger-S\"amann \cite{KS} and we refer to their paper for further details and proofs. Let us start by recalling the notion of \emph{causal space}, pioneered by Kronheimer-Penrose \cite{KrPe67}.
\begin{definition}[Causal space  $(X,\ll,\leq)$]
A \emph{causal space}  $(X,\ll,\leq)$ is a set $X$ endowed with a preorder $\leq$ and a transitive relation $\ll$ contained in $\leq$.
\end{definition}

We write $x<y$ when $x\leq y, x\neq y$. We say that $x$ and $y$ are \emph{timelike} (resp. \emph{causally}) related if $x\ll y$  (resp. $x\leq y$).  Let $A\subset X$ be an arbitrary subset of $X$. We define the \emph{chronological} (resp. \emph{causal}) future of $A$ the set
\begin{align*}
I^{+}(A)&:=\{y\in X\,:\, \exists x\in A,\, x\ll y\} \\
J^{+}(A)&:=\{y\in X\,:\, \exists x\in A,\, x\leq y\}
\end{align*}
respectively. Analogously, we define $I^{-}(A)$ (resp. $J^{-}(A)$) the   \emph{chronological} (resp. \emph{causal}) past of $A$. In case $A=\{x\}$ is a singleton, with a slight abuse of notation, we will write $I^{\pm}(x)$ (resp. $J^{\pm}(x)$) instead of  $I^{\pm}(\{x\})$ (resp. $J^{\pm}(\{x\})$).

\begin{definition}[Lorentzian pre-length space $(X,\sfd, \ll, \leq, \tau)$]
A \emph{Lorentzian pre-length space} $(X,\sfd, \ll, \leq, \tau)$ is a  causal space $(X,\ll,\leq)$ additionally  equipped with a proper metric $\sfd$  (i.e. closed and bounded subsets are compact) and a lower semicontinuous function $\tau: X\times X\to [0,\infty]$,  called \emph{time-separation function}, satisfying
\begin{equation}\label{eq:deftau}
\begin{split}
\tau(x,y)+\tau(y,z)\leq \tau (x,z) &\quad\forall x\leq y\leq z \quad \text{reverse triangle inequality} \\
\tau(x,y)=0, \; \text{if } x\not\leq y, & \quad  \tau(x,y)>0 \Leftrightarrow x\ll y.
\end{split}
\end{equation}
\end{definition}
\noindent
Note that the lower semicontinuity of $\tau$ implies that $I^{\pm}(x)$ is open, for any $x\in X$.
\\We endow $X$ with the metric topology induced by $\sfd$. All the topological concepts on $X$ will be formulated in terms of such metric topology.
\\

If $(X,\sfd, \ll, \leq, \tau)$ is a Lorentzian pre-length space, notice that setting $x \tilde{\leq} y$  (resp. $x \tilde{\ll} y$) if and only if $y \leq x$ (resp. $y\ll x$) and $\tilde{\tau}(x,y):=\tau(y,x)$, we obtain a new Lorentzian pre-length space $(X,\sfd, \tilde{\ll},\tilde{\leq}, \tilde{\tau})$. The latter is said to be the \emph{causally reversed} of the former.

Throughout the paper, $I\subset \R$ will denote an arbitrary interval. 

\begin{definition}[Causal/timelike curves]
A non-constant curve $\gamma:I\to X$ is called (future-directed) \emph{timelike} (resp. \emph{causal}) if $\gamma$ is locally Lipschitz continuous (with respect to $\sfd$) and if for all $t_{1}, t_{2}\in I$, with $t_{1}<t_{2}$, it holds $\gamma_{t_{1}}\ll \gamma_{t_{2}}$ (resp. $\gamma_{t_{1}}\leq \gamma_{t_{2}}$). We say that $\gamma$ is a \emph{null} curve if, in addition to being causal, no two points on $\gamma(I)$ are related with respect to $\ll$. 
\end{definition}

It was proved in \cite[Proposition 5.9]{KS} that for strongly causal continuous Lorentzian metrics, this notion of causality coincides with the classical one.

The length of a causal curve is defined via the time separation function, in analogy to the theory of length  metric spaces.  
\begin{definition}[Length of a causal curve]\label{def:Ltau}
For $\gamma:[a,b]\to X$ future-directed causal we set 
\begin{equation*}
{\rm L}_{\tau}(\gamma):=\inf\left\{ \sum_{i=0}^{N-1} \tau(\gamma_{t_{i}}, \gamma_{t_{i+1}})  \,:\, a=t_{0}<t_{1}<\ldots<t_{N}=b, \; N\in \N \right\}.
\end{equation*}
In case the interval is half-open, say $I=[a,b)$, then the infimum is taken over all partitions with $a=t_{0}<t_{1}<\ldots<t_{N}<b$ (and analogously for the other cases).
\end{definition}

It was proved in \cite[Proposition 2.32]{KS} that for smooth strongly causal spacetimes $(M,g)$, this notion of length coincides with the classical one: ${\rm L}_{\tau}(\gamma)={\rm L}_{g}(\gamma)$.
\\A future-directed causal curve $\gamma:[a,b]\to X$ is \emph{maximal} if  it realises the
time separation, i.e. if ${\rm L}_{\tau}(\gamma)=\tau(\gamma_{a}, \gamma_{b})$.

In case the time separation function is continuous with $\tau(x,x)=0$ for every $x\in X$ (as it will be throughout the paper, since we will assume that $X$ is a globally hyperbolic geodesic Lorentzian space), then any timelike maximal $\gamma$ with finite $\tau$-length has a (continuous, monotonically strictly increasing) reparametrisation $\lambda$ by $\tau$-arc-length, i.e.
$\tau( \gamma_{\lambda(s_{1})}, \gamma_{\lambda(s_{2})})=s_{2}-s_{1}$ for all $s_{2}\leq s_{1}$ in the corresponding interval (see \cite[Corollary 3.35]{KS}).

We therefore adopt the following convention. 
\begin{definition}\label{D:geodesic}
A curve $\gamma$ 
will be called \emph{geodesic} if it is maximal and continuous when parametrized by $\tau$-arc-lenght, i.e.
the set of causal geodesics is
\begin{equation}\label{E:geodesic}
 \Geo(X):=\{ \gamma\in C([0,1], X)\,:  \, \tau(\gamma_{s}, \gamma_{t})=(t-s) \tau(\gamma_{0}, \gamma_{1})\, \forall s<t\}.   
\end{equation}
The set of timelike geodesic is 
described as follows: 
\begin{equation}\label{E:timegeo}
\TGeo(X):=\{ \gamma \in \Geo(X):  \, \tau(\gamma_{0}, \gamma_{1})>0\}.
\end{equation}
\end{definition}

Given $x\leq y\in X$ we also set
\begin{align}
\Geo(x,y)&:=\{ \gamma\in \Geo(X)\,:\, \gamma_{0}=x, \, \gamma_{1}=y\}  \label{eq:defGeo(x,y)} \\
\fI(x,y,t)&:=\{\gamma_{t}\,:\, \gamma\in \Geo(x,y)\}  \label{eq:defI(x,y,t)} 
\end{align}
respectively the space of geodesics, and the set of $t$-intermediate points  from $x$ to $y$.
\\ If $x\ll y\in X$ we also set 
$$\TGeo(x,y):=\{ \gamma\in \TGeo(X)\,:\, \gamma_{0}=x, \, \gamma_{1}=y\}. $$
Given two subsets $A,B\subset X$ we call
\begin{equation}\label{eq:defI(A,B,t)}
\fI(A,B,t):=\bigcup_{x\in A, y\in B} \, \fI(x,y,t) 
\end{equation}
the subset of $t$-intermediate points of geodesics from points in $A$ to points in $B$. 
\\

\noindent
A Lorentzian pre-length space $(X,\sfd, \ll, \leq,\tau)$ is called
\begin{itemize}
\item \emph{non-totally imprisoning} if for every compact set $K\Subset X$ there is constant $C>0$ such that the $\sfd$-arc-length of all causal curves contained in $K$ is bounded by $C$;
\item  \emph{globally hyperbolic} if it is non-totally imprisoning and  for every $x,y\in X$ the set (called ``causal diamond'') $J^{+}(x)\cap J^{-}(y)$ is compact in $X$;
\item  \emph{geodesic} if for all $x,y\in X$ with $x \leq y$, $\Geo(x,y) \neq \emptyset $.
\end{itemize}
It was proved in \cite[Theorem 3.28]{KS} that for a globally hyperbolic  Lorentzian geodesic (actually length would suffice) space $(X,\sfd, \ll, \leq,\tau)$, the time-separation function $\tau$ is finite and continuous. 

The next useful result was proved by Minguzzi  (see \cite[Corollary 3.8]{Min23}).

\begin{proposition}\label{prop:GH->KGH}
Let $(X,\sfd, \ll, \leq,\tau)$ be a Lorentzian geodesic space. Then $X$ is globally hyperbolic if and only if
\begin{enumerate} 
\item[(i)]   for every $K_{1}, K_{2}\Subset X$ compact subsets,  the set (called ``causal emerald'')  $J^{+}(K_{1})\cap J^{-}(K_{2})$ is compact in $X$;
\item[(ii)]  the causal relation $\{x\leq y\}\subset X\times X$ is a closed subset (i.e.\;$X$ is causally closed).
\end{enumerate}
\end{proposition}
In the sequel we will use that global hyperbolicity implies (i) and (ii). Even if not used in the present work, the reverse implication is interesting at a conceptual level, as (i) and (ii) do not depend on $\sfd,\tau, \ll$, but only on the causal relation $\leq$ and on the topology induced by $\sfd$. \\

Using Proposition \ref{prop:GH->KGH}(i),  it is readily seen that if $X$ is  globally hyperbolic and $K_{1}, K_{2}\Subset X$ are compact subsets  then 
\begin{equation}\label{I(K1K2)}
\fI(K_{1},K_{2},s)\Subset \bigcup_{t\in [0,1]}  \fI(K_{1},K_{2},t) \Subset X, \quad \forall s\in [0,1].
\end{equation}

In the proof of the singularity theorem, we will use a slight variation of the time separation function associated to a subset $V\subset X$.
Recall that a subset $V\subset X$ is called \emph{achronal} if $x\not \ll y$ for every $x,y\in V$. In particular, if $V$ is achronal, then $I^{+}(V)\cap I^{-}(V)= \emptyset$, so we can define the \emph{signed time-separation} to $V$, $\tau_{V}:X\to [-\infty, +\infty]$, by
\begin{equation}\label{eq:deftauV}
\tau_{V}(x):=
\begin{cases}
\sup_{y\in V} \tau(y,x), &\quad \text{ for }x\in I^{+}(V)\\
-\sup_{y\in V} \tau(x,y),& \quad \text{ for }x\in I^{-}(V) \\
0 &\quad \text{ otherwise}
\end{cases}.
\end{equation}
Note that $\tau_{V}$ is lower semi-continuous on $I^{+}(V)$ (and upper semi-continuous on $I^{-}(V)$), as supremum of continuous functions.
\\In order for these suprema to be attained, global hyperbolicity and geodesic property of $X$ alone are not sufficient. One should rather demand additional compactness properties of the set $V$. The following notion, introduced by Galloway \cite{Ga} in the smooth setting, is well suited to this aim.

\begin{definition}[Future timelike complete (FTC) subsets]\label{def:FTC}
A subset $V\subset X$ is \emph{future timelike complete} (FTC), if for each point  $x\in I^{+}(V)$, the intersection $J^{-}(x)\cap V \subset V$ has compact closure (w.r.t. $\sfd$) in $V$. Analogously, one defines \emph{past timelike completeness} (PTC). A subset that is both   FTC and PTC is called \emph{timelike complete}.
\end{definition}

\noindent
We denote  with $\overline{C}$ the topological  closure (with respect to $\sfd$) of a subset $C\subset X$.

\begin{lemma}\label{L:initialpoint} 
Let $(X,\sfd, \ll, \leq, \tau)$ be a globally hyperbolic Lorentzian geodesic  space and let $V\subset X$ be an achronal FTC (resp. PTC) subset. Then  for each $x\in I^{+}(V)$ (resp. $x\in I^{-}(V)$) there exists a point $y_{x}\in V$ with $\tau_{V}(y_{x})=\tau(y_{x},x)>0$ (resp. $\tau_{V}(y_{x})=-\tau(x,y_{x})<0$).
\end{lemma}

\begin{proof}
Fix a point  $x\in I^{+}(V)$ (for $x\in I^{-}(V)$ the proof is analogous). By the very defnition of $\tau_{V}$ and  \eqref{eq:deftau}, it holds $\tau_{V}(x)>0$  and $\tau(\cdot, x)\equiv 0$  outside of $J^{-}(x)$.  Since by  global hyperbolicity \cite[Theorem 3.28]{KS}  the function $\tau(\cdot, x):X\to \R$ is finite and continuous,  then it admits maximum on the compact set $K:=\overline{J^{-}(x)\cap V}\subset V$ at some point $y_{x}$ . Thus
$$
\tau(y_{x},x)=\max_{y\in K} \tau(y,x)=\sup_{y\in V} \tau(y,x)=\tau_{V}(x)>0.
$$ 
\end{proof}

\begin{remark}\label{R:c1Lip}
Lemma \ref{L:initialpoint}  and reverse triangle inequality \eqref{eq:deftau}implies that
$$
\tau_{V}(x) - \tau_{V}(z) \geq \tau(y_{z},x)-\tau(y_{z},z)  \geq  \tau(z,x), \quad \forall x,z\in I^{+}(V),\, z\leq x.
$$
\end{remark}

 In analogy to the metric setting, it is natural to introduce the next notion of timelike non-branching.

\begin{definition}[Timelike non-branching]\label{def:TNB}
A  Lorentzian pre-length space $(X,\sfd, \ll, \leq, \tau)$ is said to be \emph{forward timelike non-branching}  if and only if for any $\gamma^{1},\gamma^{2} \in \TGeo(X)$, it holds:
\begin{equation}\label{eq:defNonBranch}
\exists \;  \bar t\in (0,1) \text{ such that } \ \forall t \in [0, \bar t\,] \quad  \gamma_{ t}^{1} = \gamma_{t}^{2}   
\quad 
\Longrightarrow 
\quad 
\gamma^{1}_{s} = \gamma^{2}_{s}, \quad \forall s \in [0,1].
\end{equation}
It is said to be \emph{backward timelike non-branching} if the reversed causal structure is forward timelike non-branching. In case it is both forward and backward timelike non-branching it is said \emph{timelike non-branching}.
\end{definition}  

By Cauchy Theorem, it is clear that if $(M,g)$ is a space-time whose Christoffel symbols are locally-Lipschitz (e.g. in case $g\in C^{1,1}$) then the associated synthetic structure is timelike non-branching. 
It is expected that for spacetimes with a metric of lower regularity (e.g. $g\in C^{1}$ or $g\in C^{0}$) timelike branching can occur.  It is also expected that   timelike branching can occur in closed cone structures (see Remark \ref{rem:OtherExamples}) when the Lorentz-Finsler norm is not strictly convex (see \cite[Remark 2.8]{Min}).

\begin{definition}[Measured Lorentzian pre-length space $(X,\sfd,  \mm, \ll, \leq, \tau)$]
A \emph{measured Lorentzian pre-length space} $(X,\sfd, \mm, \ll, \leq, \tau)$ is a Lorentzian pre-length space endowed with a Radon non-negative measure $\mm$ with $\supp \, \mm=X$.
We say that $(X,\sfd, \mm, \ll, \leq, \tau)$ is globally hyperbolic (resp. geodesic) if $(X,\sfd, \ll, \leq, \tau)$ is so.
\end{definition}

Recall that a Radon measure $\mm$ on a  proper metric space $X$ is a Borel-regular measure which is finite on compact subsets. In this framework, it is well known (see for instance \cite[Section 1.6]{KrPa}) that Suslin sets  are $\mm$-measurable. For the sake of this paper it will be enough to recall that   Suslin sets (also called analytic sets) are precisely images via continuous mappings of Borel subsets in complete and separable metric spaces (for more details see \cite{KrPa, Srivastava}).

\begin{remark}[Case of  a spacetime with a continuous Lorentzian metric]\label{rem:C0metrics}
Let $M$ be a smooth manifold, $g$ be a continuous Lorentzian metric over $M$ and assume that $(M,g)$ is time-oriented (i.e. there is a continuous timelike
vector field).  Note that, for $C^0$-metrics, the natural class of differentiability of the manifolds is $C^{1}$; now, $C^{1}$ manifolds always
possess a $C^{\infty}$ subatlas, and one can choose some such sub-atlas whenever convenient.  
 
A causal (respectively timelike) curve in $M$ is by definition a locally Lipschitz
curve whose tangent vector is causal (resp. timelike) almost everywhere. It would also
be possible to start from absolutely continuous (AC for short) curves, but since causal  AC
curves always admit a re-parametrisation that is Lipschitz \cite[Sec. 2.1, Rem. 2.3]{Min}, we do not loose in generality with the above convention.

Denote with  ${\rm L}_{g}(\gamma)$ the $g$-length of a causal curve $\gamma:I\to M$, i.e. ${\rm L}_{g}(\gamma):=\int_{I}\sqrt{-g(\dot \gamma, \dot \gamma)}\, \dd t$.
The time separation function $\tau:M\times M\to [0,\infty]$ is then defined in the usual way, i.e. 
$$\tau(x,y):=\sup\{{\rm L}_{g}(\gamma): \gamma \text{ is future directed causal from $x$ to $y$}\}, \quad  \text{if $x\leq y$},$$
and $\tau(x,y)=0$ otherwise. Note that the reverse triangle inequality \eqref{eq:deftau} follows directly from the definition. It is easy to check that  an  ${\rm L}_{g}$-maximal curve $\gamma$ is also ${\rm L_{\tau}}$-maximal, and  ${\rm L}_{g}(\gamma) = {\rm L}_{\tau}(\gamma)$ (see for instance \cite[Remark 5.1]{KS}).  Also, we fix a complete Riemannian
metric $h$ on $M$ and denote by $\sfd^{h}$ the associated distance function.

For a  spacetime with a Lorentzian $C^{0}$-metric: 
\begin{itemize}
\item Recall that a Cauchy hypersurface is a subset  which is met exactly  once by every inextendible causal curve. It is a well known fact that, even for $C^{0}$-metrics, a Cauchy hypersurface is a closed acausal topological hypersurface \cite[Proposition 5.2]{SaC0}.   Global hyperbolicity is equivalent to the existence of a Cauchy hypersurface \cite[Theorem 5.7, Theorem 5.9]{SaC0} which in turn implies strong causality \cite[Proposition 5.6]{SaC0}.
\item By \cite[Proposition 5.8]{KS},   if $g$ is a causally plain (or, more strongly, locally Lipschitz) Lorentzian $C^{0}$-metric on $M$ then the associated synthetic structure is a pre-length Lorentzian space.
More strongly, from   \cite[Theorem 3.30 and Theorem 5.12]{KS} and combining the above items, if $g$ is a   globally hyperbolic and causally plain Lorentzian $C^{0}$-metric on $M$ then the associated synthetic structure is a  globally hyperbolic Lorentzian geodesic space.

\item Any  Cauchy hypersurface is causally complete. More strongly, if $V\subset M$ is Cauchy hypersurface then for every $x\in J^{+}(V)$ it holds that $J^{-}(x)\cap J^{+}(V)$ is compact (and analogous statement for $x\in J^{-}(V)$). This fact is classical and well known in the smooth setting (see for instance \cite[Lemma 14.40]{O'Neill} or  \cite[Theorem 8.3.12]{Wald}) and extendable to $C^{0}$-metrics along the lines of the proof of \cite[Theorem 5.7]{SaC0}.
\end{itemize}
\end{remark}

\begin{remark}[Other classes of examples]\label{rem:OtherExamples}
\begin{itemize}
\item \textbf{Closed cone structures}. 
Several results from smooth causality theory can be extended to cone structures on smooth
manifolds. One of the motivations for such generalisations comes from  the problem of constructing
smooth time functions in stably causal or globally hyperbolic spacetimes. Fathi and Siconolfi \cite{FaSi} analysed continuous cone structures with tools from weak KAM theory, Bernard and Suhr \cite{BeSu}
studied Lyapunov functions for closed cone structures and showed (among other results) the equivalence between global hyperbolicity and the existence of steep temporal functions in this framework, Minguzzi \cite{Min} gave a deep and comprehensive analysis of causality theory for closed cone structures, including embedding and singularity theorems in this framework. Closed cone structures provide
a rich source of examples of Lorentzian pre-length and length spaces, which can be seen as the synthetic-Lorentzian analogue of Finsler manifolds (see \cite[Section 5.2]{KS} for more details). 

\item \textbf{Outlook on examples, towards  quantum gravity}. The framework of Lorentzian synthetic spaces allows to handle situations
where one may not have the structure of a manifold or a Lorentz(-Finsler) metric.  The optimal transport tools developed in the paper can provide a new perspective on curvature in those cases where there is no classical notion of curvature (Riemann tensor, Ricci
and sectional curvature, etc.).  A remarkable example of such a situation is given by certain approaches to quantum gravity, see for instance \cite{MP} where it is shown that from only a countable dense set of events and the causality relation, it is possible to reconstruct a globally hyperbolic spacetime in a purely order theoretic manner. In particular, two approaches to quantum gravity are linked to Lorentzian synthetic spaces: the one of \emph{causal Fermion systems} \cite{FinsterPrimer, FinsterBook} and the  \emph{theory of causal sets} \cite{BLMS}. The basic idea in both cases is that the structure of spacetime needs to be adjusted
on a microscopic scale to include quantum effects. This leads to non-smoothness of
the underlying geometry, and  the classical  structure of Lorentzian manifold emerges only in the macroscopic regime.   For the connection to the theory of Lorentzian (pre-)length spaces we refer to \cite[Section 5.3]{KS}, \cite[Section 5.1]{FinsterPrimer}. 
Let us mention that the link with causal Fermion systems looks particularly promising: indeed the two cornerstones, used to define synthetic timelike-Ricci curvature lower bounds, are  \emph{Loretzian-distance} and \emph{measure}, and  a causal Fermion system is naturally endowed with both (the reference measure in this setting is called \emph{universal measure}). 
\end{itemize}
\end{remark}

\subsection{Measures and weak/narrow convergence}
In this subsection we briefly recall some basic notions of convergence of  measures that will be used in the paper. Standard references for the topic  are \cite{AGSBook,Vil}. 

Given a complete and separable (in particular, everything hold for proper) metric space $(X,\sfd)$, we denote by
$\BorelSets X$ the collection of its Borel sets and 
by $\mathcal P(X)$ (resp.  $\mathcal{P}_{c}(X)$)  the  collection of all Borel probability
measures (resp. with compact support).
\\We say that $(\mu_{n})\subset \mathcal P(X)$ \emph{narrowly converges} to $\mu_{\infty}\in \mathcal P(X)$ provided 
\begin{equation}\label{eq:defNarrowConv}
\lim_{n\to\infty}\int f\,\mu_n=\int f\,\mui\qquad\forevery f\in \Cb X ,
\end{equation}
where $\Cb X$ denotes the space of bounded and continuous functions.

Relative narrow compactness in $\prob X$ can be characterized by Prokhorov's Theorem.
Let us first recall that a set  $\mathcal K\subset\prob X$ is said to be tight provided for every $\eps>0$ 
there exists a compact set $K_\eps\subset X$ such that
\[
\mu(X\setminus K_\eps)\leq\eps\quad\forevery \mu\in \mathcal K.
\]
The we have the following classical result:
\begin{theorem}[Prokhorov]\label{thm:prok}
Let $(X,\sfd)$ be complete and separable. A subset  $\mathcal
K\subset\prob X$ is tight if and only if it is precompact in the narrow topology.
\end{theorem}
We next recall a useful tightness criterion for measures in $\Prob(X\times X)$ (for the proof see for instance \cite[Lemma 5.2.2]{AGSBook}).  To this aim,  denote with $P_{1}, P_{2}:X\times X\to X$ the projections onto the first and second factor. The push-forward is defined as $(P_{i})_{\sharp} \pi (A):=\pi(P_{i}^{-1}(A))$ for any  $A\in \BorelSets X$.
\begin{lemma}[Tightness criterion in  $\prob {X\times X}$]\label{lem:tightXxX}
A subset  $\mathcal K\subset\prob {X\times X}$ is tight if and only if $(P_{i})_{\sharp} \mathcal K\subset\prob X$ is tight for $i=1,2$.
\end{lemma}
 
We next recall a useful property concerning passage to the limit in \eqref{eq:defNarrowConv} when $f$ is possibly unbounded, but a  ``uniform integrability'' condition holds. 
\begin{definition}[Uniform integrability]
We say that a Borel function $g:X\to [0,+\infty]$ is \emph{uniformly integrable} w.r.t. a given set  $\mathcal K\subset \prob X$ if
\begin{equation}\label{eq:defUnifInt}
\limsup_{k\to \infty} \sup_{\mu\in \mathcal K}  \int_{\{g\geq k\}} g \, \mu = 0.
\end{equation}
\end{definition}

\begin{lemma}[Lemma 5.1.7 \cite{AGSBook}]\label{lem:UnifIntConv}
Let  $(\mu_{n})\subset \mathcal P(X)$ be narrowly convergent to $\mu_{\infty}\in \mathcal P(X)$. 
If $f:X\to [0,\infty)$  is continuous and  uniformly integrable with respect to the set $\{\mu_{n}\}_{n\in \N}$, then 
$$
 \lim_{n\to\infty}\int f\,\mu_n=\int f\,\mui .
$$
Conversely, if $f:X\to [0,\infty)$  is continuous, $f\in L^{1}(\mu_{n})$ for every $n\in \N$ and
\begin{equation}\label{eq:limsupUnifInt}
\limsup_{n\to \infty} \int_{X} f\, \mu_{n} \leq \int_{X} f \, \mui <+\infty,
\end{equation}
then $f$ is uniformly integrable with respect to the set $\{\mu_{n}\}_{n\in \N}$.
\end{lemma}

\subsection{Relative entropy and basic properties}

We denote $\mathcal{P}_{ac}(X)$ the space of probability measures absolutely continuous with respect to $\mm$.

\begin{definition}
Given a probability measure $\mu \in \mathcal{P}(X)$ we define 
its relative entropy by
\begin{equation}
\Ent(\mu|\mm) = \int_{X} \rho \log(\rho) \, \mm,
\end{equation}
if $\mu = \rho \, \mm$ is absolutely continuous with respect to $\mm$ and $(\rho\log(\rho))_{+}$ is  $\mm$-integrable. 
Otherwise we set $\Ent(\mu|\mm) = +\infty$.
\end{definition}
A simple application of Jensen inequality using the convexity of $(0,\infty)\ni t\mapsto t \log t$ gives 
\begin{equation}\label{E:jensenE}
\Ent(\mu|\mm)\geq -\log \mm(\supp \, \mu)>-\infty,\quad \forall \mu\in  \mathcal{P}_{c}(X).
\end{equation}
We set  $\Dom(\Ent(\cdot|\mm)):=\{\mu\in \mathcal{P}(X)\,:\, \Ent(\mu|\mm)\in \R\}$ to be the finiteness domain of the entropy.
An important property of the relative entropy is the (joint) lower-semicontinuity under narrow convergence in case the reference measures are probabilities  (for a proof, see for instance  \cite[Lemma 9.4.3]{AGSBook}): 
\begin{equation}\label{eq:Entljointsc} 
\mm_{n},\mm_{\infty}\in \Prob(X), \; \mm_{n}\to \mm_{\infty}, \; \mu_{n}\to \mui \; \text{narrowly }  \quad \Longrightarrow \quad  \liminf_{n\to \infty} \Ent(\mu_{n}|\mm_{n}) \geq \Ent(\mui|\mm_{\infty}).
\end{equation}
In particular, for a general fixed reference measure $\mm$ it holds:
\begin{equation}\label{eq:Entlsc} 
\mu_{n}\to \mui \; \text{narrowly and } \mm\Big( \bigcup_{n\in \N} \supp\, \mu_{n} \Big) <\infty \quad \Longrightarrow \quad  \liminf_{n\to \infty} \Ent(\mu_{n}|\mm) \geq \Ent(\mui|\mm).
\end{equation}

\section{Optimal transport in Lorentzian synthetic spaces}\label{sec:OTLorentz}

\subsection{The $\ell_{p}$-optimal transport problem}

Given $\mu,\nu\in \mathcal{P}(X)$, denote
\begin{align*}
 \Pi(\mu,\nu)&:=\{\pi\in  \mathcal{P}(X\times X) \,:\, (P_{1})_{\sharp}\pi=\mu, \, (P_{2})_{\sharp}\pi=\nu \}, \nonumber \\
 \Pi_{\leq}(\mu,\nu)&:=\{\pi\in  \Pi(\mu,\nu) \,:\,  \pi(X^{2}_{\leq})=1 \}, 
\nonumber \\
  \Pi_{\ll}(\mu,\nu)&:=\{\pi\in  \Pi(\mu,\nu) \,:\,  \pi(X^{2}_{\ll})=1 \} 
\end{align*}
where $X^{2}_{\leq}:=\{(x,y) \in X^{2}\,:\, x\leq y \}$ and   $X^{2}_{\ll}:=\{(x,y) \in X^{2}\,:\, x\ll y \}$.

\begin{definition}\label{def:Wp}
Let  $(X,\sfd, \ll, \leq, \tau)$ be a Lorentzian pre-length space and let $p\in (0,1]$. Given $\mu,\nu\in \mathcal{P}(X)$, the \emph{$p$-Lorentz-Wasserstein distance} is defined by
\begin{equation}\label{eq:defWp}
\ell_{p}(\mu,\nu):= \sup_{\pi \in \Pi_{\leq}(\mu,\nu)} \left(  \int_{X\times X}  \tau(x,y)^{p} \, \pi(\dd x\dd y)\right)^{1/p}.
\end{equation}
When $\Pi_{\leq}(\mu,\nu)=\emptyset$ we set $\ell_{p}(\mu,\nu):=-\infty$.
\end{definition}
Note that Definition \ref{def:Wp} extends to Lorentzian pre-length spaces the corresponding notion given in the smooth Lorentzian setting in \cite{EM17}  (see also \cite{McCann, MoSu}, and \cite{Suhr} for $p=1$); when $\Pi_{\leq}(\mu,\nu)=\emptyset$ we adopt the convention of McCann  \cite{McCann} (note that \cite{EM17} set $\ell_{p}(\mu,\nu)=0$ in this case).
A coupling  $\pi\in  \Pi_{\leq}(\mu,\nu)$ maximising in \eqref{eq:defWp} is said \emph{$\ell_{p}$-optimal}. The set of \emph{$\ell_{p}$-optimal} couplings from $\mu$ to $\nu$ is denoted by $  \Pi_{\leq}^{p\text{-opt}}(\mu,\nu)$.

\begin{remark} [An equivalent formulation of \eqref{eq:defWp}]\label{rem:eqellq}
Set 
\begin{equation}\label{eq:defell}
\ell^{p}(x,y):=
\begin{cases}
\tau(x,y)^{p} \quad &  \text{if } x\leq y \\
-\infty \quad & \text{otherwise}
\end{cases}.
\end{equation}
Notice that for every $\pi\in \Pi_{\leq}(\mu,\nu)$ it holds  $ \int_{X\times X}  \tau(x,y)^{p} \, \pi(\dd x\dd y)=   \int_{X\times X}  \ell(x,y)^{p} \, \pi(\dd x\dd y)\in [0,+\infty]$.
Moreover, using the convention that $\infty-\infty=-\infty$, we have that if $\pi\in \Pi(\mu,\nu)$ satisfies $$ \int_{X\times X}  \ell(x,y)^{p} \, \pi(\dd x\dd y)>-\infty,$$ then $\pi\in \Pi_{\leq}(\mu,\nu)$.
Thus the maximization problem  \eqref{eq:defWp} is equivalent (i.e. the $\sup$ and the set of maximisers coincide) to the  maximisation problem
\begin{equation}\label{eq:supell}
\sup_{\pi \in \Pi(\mu,\nu)} \left(  \int_{X\times X}  \ell^{p}(x,y) \, \pi(\dd x\dd y)\right)^{1/p}.
\end{equation}
The advantage of the formulation \eqref{eq:supell} is that, when $X$ is globally hyperbolic geodesic (so that $\tau$ is continuous) then  $\ell^{p}$ is upper semi-continuous on $X\times X$. Thus, one can apply to the Monge-Kantorovich problem \eqref{eq:supell} standard optimal transport techniques (e.g. \cite{Vil}). 
\end{remark}

We will adopt the following standard notation: given $\mu,\nu\in \Prob(X)$, we denote with $\mu \otimes \nu\in \Prob(X^{2})$ the product measure; given $u,v:X\to \R\cup \{+\infty\} $ we denote with $u\oplus v: X^{2}\to \R\cup \{+\infty\}$ the function  $u\oplus v(x,y):=u(x)+v(y)$.

\begin{proposition}\label{prop:ExMaxellp}
Let  $(X,\sfd, \ll, \leq, \tau)$ be a  globally hyperbolic Lorentzian geodesic space and let $\mu,\nu\in \mathcal{P}(X)$. 
If  $ \Pi_{\leq}(\mu,\nu)\neq \emptyset$ and if  there exist measurable functions $a,b:X\to \R$, with $a\oplus b \in L^{1}(\mu\otimes \nu)$ such that $\ell^{p}\leq a\oplus b$  on $\supp \, \mu \times \supp \, \nu$ (e.g. when $\mu,\nu\in \Prob_{c}(X)$) then the $\sup$ in \eqref{eq:defWp} is attained and finite. 
\end{proposition}

\begin{proof}
The claim follows from Remark \ref{rem:eqellq} combined with \cite[Theorem 4.1]{Vil} (see also \cite[Theorem 1.3]{Vil:topics}).
\end{proof}

We next show that $\ell_{p}$ satisfies the reverse triangle inequality. 
This was proved in the smooth Lorentzian setting by 
Eckstein-Miller \cite[Theorem 13]{EM17}, and it is  the natural Lorentzian 
analogue of the fact that the Kantorovich-Rubinstein-Wasserstein distances 
$W_{p}$, $p\geq 1$, in the metric space setting satisfy 
the usual triangle inequality (see for instance \cite[Section 6]{Vil}).

We first isolate the causal version of the Gluing Lemma, a classical tool in Optimal Transport theory (see for instance \cite{Vil}).
\begin{lemma}[Gluing Lemma]\label{L:gluing}
Let  $(X,\sfd, \ll, \leq, \tau)$ be a Lorentzian pre-length space and let $\mu_{i}\in\Prob(X)$ 
for $i = 1,2,3$. If $\pi_{12} \in \Pi_{\leq}(\mu_{1},\mu_{2})$ and 
$\pi_{23} \in \Pi_{\leq}(\mu_{2},\mu_{3})$ are given, then there exists 
$\pi_{123} \in \mathcal{P}(X\times X \times X)$ such that 
$$
(P_{12})_{\sharp}\pi_{123}= \pi_{12}, \quad 
(P_{23})_{\sharp}\pi_{123}= \pi_{23}, \quad 
(P_{13})_{\sharp}\pi_{123} \in \Pi_{\leq}(\mu_{1},\mu_{3}).
$$
\end{lemma}

\begin{proof}
The proof goes along the same lines of the classical Gluing Lemma (see for instance \cite[Lemma 7.6]{Vil:topics}). 
Disintegrate the coupling $\pi_{12}$ with respect to $P_{2}$ and 
the coupling $\pi_{23}$ with respect to $P_{1}$ and obtain the following formula 
$$
\pi_{12} = \int_{X} (\pi_{12})_{x} \, \mu_{2}(\dd x), \quad 
\pi_{23} = \int_{X} (\pi_{23})_{x} \, \mu_{2}(\dd x), \qquad 
(\pi_{12})_{x}, (\pi_{23})_{x} \in\mathcal{P}(X\times X), 
$$
with $(\pi_{12})_{x}(X \times \{ x\}) = 
(\pi_{23})_{x}(\{ x\}\times X) = 1$, $\mu_{2}$-a.e. .  
Since $\pi_{12}$ and $\pi_{23}$ are causal couplings,  we have
$$
(\pi_{12})_{x}(X_{\leq}^{2}) = (\pi_{23})_{x}(X_{\leq}^{2}) = 1, \quad \text{for $\mu_{2}$-a.e. $x \in X$}.
$$
In particular, for $(\pi_{12})_{x}$-a.e. $(z,x)$ and 
for $(\pi_{23})_{x}$-a.e. $(x,y)$,   the transitive property of 
$\leq$ gives that $z \leq y$.
Hence defining 
$$
\pi_{123} = \int_{X} (P_{14})_{\sharp}((\pi_{12})_{x}\otimes  (\pi_{23})_{x})
\,\mu_{2}(\dd x),
$$
the first two claims are obtained by the classical Gluing Lemma  \cite[Lemma 7.6]{Vil:topics} (or \cite[Chapter 1]{Vil}),  
while the last one follows from the previous argument.
\end{proof}

\begin{proposition}[$\ell_{p}$ satisfies the reverse triangle inequality]\label{prop:RTIellp}
Let  $(X,\sfd, \ll, \leq, \tau)$ be a Lorentzian pre-length space and let $p\in (0,1]$. Then $\ell_{p}$ satisfies the reverse triangle inequality:
\begin{equation}\label{eq:RTIellq}
\ell_{p}(\mu_{0},\mu_{1})+ \ell_{p}(\mu_{1},\mu_{2})
\leq \ell_{p}(\mu_{0}, \mu_{2}), 
\quad \forall \mu_{0},\mu_{1},\mu_{2}\in \mathcal{P}(X),
\end{equation}
where we adopt the convention that $\infty-\infty=-\infty$ to interpret the left hand side of \eqref{eq:RTIellq}.
\end{proposition}

\begin{proof}
We assume $\ell_{p}(\mu_{0}, \mu_{1}), \ell_{p}(\mu_{1}, \mu_{2})>-\infty$, 
otherwise the claim is trivial. 
\\We first consider the case when $\ell_{p}(\mu_{0}, \mu_{1}), \ell_{p}(\mu_{1}, \mu_{2})<\infty$. 
By the very definition \eqref{eq:defWp}  of $\ell_{p}$, 
for any $\ve>0$ we can find $\pi_{01} \in \Pi_{\leq}(\mu_{0},\mu_{1})$
$\pi_{12} \in \Pi_{\leq}(\mu_{1},\mu_{2})$ such that 
$$
\ell_{p}(\mu_{0}, \mu_{1}) 
\leq \left(\int_{X\times X} \tau(x,y)^{p}\,\pi_{01}(\dd x\dd y)\right)^{1/p} + \ve, \quad
\ell_{p}(\mu_{1}, \mu_{2}) 
\leq \left(\int_{X\times X} \tau(x,y)^{p}\,\pi_{12}(\dd x\dd y)\right)^{1/p} + \ve. 
$$
We denote with $\pi_{012} \in \mathcal{P}(X^{3})$ the measure given by 
the Gluing Lemma \ref{L:gluing}. Recalling that for $\pi_{012}$-a.e. 
$(x,z,y) \in X^{3}$ it holds $x \leq z \leq y$, we can use \eqref{eq:deftau} to compute
\begin{align*}
\ell_{p}(\mu_{0}, \mu_{2}) 
&~ \geq 
\left(\int_{X\times X} \tau(x,y)^{p} \,(P_{13})_{\sharp}\pi_{012}(\dd x\dd y)\right)^{1/p} \\
&~ = \left(\int_{X\times X\times X} \tau(x,y)^{p} \,\pi_{012}(\dd x\dd z\dd y) 
\right)^{1/p}\\
&~ \geq \left(\int_{X\times X\times X} [\tau(x,z) + \tau(z,y)]^{p} \,\pi_{012}(\dd x\dd z\dd y) 
\right)^{1/p} \\
&~ \geq \left(\int_{X\times X} \tau(x,z)^{p} \,\pi_{012}(\dd x\dd z\dd y) 
\right)^{1/p} + 
\left(\int_{X\times X} \tau(z,y)^{p} \,\pi_{012}(\dd x\dd z\dd y) 
\right)^{1/p} \\
&~ = \left(\int_{X\times X} \tau(x,z)^{p} \,\pi_{01}(\dd x\dd z) 
\right)^{1/p} + 
\left(\int_{X\times X} \tau(z,y)^{p} \,\pi_{12}(\dd z\dd y) 
\right)^{1/p} \\
&~ \geq  \ell_{p}(\mu_{0}, \mu_{1}) + \ell_{p}(\mu_{1}, \mu_{2}) - 2\ve,
\end{align*}
proving the inequality, by the arbitrariness of $\ve>0$. If  
one of $\ell_{p}(\mu_{0}, \mu_{1})$, $\ell_{p}(\mu_{1}, \mu_{2})$ 
is not bounded from above, then simply take a sequence of couplings 
with diverging cost; repeating the above calculations we obtain that also 
$\ell_{p}(\mu_{0}, \mu_{2}) = \infty$, proving the claim.
\end{proof}

\smallskip
\subsection{Cyclical monotonicity}
\label{Ss:cyclical}

The notion of cyclical monotonicity is very useful to relate an optimal coupling with its support.

\begin{definition}[$\tau^{p}$-cyclical monotonicity and $\ell^{p}$-cyclical monotonicity]\label{D:monotonicity}
Fix $p\in (0,1]$ and let  $(X,\sfd, \ll, \leq, \tau)$ be a Lorentzian pre-length space.
A subset $\Gamma\subset X^{2}_{\leq}$ is said to be $\tau^{p}$-cyclically monotone (resp. $\ell^{p}$-cyclically monotone) if, for any $N\in \N$ and any family $(x_{1}, y_{1}), \ldots, (x_{N}, y_{N})$ of points in $\Gamma$, the next inequality holds:
\begin{equation}\label{eq:taupcyclmon}
\sum_{i=1}^{N}\tau(x_{i}, y_{i})^{p} \geq \sum_{i=1}^{N}\tau(x_{i+1}, y_{i})^{p},
\end{equation}
(resp. $ \sum_{i=1}^{N}\ell(x_{i}, y_{i})^{p} \geq \sum_{i=1}^{N}\ell(x_{i+1}, y_{i})^{p})$ with the convention $x_{N+1}=x_{1}$. A coupling  is said to be $\tau^{p}$-cyclically monotone (resp. $\ell^{p}$-cyclically monotone) if it is concentrated on an $\tau^{p}$-cyclically monotone set (resp. $\ell^{p}$-cyclically monotone set).
\end{definition}

\begin{remark}\label{rem:tauellpCM}
Notice that $\Gamma\subset X^{2}_{\leq}$ is $\ell^{p}$-cyclically monotone if and only if \eqref{eq:taupcyclmon} holds for those families with $x_{i+1}\leq y_{i}$ for all $i\in \{1,\dots,N\}$. It is then clear that 
\begin{equation}\label{eq:ImplicCyclMon}
\text{$\tau^{p}$-cyclical monotonicity }\Rightarrow \text{ $\ell^{p}$-cyclical monotonicity.}
\end{equation}
Note if $P_{1}(\Gamma)\times P_{2}(\Gamma) \subset X^{2}_{\leq}$ then $\ell^{p}$-cyclical monotonicity is equivalent to $\tau^{p}$-cyclical monotonicity.
\end{remark}

\begin{proposition}[Optimality $\Leftrightarrow$ cyclical monotonicity]\label{P:OptiffMon}
Fix $p\in (0,1]$. Let  $(X,\sfd, \ll, \leq, \tau)$ be a  Lorentzian pre-length  space and let $\mu,\nu\in \mathcal{P}(X)$.  Assume that  $ \Pi_{\leq}(\mu,\nu)\neq \emptyset$ and that  there exist measurable functions $a,b:X\to \R$, with $a\oplus b \in L^{1}(\mu\otimes \nu)$ such 
that $\ell^{p}\leq a\oplus b$, $\mu\otimes \nu$-a.e.. Then the following holds.
\begin{enumerate}
\item If $\pi$ is  $\ell_{p}$-optimal then $\pi$ is  $\ell^{p}$-cyclically monotone. 
\item If $\pi(X^{2}_{\ll}) = 1$ and $\pi$ is $\ell^{p}$-cyclically monotone then $\pi$ is $\ell_{p}$-optimal. 
\end{enumerate}

\end{proposition}

\begin{proof}
The result follows from \cite{biacar:cmono}, dealing with optimal transport (minimisation)  problems associated to general  Borel cost functions $c(\cdot, \cdot):X^{2}\to [0,+\infty]$. 
Of course, the (maximising) optimal couplings in $\Pi(\mu,\nu)$ for the cost $\ell^{p}$ are the same as for the cost $\ell^{p}-(a\oplus b)$, which is non-positive  $\mu\otimes \nu$-a.e.; hence we enter in the framework of  \cite{biacar:cmono}. 
\\The first claim thus follows from \cite[Lemma 5.2]{biacar:cmono} (see also  \cite[Proposition B.16]{biacar:cmono}).
\\For the second claim, notice that \cite[Theorem 5.6]{biacar:cmono} provides a general 
condition to ensure that an $\ell^{p}$-cyclically monotone coupling is $\ell_{p}$-optimal. Thanks to \cite[Corollary 5.7, Proposition 5.8]{biacar:cmono} it will be enough to 
verify the existence of countably many Borel sets $A_{i},B_{i} \subset X$ 
such that 
$$
\pi\Big(\bigcup_{i\in \N} A_{i}\times B_{i} \Big) = 1, \qquad  
\bigcup_{i\in \N} A_{i}\times B_{i} \subset X^{2}_{\leq}.
$$
The existence of such sets (that can actually chosen to be open) follows directly from the fact that $X^{2}_{\ll}=\{\tau>0\}\subset X^{2}$ is open by the lower semicontinuity of $\tau$.
\end{proof}

\begin{remark}
Thanks to \cite[Proposition 5.8]{KS},  Proposition \ref{P:OptiffMon} is valid  for  a causally plain (so, in particular, for a locally-Lipschitz) Lorentzian $C^{0}$-metric  $g$ on a space-time $M$.
\\In case  $(X,\sfd, \ll, \leq, \tau)$ is a  globally hyperbolic Lorentzian geodesic space (as it will be for most of the paper), the first claim in Proposition \ref{P:OptiffMon} follows from more standard literature (see e.g. \cite[Theorem 3.2]{AmbrosioPratelliL1}), thanks to Remark \ref{rem:eqellq}.
\end{remark}

\noindent
We will later see that for $\tau^{p}$-cyclically monotone causal couplings, $\ell_{p}$-optimality holds true (Theorem \ref{thm:OptiffMon}). To conclude we report a standard fact about optimal couplings.  
\begin{lemma}[Restriction]\label{L:restriction}
Fix $p\in (0,1]$. Let  $(X,\sfd, \ll, \leq, \tau)$ be a  Lorentzian pre-length  space and  let $\mu,\nu\in \mathcal{P}(X)$.
Then for every  $\pi \in \Pi_{\leq}^{p\text{-opt}}(\mu,\nu)$ and every measurable function 
$f : X \times X \to [0,\infty)$ with $\int f\, \pi =1$ and $f \in L^{\infty}(\pi)$,  also the coupling $f \pi$ is optimal, i.e. denoting with 
$$
\mu_{f} : = (P_{1})_{\sharp} f \pi , \qquad \nu_{f} : = (P_{2})_{\sharp} f\pi,
$$
it holds true $f \pi \in \Pi_{\leq}^{p\text{-opt}}(\mu_{f},\nu_{f})$.
\end{lemma}

\begin{proof}
Trivially $f \pi \in \Pi_{\leq}(\mu_{f},\nu_{f})$ hence we will only be concerned about optimality.
Assume by contradiction the existence 
of $\hat \pi \in \Pi_{\leq}(\mu_{f},\nu_{f})$ with 
$$
\int_{X\times X} \tau(x,y)^{p} f(x,y) \pi(\dd x\dd y) < \int_{X\times X} \tau(x,y)^{p} \hat \pi(\dd x\dd y). 
$$
Consider then the new coupling 
$$
\bar \pi : = \pi - \frac{f}{\| f\|_{\infty}} \pi + \frac{1}{\| f\|_{\infty}} \hat \pi.
$$
By linearity, $\bar \pi$ has the same marginals of $\pi$ and it is causal, 
i.e. $\bar \pi \in \Pi_{\leq}(\mu,\nu)$.
Finally 
\begin{align*}
\int_{X\times X} \tau(x,y)^{p} \bar \pi (\dd x\dd y)
&~ = \int_{X\times X} \tau(x,y)^{p} \pi (\dd x\dd y) 
+ \frac{1}{\| f\|_{\infty}} \int_{X\times X} \tau(x,y)^{p}(\hat \pi- f\pi )(\dd x\dd y)\\
&~ >  \int_{X\times X} \tau(x,y)^{p} \pi (\dd x\dd y),
\end{align*}
giving a contradiction.
\end{proof}

\subsection{Stability of optimal couplings}
\label{Ss:stabilityoptimal}

While in the Riemannian framework stability 
of optimal couplings follows by stability of cyclical monotonicity, 
in the Lorentzian setting, due to the upper semicontinuity of the cost function $\ell^{p}$ (opposed to continuity of the Riemannian cost $\sfd^{p}$), a more 
refined analysis is needed.

Building on the previous Proposition \ref{P:OptiffMon}, 
we can establish a first basic stability property with respect to narrow convergence 
valid for a special class of optimal couplings. 

\begin{lemma}[Stability of $\ell_p$-optimal couplings I]\label{lem:ConvEllpNarrow1}
Let  $(X,\sfd, \ll, \leq, \tau)$ be a globally hyperbolic Lorentzian geodesic space and fix $p\in (0,1]$.  Let
$(\mu^{1}_{n}), (\mu^{2}_{n}) \subset \prob X$  be narrowly convergent to some  $\mui^{1}, \mui^{2}\in \prob X$ and assume that,
for every $n\in \N$, there exists an  $\ell_{p}$-optimal coupling  $\pi_{n}\in  \Pi_{\leq}^{p\text{-opt}}(\mu^{1}_{n}, \mu^{2}_{n})$ which is also $\tau^{p}$-cyclically monotone. 

Then $(\pi_{n})$ is narrowly relatively compact in $\prob{X^{2}}$, moreover any narrow limit point $\pi_{\infty}$ belongs to 
$\Pi_{\leq}(\mui^{1}, \mui^{2})$ and is $\ell_{p}$-optimal, 
provided $\pi_{\infty}(X^{2}_{\ll}) = 1$.
\end{lemma}

\begin{proof}
By Prokhorov Theorem \ref{thm:prok}, the subsets $\{\mu^{1}_{n}\}_{n\in \N},\,\{\mu^{2}_{n}\}_{n\in \N} \subset \prob X$ are tight.  Lemma \ref{lem:tightXxX} implies that $\{\pi_{n}\}_{n\in \N}\subset  \prob {X\times X}$ is tight as well, and then (again by Theorem \ref{thm:prok}) it converges narrowly, up to a subsequence, to some $\pi_{\infty}\in \prob{X\times X}$. Using the continuity of the projection maps, it is readily seen that $\pi_{\infty}\in \Pi(\mui^{1}, \mui^{2})$.  Global hyperbolicity together with Proposition \ref{prop:GH->KGH}(ii) further implies that   $\pi_{\infty}\in \Pi_{\leq}(\mui^{1}, \mui^{2})$. 
To conclude optimality it is enough to observe that $\tau^{p}$ is continuous and 
therefore $\tau^{p}$-cyclical monotonicity is preserved under narrow convergence
(note that the same claim would be false for $\ell^{p}$-cyclically monotone
sets) and apply the second point of Proposition \ref{P:OptiffMon} together with \eqref{eq:ImplicCyclMon} (see also Theorem \ref{thm:OptiffMon} below, for a more self-contained proof of the implication $\pi$ is $\tau^{p}$-cyclically monotone $\Rightarrow$ $\pi$ is $\ell_{p}$-optimal).
\end{proof}

To obtain stronger stability properties,  
we will use $\Gamma$-convergence techniques.
For the rest of this section,
 $(X,\sfd, \ll, \leq, \tau)$ will be a  globally hyperbolic Lorentzian geodesic space and we also fix $p\in (0,1]$.  
 Let
$(\mu^{1}_{n}), (\mu^{2}_{n}) \subset \prob X$  be narrowly convergent to some  $\mui^{1}, \mui^{2}\in \prob X$.
Associated with them we define 
$$
F_{n}, F_{\infty} : \mathcal{P}(X^2) \to \R \cup \{ \pm \infty \}, 
\quad
F_{i} (\pi) = 
\begin{cases}
\int_{X\times X} \tau(x,y)^{p} \, \pi(\dd x\dd y), & \pi \in 
\Pi_{\leq}(\mu^{1}_{i}, \mu^{2}_{i}) \\
-\infty, & {\text {otherwise}},
\end{cases}
$$
for $i = n,\infty$.

\begin{lemma} ($\limsup$-inequality)\label{L:limsup}
Let $\{\pi_{i}\}_{i\in \N\cup \{\infty\}} \subset  \mathcal{P}(X^{2})$ be such that $\pi_{n} \to\pi_{\infty}$ narrowly and $\tau^{p}$ is uniformly integrable with respect to $\{\pi_{i}\}_{i\in \N\cup \{\infty\}}$ (in particular, the second condition is satisfied if there exists a compact subset containing $\supp\, \pi_{n}$ for all $n\in \N$).
Then
\begin{equation}\label{eq:FntoFinfty}
F_{\infty}(\pi_{\infty}) \geq \limsup_{n\to \infty} F_{n}(\pi_{n}).
\end{equation}
If moreover, $\pi_{n}(X^{2}_{\leq})=1$ for all $n\in \N$, then also $\pi_{\infty}(X^{2}_{\leq})=1$ and 
$$
F_{\infty}(\pi_{\infty}) = \lim_{n\to \infty} F_{n}(\pi_{n}).
$$
\end{lemma}

\begin{proof}
Without loss of generality we can assume that  $\pi_{n}(X^{2}_{\leq})=1$ definitively, 
otherwise the claim \eqref{eq:FntoFinfty} is trivial.
Since by assumption $X^{2}_{\leq}\subset X^{2}$ is closed, it follows that 
$$
\pi_{\infty}(X^{2}_{\leq})\geq \limsup_{n\to \infty} \pi_{n}(X^{2}_{\leq})=1.
$$ 
Using that (from global hyperbolicity) $\tau^{p}$ is continuous on $X^{2}_{\leq}$ together with Lemma \ref{lem:UnifIntConv}, we conclude that $F_{n}(\pi_{n}) \to F_{\infty}(\pi_{\infty})$.
\end{proof}

For the liminf inequality we have to select a particular family of couplings of
$(\mu_{n}^{1}),(\mu_{n}^{2})$.

\begin{lemma}[Existence of a recovery sequence]\label{lem:recoveryseq}
Assume that there exists a compact subset $\cK\subset X$ such that $\supp\, \mu^{1}_{n}, \, \supp\, \mu^{2}_{n}\subset \cK$ for all $n\in \N$.
Assume that, for each $n \in \N$, the sets 
$\Pi_{\leq}(\mu^{1}_{n},\mu^{1}_{\infty})$ and 
$\Pi_{\leq}(\mu^{2}_{\infty},\mu^{2}_{n})$ are both not empty. 
Then, for any $\pi \in  \Pi(\mu^{1}_{\infty}, \mu^{2}_{\infty})$, there exists a sequence 
$\pi_{n} \in \Pi(\mu^{1}_{n}, \mu^{2}_{n})$ such that $F_{\infty}(\pi) \leq \liminf_{n\to \infty}F_{n}(\pi_{n})$.
\end{lemma}

\begin{proof}
Fix any $ \pi \in \mathcal{P}(X^{2})$. If 
$\pi \notin \Pi_{\leq}(\mu^{1}_{\infty}, \mu^{2}_{\infty})$ then the claim is trivial
(just take as a recovery sequence $\pi$ itself). 
Assume then $\pi \in \Pi_{\leq}(\mu^{1}_{\infty}, \mu^{2}_{\infty})$.
By assumption 
there exists $\pi_{n}^{1} \in \Pi_{\leq}(\mu^{1}_{n},\mu^{1}_{\infty})$ and 
$\pi_{n}^{2} \in \Pi_{\leq}(\mu^{2}_{\infty},\mu^{2}_{n})$. 
Then, by Gluing Lemma \ref{L:gluing} and transitivity of $\leq$, we obtain a 
$\hat \pi_{n} \in \mathcal{P}(X\times X \times X \times X)$ such that 
$$
(P_{12})_{\sharp} \hat \pi_{n} = \pi^{1}_{n}, \quad   
(P_{23})_{\sharp} \hat \pi_{n} = \pi, \quad
(P_{34})_{\sharp} \hat \pi_{n} = \pi^{2}_{n}, \quad 
(P_{14})_{\sharp} \hat \pi_{n} \in \Pi_{\leq} (\mu_{n}^{1},\mu_{n}^{2}).
$$
Recalling that $\tau$ is non-negative and satisfies reverse triangle inequality, we get: 
\begin{align*}
F((P_{14})_{\sharp} \hat \pi_{n})  
= &~\int_{X\times X} \tau(x,y)^{p} \, (P_{14})_{\sharp}  \hat \pi_{n} (\dd x\dd y) \\
 = &~ \int_{X\times X\times X \times X} \tau(P_{14}(x,z,w,y))^{p} \, \hat \pi_{n} (\dd x\dd y\dd z\dd w)  \\
 \geq &~
\int_{X\times X\times X \times X} (\tau(x,z) + \tau(z,w) + \tau(w,y))^{p} \, \hat \pi_{n} (\dd x\dd y\dd z\dd w) \\
 \geq &
\int_{X\times X} \tau(z,w)^{p} \pi(\dd z\dd w)  \\
\geq &~ F(\pi).
\end{align*}
Thus the sequence $\pi_{n}:=(P_{14})_{\sharp} \hat \pi_{n}$ satisfies the claim.
\end{proof}

\begin{theorem}[Stability of $\ell_{p}$-optimal couplings II]
Assume that there exists a compact subset $\cK\subset X$ such that $\supp\, \mu^{1}_{n}, \, \supp\, \mu^{2}_{n}\subset \cK$ for all $n\in \N$, and that  for each $n \in \N$ the sets 
$\Pi_{\leq}(\mu^{1}_{n},\mu^{1}_{\infty})$ and 
$\Pi_{\leq}(\mu^{2}_{\infty},\mu^{2}_{n})$ are both not empty. 

Then
$\ell_{p}(\mu^{1}_{n}, \mu^{2}_{n})$ converges to 
$\ell_{p}(\mu^{1}_{\infty}, \mu^{2}_{\infty})$ and any narrow-limit point of
$\Pi_{\leq}^{p\text{-opt}}(\mu^{1}_{n}, \mu^{2}_{n})$ belongs to 
$\Pi_{\leq}^{p\text{-opt}}(\mu^{1}_{\infty}, \mu^{2}_{\infty})$.
\end{theorem}

\begin{proof}
For the first claim, notice that from Lemma \ref{L:limsup} and the equintegrability of $\tau^{p}$ granted by the assumptions,
it readily follows that $\limsup_{n\to \infty}  \ell_{p}(\mu^{1}_{n}, \mu^{2}_{n}) \leq \ell_{p}(\mu^{1}_{\infty}, \mu^{2}_{\infty})$. 
Also, Lemma \ref{lem:recoveryseq} gives $\ell_{p}(\mu^{1}_{\infty}, \mu^{2}_{\infty}) \leq \liminf_{n\to \infty}  \ell_{p}(\mu^{1}_{n}, \mu^{2}_{n})$. Hence, $\ell_{p}(\mu^{1}_{n}, \mu^{2}_{n})\to \ell_{p}(\mu^{1}_{\infty}, \mu^{2}_{\infty})$.

For the second claim, if $\pi_{n}\in \Pi_{\leq}^{p\text{-opt}}(\mu^{1}_{n}, \mu^{2}_{n})$ converges narrowly to $\pi$ then, by the continuity of the projections and the causal closedness of $X$, we have that $\pi\in \Pi_{\leq}(\mu^{1}_{\infty}, \mu^{2}_{\infty})$ and 
$$\int \tau(x,y)^{p} \, \pi(\dd x\dd y)= \lim_{n\to \infty} \int \tau(x,y)^{p} \, \pi_{n}(\dd x\dd y)= \lim_{n\to \infty} \ell_{p}(\mu^{1}_{n}, \mu^{2}_{n})^{p}= \ell_{p}(\mu^{1}_{\infty}, \mu^{2}_{\infty})^{p},$$
where the first identity follows  from Lemma \ref{L:limsup} and last by the previous part of the proof.  We conclude that $\pi\in \Pi_{\leq}^{p\text{-opt}}(\mu^{1}_{\infty}, \mu^{2}_{\infty})$.
\end{proof}

Another simple criterion, based on ideas from $\Gamma$-convergence, to ensure stability of $\ell_{p}$-optimal couplings is the following one.

\begin{lemma}\label{L:use}
Let  $(X,\sfd, \ll, \leq, \tau)$ be a  Lorentzian globally hyperbolic geodesic space and fix $p\in (0,1]$.  Let
$(\mu^{1}_{n}), (\mu^{2}_{n}) \subset \prob X$  be narrowly convergent to some  
$\mui^{1}, \mui^{2}\in \prob X$.
Assume moreover the existence of 
an optimal $\bar \pi_{\infty} \in \Pi_{\leq}^{p\text{-opt}}(\mui^{1}, \mui^{2})$
and of a sequence $\bar \pi_{n} \in \Pi_{\leq}(\mu_{n}^{1}, \mu_{n}^{2})$ 
such that $F_{n}(\bar \pi_{n}) \to F_{\infty}(\bar \pi_{\infty})$.

Then for any  $\tau^{p}$-uniformly integrable sequence $\pi_{n} \in \Pi_{\leq}^{p{\text-opt}}(\mu_{n}^{1}, \mu_{n}^{2})$, any limit measure $\pi_{\infty}$ in the narrow topology is $\ell_p$-optimal, 
i.e. $\pi_{\infty} \in \Pi_{\leq}^{p\text{-opt}}(\mui^{1}, \mui^{2})$.
\end{lemma}

\begin{proof}
Consider any limit point $\pi_{\infty}$ of a $\tau^{p}$-uniformly integrable sequence  
$\pi_{n} \in \Pi_{\leq}^{p\text{-opt}}(\mu_{n}^{1}, \mu_{n}^{2})$ and 
$\bar \pi_{\infty} \in \Pi_{\leq}^{p\text{-opt}}(\mui^{1}, \mui^{2})$ limit point of 
$\bar \pi_{n} \in \Pi_{\leq}(\mu_{n}^{1}, \mu_{n}^{2})$. From Lemma \ref{L:limsup} we have  that
\begin{equation}\label{E:Gammaconv}
F_{\infty}(\pi_{\infty}) \geq \limsup_{n \to \infty} F_{n}(\pi_{n}) 
\geq 
\limsup_{n \to \infty} F_{n}(\bar \pi_{n})  
= F_{\infty}(\bar \pi_{\infty}) \geq F_{\infty}(\pi_{\infty})
\end{equation}
giving optimality of $\pi_{\infty}$.
\end{proof}

From Lemma \ref{L:use} we obtain another stability result. 
For this scope we introduce the following notation: 
$$
D : = \Big\{ \nu \in \mathcal{P}_{c}(X) \colon \nu = \sum_{i\leq k} \alpha_{i}\delta_{x_{i}}, \text{ for some }k \in \N \Big\}.
$$

\begin{theorem}[Stability of $\ell_{p}$-optimal couplings III]\label{T:stability}
Let  $(X,\sfd, \ll, \leq, \tau)$ be a globally hyperbolic Lorentzian geodesic space and fix $p\in (0,1]$.  Let also
$\mu_{0}, \mu_{1}\in \mathcal{P}_{c}(X)$ be given 
and assume the existence of 
$\pi \in \Pi_{\leq}^{p\text{-opt}}(\mu_{0}, \mu_{1})$ with 
$\supp\,\pi \subset X^{2}_{\ll}$.
 Let $(\mu_{1,n}) \subset D$  with $\supp \, \mu_{1,n}\subset \supp \, \mu_{1}$ be a sequence  narrowly convergent to $\mu_{1}$ 
with  $\ell_{p}(\mu_{0},\mu_{1,n})\in [0,\infty)$. Then there exists another sequence 
$(\bar \mu_{1,n} ) \subset D$ such that the following holds true.
\smallskip
The sequence $(\bar \mu_{1,n})$ still narrowly converges to $\mu_{1}$ and 
$\bar \mu_{1,n}$ is absolutely continuous with respect to $\mu_{1,n}$.
Moreover, for any  sequence $\pi_{n}\in \Pi_{\leq}^{p\text{-opt}}(\mu_{0},\bar \mu_{1,n})$, 
any limit measure $\pi_{\infty}$ in the narrow topology is $\ell_{p}$-optimal, 
i.e. $\pi_{\infty} \in \Pi_{\leq}^{p\text{-opt}}(\mu_{0}, \mu_{1})$, and 
$$
\ell_{p}(\mu_{0},\bar \mu_{1,n}) \to \ell_{p}(\mu_{0}, \mu_{1}).
$$
\end{theorem}

\begin{proof}
{\bf Step 1.} Restricting $\mu_{0}$.\\
Since $\supp\, \pi$ is compact and $X_{\ll}^{2}$ is an open set, for any $\ve>0$ we find finitely many points $(x_{i,\ve},y_{i,\ve})$ with $i = 1,\dots,k_{\ve}$
such that $\supp\, \pi \subset \cup_{i\leq k_{\ve}} B_{\ve}(x_{i,\ve})\times B_{\ve}(y_{i,\ve}) 
\subset X_{\ll}^{2}$.
In particular, for any $\ve>0$, $\mu_{1}(\cup_{i\leq k_{\ve}}B_{\ve}(y_{i,\ve}) )= 1$. 
Then narrow convergence implies that 
\begin{equation}\label{eq:mu1inAeps1}
\liminf_{n\to \infty} \mu_{1,n}(A_{\ve}) \geq \mu_{1}(A_{\ve}) = 1, \qquad A_{\ve} := \cup_{i\leq k_{\ve}}B_{\ve}(y_{i,\ve}).
\end{equation}
Since we are interested in obtaining a sequence $\{ \bar \mu_{1,n}\}$ absolutely continuous with respect to 
$\mu_{1,n}$, we can  restrict and normalize $\mu_{1,n}$ to $A_{\ve}$ obtaining (thanks to 
\eqref{eq:mu1inAeps1}) a new sequence still converging narrowly to $\mu_{1}$.
Hence, without loss of generality, we will assume $\mu_{1,n}(A_{\ve}) = 1$ for every $n\in \N$.

\smallskip
{\bf Step 2.} Construction of the approximations.\\
 By assumption  $\mu_{1,n} = \sum_{i \leq h_{n}} \alpha_{i,n}\delta_{y_{i,n}}$, 
with $\sum_{i\leq h_{n}}\alpha_{i,n} = 1$,  $\alpha_{i,n}\geq 0$ and
from {\bf Step 1} we have $\mu_{1,n}(A_{\ve}) = 1$. 
Let $\{B_{i,\ve}\}_{i = 1}^{k_\ve}$ be a pairwise disjoint  covering  of $\supp\, \pi$, where each $B_{i,\ve}$ 
is a Borel subset of $B_{\ve}(x_{i,\ve})\times B_{\ve}(y_{i,\ve}) \subset X_{\ll}^{2}$.
We define the following approximations: 
\begin{equation}\label{E:definitionpi1en}
\pi_{i,\ve} : = \pi\llcorner_{B_{i,\ve}},
\quad \pi_{i,\ve,n} : = (P_{1})_{\sharp}\pi_{i,\ve} 
\otimes \left(\sum_{y_{i,n} \in P_{2}(B_{i,\ve})}\alpha_{i,n}\delta_{y_{i,n}}\right)\Big/\left(\sum_{y_{i,n} \in P_{2}(B_{i,\ve})}\alpha_{i,n}\right),
\end{equation}
and set $\pi_{\ve, n} : = \sum_{i\leq k_{\ve}}\pi_{i,\ve,n}$. 
 Observe that:
\begin{align}
(P_{2})_{\sharp}\pi_{\ve, n}
&~ = \sum_{i\leq k_{\ve}}(P_{2})_{\sharp} \pi_{i,\ve,n}   \nonumber  \\
&~ = \sum_{i\leq k_{\ve}} \left( \frac{\pi(B_{i,\ve})} {\sum_{y_{i,n} \in P_{2}(B_{i,\ve})}\alpha_{i,n}} \sum_{y_{i,n} \in P_{2}(B_{i,\ve})}\alpha_{i,n}\delta_{y_{i,n}} \right)  \nonumber   \\
&~ = \sum_{i\leq k_{\ve}} \left( \frac{\pi(B_{i,\ve})}{\mu_{1,n}(P_{2}(B_{i,\ve}))} 
\sum_{y_{i,n} \in P_{2}(B_{i,\ve})}\alpha_{i,n}\delta_{y_{i,n}} \right)=:\bar{\mu}_{1,n}.\label{eq:P2pien} 
\end{align} 
Moreover
\begin{equation}\label{eq:P1pien} 
(P_{1})_{\sharp}\pi_{\ve, n} = \sum_{i\leq k_{\ve}} (P_{1})_{\sharp} \pi_{i,\ve}  
 = (P_{1})_{\sharp}\pi  = \mu_{0}. 
\end{equation}
It is clear from \eqref{eq:P2pien} and \eqref{eq:P1pien}  that $\pi_{\ve, n} \in \Pi_{\leq}(\mu_{0},\bar \mu_{1,n})$ and that $\bar{\mu}_{1,n}\ll\mu_{1,n}$. 
Notice indeed that $\pi_{i,\ve,n}$ is concentrated over $P_{1}(B_{i,\ve})\times P_{2}(B_{i,\ve})$. 
Being $B_{i,\ve}$ a subset of the product of two balls inside $X_{\ll}^{2}$, causality 
of $\pi_{i,\ve,n}$ and of $\pi_{\ve, n}$ then follow.

\smallskip
{\bf Step 3.} Convergence of the approximations.\\
We now estimate the difference between $\pi_{\ve,n}$ and $\pi$ by checking first the 
difference between $\pi_{i,\ve,n}$ and $\pi_{i,\ve}$  in duality with $(f_{1},f_{2}) \in C_{b}(X)^{2}$:
from \eqref{E:definitionpi1en} we deduce that
$$
\int f_{1}(x)f_{2}(y) \pi_{i,\ve,n}(\dd x\dd y) = 
\int f_{1}(x) \pi_{i,\ve}(\dd x\dd y)
\frac{ \sum_{y_{i,n \in P_{2}(B_{i,\ve})}} \alpha_{i,n}f_{2}(y_{i,n})}
{\sum_{y_{i,n \in P_{2}(B_{i,\ve})}} \alpha_{i,n}}.
$$
Since $\supp \, \mu_{1}$ is compact, we have that $f_{2}|_{\supp \, \mu_{1}}$ is uniformly continuous. Denoting with $\omega_{f_{2}}(\ve)$ the modulus of continuity of $f_{2}|_{\supp \, \mu_{1}}$ at distance $\ve$, 
and recalling that $P_{2}(B_{i,\ve}) \subset B_{\ve}(y_{i,\ve})$ we estimate
$$
\left| \int f_{1}(x)f_{2}(y) \pi_{i,\ve,n}(\dd x\dd y) -
\int  f_{1}(x) \pi_{i,\ve}(\dd x\dd y) f_{2}(y_{i,\ve}) \right| 
\leq \omega_{f_{2}}(\ve)\int |f_{1}(x)| \pi_{i,\ve}(\dd x\dd y). 
$$
In the same way: 
$$
\left|\int f_{1}(x)f_{2}(y) \pi_{i,\ve}(\dd x\dd y) - 
\int f_{1}(x) \pi_{i,\ve}(\dd x\dd y)f_{2}(y_{i,\ve})\right| \leq
\omega_{f_{2}}(\ve)\int |f_{1}(x)| \pi_{i,\ve}(\dd x\dd y). 
$$
Hence, summing over all $i\leq k_{\ve}$, we obtain 
$$
\left| \int_{X\times X} f_{1}(x)f_{2}(y) \pi_{\ve,n}(\dd x\dd y)
- \int_{X\times X} f_{1}(x)f_{2}(y) \pi(\dd x\dd y) \right| \leq 2 \omega_{f_{2}}(\ve) \|f_{1} \|_{\infty}.
$$

Recall that every $\varphi\in C(\supp \, \mu_{0}\times \supp \, \mu_{1}; \R)$ can be approximated in $C^{0}$-norm by finite linear combinations of product functions $f_{i,1}\otimes f_{i,2}$ with $f_{i,1}\in C(\supp\, \mu_{0};\R), \,f_{i,2}\in  C(\supp\, \mu_{1};\R)$.  Thus,  letting $\ve_{n}\downarrow 0$ be such that $\liminf_{n\to \infty} \mu_{1,n}(A_{\ve_{n}})=1$ (the existence of the sequence $(\ve_{n})$ is granted by \eqref{eq:mu1inAeps1}) and   defining  $\pi_{n}:=\pi_{\ve_{n},n}$ for every $n\in \N$, we have 
$$
\pi_{n} \in \Pi_{\leq}(\mu_{0},\bar \mu_{1,n}),  \quad 
\bar \mu_{1,n} \ll \mu_{1,n}, \quad \pi_{n} \to \pi \text{ narrowly}.
$$
In particular the last convergence implies 
that $\bar \mu_{1,n}$ converges narrowly to $\mu_{1}$, applying $(P_{2})_{\sharp}$.
Since for large $n$ the construction gives $\supp \, \pi_{n}\subset (\supp \, \pi)^{\ve} \Subset X^{2}$ (here $(\supp \, \pi)^{\ve}$ denotes an 
$\ve$-enlargement  of $\supp \, \pi$ with respect to $\sfd$), we have that $(\pi_{n})$ is $\tau^{p}$-uniformly integrable  and thus $F_{n}(\pi_{n})\to F_{\infty}(\pi)$ by Lemma \ref{L:limsup}. The conclusion then follows from Lemma \ref{L:use}.
\end{proof}

\begin{remark}
The previous stability results can be seen as the Lorentzian counterpart of the metric fact that $W_{p}(\mu_{n},\mui)\to 0$ if and only if  $\mu_{n}\to \mui$ narrowly and $(\mu_{n})$ has uniformly integrable $p$-moments.  The remarkable differences in the Lorentzian setting are first that the cost is not continuous implying the $\ell_{p}$-optimality does not pass to the limit automatically, and second  that $\ell_{p}(\mu_{n},\mui)\to 0$ \emph{does not} imply $\mu_{n}\to \mui$ narrowly: it is easy to construct a counterexample (e.g. already in 1+1 dimensional Minkowski space-time) using that if $\supp \,\mu_{1}$ and $\supp \, \mu_{2}$ are contained in the light cone of a given common point then $\ell_{p}(\mu_{1},\mu_{2})=0$.
\end{remark}


\subsection{Kantorovich duality}

In the smooth Lorentzian setting, Kantorovich duality has been studied in \cite{Suhr, KellSuhr} in case $p=1$ and in \cite{McCann} for $p\in (0,1)$,   see also \cite{BertrandPuel, BertrandPratelliPuel} for relativistic cost functions in $\R^{n}$. 
In this section we study Kantorovich duality in the Lorentzian synthetic setting.
The following definition, relaxing the notion of $q$-separation introduced by McCann  \cite[Definition 4.1]{McCann} in the smooth Lorentzian setting will turn out to be very useful. Recall the definition   \eqref{eq:defell} of the cost function $\ell^{p}$.

\begin{definition}[Timelike $p$-dualisable]\label{def:Tpdual}
 Let  $(X,\sfd, \ll, \leq, \tau)$ be a Lorentzian pre-length space and let $p\in (0,1]$. We say that $(\mu,\nu)\in \mathcal{P}(X)^{2}$ is \emph{timelike $p$-dualisable (by $\pi\in \Pi_{\ll}(\mu,\nu)$)}  if 
 \begin{enumerate}
\item  $\ell_{p}(\mu,\nu)\in (0,\infty)$;
\item  $\pi\in  \Pi_{\leq}^{p\text{-opt}}(\mu,\nu)$ and $\pi(X^{2}_{\ll})=1$;
\item there exist measurable functions $a,b:X\to \R$, with $a\oplus b \in L^{1}(\mu\otimes \nu)$ such that  $\ell^{p}\leq a\oplus b$ on $\supp \, \mu \times  \supp \, \nu $.
\end{enumerate}
 \end{definition}

The motivation for considering timelike $p$-dualisable pairs of measures is twofold: firstly the $p$-optimal coupling $\pi(\dd x\dd y)$ matches events described by $\mu(\dd x)$  with events described by $\nu(\dd y)$ so that $x\ll y$, secondly
 Kantorovich duality holds (cf. \cite[Proposition 2.7]{Suhr} in smooth Lorentzian setting and in case $p=1$):
 
 \begin{proposition} [Weak Kantorovich duality I] \label{prop:WKDI}
Fix $p\in (0,1]$. Let  $(X,\sfd, \ll, \leq, \tau)$ be a globally hyperbolic Lorentz geodesic space. If $(\mu,\nu)\in \mathcal{P}(X)^{2}$  is timelike $p$-dualisable, then (weak) Kantorovich duality holds:
 \begin{equation}\label{eq:KantDual}
 \ell_{p}(\mu,\nu)^{p}=\inf \left\{ \int_{X} u\, \mu+ \int_{X} v\, \nu \right\} ,
 \end{equation}
 where the $\inf$ is taken over all measurable functions $u:\supp\, \mu\to \R\cup \{+\infty\}$ and  $v:\supp\, \nu\to \R\cup \{+\infty\}$ with $u\oplus v\geq \ell^{p}$ on $\supp \, \mu \times  \supp \, \nu $ and  $u\oplus v\in L^{1}(\mu\otimes \nu)$. Furthermore, the value of the right hand side does not change if one restricts the $\inf$ to bounded and continuous functions. 
 \end{proposition}
 
 \begin{proof}
 The claim follows from  Remark \ref{rem:eqellq} combined with   \cite[Theorem 1.3]{Vil:topics}.
 \end{proof}
\begin{remark}
The notion of timelike $p$-dualisabily is not empty,  indeed for instance if $X$ is globally hyperbolic and  $\mu,\nu \in \Prob_{c}(X)$ admit an optimal coupling $\pi\in  \Pi_{\leq}^{p\text{-opt}}(\mu,\nu)$ concentrated on $X^{2}_{\ll}$, then all the three conditions are satisfied. The only one requiring a comment is the last one: since by global hyperbolicity $\tau:X^{2}\to \R$ is continuous then it is bounded on the compact set $\supp \, \mu \times \supp \, \nu \Subset X^{2}$ and we can choose $a$ and $b$ to be constant functions.
\end{remark}

Under stronger assumptions on the causality relation on $(\mu,\nu)$, (weak) Kantorovich duality holds for general Lorentzian pre-length spaces:
\begin{proposition}[Weak Kantorovich duality II] \label{prop:WKDI}
Fix $p\in (0,1]$. Let  $(X,\sfd, \ll, \leq, \tau)$ be a Lorentzian pre-length  space and let $(\mu,\nu)\in \mathcal{P}(X)^{2}$ such that $(\mu\otimes\nu)\, (X^{2}_{\leq})=1$.  Assume that there exist measurable functions $a,b:X\to \R$, with $a\oplus b \in L^{1}(\mu\otimes \nu)$ such that $\tau^{p}\leq a\oplus b$, $\mu\otimes \nu$-a.e..
\\ Then (weak) Kantorovich duality \eqref{eq:KantDual} holds.
\end{proposition}

\begin{proof}
The result follows from \cite[Theorem 1]{BeiSch} where (weak) Kantorovich duality (for the minimum optimal transport problem) is proved to hold for general  $\mu\otimes\nu$-a.e. finite Borel costs with values in $[0,\infty]$, observing that the cost $(a\oplus b)-\ell^{p}$ takes values in $[0,\infty]$,  $\mu \otimes \nu$-a.e..
\end{proof}

We next discuss the validity of the  \emph{strong} Kantorovich duality, i.e. the existence of optimal functions (called Kantorovich potentials) achieving 
the infimum in the right hand side of   \eqref{eq:KantDual}.
  To that aim the next definition is key.

\begin{definition}[$\ell^{p}$-concave functions, $\ell^{p}$-transform and  $\ell^{p}$-subdifferential]
Fix $p\in (0,1]$ and let $U,V\subset X$.  A measurable function $\varphi:U\to \R$  is \emph{$\ell^{p}$-concave relatively to $(U,V)$} if there exists a function $\psi:V \to \R$ such that
$$
\varphi(x)=\inf_{y\in V} \psi(y)- \ell^{p}(x,y), \quad \forall x\in U.
$$
The function 
\begin{equation}\label{eq:defellptransform}
\varphi^{(\ell^{p})}: V\to \R \cup \{-\infty\}, \quad \varphi^{(\ell^{p})}(y):=\sup_{x\in U} \varphi(x)+ \ell^{p} (x,y)
\end{equation}
is called \emph{$\ell^{p}$-transform of $\varphi$}. 
The \emph{$\ell^{p}$-subdifferential $\partial_{\ell^{p}} \varphi\subset (U\times V )\cap X^{2}_{\leq}$} is defined by
$$
\partial_{\ell^{p}} \varphi := \{(x,y)\in (U\times V )\cap X^{2}_{\leq}\,:\,  \varphi^{(\ell^{p})}(y) - \varphi(x) = \ell^{p}(x,y)\}.
$$
Replacing $\ell^{p}$ with $\tau^{p}$ in all the definitions above, one obtains the notions of $\tau^{p}$-concave functions, $\tau^{p}$-transform and  $\tau^{p}$-subdifferential.
\end{definition}

Let us explicitely observe that, by the very definition \eqref{eq:defellptransform} of $\ell^{p}$-transform it holds
\begin{equation}\label{eq:ellptransform+varphi}
\varphi^{(\ell^{p})}(y)- \varphi(x) \geq \ell^{p}(x,y), 
\quad \forall (x,y) \in U\times V,
\end{equation}
and analogous inequality replacing $\ell^{p}$ with $\tau^{p}$.

\begin{definition}[Strong Kantorovich duality]\label{def:StrongKantDual}
Fix $p\in (0,1]$. We say that $(\mu,\nu)\in \mathcal{P}(X)^{2}$ satisfies \emph{strong $\ell^{p}$-Kantorovich duality}  if 
\begin{enumerate}
\item  $\ell_{p}(\mu,\nu)\in (0,\infty)$;
\item  there exists  Borel subsets $A_{1}\subset \supp\, \mu, A_{2}\subset \supp \, \nu$ with $\mu(A_{1})=\nu(A_{2})=1$,  and there exists $\varphi: A_{1} \to \R$  which is  $\ell^{p}$-concave  relatively to $(A_{1}, A_{2})$ and
 satisfying
$$
 \ell_{p}(\mu,\nu)^{p}= \int_{X} \varphi^{(\ell^{p})}(y) \, \nu (\dd y) - \int_{X} \varphi(x) \, \mu(\dd x).
$$
\end{enumerate}
Replacing $\ell^{p}$ with $\tau^{p}$ in condition 2 above, one obtains the notion of  \emph{strong $\tau^{p}$-Kantorovich duality}.
\end{definition}

\begin{remark}\label{rem:NecSufOpt}
Using \eqref{eq:ellptransform+varphi}, it is immediate to check that if  $(\mu,\nu)\in \mathcal{P}(X)^{2}$  satisfies strong $\ell^{p}$-Kantorovich duality then the following holds. A coupling $\pi\in \Pi_{\leq}(\mu,\nu)$ is $\ell_{p}$-optimal if and only if
\begin{equation*}
\varphi^{(\ell^{p})}(y)  - \varphi(x)  =\ell^{p}(x,y)=\tau(x,y)^{p}, \quad \text{for $\pi$-a.e. $(x,y)$}, 
\end{equation*}
i.e. if and only if $\pi (\partial_{\ell^{p}} \varphi)=1$.  Analogously, if $(\mu,\nu)\in \mathcal{P}(X)^{2}$  satisfies strong $\tau^{p}$-Kantorovich duality then $\pi\in \Pi_{\leq}(\mu,\nu)$ is $\ell_{p}$-optimal if and only if
$\pi (\partial_{\tau^{p}} \varphi)=1$.  
\end{remark}

\begin{remark}[The case $\supp \, \mu\times \supp\, \nu \subset X^{2}_{\leq}$]
In case $\supp \, \mu\times \supp\, \nu \subset X^{2}_{\leq}$, it is readily seen from the definitions above that $\varphi:\supp\, \mu\to \R$ is $\ell^{p}$-concave (relatively to $(\supp \, \mu, \supp \, \nu)$) if and only if it is $\tau^{p}$-concave, moreover $\varphi^{(\ell^{p})}=\varphi^{(\tau^{p})}$, and $ \partial_{\ell^{p}} \varphi= \partial_{\tau^{p}} \varphi$. It follows that also the notions of strong $\ell^{p}$-Kantorovich duality and strong $\tau^{p}$-Kantorovich duality coincide in this case.
\end{remark}

We next relate $\tau^{p}$-cyclical monotonicity with strong $\tau^{p}$-Kantorovich duality.

\begin{theorem}[$\tau^p$-cyclical monotonicity $\Rightarrow$ strong $\tau^p$-Kantorovich duality]\label{thm:OptiffMon}
Fix $p\in (0,1]$. 
Let  $(X,\sfd, \ll, \leq, \tau)$ be a Lorentzian pre-length 
space and let $\mu,\nu\in \mathcal{P}(X)$ with  $\ell_{p}(\mu,\nu) \in (0,\infty)$.  
and that  $\ell_{p}(\mu,\nu) \in (0,\infty)$. 
For any $\pi\in \Pi_{\leq} (\mu,\nu)$ the following holds.

If $\pi$ is $\tau^{p}$-cyclically monotone then $\pi$ is $\ell_{p}$-optimal. Moreover, $(\mu,\nu)$ satisfies strong $\tau^{p}$-Kantorovich duality and $\pi  (\partial_{\tau^{p}}\varphi)=1$.
\end{theorem}

\begin{proof}
The proof consists in constructing a  $\tau^{p}$-concave function $\varphi$ such that $\pi( \partial_{\tau^{p}}\varphi)=1$.
\\Let $\Gamma\subset X^{2}_{\leq}$ be a Borel  $\tau^{p}$-cyclically monotone set such that $\pi(\Gamma)=1$, and $\tau|_{\Gamma}$ is real valued. It follows that $\tau^{p}$ is real valued on $P_{1}(\Gamma)\times P_{2}(\Gamma)$. Notice that $P_{i}(\Gamma)\subset X$ is a Suslin set, for $i=1,2$.
\smallskip

\textbf{ Step 1. } Definition of $\varphi_{(x_{0}, y_{0})}=\varphi$,  and proof that $\varphi(x_{0})=0$.
\\Fix $(x_{0}, y_{0})\in  \Gamma$. Define $\varphi_{(x_{0}, y_{0})}=\varphi: P_{1}(\Gamma) \to \R \cup \{\pm \infty\}$ by
\begin{equation}\label{eq:defvarphi}
\varphi_{(x_{0}, y_{0})}(x)=\varphi(x):= \inf \left\{ \sum_{i=0}^{k} \left[  \tau(x_{i}', y_{i}')^{p}-  \tau(x_{i+1}', y_{i}')^{p} \right] \right\}
\end{equation}
where the $\inf$ is taken over all $k\in \N$ and all ``chains''
$$
\left\{( x_{i}', y_{i}') \right\}_{0\leq i\leq k+1} \subset \Gamma \text{ with }    x_{k+1}'=x, \; (x_{0}', y_{0}')=(x_{0}, y_{0}).
$$
Let us stress that $y_{k+1}'$ does not enter in the expression of the right hand side of \eqref{eq:defvarphi} (this will be useful in step 3). It is readily seen that
\begin{equation}\label{eq:varphi(x0)=0}
\varphi(x_{0})=0.
\end{equation}
Indeed on the one hand $\varphi(x_{0})\leq  \tau(x_{0}, y_{0})^{p}-  \tau(x_{0}, y_{0})^{p}=0$. On the other hand, since by assumption  $\Gamma$ is $\tau^{p}$-cyclically monotone,  then the right hand side of \eqref{eq:defvarphi} is non-negative. Thus \eqref{eq:varphi(x0)=0} is proved.

\smallskip
\textbf{Step 2.}  We show that $\varphi$ is real-valued on $P_{1}(\Gamma)$  and  measurable.
\\Fix $x\in P_{1}(\Gamma)$. The very definition \eqref{eq:defvarphi} of $\varphi=\varphi_{(x_{0}, y_{0})}$ gives
$$
\varphi(x)+ \left[  \tau(x, y)^{p} -\tau(x_{0}, y)^{p} \right]   \geq \varphi(x_{0})\overset{\eqref{eq:varphi(x0)=0}}{=}0,
$$
where $y$ is such that $(x,y) \in \Gamma$.
In particular,  $\varphi(x)>-\infty$. Analogously, 
$$
\varphi(x_{0})+ \left[  \tau(x_{0}, y_{0})^{p} -\tau(x, y_{0})^{p} \right]  \geq \varphi(x)
$$
and thus $\varphi(x)<+\infty$. 
Notice that, under the stronger assumption that $X$ is a globally hyperbolic Lorentzian geodesic space (so that $\tau$ is continuous),  then  $\varphi$ would be upper semi-continuous (as  infimum of a family of  continuous functions) and thus measurable.
\\ We now prove that $\varphi$ is measurable also in the general setting.  Since $\tau^{p}$ is lower semicontinuous, there exists compact subset $\Gamma_{j}\Subset \Gamma$ such that $\Gamma_{j}\subset \Gamma_{j+1}$, $\Gamma=\bigcup_{j\in \N} \Gamma_{j}$ and $\tau^{p}|_{\Gamma_{j}}$ is continuous and real valued. We choose continuous functions $c_{l}$ such that $c_{l}\uparrow \tau^{p}$. Notice that  $c_{l}|_{\Gamma_{j}}\to \tau^{p}|_{\Gamma_{j}}$ uniformly as $l\to \infty$, for every $j\in \N$, by Dini's Theorem. Define the auxiliary functions
$$
\varphi_{k,j, l}(x):= \inf \left\{ \sum_{i=0}^{k} \left[  c_{l}(x_{i}', y_{i}')-  c_{l}(x_{i+1}', y_{i}') \right] \right\}, \quad \varphi_{k,j}(x):= \inf \left\{ \sum_{i=0}^{k} \left[  \tau(x_{i}', y_{i}')^{p}-  \tau(x_{i+1}', y_{i}')^{p} \right] \right\},
$$
where the infimum is taken over all  ``chains''
$$
\left\{( x_{i}', y_{i}') \right\}_{0\leq i\leq k+1} \subset \Gamma_{j} \text{ with }    x_{k+1}'=x, \; (x_{0}', y_{0}')=(x_{0}, y_{0})\in \Gamma_{1}.
$$
The uniform convergence $c_{l}|_{\Gamma_{j}}\to \tau^{p}|_{\Gamma_{j}}$ ensures that  $\varphi_{k,j, l}|_{\Gamma_{j}}\to \varphi_{k,j}|_{\Gamma_{j}}$ pointwise.  The monotonicity of the quantities in $j$ and $k$ gives
$$
\varphi(x)=\lim_{k\to \infty} \lim_{j\to \infty} \varphi_{k,j}(x)= \lim_{k\to \infty} \lim_{j\to \infty} \lim_{l\to \infty} \varphi_{k,j,l}(x), \quad \forall x\in P_{1}(\Gamma).
$$
As each $\varphi_{k,j,l}$ is upper semi continuous, we get that $\varphi$ is measurable.

\smallskip

\textbf{Step 3.}  We show that $\varphi$ is $\tau^{p}$-concave relatively to $(P_{1}(\Gamma), P_{2}(\Gamma))$.
\\Define $\psi_{(x_{0}, y_{0})}=\psi: P_{2}(\Gamma)\to \R\cup\{-\infty\}$ by
\begin{equation}\label{eq:defpsiPFSKD}
\psi_{(x_{0},y_{0})}(y)=\psi(y):=\inf\left\{ \sum_{i=0}^{k} \left[  \tau(x_{i}', y_{i}')^{p}-  \tau(x_{i+1}', y_{i}')^{p} \right] +\tau(x_{k+1}', y)^{p} \right\},
\end{equation}
where the $\inf$ is taken over all $k\in \N$ and all chains
$$
\left\{( x_{i}', y_{i}') \right\}_{0\leq i\leq k+1} \subset \Gamma \text{ with }    y_{k+1}'=y, \; (x_{0}', y_{0}')=(x_{0}, y_{0}).
$$
Notice that, for every  $x\in P_{1}(\Gamma)$ there exists  $y\in P_{2}(\Gamma)$ such that $(x,y)\in \Gamma$; thus, any chain  in the definition \eqref{eq:defvarphi} of $\varphi(x)$ can be concatenated with $(x,y)$, giving an admissible chain for the definition  \eqref{eq:defpsiPFSKD} of  $\psi(y)$. It follows that $\varphi(x)+ \tau(x,y)^{p} \geq  \psi(y)$ and thus
$$
\varphi(x)\geq \inf_{y\in P_{2}(\Gamma)} \psi(y)- \tau(x,y)^{p}, \quad \forall  x\in P_{1}(\Gamma).
$$
Conversely, it is readily seen from the definition \eqref{eq:defvarphi} (resp.  \eqref{eq:defpsiPFSKD}) of $\varphi(x)$ (resp.  $\psi(y)$)  that (recall that $y_{k+1}'$ does not play any role in \eqref{eq:defvarphi}) 
$$
\varphi(x)\leq  \psi(y)- \tau(x,y)^{p}, \quad \forall (x,y)\in P_{1}(\Gamma)\times P_{2}(\Gamma).
$$
We conclude that 
$$
\varphi(x)= \inf_{y\in P_{2}(\Gamma)} \psi(y)- \tau(x,y)^{p}, \quad \forall x\in P_{1}(\Gamma).
$$
It follows that $\psi$ is real valued on $P_{2}(\Gamma)$ and  $\varphi$ is  $\tau^{p}$-concave relatively to $(P_{1}(\Gamma), P_{2}(\Gamma))$.
\\

\textbf{Step 4.} We show that $\Gamma \subset \partial_{\tau^{p}} \varphi$. 
\\Let $(\bar x, \bar y)\in \Gamma$. From the definition \eqref{eq:defvarphi} of $\varphi(x)$ we have
$$
\varphi(\bar x)+ \left[  \tau(\bar x, \bar y)^{p} -\tau(x, \bar y)^{p} \right]   \geq \varphi(x), \quad \forall x\in P_{1}(\Gamma),
$$
which can be rewritten as
$$
\varphi(\bar x)+  \tau(\bar x, \bar y)^{p} \geq \sup_{x\in P_{1}(\Gamma)} \varphi(x)+ \tau(x, \bar y)^{p}=\varphi^{(\tau^{p})}(\bar y).
$$
Since the inequality $\varphi(\bar x)+  \tau(\bar x, \bar y)^{p} \leq \varphi^{(\tau^{p})}(\bar y)$ is trivial from the definition of $\tau^{p}$-transform, we conclude that equality holds and thus $(\bar x, \bar y) \in \partial_{\tau^{p}} \varphi$.
\\

\textbf{Step 5.}  Conclusion: we claim that
\begin{equation}\label{eq:strongKantDualPf}
 \ell_{p}(\mu,\nu)^{p}= \int_{X} \varphi^{(\tau^{p})}(y) \, \nu (\dd y) - \int_{X} \varphi(x) \, \mu(\dd x)= \int_{X^{2}} \tau(x,y)^{p} \, \pi(\dd x\dd y).
\end{equation}
From Step 4, we know that
$$
\varphi^{(\tau^{p})}(y)-\varphi(x)=\tau(x,y)^{p}, \quad \text{for all $(x,y)\in \Gamma$},
$$
which integrated with respect to $\pi$  gives the second identity of  \eqref{eq:strongKantDualPf}.
\\On the other hand, integrating  the inequality
$$
\varphi^{(\tau^{p})}(y)-\varphi(x)\geq \tau(x,y)^{p}, \quad \text{ $\mu\otimes \nu$-a.e. $(x,y)$,}
$$
with respect to any $\pi'\in \Pi_{\leq} (\mu,\nu)$ gives that
$$
 \ell_{p}(\mu,\nu)^{p}=\sup_{\pi'\in \Pi_{\leq} (\mu,\nu)} \int_{X^{2}}  \tau(x,y)^{p} \, \pi'(\dd x \dd y) \leq \int_{X} \varphi^{(\tau^{p})}(y) \, \nu (\dd y) - \int_{X} \varphi(x) \, \mu(\dd x).
$$
The claimed  \eqref{eq:strongKantDualPf} follows.
\end{proof}

\begin{definition}[Strongly timelike  $p$-dualisability] \label{def:TStrongDual}
A pair $(\mu,\nu) \in (\mathcal{P}(X))^{2}$ is said to be \emph{strongly timelike $p$-dualisable} if
\begin{enumerate}
\item  $(\mu,\nu)$ is timelike $p$-dualisable;
\item  there exists a measurable $\ell^{p}$-cyclically monotone set $\Gamma\subset X^{2}_{\ll} \cap (\supp \, \mu \times \supp \,\nu)$ such that a coupling  $\pi\in \Pi_{\leq}(\mu,\nu)$ is $\ell_{p}$-optimal if  and only if $\pi$ is concentrated on $\Gamma$, i.e. $\pi(\Gamma)=1$.
\end{enumerate}
\end{definition}

\begin{remark}
Let  $(\mu,\nu)$ be timelike  $p$-dualisable and satisfying strong $\ell^{p}$-Kantorovich duality (resp.  strong $\tau^{p}$-Kantorovich duality). It follows from Remark \ref{rem:NecSufOpt} that if  $\Gamma:=\partial_{\ell^{p}} \varphi \subset X^{2}_{\ll}$ (resp. $\Gamma:=\partial_{\tau^{p}} \varphi \subset X^{2}_{\ll}$) then  $(\mu,\nu)$ is  strongly timelike $p$-dualisable.
\end{remark}

\noindent
In the next two corollaries we show that the notion of strongly timelike $p$-dualisability is non-empty:

\begin{corollary}\label{lem:q-dualPcX}
Fix $p\in (0,1]$. Let  $(X,\sfd, \ll, \leq, \tau)$ be a   globally hyperbolic Lorentzian geodesic space and assume that $\mu,\nu\in \mathcal{P}(X)$ satisfy:
\begin{enumerate}
\item there exist measurable functions $a,b:X\to \R$  with $a\oplus b \in L^{1}(\mu\otimes \nu)$ such that $\tau^{p}\leq a \oplus b$ on $\supp \, \mu \times  \supp \, \nu $;
\item $\supp \, \mu \times \supp \, \nu \subset X^{2}_{\ll}$.
\end{enumerate}
Then $(\mu,\nu)$ satisfies strong $\tau^{p}$-Kantorovich duality and is strongly timelike $p$-dualisable.
\end{corollary}

\begin{proof}
The fact that  there exists $\pi\in \Pi_{\leq}^{p\text{-opt}}(\mu,\nu)$ follows from Proposition \ref{prop:ExMaxellp}; morover, since $\supp \, \pi \subset \supp \, \mu \times \supp \, \nu \subset X^{2}_{\ll}$, we infer that $(\mu,\nu)$ is timelike $p$-dualisable.
\\ From part 1 of Proposition \ref{P:OptiffMon}  we have $\pi$ is $\ell^{p}$-cyclically monotone and thus, from Remark \ref{rem:tauellpCM},  also  $\tau^{p}$-cyclically monotone since $\supp \, \mu \times \supp\, \nu \subset X^{2}_{\leq}$.
\\Using now Theorem \ref{thm:OptiffMon}  we infer that $(\mu,\nu)$ satisfies strong $\tau^{p}$-Kantorovich duality.
Setting $\Gamma:=\partial_{\tau^{p}} \varphi  \subset \supp \, \mu \times \supp \,\nu$, it is a direct consequence of the assumptions that $\Gamma \subset X^{2}_{\ll}$ and thus Remark \ref{rem:NecSufOpt} yields that condition 2 of Definition \ref{def:TStrongDual} is  satisfied.
\end{proof}

In the next corollary we show that, in case $\nu$ is a Dirac measure,   the strongly timelike $p$-dualisability is  equivalent to the  timelike  $p$-dualisability.

\begin{corollary}\label{cor:STDualDelta}
Let  $(X,\sfd, \ll, \leq, \tau)$ be a Lorentzian pre-length space and let $p\in (0,1]$. Fix $\bar x\in X$ and let $\nu:=\delta_{\bar x}$. Assume that $\mu \in \mathcal{P}(X)$ satisfies:
\begin{equation}\label{eq:DualnuDirac}
\tau(\cdot, \bar x)^{p} \in L^{1}(X,\mu) \quad \text{and} \quad \tau(\cdot, \bar{x})>0 \text{ $\mu$-a.e. }. 
\end{equation}
Then $(\mu,\nu)$ is strongly timelike  $p$-dualisable. In other terms, in case $\nu$ is a Dirac measure,   the strongly timelike  $p$-dualisability is  equivalent to the timelike  $p$-dualisability.
\end{corollary}

\begin{proof}
Let $\pi:=\mu \otimes \delta_{\bar x}$ and choose $b\equiv 0, a(x):=\tau(x,\bar x)^{p}$.  Noticing that  $\Pi(\mu,\delta_{\bar x})= \{\pi\}$ we get that \eqref{eq:DualnuDirac} implies:  $\ell_{p}(\mu,\delta_{\bar x}) \in (0,\infty)$,  $\pi$  is the unique $\ell_{p}$-optimal coupling for $(\mu,\delta_{\bar x})$ and $\pi(X^{2}_{\ll})=1$. It follows that $(\mu,\nu)$ is strongly timelike  $p$-dualisable.
\end{proof}

\subsection{$\ell_p$-optimal dynamical plans}

Let us start by introducing some classical notation.
The evaluation map is defined by 
\begin{equation}\label{def:eet}
\ee_{t}: C([0,1], X) \to X, \quad \gamma\mapsto \ee_{t}(\gamma):=\gamma_{t}, \quad \forall t\in [0,1].
\end{equation}
The stretching/restriction operator ${\rm restr}_{s_{1}}^{s_{2}}: C([0,1], X) \to C([0,1], X)$ is defined by
\begin{equation}\label{def:restr}
({\rm restr}_{s_{1}}^{s_{2}} \gamma)_{t}:= \gamma_{(1-t)s_{1}+t s_{2}}, \quad \forall s_{1},s_{2}\in [0,1], s_{1}<s_{2}, \, \forall t\in [0,1].
\end{equation}

\begin{definition}[$\ell_p$-optimal dynamical plans and $\ell_p$-geodesics]\label{def:ellp-DOP}
Let  $(X,\sfd, \ll, \leq, \tau)$ be a Lorentzian pre-length space and let $p\in (0,1]$. We say that $\eta\in \mathcal{P}(\Geo(X))$ is an  $\ell_p$-optimal dynamical plan from $\mu_0\in \mathcal{P}(X)$ to $\mu_1\in \mathcal{P}(X)$ if $(\ee_0)_\sharp \eta=\mu_0, (\ee_1)_\sharp \eta=\mu_1$ and 
\begin{equation}\label{eq:defODP}
(\ee_0, \ee_1)_{\sharp} \eta \quad \text{ belongs to }\Pi^{p\text{-opt}}_{\leq} ((\ee_0)_\sharp \eta, (\ee_1)_\sharp \eta).
\end{equation}
The set of $\ell_p$-optimal dynamical plans from $\mu_{0}$ to $\mu_{1}$ is denoted by ${\rm OptGeo}_{\ell_{p}}(\mu_{0}, \mu_{1})$.
We say that a curve curve $[0,1] \ni t \mapsto  \mu_{t} \in \mathcal{P}(X)$ 
is an $\ell_{p}$-geodesic if there exists an $\ell_p$-optimal dynamical plan $\eta$ from $\mu_0$ to $\mu_1$ such that $\mu_t=(\ee_t)_\sharp \eta$, for all $t\in [0,1]$.
\end{definition}

\begin{remark}[On the notion of $\ell_p$-geodesics]\label{R:dyngeo}
In analogy with the metric setting (cf. \cite[Definition 7.20]{Vil}), one could have defined an $\ell_p$-geodesic to be a curve $[0,1] \ni t \mapsto  \mu_{t} \in \mathcal{P}(X)$ continuous in narrow topology and satisfying 
\begin{equation}\label{eq:defgeodellq}
\ell_{p}(\mu_{s}, \mu_{t})=(t-s) \ell_{p}(\mu_{0}, \mu_{1}) \in (0,\infty), \quad \text{for all }0\leq s< t\leq 1.
\end{equation}
We claim that if  $\eta\in {\rm OptGeo}_{\ell_{p}}(\mu_{0}, \mu_{1})$, then curve 
$$
\mu_t:=(\ee_t)_\sharp \eta, \quad t\in [0,1],
$$
satisfies both properties: continuity in narrow topology and identity \eqref{eq:defgeodellq}.

Firstly if $t \to t_0$, 
then for any continuous and bounded function $f$, it holds:
$$
\int_X f(x) \,\mu_t(dx) = \int_{\Geo(X)} f(\gamma_t) \, \eta(\dd \gamma).
$$
Continuity in narrow topology follows from dominated convergence Theorem. Next, let us prove \eqref{eq:defgeodellq}.
By construction,
$(\ee_s,\ee_t)_\sharp \eta \in \Pi_{\leq}(\mu_s,\mu_t)$ for all $s \leq t$. Thus
\begin{equation}\label{eq:ellpmusmutt-s}
\ell_p(\mu_s,\mu_t)^p \geq 
\int \ell^p(\gamma_s,\gamma_t)
\,\eta(\dd \gamma)
= (t-s)^p \ell_p(\mu_0,\mu_1)^p, \quad \text{for all } 0\leq s\leq t\leq 1,
\end{equation}
where the last identity follows from the optimality of $(\ee_0, \ee_1)_{\sharp} \eta$.
The inequality \eqref{eq:ellpmusmutt-s} turns into an identity by applying the reverse triangle inequality: indeed, applying \eqref{eq:ellpmusmutt-s} twice with suitable choices of intermediate times, we get
$$
\ell_p(\mu_0,\mu_1) \geq 
\ell_p(\mu_0,\mu_s) + \ell_p(\mu_s,\mu_t)+
\ell_p(\mu_t,\mu_1)
\geq 
\ell_p(\mu_0,\mu_1), 
$$
forcing identities in the inequalities 
and showing also that 
$(\ee_s,\ee_t)_\sharp \eta \in \Pi_{\leq}^{p\text{-opt}}(\mu_s,\mu_t)$, 
for all $s\leq t \in [0,1]$.
\end{remark}

\begin{proposition}\label{prop:ellqgeodCS}
Fix $p\in (0,1)$ and let  $(X,\sfd, \ll, \leq, \tau)$ be a globally hyperbolic, Lorentzian geodesic space. Let $\mu_{0},\mu_{1}\in \Prob_{c}(X)$ such that there exists $\pi\in \Pi^{p\text{-opt}}_{\leq} (\mu_{0}, \mu_{1})$.
Then
\begin{enumerate}
\item There always exists an $\ell_{p}$-optimal dynamical plan (and thus an $\ell_{p}$-geodesic) from $\mu_{0}$ to $\mu_{1}$.

\item If $\eta\in {\rm OptGeo}_{\ell_{p}}(\mu_{0}, \mu_{1})$ then $\eta^{s_{1},s_{2}}:= ({\rm restr}_{s_{1}}^{s_{2}})_{\sharp} \eta\in {\rm OptGeo}_{\ell_{p}}((\ee_{s_{1}})_{\sharp}\eta, (\ee_{s_{2}})_{\sharp}\eta)$, for all $s_{1}<s_{2}$, $s_{1},s_{2}\in [0,1]$. 

\item  Let $\eta\in {\rm OptGeo}_{\ell_{p}}(\mu_{0}, \mu_{1})$ and let $\tilde{\eta}$ be a non-negative measure on $C([0,1], X)$ such that $\tilde{\eta} \leq \eta^{s_1,s_2}$, for some $s_1,s_2 \in [0,1]$ and $\tilde{\eta}( C([0,1], X))>0$. Then $\eta':=\frac{1} {\tilde{\eta}( C([0,1], X))} \, \tilde\eta$ 
is an $\ell_{p}$-optimal dynamical plan.  

\item Assume that $X$ is timelike non-branching, $\eta\in {\rm OptGeo}_{\ell_{p}}(\mu_{0}, \mu_{1})$ and it is concentrated on $\TGeo(X)$. 

\begin{enumerate}
\item If $(s_{1},s_{2})\neq (0,1)$ then $\eta'$ as in 3. is the unique element of  $ {\rm OptGeo}_{\ell_{p}}((\ee_{0})_{\sharp}\eta', (\ee_{1})_{\sharp}\eta')$, and $(\mu'_{t}:= (\ee_{t})_{\sharp}\eta')_{t\in [0,1]}$ is the unique $\ell_{p}$-geodesic joining its endpoints.
\item  There exists a set $\Gamma\subset \TGeo(X)$ such that $\eta(\Gamma)=1$ and satisfying the following property: if $\gamma^{1}, \gamma^{2}\in \Gamma$ cross at some intermediate time $t_{0}\in (0,1)$, i.e. there exists $t_{0}\in (0,1)$ such that $\gamma^{1}_{t_{0}}=\gamma^{2}_{t_{0}}$, then $\gamma^{1}_{t}=\gamma^{2}_{t}$ for all $t\in [0,1]$.

\item 
Assume that $\eta$ can be written as $\eta= \lambda_{1} \eta^{1}+ \lambda_{2 }\eta^{2}$, for some $\eta^{i}\in \Prob(C([0,1], X))$, $\lambda_{i}\in (0,1)$ for $i=1,2$,   $\lambda_{1}+\lambda_{2}=1$, and each $\eta^i$ is concentrated on a set $\Gamma^i$ such that $\Gamma^1\cap \Gamma^2=\emptyset$.  Then $\eta^{1},\eta^{2}$ are $\ell_{p}$-optimal dynamical plans and they satisfy  $(\ee_{t})_{\sharp} \eta^{1} \perp (\ee_{t})_{\sharp} \eta^{2}$ for all  $t\in (0,1)$.
\end{enumerate}

\item  Every $\ell_{p}$-geodesic $(\mu_{t}=(\ee_{t})_{\sharp} \eta)_{t\in [0,1]}$ from $\mu_{0}$ to $\mu_{1}$ is an absolutely continuous curve in the $W_{1}$-Kantorovich Wasserstein space $(\Prob(X), W_{1})$ w.r.t. $\sfd$, with length
\begin{equation}\label{eq:LengthW1mut}
{\rm L}_{W_{1}}\big((\mu_{t})_{t\in [0,1]} \big) \leq \int {\rm L}_{\sfd}(\gamma) \, \eta(\dd \gamma)\leq \bar{C}<\infty
\end{equation}
where  $\bar{C}>0$ depends only on the compact subset $J^{+}(\supp\, \mu_{0}) \cap J^{-} (\supp\, \mu_{1})\Subset X$.
\end{enumerate}
\end{proposition}

\begin{proof}

First of all notice that by global hyperbolicity  and Proposition \ref{prop:GH->KGH} (i), we have
$$\bigcup_{t\in [0,1]} \supp\, \mu_{t}\subset J^{+}(\supp\, \mu_{0}) \cap J^{-} (\supp\, \mu_{1})\Subset X,$$ 
with  $J^{+}(\supp\, \mu_{0}) \cap J^{-} (\supp\, \mu_{1}) 
= : {\mathcal X} $ compact set.

\smallskip
\textbf{Proof of 1.}\\ 
Let $\pi\in \Pi^{p\text{-opt}}_{\leq} (\mu_{0}, \mu_{1})$ be given by the assumptions.
Consider the set 
$$
G : = \{(x,y,\gamma) \in {\mathcal X} \times {\mathcal X} \times \Geo(X) \colon x = \gamma_0, y = \gamma_1 \}.
$$
By the continuity of $\tau$, the set $\Geo(X)$ is closed inside the complete and separable metric space $C([0,1],{\mathcal X})$. Hence $G$ is closed as well. Since $X$ is geodesic, 
$P_{12}(G)$ contains all the couples $(x,y) \in X\times X$ such that $x \leq y$.
Invoking a classical selection theorem (for instance \cite[Theorem 5.5.2]{Srivastava}), there exists an analytic map $S : X\times X \to \Geo(X)$, such that $(x,y,S(x,y)) \in G$ (an overview on analytic sets will be given in Section \ref{Ss:disintegration}).
In other words, there is a measurable selection to join two endpoints $x$ and $y$ by a geodesic. Define
$$
\eta:=S_\sharp \pi \in \mathcal{P}(\Geo(X)).
$$
Since  
$(\ee_0, \ee_1)\circ S={\rm Id}$, 
it is clear that $\eta$ is an $\ell_p$-optimal dynamical plan according to Definition \ref{def:ellp-DOP}.

\smallskip
\textbf{Proof of 2.}\\ 
In Remark \ref{R:dyngeo} we have shown that 
$(\ee_{s_1},\ee_{s_2})_\sharp \eta \in \Pi_{\leq}^{p\text{-opt}}(\mu_{s_1},\mu_{s_2}). 
$
Since 
$(\ee_s,\ee_t)_\sharp \eta = 
(\ee_0,\ee_1)_\sharp \eta^{s,t}$, 
and $\eta^{s,t}(\Geo(X)) = 1$ (recall that ${\rm restr}_{s_{1}}^{s_{2}}$ maps $\Geo(X) $ into itself), 
the claim follows.

\smallskip
\textbf{Proof of 3.}\\ 
By the previous point, $\eta^{s_1,s_2}$ is an $\ell_p$-optimal dynamical plan; thus, we can assume with no loss in generality that $s_1 = 0$ and $s_2=1$.
Since $\tilde \eta \leq \eta$, then necessarily $\tilde \eta$ is concentrated on $\Geo(X)$ and therefore $\eta' (\Geo (X)) = 1$.
The optimality of $(\ee_0,\ee_1)_{\sharp}\eta'$ follows from Lemma \ref{L:restriction}, 
by noticing that $(\ee_0,\ee_1)_{\sharp}\tilde \eta \leq (\ee_0,\ee_1)_{\sharp}\eta$.

\smallskip
\textbf{Proof of 4.}\\ 
Concerning parts $(a)$ and $(b)$,
they are routine properties of Wasserstein geodesics in non-branching spaces. A proof of the claims can be obtained for instance  from the same proof of \cite[Theorem 7.30]{Vil}, part $(iii)$: the assumption of having a coercive Lagrangian action is never used; a selection theorem like the one we obtained in 1. is sufficient to follow verbatim the same proof. 

\smallskip
Concerning part $(c)$: the optimality of $\eta^{1}$ and $\eta^{2}$ follows directly from 3.  Call $\mu_{t}:=(\ee_{t})_{\sharp} \eta$ and $\mu^{i}_{t}:=(\ee_{t})_{\sharp} \eta^{i}$, for $i=1,2$, $t\in [0,1]$. Assume by contradiction that for some $t_{0}\in (0,1)$ there exists 
\begin{equation}\label{eq:defpropA}
A \Subset X \text{ compact subset s.t. } A\subset \ee_{t_0}(\Gamma^1)\cap \ee_{t_0}(\Gamma^2),  \; \mu_{t_{0}}^{1}(A)>0, \;  \mu_{t_{0}}^{1}\llcorner A \ll \mu_{t_{0}}^{2}.
\end{equation}  
From 3. we know that
$$
 \bar{\eta}:=\frac{1}{\mu_{t_{0}} (A)} \eta \llcorner (\ee_{t_{0}}^{-1}(A)), \quad \bar{\eta}^{i}:=\frac{1}{\mu^{i}_{t_{0}} (A)} \eta^{i} \llcorner (\ee_{t_{0}}^{-1}(A)\cap \Gamma^i) \quad \text{ for }i=1,2,
$$
are all $\ell_{p}$-optimal dynamical plans.  We claim that
\begin{equation}\label{eq:gamma1=gamma2bareta}
\gamma^{1}_{t_{0}}=\gamma^{2}_{t_{0}} \text{ for some }\gamma^{i}\subset \Gamma^i \; \Longrightarrow \gamma^{1}=\gamma^{2}.
\end{equation}
Indeed, if $\gamma^{1}_{t_{0}}=\gamma^{2}_{t_{0}}$, then $\gamma^3$ defined by concatenation 
$$
\gamma^3(t)=
\begin{cases}
\gamma^1(t) \quad \text{for } t\in [0, t_0]\\
\gamma^2(t) \quad \text{for } t\in [t_0,1],
\end{cases}
$$
would be a timelike geodesic, coinciding with $\gamma^1$ on $[0,t_0]$. The forward non-branching property implies that $\gamma^1=\gamma^2$ on $[t_0, 1]$. The analogous argument on $[t_0,1]$ with the backward non-branching property gives $\gamma^1=\gamma^2$ on $[0, t_0]$ and thus the claim \eqref{eq:gamma1=gamma2bareta}.

The combination of   \eqref{eq:defpropA} with  \eqref{eq:gamma1=gamma2bareta}  gives that  $\ee_{t_0}^{-1}(A)\cap \Gamma^1= \ee_{t_0}^{-1}(A)\cap \Gamma^2$. Hence $$\Gamma^1\cap \Gamma^2\supset \ee_{t_0}^{-1}(A)\cap \Gamma^1\cap \Gamma^2 \neq \emptyset,$$ yielding a contradiction. 

\smallskip
\textbf{Proof of 5.} \\
From the  non-totally imprisoning property,  it follows that 
\begin{align}
\sup& \left\{ {\rm L}_{\sfd}(\gamma) \,: \,  \gamma\in \supp\, \eta \right\}  \nonumber \\
& \leq \sup \left\{ {\rm L}_{\sfd}(\gamma) \,: \,  \gamma(I) \subset J^{+}(\supp\, \mu_{0}) \cap J^{-} (\supp\, \mu_{1}),\; \gamma: I\to X \text{ causal }  \right\} =: \bar{C}< \infty.
\end{align}
In particular  $\int {\rm L}_{\sfd}(\gamma) \, \eta(\dd \gamma)\leq \bar{C}<\infty$. The claim \eqref{eq:LengthW1mut} follows from   \cite[Theorem 4]{Lisini}.
\end{proof}

\section{Synthetic Ricci curvature lower bounds}

\subsection{Timelike Curvature Dimension condition}
The goal  of this section  to give a synthetic  formulation of the strong energy condition (and more generally a synthetic formulation of  $\Ric_g\geq -K g$ in the timelike directions and ${\rm dim}\leq N$) for  a  measured pre-length space $(X,\sfd, \mm, \ll, \leq, \tau)$. 
Let us recall the characterization of Ricci curvature bounded below and dimension bounded above in the smooth Lorentzian globally hyperbolic setting proved by McCann \cite[Corollary 6.6, Corollary 7.5]{McCann} (see also \cite[Corollary 4.4]{MoSu}).

\begin{theorem}\label{thm:CharLorRic}
Let $(M^{n},g)$ be a globally hyperbolic spacetime and $0< p <1$. Then the following are equivalent:
\begin{enumerate} 
\item  $\Ric_{g}(v,v) \geq -K g(v,v) $, for every timelike $v \in TM$.
\item  For any couple $(\mu_{0},\mu_{1})\in (\Dom(\Ent(\cdot|\vol_{g})))^{2}$ which is timelike $p$-dualisable (in the sense of Definition \ref{def:Tpdual})  there exists a (unique) $\ell_{p}$-geodesic $(\mu_{t})_{t\in [0,1]}$ joining them such that the function $[0,1]\ni t\mapsto e(t) : = \Ent(\mu_{t}|\vol_{g})$ is semi-convex (and thus in particular it is locally Lipschitz in $(0,1)$) and it satisfies:
\begin{equation}\label{eq:conve}
e''(t) - \frac{1}{n} e'(t)^{2 } \geq K \int_{M\times M} \tau(x,y)^{2} \, \pi(\dd x\dd y),
\end{equation}
in the distributional sense on $[0,1]$.
\item    For any couple $(\mu_{0},\mu_{1})\in (\Dom(\Ent(\cdot|\vol_{g}))\cap \Prob_{c}(X))^{2}$ which is strongly timelike $p$-dualisable (in the sense of Definition \ref{def:TStrongDual})  there exists a (unique) $\ell_{p}$-geodesic $(\mu_{t})_{t\in [0,1]}$ joining them and satisfying \eqref{eq:conve}.
\end{enumerate}
\end{theorem}

\begin{proof}
The equivalence of 1 and 2 was proved in McCann \cite[Corollary 6.6, Corollary 7.5]{McCann} (see also \cite[Corollary 4.4]{MoSu}). Trivially $2 \Longrightarrow 3$. The implication  $3 \Longrightarrow 1$ can be proved along the lines of \cite[Corollary 4.4]{MoSu} using Corollary \ref{lem:q-dualPcX}.
\end{proof}

The following definition is thus natural.

\begin{definition}[$\mathsf{TCD}^{e}_{p}(K,N)$ and  $\mathsf{wTCD}^{e}_{p}(K,N)$ conditions]\label{def:TCD(KN)}
Fix $p\in (0,1)$, $K\in \R$, $N\in (0,\infty)$. We say that  a  measured Lorentzian pre-length space $(X,\sfd,\mm, \ll, \leq, \tau)$ satisfies  $\mathsf{TCD}^{e}_{p}(K,N)$ 
(resp. $\mathsf{wTCD}^{e}_{p}(K,N)$)  if the following holds.
For any couple $(\mu_{0},\mu_{1})\in (\Dom(\Ent(\cdot|\mm)))^{2}$ which is   timelike $p$-dualisable   
(resp. $(\mu_{0},\mu_{1})\in [\Dom(\Ent(\cdot|\mm))\cap  \Prob_{c}(X)]^{2}$  
which is   strongly timelike $p$-dualisable) by some 
$\pi\in \Pi^{p\text{-opt}}_{\ll}(\mu_{0},\mu_{1})$,  
there exists an  $\ell_{p}$-geodesic $(\mu_{t})_{t\in [0,1]}$ such that  
the function $[0,1]\ni t\mapsto e(t) : = \Ent(\mu_{t}|\vol_{g})$ is 
semi-convex (and thus in particular it is locally Lipschitz in $(0,1)$) and it satisfies
\begin{equation}\label{eq:conveKN}
e''(t) - \frac{1}{N} e'(t)^{2 } \geq K \int_{X\times X} \tau(x,y)^{2} \, \pi(\dd x\dd y),
\end{equation}
in the distributional sense on $[0,1]$.
\end{definition}

Definition \ref{def:TCD(KN)} corresponds to a differential/infinitesimal formulation of the $\mathsf{TCD}^{e}_{p}(K,N)$ condition. In order to have also an integral/global  formulation it is convenient  to introduce  the following entropy (cf. \cite{EKS}) 
\begin{equation}\label{eq:defSN}
U_{N}(\mu|\mm) : = \exp\left(-\frac{\Ent(\mu|\mm)}{N} \right).
\end{equation}
It is clear that \eqref{eq:Entlsc}   implies the upper-semicontinuity of $U_{N}$ under narrow convergence:
\begin{equation}\label{eq:UNusc} 
\mu_{n}\to \mui \; \text{narrowly and } \mm\Big( \bigcup_{n\in \N} \supp\, \mu_{n} \Big) <\infty\quad \Longrightarrow \quad  \limsup_{n\to \infty} U_{N}(\mu_{n}|\mm) \leq U_{N}(\mui|\mm).
\end{equation}
It is straightforward to check that  $[0,1]\ni t\mapsto e(t)$ is semi-convex and  satisfies \eqref{eq:conve} if and only if  $[0,1]\ni t\mapsto u_{N}(t):= \exp(- e(t)/N)$ is semi-concave and satisfies
\begin{equation}\label{eq:sN''}
u_{N}''\leq -\frac{K}{N} \|\tau\|^{2}_{L^{2}(\pi)} \,  u_{N}.
\end{equation} 
Set
\begin{equation}\label{eq:deffsfc}
\fs_{\kappa}(\vartheta):=
\begin{cases}
\frac{1}{\sqrt{\kappa}} \sin(\sqrt{\kappa} \vartheta),  \quad & \kappa>0\\
\vartheta, &\kappa=0\\
\frac{1}{\sqrt{-\kappa}} \sinh(\sqrt{-\kappa} \vartheta), \quad  &\kappa<0\\
\end{cases}, \qquad 
\fc_{\kappa}(\vartheta):=
\begin{cases}
\cos(\sqrt{\kappa} \vartheta),  \quad & \kappa\geq 0\\
\cosh(\sqrt{-\kappa} \vartheta) \quad  &\kappa<0\\
\end{cases}, 
\end{equation}
and
\begin{equation}\label{eq:sigmakappa}
\sigma_{\kappa}^{(t)}(\vartheta):=
\begin{cases}
\frac{\fs_{\kappa}(t\vartheta)}{\fs_{\kappa}(\vartheta)}, \quad & \kappa\vartheta^{2}\neq 0 \text{ and } \kappa\vartheta^{2}<\pi^{2} \\
t,\quad & \kappa \vartheta^{2}=0\\
+\infty \quad &  \kappa \vartheta^{2}\geq \pi^{2}
\end{cases}.
\end{equation}
Note that the function $\kappa\mapsto \sigma_{\kappa}^{(t)}(\vartheta)$ is non-decreasing for every fixed $\vartheta, t$.
With the above notation, the differential inequality \eqref{eq:sN''}  is equivalent to the integrated version (cf. \cite[Lemma 2.2]{EKS}):
\begin{equation}\label{eq:sNconc}
u_{N}(t) \geq \sigma^{(1-t)}_{K/N} \left(\|\tau\|_{L^{2}(\pi)}\right) u_{N}(0) + \sigma^{(t)}_{K/N} \left( \|\tau\|_{L^{2}(\pi)} \right) u_{N}(1).
\end{equation} 
We thus proved the following proposition.

\begin{proposition}\label{prop:equivTCD}
Fix $p\in (0,1)$, $K\in \R$ and $N\in (0,\infty)$. The measured Lorentzian pre-length space $(X,\sfd, \mm, \ll, \leq, \tau)$ satisfies  (resp. weak) $\mathsf{TCD}^{e}_{p}(K,N)$ if and only if  for any couple $(\mu_{0},\mu_{1})\in \big(\Dom(\Ent(\cdot|\mm))\big)^{2}$ which is timelike $p$-dualisable  
(resp. $(\mu_{0},\mu_{1})\in [\Dom(\Ent(\cdot|\mm))\cap  \Prob_{c}(X)]^{2}$  which is   
strongly timelike $p$-dualisable) by some   $\pi\in \Pi^{p\text{-opt}}_{\ll}(\mu_{0},\mu_{1})$, 
there exists an $\ell_{p}$-geodesic $(\mu_{t})_{t\in [0,1]}$ such that  the function $[0,1]\ni t\mapsto u_{N}(t) : = U_{N}(\mu_{t}|\mm)$  satisfies \eqref{eq:sNconc}.
\end{proposition}

As an example of geometric application of the $\mathsf{TCD}^{e}_{p}(K,N)$ we next show a timelike Brunn-Minkowski  inequality (for the Riemannian/metric counterparts see \cite{sturm:II, EKS, CM2}).

\begin{proposition}[A  timelike Brunn-Minkowski  inequality]\label{prop:BrunnMnk}
Let $(X,\sfd, \mm, \ll, \leq, \tau)$ be a measured  Lorentzian pre-length space satisfying (resp. weak)  $\mathsf{TCD}^{e}_{p}(K,N)$, for some $K\in \R, N\in [1,\infty), p\in (0,1)$.
Let $A_{0}, A_{1}\subset X$ be measurable subsets with $\mm(A_{0}), \mm(A_{1})\in (0,\infty)$. Calling $\mu_{i}:=1/\mm(A_{i}) \, \mm\llcorner_{A_{i}}$, $i=1,2$, assume that $(\mu_{0},\mu_{1})$ is (resp. strongly) timelike $p$-dualisable. 
Then
\begin{equation}\label{eq:BrunnMink}
\mm(A_{t})^{1/N}\geq \sigma^{(1-t)}_{K/N} \left(\Theta \right)\, \mm(A_{0})^{1/N} +  \sigma^{(t)}_{K/N} \left(\Theta \right) \, \mm(A_{1})^{1/N}
\end{equation}
where  $A_{t}:=\cI(A_{0},A_{1},t)$ defined in \eqref{eq:defI(A,B,t)} is the set of $t$-intermediate points of geodesics from $A_{0}$ to $A_{1}$, and $\Theta$ is the maximal/minimal time-separation between points in $A_{0}$ and $A_{1}$, i.e.:
\begin{equation*}
\Theta:=
\begin{cases}
\sup\{\tau(x_{0}, x_{1}): x_{0}\in A_{0}, x_{1}\in A_{1}\} & \text{if } K< 0,\\
\inf\{\tau(x_{0}, x_{1}): x_{0}\in A_{0}, x_{1}\in A_{1}\} \quad &\text{if } K\geq 0 .
\end{cases}
\end{equation*}
In particular, if $K\geq 0$ it holds:
\begin{equation*}
\mm(A_{t})^{1/N}\geq (1-t)\, \mm(A_{0})^{1/N} + t  \, \mm(A_{1})^{1/N}.
\end{equation*}
\end{proposition}

\begin{proof}
Let  $(\mu_{t})_{t\in [0,1]}$ be the  $\ell_{p}$-geodesic given by Proposition \ref{prop:equivTCD}, satisfying
$$
U_{N}(\mu_{t}|\mm)\geq \sigma^{(1-t)}_{K/N} \left(\|\tau\|_{L^{2}(\pi)}\right)\, \mm(A_{0})^{1/N} +  \sigma^{(t)}_{K/N} \left( \|\tau\|_{L^{2}(\pi)} \right) \, \mm(A_{1})^{1/N}.
$$
Since $\mu_{t}=\rho_{t} \mm$ is concentrated on $A_{t}$, which is Suslin,  applying Jensen's inequality twice gives:
\begin{equation}\label{eq:twiceJensen}
U_{N}(\mu_{t}|\mm)=\exp\left(-\frac{1}{N}\int\log \rho_{t} \, \mu_{t}  \right) \leq \int \rho_{t}^{-1/N} \, \mu_{t}= \int_{A_{t}} \rho_{t}^{1-\frac{1}{N}}  \mm\leq \mm(A_{t})^{1/N}.
\end{equation}
The claim follows observing that $\vartheta\mapsto\sigma_{K/N}(\vartheta)$ is non-increasing for $K\leq 0$ (resp. non-decreasing for $K>0$) and that $\|\tau\|_{L^{2}(\pi)}\leq \Theta$ (resp. $\|\tau\|_{L^{2}(\pi)}\geq \Theta$).
Notice that in the case of $\mathsf{wTCD}$, we first assume $A_{0},A_{1}$ to be compact 
and then obtain the full claim arguing by inner regularity of $\mm$ with respect to compact sets.
\end{proof}

The Brunn–Minkowski inequality implies further geometric consequences
like a  timelike Bishop-Gromov  volume growth estimate and a  timelike Bonnet-Myers  theorem. 
In order to state them, let us introduce some notation. Fix $x_{0}\in X$ and let 
$$B^{\tau}(x_{0},r):=\{x\in I^{+}(x_{0})\cup \{x_{0}\}: \tau(x_{0}, x)< r\}$$ 
be the $\tau$-ball of radius $r$ and center $x_{0}$.   Since typically the volume of a $\tau$-ball is infinite (e.g. in Minkowski space  it is the region below an hyperboloid), it is useful to localise volume estimates using star-shaped sets. To this aim, we say that    $E\subset I^{+}(x_{0})\cup\{x_{0}\}$ is $\tau$-star-shaped with respect to $x_{0}$ if 
$ \fI(x_{0},x,t) \subset E$ for every $x\in E$ and $t \in (0,1]$.
Define
$$
v(E,r):=\mm(\overline{B^{\tau}(x_{0},r)}\cap E), \quad s(E, r):=\limsup_{\delta\downarrow 0} \frac{1}{\delta}  \mm\big((\overline{B^{\tau}(x_{0},r+\delta)}\setminus B^{\tau}(x_{0},r) )\cap E\big)  
$$
the volume of the $\tau$-ball of radius $r$ (respectively of the   $\tau$-sphere of radius $r$) intersected with the compact subset  $E\subset I^{+}(x_{0})\cup\{x_{0}\}$, $\tau$-star-shaped with respect to $x_{0}$. Let us mention that, for \emph{smooth} Lorentzian manifolds, a timelike Bishop-Gromov inequality was obtained in \cite{EhSa98}.

\begin{proposition}[A  timelike Bishop-Gromov  inequality]\label{prop:BisGro}
Let $(X,\sfd, \mm, \ll, \leq, \tau)$ be a measured globally hyperbolic Lorentzian geodesic space satisfying $\mathsf{wTCD}^{e}_{p}(K,N)$, for some $K\in \R, N\in [1,\infty), p\in (0,1)$.
Then, for each $x_{0}\in X$, each compact subset  $E\subset I^{+}(x_{0})\cup\{x_{0}\}$ $\tau$-star-shaped with respect to $x_{0}$, and each $0<r<R\leq \pi \sqrt{N/(K \vee 0)}$, it holds:
\begin{equation}\label{eq:BS}
\frac{s(E, r)}{s(E,R)} \geq \left(\frac{\fs_{K/N}(r)}  {\fs_{K/N}(R)} \right)^{N}, \quad  \frac{v(E, r)}{v(E,R)} \geq \frac{\int_{0}^{r} \fs_{K/N} (t)^{N} \dd t}   {\int_{0}^{R}\fs_{K/N} (t)^{N} \dd t }.
\end{equation}
\end{proposition}

\begin{proof}
We briefly sketch the argument.  The basic idea is to apply Proposition \ref{prop:BrunnMnk} to $A_{0}:=B^{\tau}(x_{0},\varepsilon)\cap E$ and $A_{1}:=\big(\overline{B^{\tau}(x_{0},R+\delta R)}\setminus B^{\tau}(x_{0},R) \big)  \cap E$.  Observe that, for $\ve>0$  small enough, it holds $A_{0}\times A_{1}\subset X^{2}_{\ll}$ and thus the measures $(\mu_{0},\mu_{1})$ in the statement of  Proposition \ref{prop:BrunnMnk} are strongly timelike $p$-dualisable thanks to Corollary \ref{lem:q-dualPcX}. Thus we can apply  Proposition \ref{prop:BrunnMnk} and follow verbatim the proof of \cite[Theorem 2.3]{sturm:II} replacing  the coefficients $\tau^{(t)}_{K/N}(\vartheta)$ with $\sigma^{(t)}_{K/N}(\vartheta)$.
\end{proof}

\begin{proposition}[A  timelike Bonnet-Myers  inequality]\label{prop:BonMy}
Let $(X,\sfd, \mm, \ll, \leq, \tau)$ be a measured globally hyperbolic Lorentzian geodesic space satisfying $\mathsf{wTCD}^{e}_{p}(K,N)$, for some $K>0, N\in [1,\infty), p\in (0,1)$.
Then
\begin{equation}\label{eq:BonMy}
\sup_{x,y\in X} \tau(x,y) \leq \pi \sqrt{\frac{N}{K}}.
\end{equation}
In particular, for any causal curve $\gamma$ it holds $\LL_{\tau}(\gamma)\leq \pi \sqrt{\frac{N}{K}}$.
\end{proposition}

\begin{proof}
Assume by contradiction that there exist $x'_{0},x'_{1}\in X$ with $\tau(x'_{0},x'_{1})\geq \pi \sqrt{N/K}+4\ve$, for some  $\ve>0$. Let $\delta>0$ and  $x_{0},y_{0}\in X$ be such that   
$$B^{\sfd}(x_{0},\delta)\subset  I^{+}(x_{0}'),\; B^{\sfd}(x_{1},\delta)\subset I^{-}(x_{1}'), \:  \inf\{ \tau(x,y):  x\in B^{\sfd}(x_{0},\delta), \, y\in B^{\sfd}(x_{1},\delta)\} \geq \pi \sqrt{N/K}+\ve,$$
where $B^{\sfd}(x,r)$ denotes the $\sfd$-metric ball of radius $r$ centred at $x$.
 From Corollary \ref{lem:q-dualPcX} it follows that $A_{0}:=B^{\sfd}(x_{0},\delta),  A_{1}:=B^{\sfd}(x_{0},\delta)$ satisfy the assumptions of Proposition \ref{prop:BrunnMnk}. Note that, for this choice of sets,  $\Theta\geq \pi \sqrt{N/K}+\ve$ and thus $\mm(A_{1/2})=+\infty$.  However, $A_{1/2}\subset J^{+}(x_{0}')\cap J^{-}(x_{1}')$ is relatively compact by global hyperbolicity and thus it has finite $\mm$-measure. Contradiction.
\end{proof}

\subsection{Timelike Measure Contraction Property}

A weaker variant of the $\mathsf{TCD}^{e}_{p}(K,N)$ condition is obtained by considering $(K,N)$-convexity properties only for those $\ell_{p}$-geodesics $(\mu_{t})_{t\in [0,1]}$ where $\mu_{1}$ 
is a Dirac delta. 
In the metric measure setting, such a variant goes under the name of 
Measure Contraction Property ($\mathsf{MCP}$ for short) and was developed  independently by Sturm \cite{sturm:II} and Ohta \cite{Ohta1}.

\begin{definition}\label{def:TMCP}
Fix $p\in (0,1)$, $K\in \R$, $N\in (0,\infty)$. The measured Lorentzian pre-lengh space $(X,\sfd, \mm, \ll, \leq, \tau)$ satisfies $\mathsf{TMCP}^{e}(K,N)$ if and only if  for any $\mu_{0}\in \Prob_{c}(X)\cap \Dom(\Ent(\cdot|\mm))$ and for any $x_{1}\in X$
such that  $x\ll x_{1}$ for $\mu_{0}$-a.e. $x\in X$, there exists an $\ell_{p}$-geodesic $(\mu_{t})_{t\in [0,1]}$ from $\mu_{0}$ to  $\mu_{1}=\delta_{x_{1}}$
such that 
\begin{equation}\label{eq:defTMCP(KN)}
U_{N}(\mu_{t}|\mm) \geq \sigma^{(1-t)}_{K/N} \left( \|\tau(\cdot, x_{1})\|_{L^{2}(\mu_{0})}\right)\, U_{N}(\mu_{0}|\mm), \quad \forall t\in [0,1).
\end{equation}
\end{definition}

\begin{remark}
The validity of the $\mathsf{TMCP}^{e}(K,N)$ condition is independent of the choice of $p\in (0,1)$ in Definition \ref{def:TMCP}. Indeed for any $p,q\in (0,1)$, a curve $(\mu_{t})_{t\in [0,1]}$ with $\mu_{1}=\delta_{\bar{x}}$ is an $\ell_{p}$-geodesic if and only if it is an $\ell_{q}$-geodesic.  The claim follows easily by the very definition of $\ell_{p}$-geodesic (see Definition \ref{def:ellp-DOP}) and by the fact that the only coupling (and thus optimal) from $\mu_0$ to $\mu_{1}=\delta_{\bar{x}}$ is $\pi=\mu_0\otimes \delta_{\bar{x}}$.
\end{remark}

\begin{remark}[Geometric Properties]\label{rem:GeomPropTMCP}
As in the Riemannian/metric case  \cite{sturm:II}, many properties valid for $\TCD^{e}_{p}(K,N)$  remain true also for  $\TMCP^{e}(K,N)$. More precisely, this is the case for:
\begin{itemize}
\item  Timelike Bishop-Gromov  inequality,  Proposition \ref{prop:BisGro};
\item  Timelike Bonnet-Myers  inequality, Proposition \ref{prop:BonMy}.
\end{itemize}
Actually, in Section \ref{Subsec:TBGBM}, the above results will be improved to sharp forms in case of timelike non-branching $\TMCP^{e}(K,N)$ spaces. Such an improvement will be a product of the techniques developed in Section \ref{Ss:Nonbranching}  and  Section \ref{Sec:LocTMCP}.
\end{remark}

\begin{remark}\label{R:TMCPgeod}
If a Lorentzian pre-lengh space $(X,\sfd, \mm, \ll, \leq, \tau)$ satisfies $\mathsf{TMCP}^{e}(K,N)$, 
then for any $x_{1}\in X$ and $\mm$-a.e. $x\ll x_{1}$ there exists $\gamma \in \TGeo(X)$ such that 
$\gamma_{0} = x$ and $\gamma_{1} = x$.
If in addition $X$ is globally hyperbolic, it follows that $X$ is time-like geodesic. Indeed, given any $x_{1}\in X$ and $x \ll x_{1}$ by $\TMCP^{e}(K,N)$ there is 
a sequence $x_{n} \to x$ and $\gamma^{n} \in \TGeo(X)$ 
with $\gamma^{n}_{0} = x_{n}$ and $\gamma_{1}^{n} = x_{1}$. 
From global hyperbolicity and Proposition \ref{prop:GH->KGH}(i), it follows the existence of a limit $\gamma^{\infty} \in \TGeo(X)$
with $\gamma^{\infty}_{0} = x$ and $\gamma^{\infty}_{1} = x_{1}$ giving that $X$ is timelike geodesic.

If instead $x \leq y$ one needs to further assume $X$ to be causally path connected, i.e. for 
any $x,y \in X$ such that $x \leq y$ there exists a causal curve $\gamma$ with $\gamma_{0} = x$
and $\gamma_{1} = y$.
Hence if a Lorentzian pre-lengh space $(X,\sfd, \mm, \ll, \leq, \tau)$ satisfies $\mathsf{TMCP}^{e}(K,N)$, it is globally hyperbolic and causally path connected, then it is geodesic.
\end{remark}

\begin{lemma}\label{lem:TCDscaling}
Fix $p\in (0,1)$, $K\in \R$, $N\in (0,\infty)$. Let 
the measured  Lorentzian pre-length space $(X,\sfd, \mm, \ll, \leq, \tau)$ satisfy 
$\TCD^{e}_{p}(K,N)$ (resp. $\wTCD^{e}_{p}(K,N),\TMCP^{e}(K,N)$).
Then 
\begin{enumerate}
\item  \emph{Consistency}: $(X,\sfd, \mm, \ll, \leq, \tau)$ satisfy  
$\TCD^{e}_{p}(K',N')$ (resp. $\wTCD^{e}_{p}(K',N'),\TMCP^{e}(K',N')$) for every $K'\leq K$ and $N'\geq N$.
\item \emph{Scaling}:  The rescaled space $(X, a\cdot \sfd, b\cdot \mm, \ll, \leq, r\cdot \tau)$, for $a,b,r>0$ satisfies 
$\TCD^{e}_{p}(K/r^{2},N)$ (resp. $\wTCD^{e}_{p}(K/r^{2},N), \TMCP^{e}(K/r^{2},N)$).
\end{enumerate}
\end{lemma}

\begin{proof} 
1.  Consistency for $\TCD^{e}_{p}(K,N)$ follows directly by the definition \eqref{eq:conveKN}. 
\\ Regarding $\TMCP^{e}$: the consistency in $K$ follows by the fact that the map $\kappa\mapsto \sigma_{\kappa}^{(t)}(\vartheta)$ is monotone increasing. For the consistency in $N$, observe that taking the logarithm of \eqref{eq:defTMCP(KN)} one obtains the equivalent condition
\begin{equation}\label{eq:MCPEnt}
\Ent(\mu_{t}|\mm)\leq \Ent(\mu_{0}|\mm) -N \log\left( \sigma^{(1-t)}_{K/N} \left(\|\tau(\cdot, x_{1})\|_{L^{2}(\mu_{0})}\right)\right). 
\end{equation}
It follows from \cite[Lemma 1.2]{sturm:II} that
$$
 \left( \sigma^{(t)}_{K/N'} (\vartheta)\right)^{N'}\leq t^{N'-N}\;  \left( \sigma^{(t)}_{K/N} (\vartheta)\right)^{N} \leq  \left( \sigma^{(t)}_{K/N} (\vartheta)\right)^{N}   \quad \forall t\in [0,1], \, K\in \R, \, N'\geq N,
$$
giving that the function $N\mapsto -N \log\left( \sigma^{(1-t)}_{K/N}(\vartheta)\right)$ is non-decreasing for every fixed $K,t, \vartheta$.  
\\2. Follows by the very definitions,  observing that $\Ent(\mu|b\cdot \mm)=\Ent(\mu|\mm)-\log (b)$,  $\| r\cdot \tau \|_{L^{2}(\pi)}= r \|\tau \|_{L^{2}(\pi)}$ and that $\sigma^{(t)}_{\kappa/r^{2}}(r\cdot \vartheta)= \sigma^{(t)}_{\kappa}(\vartheta)$.
\end{proof}

We refer to Appendix \ref{sec:AppTMCPsmooth} for a discussion of  $\TMCP^{e}(K,N)$  in  case of  smooth Lorentzian manifolds.

\begin{proposition}[$\mathsf{wTCD}^{e}_{p}(K,N) \Rightarrow \mathsf{TMCP}^{e}(K,N)$]\label{prop:CD->MCP}
Fix $p\in (0,1)$, $K\in \R$, $N\in (0,\infty)$. The  $\mathsf{wTCD}^{e}_{p}(K,N)$ condition implies $\mathsf{TMCP}^{e}(K,N)$ for globally hyperbolic Lorentzian geodesic spaces.
\end{proposition}

\begin{proof}
{\bf Step 1.}\\
Let $\mu_{0}=\rho_{0}\, \mm\in \Dom(\Ent(\cdot|\mm))\cap \Prob_{c}(X)$ and  $x_{1}\in X$ be such that $x\ll x_{1}$ for $\mu_{0}$-a.e. $x\in X$.

\noindent
For each $\ve >0$ consider $K_{\ve}\Subset \supp \, \mu_{0} \Subset  X$ compact subset such that (the last condition will be used  later in step 2)
$$
\int_{X\setminus K_{\ve}} \rho_{0} \,| \log(\rho_{0})| \, \mm \leq \ve, \qquad \mu_{0}(K_{\ve}) \geq 1 - \ve, \qquad K_{\ve} \times \{ x_{1}\} \subset \{\tau>0\}\subset X^{2},
$$
and consider the restricted measure $\mu_{0}^{\ve} : = \mu_{0}\llcorner_{K_{\ve}}/\mu_{0}(K_{\ve})$. A straightforward computation gives
\begin{equation}\label{E:decompo}
\Ent(\mu_{0}|\mm) = \int_{X\setminus K_{\ve}} \rho \log(\rho) \, \mm + \Ent(\mu_{0}^{\ve}|\mm)\, \mu_{0}(K_{\ve}) + \mu_{0}(K_{\ve}) \, \log(\mu_{0}(K_{\ve})).
\end{equation}
Hence 
$$
\Ent(\mu_{0}|\mm) \geq \Ent(\mu_{0}^{\ve}|\mm)(1-\ve) -2\ve,
$$
giving
\begin{equation}\label{eq:UNmu0eps}
U_{N}(\mu_{0}^{\ve}|\mm) \geq \exp\left( - \frac{2\ve}{N (1-\ve)}\right)  U_{N}(\mu_{0}|\mm)^{1/(1-\ve)}.
\end{equation}
\smallskip

{\bf Step 2.}\\
Fix $\ve \ll 1$.
Since the set $\{\tau>0\}\subset X\times X$ is open and by construction it contains $K_{\ve} \times \{ x_{1}\}$, for $\eta>0$ small enough it holds
\begin{equation}\label{eq:defKepsBetatau>0}
K_{\ve} \times B_{\eta}(x_{1})\subset \{\tau>0\}.
\end{equation}
 Define 
$\mu_{1}^{\eta}:=\mm\llcorner B_{\eta}(x_{1})/\mm(B_{\eta}(x_{1}))$.
By Corollary \ref{lem:q-dualPcX}, we know that $(\mu_{0}^{\ve}, \mu_{1}^{\eta})$ is strongly timelike $p$-dualisable. It also clear that $\mu_{0}^{\ve}, \mu_{1}^{\eta}\in \Dom(\Ent(\cdot|\mm))\cap \Prob_{c}(X)$.

The $\mathsf{wTCD}^{e}_{p}(K,N)$ condition thus implies that for each $\ve,\eta>0$ small enough there exists an $\ell_{p}$-optimal coupling $\pi_{\ve, \eta}\in \Pi_{\leq}(\mu_{0}^{\ve}, \mu_{1}^{\eta})$ 
and an  $\ell_{p}$-geodesic $(\mu_{t}^{\ve,\eta})_{t\in [0,1]}$ 
joining  $\mu_{0}^{\ve}$ and  $\mu_{1}^{\eta}$  verifying for all $t \in [0,1]$:
\begin{align}
U_{N}(\mu_{t}^{\ve,\eta}|\mm) \geq &~  \sigma^{(1-t)}_{K/N} \left(\|\tau\|_{L^{2}( \pi_{\ve,\eta})}\right) U_{N}(\mu_{0}^{\ve}|\mm) + \sigma^{(t)}_{K/N} \left(\|\tau\|_{L^{2}( \pi_{\ve,\eta})}\right) U_{N} (\mu_{1}^{\eta}|\mm) \nonumber \\
\geq &~  \sigma^{(1-t)}_{K/N} \left(\|\tau\|_{L^{2}( \pi_{\ve,\eta})}\right)  U_{N}(\mu_{0}^{\ve}|\mm) \nonumber \\
\geq &~   \sigma^{(1-t)}_{K/N} \left(\|\tau\|_{L^{2}( \pi_{\ve,\eta})}\right)    \exp\left( - \frac{2\ve}{N (1-\ve)}\right)  U_{N}(\mu_{0}|\mm)^{1/(1-\ve)}, \label{eq:P2.7step2}
\end{align}
where in the last inequality we used \eqref{eq:UNmu0eps}.
\smallskip

{\bf Step 3.}\\
In this last step we pass into the limit, first as $\eta\to0$, then as $\ve\to 0$. 
\\First of all it is clear that $\mu_{0}^{\ve}\to \mu_{0}$ and  $\mu_{1}^{\eta}\to \mu_{1}$ narrowly. Global hyperbolicity, via Proposition \ref{prop:GH->KGH}(i), implies that
 $$ \bar{K}:=\bigcup_{s\in [0,1]} \fI(K_{\ve_{0}} , B_{\eta_{0}}(x_{1}),s)\Subset X$$
 is a  compact subset, see \eqref{eq:defI(A,B,t)},\eqref{I(K1K2)}. It is easily  seen that  
\begin{equation}\label{eq:mutveetaSuppComp}
\supp\, \mu_{t}^{\ve,\eta}\subset  \fI(K_{\ve}, B_{\eta}(x_{1}),t)\subset  \bar{K}, \quad \forall t\in [0,1], \eta\in [0, \eta_{0}].
\end{equation}
Fix $\ve\in (0, \ve_{0})$ and a sequence $(\eta_{n})$ with $\eta_{n}\downarrow 0$. We aim to construct a limit $\ell_{p}$-geodesic $(\mu^{\ve}_{t})_{t\in [0,1]}$ from $\mu^{\ve}_{0}$ to $\mu_{1}=\delta_{x_{1}}$.
From \eqref{eq:LengthW1mut} we get that 
$$
\sup_{n\in \N} {\rm L}_{W_{1}} \left((\mu^{\ve, \eta_{n}}_{t})_{t\in [0,1]}\right) \leq \bar{C}<\infty.
$$
By the metric Arzel\'a-Ascoli Theorem we deduce that there exists a limit continuous curve $(\mu^{\ve}_{t})_{t\in [0,1]} \subset (\Prob(\bar{K}), W_{1})$ such that (up to a sub-sequence)
$ W_{1}\left(\mu^{\ve, \eta_n}_{t}, \mu^{\ve}_{t} \right)\to 0 $ and thus $\mu^{\ve, \eta_n}_{t} \to \mu^{\ve}_{t}$ narrowly, as $n\to \infty$.
Lemma \ref{lem:ConvEllpNarrow1}  yields that 
\begin{equation}\label{eq:ellqnuepst}
\ell_{p}(\mu^{\ve}_{0},\mu^{\ve}_{t})=\lim_{n\to \infty}  \ell_{p}(\mu^{\ve}_{0}, \mu^{\ve,\eta_{n}}_{t})=  t \;\lim_{n\to \infty}  \ell_{p}(\mu^{\ve}_{0}, \mu^{\eta_{n}}_{1}) =  t\,\ell_{p}(\mu^{\ve}_{0},\mu_{1}).
\end{equation}
In other terms, the curve   $(\mu^{\ve}_{t})_{t\in [0,1]}$ is an $\ell_{p}$-geodesic from $\mu^{\ve}_{0}$ to $\mu_{1}=\delta_{x_{1}}$. 
The upper-semicontinuity of $U_{N}(\cdot|\mm)$ under narrow convergence \eqref{eq:UNusc} yields
\begin{equation}\label{eq:UNuscveeta}
\limsup_{i\to \infty} \, U_{N}(\mu_{t}^{\ve,\eta_{n_{i}}}|\mm) \leq U_{N} ( \mu^{\ve}_{t}), \quad \forall t\in [0,1].
\end{equation}
Moreover,  it is readily seen that $\pi^{\ve,\eta_{n_{i}}}\to \mu^{\ve}_{0}\otimes \delta_{x_{1}} $ narrowly  and 
\begin{equation}\label{eq:limsigmaeta0}
\lim_{i\to \infty}\sigma^{(1-t)}_{K/N} \left( \| \tau\|_{L^{2}( \pi_{\ve,\eta_{n_{i}}})}\right) =  \sigma^{(1-t)}_{K/N} \left( \| \tau (\cdot,x_{1}) \|_{L^{2}( \mu^{\ve}_{0}) } \right).
\end{equation}
Combining \eqref{eq:P2.7step2},\eqref{eq:UNuscveeta} and \eqref{eq:limsigmaeta0} gives 
\begin{equation}\label{eq:UNConcmutve}
U_{N}(\mu_{t}^{\ve}|\mm) \geq    \sigma^{(1-t)}_{K/N} \left( \| \tau (\cdot,x_{1}) \|_{L^{2}( \mu^{\ve}_{0}) } \right)   \exp\left( - \frac{2\ve}{N (1-\ve)}\right)  U_{N}(\mu_{0}|\mm)^{1/(1-\ve)}, \quad \forall t\in [0,1].
\end{equation}
In order to conclude the proof we now pass to the limit as $\ve\downarrow 0$ in \eqref{eq:UNConcmutve}. Observe that
$$
\bar{K}':= \bigcup_{s\in [0,1]} \fI(\supp\, \mu_{0} , x_{1},s)\Subset X
$$
is a compact subset by global hyperbolicity (via Proposition \ref{prop:GH->KGH}(i)) and 
$$
\supp\, \mu_{t}^{\ve}\subset  \fI(\supp \, \mu_{0}, x_{1},t)\subset  \bar{K}', \quad \forall t\in [0,1], \ve\in [0, \ve_{0}].
$$
The argument from \eqref{eq:mutveetaSuppComp} to \eqref{eq:UNConcmutve} can be adapted to show that there exists an $\ell_{p}$-geodesic $(\mu_{t})_{t\in [0,1]}$  satisfying \eqref{eq:defTMCP(KN)}.
\end{proof}

\subsection{Stability of $\TCD^{e}_{p}(K,N)$ and $\TMCP^{e}(K,N)$ conditions}
\label{Ss:Stability}

This section is of independent interest and will not be used in the rest of the paper. In the next theorem we show the stability of the $\TMCP^{e}(K,N)$ condition under convergence of Lorentzian spaces.
Throughout this part we will make use of topological embeddings to identify
spaces with their image inside a larger space. 
Recall that a topological 
embedding is a map $f:X\to Y$ between two topological spaces $X$ and $Y$ such that 
$f$ is continuous, injective and with continuous inverse between $X$ and $f(X)$.
We also say that a space $X$ is \emph{pointed}, if a reference point $\star\in X$ is specified.

\begin{theorem}[Stability of $\TMCP^{e}(K,N)$]\label{thm:StabTMCP}
Let  $\{(X_{j},\sfd_{j}, \mm_{j}, \star_{j},\ll_{j}, \leq_{j}, \tau_{j})\}_{j\in \N\cup\{\infty\}}$ be a sequence of pointed  measured Lorentzian geodesic spaces  satisfying the following properties:
\begin{enumerate}
\item There exists a globally hyperbolic  Lorentzian geodesic space $(\bar X, \bar \sfd,  \overline \ll, \bar \leq, \bar  \tau)$ such that each $(X_{j},\sfd_{j}, \mm_{j}, \ll_{j}, \leq_{j}, \tau_{j})$, $j\in \N\cup\{\infty\}$, is 
 isomorphically embedded in it, i.e. there exist 
 topological embedding  maps $\iota_{j}: X_{j}\to \bar{X}$  such that
\begin{itemize}
\item  $x^{1}_{j} \leq_{j} x^{2}_{j}$ 
if and only if $\iota_{j}(x^{1}_{j}) \bar\leq  \iota_{j}(x^{2}_{j})$, for every $j\in \N\cup\{\infty\}$, 
for every   $x^{1}_{j}, x^{2}_{j}\in X_{j}$; 
\item $\bar{\tau}  (\iota_{j}(x^{1}_{j}), \iota_{j}(x^{2}_{j}))= \tau_{j} (x^{1}_{j}, x^{2}_{j})$ for every $x^{1}_{j}, x^{2}_{j}\in X_{j}$,  for every $j\in \N\cup\{\infty\}$;
\end{itemize}

\item The measures $(\iota_{j})_{\sharp} \mm_{j}$  converge to $(\iota_{\infty})_{\sharp} \mm_{\infty}$ weakly in duality with $C_{c}(\bar X)$ in $\bar X$, i.e.
\begin{equation}\label{eq:weakconv}
\int \varphi\; (\iota_{j})_{\sharp} \mm_{j} \to \int \varphi \; (\iota_{\infty})_{\sharp} \mm_{\infty} \quad \forall \varphi \in C_{c}(\bar X),
\end{equation} 
where  $C_{c}(\bar X)$ denotes the set of continuous functions with compact support.
\item  Convergence of the reference points: $\iota_j(\star_j)\to \iota_\infty(\star_\infty)$ in $\bar{X}$.
\item  For every compact subset $\mathcal{K}\Subset \bar{X}$, it holds that
\begin{equation}\label{eq:TGeoKcompact}
\{\gamma\in \TGeo(\bar X)\colon \gamma([0,1])\subset \mathcal{K}\}\subset C([0,1], \bar{X}) \; \text{ is pre-compact.}
\end{equation}
\item There exist $p\in (0,1), K\in \R, N\in (0,\infty)$ such that  $(X_{j},\sfd_{j}, \mm_{j}, \ll_{j}, \leq_{j}, \tau_{j})$ satisfies  $\TMCP^{e}(K,N)$,  for each $j\in \N$ .   
\end{enumerate}
Then also the limit space 
$(X_{\infty},\sfd_{\infty}, \mm_{\infty}, \ll_{\infty}, \leq_{\infty}, \tau_{\infty})$  
satisfies   $\TMCP^{e}(K,N)$.
\end{theorem}

\begin{remark}\label{R:TMCPandgeo}
Even though we haven't specifically list any topological assumption on the sequence of spaces 
$X_{j}$, they actually inherit them from $\bar{X}$ via the topological embeddings $\iota_{j}$. 
The map $\iota_{j}$ preserves both the causal relations and $\tau_{j}$ hence
$(X_{j},\sfd_{j}, \mm_{j}, \ll_{j}, \leq_{j}, \tau_{j})$ are globally hyperbolic Lorentzian geodesic (by assumption) spaces.
\end{remark}

\begin{proof}
For simplicity of notation, we will identify $X_{j}$ with its isomorphic image $\iota_{j}(X_{j})\subset \bar X$ and the measure $\mm_{j}$ with $(\iota_{j})_{\sharp} \mm_{j}$, for each $j\in \N\cup \{\infty\}$.

Fix  arbitrary  $\mu_{0}^{\infty}=\rho_{0}^{\infty} \mm_{\infty}\in \Prob_{c}(X_{\infty})\cap \Dom(\Ent(\cdot|\mm_{\infty}))$ and $x_{1}^{\infty}\in X_{\infty}$ such that  $x\ll_{\infty} x_{1}^{\infty}$ for $\mu_{0}^{\infty}$-a.e. $x\in X_{\infty}$. 
Since $\mu_{0}^{\infty}$ has compact support and $\bar X$ is globally hyperbolic, we can  restrict all the arguments to a large compact subset $E\Subset \bar X$ such that
\begin{itemize} 
\item $J^{+}_{\bar{X}}(\supp \, \mu_{0}^{\infty})\cap J^{-}_{\bar{X}}(x_1^\infty) \subset E$;
\item $\mm_{\infty}(E)=\lim_{j\to \infty}\mm_j(E)$. 
\end{itemize}
Thus,  $\tilde{\mm}_j:=\mm_j(E)^{-1} \mm_j\llcorner E$ are probability measures supported in the compact subset $E\Subset \bar{X}$ and narrowly converge to $\tilde{\mm}_\infty:=\mm_\infty(E)^{-1} \mm_\infty\llcorner E$. With a slight abuse of notation, for simplicity, we will write $\mm_j$ for  $\tilde{\mm}_j$, $j\in \N\cup\{\infty\}$.
Since on compact metric spaces narrow convergence is equivalent to $W_{2}$ convergence, 
 we actually assume
$\mm_{j}\to \mm_{\infty}$ in $W_{2}^{(\bar X,  \bar \sfd)}$.
Denote with $\ggamma_{j} \in \Pi(\mm_{\infty}, \mm_{j})$ an optimal coupling  for $W_{2}^{(\bar X,  \bar \sfd)}$.
\\

\textbf{Step 1}.   We show that, up to a subsequence, for every $j\in \N$ there exists $\mu_{0}^{j}\in \Prob_{c}(X_{j})\cap \Dom(\Ent(\cdot|\mm_{j}))$, $x^{j}_{1}\in X_{j}\cap E$  such that
\begin{equation}\label{eq:Entmu0inftyoj}
\mu_{0}^{j}\left(I^{-}_{\ll_{j}} (x_{1}^{j})\right)=1, \quad x^{j}_{1}\to x^{\infty}_{1},  \quad \mu_{0}^{j}\to  \mu_{0}^{\infty} \text{ narrowly}, \quad U_{N}(\mu_{0}^{\infty}|\mm_{\infty}) \leq  \liminf_{j\to \infty} U_{N}(\mu_{0}^{j}|\mm_{j}).
\end{equation}

\textbf{Step 1a}. Let us first consider the case $\mu_{0}^{\infty}=\rho_{0}^{\infty} \mm_{\infty}\in \Prob_{c}(X_{\infty})$ has density $\rho_{0}^{\infty}\in L^{\infty}(\mm_{\infty})$ and $x_{1}^{\infty}\in X_{\infty}$ is such that 
$\supp \, \mu_{0}^{\infty} \Subset I^{-}_{\ll_{\infty}}(x_{1}^{\infty}).$  

From narrow convergence we deduce the  existence of a sequence 
$x_{1}^{j} \in \supp\,\mm_{j} \subset X_{j}\cap E$ with
$x_{1}^{j} \to x_{1}^{\infty}$ with respect to $\bar \sfd$. 
Since $\bar \tau: \bar X^{2}\to \R$ is continuous and $\supp \, \mu_{0}^{\infty}$ is compact, 
$$
\lim_{j\to \infty } 
\min_{x \in \supp \, \mu_{0}^{\infty}} \bar \tau(x,x_{1}^{j})
 = \min_{x \in \supp \, \mu_{0}^{\infty}} \bar \tau(x,x_{1}^{\infty}) = 
 \min_{x \in \supp \, \mu_{0}^{\infty}} \tau_{\infty}(x,x_{1}^{\infty}) > 0.
$$
Hence, for $j$ sufficiently large, we can assume that  
$x\bar\ll x_{1}^{j}$ for $\mu_{0}^{\infty}$-a.e. $x\in X_{\infty}$. Then since 
$I^{-}_{\bar \ll}(x_{1}^{j})$ is open,
 any narrow converging sequence of probability measures $\mu_{0}^{k} \to \mu_{0}^{\infty}$ 
satisfies
\begin{equation}\label{E:convergence}
\liminf_{k\to \infty} \mu_{0}^{k} (I^{-}_{\bar \ll}(x_{1}^{j}))
\geq \mu_{0}^{\infty}(I^{-}_{\bar \ll}(x_{1}^{j})) = 1.
\end{equation}

Define now $\ggamma_j'\in\Prob(\bar X^2)$ as $\ggamma_j'(\dd x\dd y):=\rho_{0}^{\infty}(x)\ggamma_j(\dd x\dd y)$ and $\hat{\mu}^{j}_{0}:=(P_{2})_\sharp\ggamma_j'\in\Prob(X_{j})\subset \Prob(\bar X)$.  By construction, $\ggamma_j'\ll\ggamma_j$, 
hence $\hat{\mu}^{j}_{0}\ll(P_{2})_\sharp\ggamma_j=\mm_j$. 
Let $\hat{\mu}^{j}_{0}=\hat{\rho}^{j}_{0} \mm_j$. 
It is readily checked from the definition that it holds 
$\hat{\rho}^{j}_{0}(y)=\int \rho^{\infty}_{0}(x)\,(\ggamma_j)_{y}(\dd x)$, where $\{(\ggamma_j)_y\}$ is the disintegration of $\ggamma_j$ w.r.t. the projection on the second marginal. 
In particular, 
$\| \hat \rho_{0}^{j}\|_{L^{\infty}(\mm_{j})} \leq 
\|  \rho_{0}^{\infty}\|_{L^{\infty}(\mm_{\infty})}$.

By Jensen's inequality applied to the convex function $u(z)=z\log(z)$ we have
\[
\begin{split}
\Ent(\hat{\mu}^{j}_{0}|\mm_{j})&=\int u(\hat{\rho}^{j}_{0})\,\mm_j=
\int u\left(\int  \rho^{\infty}_{0}(x)\, (\ggamma_j)_{y}(\dd x)\right)\,\mm_j(\dd y)\\
&\leq \int u( \rho^{\infty}_{0}(x))\,(\ggamma_j)_{y}(\dd x)\,\mm_j(\dd y)=\int u( \rho^{\infty}_{0}(x))\,\ggamma_j(\dd x\dd y)\\
&=\int u( \rho^{\infty}_{0})\,(P_{1})_\sharp\ggamma_j=\int u(\rho^{\infty}_{0})\,\mm_\infty=\Ent(\mu_{0}^{\infty}|\mm_{\infty}).
\end{split}
\]
Since by construction we have 
$\ggamma_j'\in\Pi(\mu^{\infty}_{0},\hat \mu^{j}_{0})$, it holds
\[
\begin{split}
\left(W_2^{(\bar X, \bar \sfd)}(\mu^{\infty}_{0}, \hat\mu^{j}_{0})\right)^{2}& \leq \int \bar\sfd^2(x,y)\,\ggamma_j'(\dd x\dd y)= \int \rho_{0}^{\infty}(x) \bar \sfd^2(x,y)\,\ggamma_j(\dd x\dd y) \\
& \leq \|\rho_{0}^{\infty}\|_{ L^{\infty}(\mm_{\infty})} \left(W_2^{(\bar X, \bar \sfd)}(\mm_\infty, \mm_j)\right)^{2},
\end{split}
\]
and therefore $W_2^{(\bar X, \bar \sfd)}(\mu^{\infty}_{0},\hat\mu^{j}_{0})\to 0$. In particular $\hat\mu^{j}_{0}\to \mu^{\infty}_{0}$ narrowly in $\bar X$.

Moreover  $\hat \mu_{0}^{j}$ has compact support, indeed $\supp\, \hat \mu_{0}^{j}\subset \supp\, \mm_j \subset E \Subset \bar{X}$. 
We will also 
 cutoff  where the density $\hat \rho^{j}_{0}$ is too small in the following manner. 
Consider the set $K_{j}: = \{ \hat \rho_{0}^{j} \geq 1/j \}$ 
that is easily verified to satisfy $\hat \mu_{0}^{j}(K_{j}) \geq 1-1/j$ and define 
$$
\bar \mu_{0}^{j} : = \hat \mu_{0}^{j}\llcorner_{K_{j}}/
\hat \mu_{0}^{j}(K_{j}).
$$
The difference between $\Ent(\bar  \mu_{0}^{j}|\mm_{j})$ 
and  $\Ent(\hat \mu_{0}^{j}|\mm_{j})$ is controlled  (see \eqref{E:decompo}) by 
$$
\int_{\{\hat \rho_{0}^{j} \leq 1/j\}} |\hat \rho_{0}^{j} \log (\hat \rho_{0}^{j}) |\mm_{j} \leq \frac{1}{j} \log(j).
$$
Hence $\bar  \mu_{0}^{j}$ still verifies all the properties we have checked for $\hat \mu_{0}^{j}$. Finally  
it is only left to restrict $\bar \mu_{0}^{j}$  to 
$I_{\bar \ll}(x_{1}^{j})$.
From \eqref{E:convergence}, adopting a diagonal argument, 
we also obtain that 
$\bar \mu_{0}^{j}(I_{\bar \ll}(x_{1}^{j})) \geq 1 - 1/j$.
Hence we define 
$$
\mu_{0}^{j} : =  \bar \mu_{0}^{j}\llcorner_{I_{\bar \ll}(x_{1}^{j})}/
\bar \mu_{0}^{j}(I_{\bar \ll}(x_{1}^{j})).
$$
Again the difference between $\Ent(\bar  \mu_{0}^{j}|\mm_{j})$ 
and  $\Ent( \mu_{0}^{j}|\mm_{j})$ is controlled  (see \eqref{E:decompo}) by 
$$
\int_{X\setminus I_{\bar \ll}(x_{1}^{j})} \hat \rho_{0}^{j} \, |\log (\hat \rho_{0}^{j}) | \, \mm_{j} \leq \log(j)  \, 
\bar \mu_{0}^{j}(X\setminus I_{\bar \ll}(x_{1}^{j})) 
\leq \frac{1}{j} \log(j).
$$
Thus \eqref{eq:Entmu0inftyoj} is proved in this case.
\\

\textbf{Step 1b}.   $\mu_{0}^{\infty}=\rho_{0}^{\infty} \mm_{\infty}\in \Prob_{c}(X_{\infty})$ has density $\rho_{0}^{\infty}\in L^{\infty}(\mm_{\infty})$,  and $x_{1}^{\infty}\in X_{\infty}$ is such that  $x\ll_{\infty} x_{1}^{\infty}$ for $\mu_{0}^{\infty}$-a.e. $x\in X_{\infty}$.

\noindent
For $n\in \N$ define 
$\mu_{0,n}^{\infty}:= \bar{c}_{n}\mu_{0}^{\infty} \llcorner \{\tau_{\infty}(\cdot, x^{\infty}_{1})\geq \frac{1}{n}\} \in \Prob(X_{\infty})$, 
where $\bar{c}_{n}\downarrow 1$ are the normalising constants. By the continuity of $\tau_{\infty}$, 
it is readily seen that $\supp \, \mu_{0,n}^{\infty} \Subset I^{-}_{\ll_{\infty}}(x_{1}^{\infty})$. Moreover
 \[
\lim_{n\to\infty}\Ent(\mu_{0,n}^{\infty}|\mm_{\infty})=
 \Ent(\mu_{0}^{\infty}|\mm_{\infty}),\qquad \lim_{n\to\infty}W_2^{(\bar X, \bar \sfd)}(\mu_{0,n}^{\infty},\mu_{0}^{\infty})=0.
\]
Then apply Step 1a to $\mu_{0,n}^{\infty}$ and conclude with a diagonal argument.
\\

\textbf{Step 1c}. General case.
 If $\rho_{0}^{\infty}$ is not bounded, for $k\in\N$ define $\rho_{0,k}^{\infty}:=\bar{c}_k \min\{\rho_{0}^{\infty},k\}$, $\bar{c}_k \downarrow 1$ being such that  $\mu_{0,k}^{\infty}:=\rho^{\infty}_{0,k}\ \mm_\infty\in\Prob(X_{\infty})$. Clearly, it holds
\[
\lim_{k\to\infty}\Ent(\mu_{0,k}^{\infty}|\mm_{\infty})=  \Ent(\mu_{0}^{\infty}|\mm_{\infty}),\qquad \lim_{k\to\infty}W_2^{(\bar X, \bar \sfd)}(\mu_{0,k}^{\infty},\mu_{0}^{\infty})=0.
\]
Then apply Step 1b to $\mu_{0,k}^{\infty}$ and conclude with a diagonal argument.
\\

\textbf{Step 2}. Conclusion. 
\\
Using the  assumption  that  $(X_{j},\sfd_{j}, \mm_{j}, \ll_{j}, \leq_{j}, \tau_{j})$ satisfies  $\TMCP^{e}(K,N)$,  we obtain  an $\ell_{p}$-geodesic $(\mu_{t}^{j})_{t\in [0,1]}$ from $\mu_{0}^{j}$ to  $\mu_{1}^{j}:=\delta_{x_{1}^{j}}$
such that 
\begin{equation}\label{eq:TMCP(KN)Xj}
U_{N}(\mu^{j}_{t}|\mm_{j}) \geq \sigma^{(1-t)}_{K/N} \left( \|\tau_j(\cdot, x_{1}^{j}) \|_{L^{2}(\mu^{j}_{0})} \right)\, U_{N}(\mu^{j}_{0}|\mm_{j}), \quad \forall t\in [0,1).
\end{equation}
From \eqref{eq:Entmu0inftyoj}, it is readily seen that $\mu^{j}_{0}\otimes \delta_{x_{1}^{j}}\to \mu^{\infty}_{0}\otimes \delta_{x_{1}^{\infty}}$ narrowly in $\bar{X}^{2}$.  Thus, recalling that $\bar \tau$ is continuous and bounded (on $E$), we infer that
\begin{equation}\label{eq:limittauxxj1}
 \|\tau_j(\cdot, x_{1}^{j}) \|_{L^{2}(\mu^{j}_{0})}^2= \int_{\bar X^{2}} \bar \tau(x,y)^{2} \, \mu^{j}_{0} \otimes  \delta_{x_{1}^{j}} (\dd x\dd y) \underset{j\to \infty}{\longrightarrow}   \int_{\bar X^{2}} \bar \tau(x,y)^{2} \, \mu^{\infty}_{0} \otimes  \delta_{x_{1}^{\infty}} (\dd x\dd y)= \|\tau_\infty(\cdot, x_{1}^{\infty}) \|^2_{L^{2}(\mu^{\infty}_{0})}.   
\end{equation}
Since by construction  $\cup_{j\in \N}\, \supp \, \mu^j_0 \bigcup (x^j_1)_{j\in \N} \subset E \Subset \bar{X}$ and $\bar{X}$ is globally hyperbolic, then there exists a compact subset $\mathcal{K}\Subset \bar{X}$ such that
\begin{equation}\label{eq:suppmujtCompact}
\bigcup_{t\in [0,1], j\in \N} \supp\, \mu^j_t \subset \mathcal{K} \Subset \bar{X}.
\end{equation}
Let $\eta^j\in \Prob(C([0,1], \bar{X})$ be the $\ell_p$-optimal dynamical plan representing the $\ell_p$-geodesic $(\mu^{j}_{t})_{t\in [0,1]}$.
Prokhoroff Theorem combined with  \eqref{eq:TGeoKcompact} and \eqref{eq:suppmujtCompact} yields that the sequence $(\eta^j)_{j\in \N}\subset \Prob(C([0,1], \bar{X}))$ is pre-compact in narrow topology, i.e.\;there exists $\eta^\infty\subset \Prob(C([0,1], \bar{X}))$ such that, up to a subsequence,  $\eta^j\to \eta^\infty$ narrowly. 

We now claim that $\eta^\infty$ is an $\ell^p$-optimal dynamical plan from $\mu^\infty_0$ to $\mu^\infty_1=\delta_{x^\infty_1}$. Indeed, by the continuity of $\ee_t: C([0,1], \bar{X})\to \bar{X}$, it is easily seen that 
\begin{equation}\label{eq:mujtomuinftyNarrow}
\mu^j_t=(\ee_t)_\sharp \eta^j \underset{j\to \infty}{\longrightarrow} (\ee_t)_\sharp \eta^\infty=:\mu^\infty_t \quad \text{narrowly, for every $t\in [0,1]$.}
\end{equation}
But since  $(\ee_i)_\sharp \eta^j= \mu^j_i$ converges narrowly to $\mu^\infty_i$ for $i=0,1$, by the uniqueness of the narrow limit we infer that $(\ee_i)_\sharp \eta^\infty=\mu^\infty_i$, for $i=0,1$. Recalling that by construction $\mu^\infty_1=\delta_{x^\infty_1}$ and that $\mu^\infty_0 \otimes \delta_{x^\infty_1}$  is the unique coupling between $\mu^\infty_0$ and $\delta_{x^\infty_1}$ (which is hence optimal), it follows that $(\ee_0,\ee_1)_\sharp \eta^\infty$ is an $\ell_p$-optimal coupling and thus $\eta^\infty$ is an $\ell_p$-optimal dynamical plan, proving the claim.

Finally, the joint upper semicontinuity of $U_{N}$ under narrow 
convergence \eqref{eq:Entljointsc} together with \eqref{eq:mujtomuinftyNarrow}  yields:
\begin{equation}\label{eq:GammaliminfjtMCP}
U_{N}(\mu^{\infty}_{t}|\mm_{\infty}) \geq \limsup_{j\in \N} U_{N}(\mu^{j}_{t}|\mm_{j}), \quad \forall t\in [0,1],
\end{equation}
obtaining in particular that $(\mu^{\infty}_{t})_{t\in [0,1]} \subset  \Prob(X_{\infty})$. 
The combination of \eqref{eq:Entmu0inftyoj}, \eqref{eq:TMCP(KN)Xj}, \eqref{eq:limittauxxj1} and  \eqref{eq:GammaliminfjtMCP} gives that 
$$
U_{N}(\mu^{\infty}_{t}|\mm_{\infty}) \geq \sigma^{(1-t)}_{K/N} \left(\|\bar \tau(\cdot, x_{1}^{\infty})\|_{L^{2}( \mu^{\infty}_{0})}\right) \, U_{N}(\mu^{\infty}_{0}|\mm_{\infty}), \quad \forall t\in [0,1).
$$
as desired.
\end{proof}

In the next theorem we show that if a sequence of $\TCD^{e}_{p}(K,N)$ Lorentzian spaces converge to a limit  Lorentzian space, then the latter is $\wTCD^{e}_{p}(K,N)$. 
The same observation of Remark \ref{R:TMCPandgeo} 
will be valid for the next theorem.
 
\begin{theorem}[Weak stability of  $\TCD^{e}_{p}(K,N)$]\label{thm:StabTCD}
Let  $\{(X_{j},\sfd_{j}, \mm_{j}, \star_{j}, \ll_{j}, \leq_{j}, \tau_{j})\}_{j\in \N\cup\{\infty\}}$ be a sequence of pointed measured Lorentzian geodesic spaces  satisfying the following properties:
\begin{enumerate}
\item There exists a globally hyperbolic  Lorentzian geodesic space $(\bar X, \bar \sfd,  \overline \ll, \bar \leq, \bar  \tau)$ such that each $(X_{j},\sfd_{j}, \mm_{j}, \ll_{j}, \leq_{j}, \tau_{j})$, $j\in \N\cup\{\infty\}$, is isomorphically embedded in it (as in 1. of Theorem \ref{thm:StabTMCP}).
\item The measures $(\iota_{j})_{\sharp} \mm_{j}$  converge to $(\iota_{\infty})_{\sharp} \mm_{\infty}$ weakly in duality with $C_{c}(\bar X)$ in $\bar X$, i.e. \eqref{eq:weakconv} holds.
\item  Convergence of the reference points: $\iota_j(\star_j)\to \iota_\infty(\star_\infty)$ in $\bar{X}$.
\item  For every compact subset $\mathcal{K}\Subset \bar{X}$, the pre-compactness condition \eqref{eq:TGeoKcompact} holds. 
\item There exist $p\in (0,1), K\in \R, N\in (0,\infty)$ such that   $(X_{j},\sfd_{j}, \mm_{j}, \ll_{j}, \leq_{j}, \tau_{j})$  satisfies  $\TCD^{e}_{p}(K,N)$, for each $j\in \N$.
\end{enumerate}
Then  the limit space $(X_{\infty},\sfd_{\infty}, \mm_{\infty}, \ll_{\infty}, \leq_{\infty}, \tau_{\infty})$  satisfies the $\wTCD^{e}_{p}(K,N)$ condition.
\end{theorem}

\begin{proof}
Without affecting generality, we will identify $X_{j}$ with its isomorphic image $\iota_{j}(X_{j})\subset \bar X$ and the measure $\mm_{j}$ with $(\iota_{j})_{\sharp} \mm_{j}$, for each $j\in \N\cup \{\infty\}$.

Fix  $\mu_{0}^{\infty}, \mu_{1}^{\infty} \in  \Dom(\Ent(\cdot|\mm_{\infty})) \cap \Prob_{c}(X_{\infty})$ strongly timelike $p$-dualisable, i.e. such that there exists 
${\pi}_{\infty}\in \Pi_{\leq_{\infty}}^{p\text{-opt}}(\mu_{0}^{\infty}, \mu_{1}^{\infty})$ 
with $\pi_{\infty}( \{\tau_{\infty}>0\})=1$
and there exists a measurable $\ell^{p}$-cyclically monotone set 
$$
\Gamma\subset (X_{\infty}^{2})_{\ll_{\infty}} \cap (\supp \, 
\mu_{0}^{\infty} \times \supp \,\mu_{1}^{\infty})
$$ 
such that a coupling  
$\pi\in \Pi_{\leq_{\infty}}(\mu_{0}^{\infty},\mu_{1}^{\infty})$ 
is $\ell_{p}$-optimal if  and only if $\pi$ is concentrated on $\Gamma$.
Since $\mu_{0}^{\infty}, \mu_{1}^{\infty}$ 
have compact support and $\bar X$ is globally hyperbolic, 
we can  restrict all the arguments to a large compact subset of $E\Subset \bar X$ such that
\begin{itemize} 
\item $J^{+}_{\bar{X}}(\supp \, \mu_{0}^{\infty})\cap J^{-}_{\bar{X}}(\supp \, \mu_{1}^{\infty}) \subset E$;
\item $\mm_{\infty}(E)=\lim_{j\to \infty}\mm_j(E)$. 
\end{itemize}  
Thus,  $\tilde{\mm}_j:=\mm_j(E)^{-1} \mm_j\llcorner E$ are probability measures supported in the compact subset $E\Subset \bar{X}$ and narrowly converge to $\tilde{\mm}_\infty:=\mm_\infty(E)^{-1} \mm_\infty\llcorner E$. With a slight abuse of notation, for simplicity, we will write $\mm_j$ for  $\tilde{\mm}_j$, $j\in \N\cup\{\infty\}$.

\smallskip
\textbf{Step 1:} We prove that, up to a subsequence, for every $j\in \N$ 
there exists $(\mu^{j}_{0}, \mu^{j}_{1}) \in \Prob(X_{j})^{2}$ timelike $p$-dualisable  such that
\begin{equation}\label{eq:claimStep1}
\begin{split}
&\text{$\mu^{j}_{0}\to \mu^{\infty}_{0}, \,  \mu^{j}_{1}\to \mu^{\infty}_{1}$ narrowly in $\bar X$}
\text{ and $\ell_{p}(\mu^{j}_{0}, \mu^{j}_{1})\to \ell_{p}(\mu^{\infty}_{0}, \mu^{\infty}_{1})$ as $j\to \infty$}.
\end{split}
\end{equation}
Since $E\Subset \bar X$ is compact, the narrow convergence $\mm_{j}\to \mm_{\infty}$ implies that  $\mm_{j}\to \mm_{\infty}$ in $W_{2}^{(\bar X, \bar \sfd)}$. Let 
\begin{equation}\label{eq:defggammaj}
\text{$\ggamma_{j}\in \Pi(\mm_{\infty}, \mm_{j})$ be a  $W_{2}^{(\bar X, \bar \sfd)}$-optimal coupling.} 
\end{equation}
Thanks to the next Lemma \ref{lem:approxpiinftylinfty}, we can approximate $\pi_{\infty}$ by
\begin{equation}\label{eq:defpiinftyn}
\begin{split}
&\pi_{\infty,n}=\rho_{\infty,n} \mm_{\infty} \otimes \mm_{\infty},\ \rho_{\infty,n}\in L^{\infty}(\mm_{\infty}\otimes\mm_{\infty}), \;\pi_{\infty,n}(\{\bar \tau>0\})=1,\;\pi_{\infty,n}\to \pi_{\infty} \text{ narrowly} \\
& \lim_{n\to \infty} \Ent( (P_{1})_{\sharp} \pi_{\infty,n}|\mm_{\infty}) = \Ent(\mu^{\infty}_{0}|\mm_{\infty}), \quad    \lim_{n\to \infty} \Ent( (P_{2})_{\sharp} \pi_{\infty,n}|\mm_{\infty}) = \Ent(\mu^{\infty}_{1}|\mm_{\infty}).
\end{split}
\end{equation} 
Define then
\begin{equation}\label{eq:deftildepijn}
 \tilde \pi_{j,n}(\dd x_{1} \dd x_{2} \dd x_{3} \dd x_{4}):= \rho_{\infty,n}(x_{1}, x_{3}) \, \ggamma_{j}(\dd x_{1} \, \dd x_{2}) \otimes  \ggamma_{j}(\dd x_{3} \, \dd x_{4}), \quad \pi_{j,n}:=(P_{24})_{\sharp} \tilde \pi_{j,n},
 \end{equation}
and observe that $\pi_{j,n}\ll \mm_{j}\otimes \mm_{j}$ and that $\pi_{j,n} \to \pi_{\infty,n}$ 
narrowly as $j\to \infty$. 
By lower semicontinuity over open subsets, 
we have that   $\liminf_{j\to \infty}\pi_{j,n}(\{\bar \tau>0\}) \geq   \pi_{\infty,n}(\{\bar \tau>0\})=1$.
Thus, calling  $c_{j,n}:= 1/ \pi_{j,n} (\{\bar \tau > 0 \})$ for $j$ large enough, it holds that 
\begin{equation}\label{eq:convtauppij}
\text{$\pi_{j,n}':= c_{j,n}  \pi_{j,n} \llcorner \{\bar \tau >0 \}  \to \pi_{\infty,n}$ narrowly \; and \;  $\lim_{j\to \infty} \int \bar{\tau}^{p} \, \pi_{j,n}' = \int \bar{\tau}^{p}\, \pi_{\infty,n}>0$}
\end{equation}
 by Lemma \ref{lem:UnifIntConv}. Let 
 \begin{equation}\label{eq:mu'ijn}
 \text{$ (\mu')_{0}^{j,n}:=(P_{1})_{\sharp} \pi'_{j,n}$, $ (\mu')_{1}^{j,n}:=(P_{2})_{\sharp} \pi'_{j,n}$}
 \end{equation}
and notice that $\ell_{p}\big((\mu')_{0}^{j,n},(\mu')_{1}^{j,n}\big)\in (0,\infty)$. 
Let
\begin{equation}\label{eq:defpijn''}
\pi_{j,n}''\in \Pi_{\bar \leq}^{p\text{-opt}} \big((\mu')_{0}^{j,n},(\mu')_{1}^{j,n}\big) 
\end{equation}
be an $\ell_{p}$-optimal coupling (whose existence is ensured by Proposition \ref{prop:ExMaxellp}). 
\\

Combining \eqref{eq:convtauppij} with Lemma \ref{lem:tightXxX}, with Prokhorov Theorem \ref{thm:prok} and with the causal closeness of $\bar X$  we deduce that there exists $\hat \pi_{\infty,n}\in \Pi_{\leq}( (P_{1})_{\sharp}\pi_{\infty,n}, (P_{2})_{\sharp} \pi_{\infty,n})$ such that, up to a subsequence,  $\pi_{j,n}''\to \hat \pi_{\infty,n}$ narrowly as $j\to \infty$. Repeating once more the tightness argument,  we deduce that there exists $\hat \pi_{\infty}\in \Pi_{\leq}(\mu^{\infty}_{0}, \mu^{\infty}_{1})$ such that, up to a subsequence,  $\hat \pi_{n,\infty}\to \hat \pi_{\infty}$ narrowly as $n\to \infty$.
We conclude that there exist sequences $(n_{k}), (j_{k})$ such that
\begin{equation}\label{eq:pi''nktohatpiinfty}
\pi''_{j_{k},n_{k}}\to  \hat \pi_{\infty} \text{ narrowly and }  \int  \bar{\tau}^{p}  \, \hat \pi_{\infty}  =\lim_{k\to \infty} \int  \bar{\tau}^{p}  \, \pi_{j_{k},n_{k}}'' \geq
\int \bar{\tau}^{p}\, \pi_{\infty},
\end{equation}
where the last inequality follows from Lemma \ref{lem:UnifIntConv}, \eqref{eq:defpiinftyn}, \eqref{eq:convtauppij} and the optimality of $\pi_{j_{k},n_{k}}''$.  
Combining \eqref{eq:pi''nktohatpiinfty} with the fact that ${\pi}_{\infty}\in \Pi_{\leq_{\infty}}^{p\text{-opt}}(\mu_{0}^{\infty}, \mu_{1}^{\infty})$, we get that
$ \hat \pi_{\infty} \in \Pi_{\leq_{\infty}}^{p\text{-opt}}(\mu_{0}^{\infty}, \mu_{1}^{\infty})$ as well.
\\ Since by assumption $(\mu^{\infty}_{0},  \mu^{\infty}_{1})$ is strongly timelike $p$-dualisable, we infer that $\hat \pi_{\infty}\{\bar \tau>0\}=1$. 
Thus  $\liminf_{k\to \infty}\pi_{j_{k},n_{k}}''(\{\bar \tau>0\}) \geq \hat \pi_{\infty}(\{\bar \tau>0\})=1$.
For $k$ large enough, set  
\begin{equation}\label{eq:defpij}
 \text{$c_{k}'':= 1/\pi_{j_{k}, n_{k}}'' (\{\bar \tau > 0 \})$,\; $\pi_{k}=c_{k}'' \pi_{j_{k}, n_{k}}''  \llcorner{\{\bar \tau >0\} }$, \; ${\mu}_{0}^{k}:=(P_{1})_{\sharp} \pi_{k}$, ${\mu}_{1}^{k}:=(P_{2})_{\sharp} \pi_{k}$}  
 \end{equation}
 and notice that
 \begin{equation}\label{eq:pijtopiinftynarrow}
 \pi_{k}\to \hat \pi_{\infty}, \quad {\mu}_{0}^{k} \to {\mu}_{0}^{\infty}, \quad {\mu}_{1}^{k} \to {\mu}_{1}^{\infty} \text{ narrowly}.
 \end{equation}
Since the restriction of an optimal coupling is optimal (Lemma \ref{L:restriction}), it follows that  $\pi_{k} \in \Pi_{\bar \leq}^{p\text{-opt}}({\mu}_{0}^{k}, {\mu}_{1}^{k})$ and by construction $ \pi_{k}(\{\bar \tau >0\}=1$.
We conclude that $({\mu}_{0}^{k}, {\mu}_{1}^{k})$  is timelike $p$-dualisable by $\pi_{k}$ and $\ell_{p}(\mu^{k}_{0}, \mu^{k}_{1})\to \ell_{p}(\mu^{\infty}_{0}, \mu^{\infty}_{1})$. 
Up to renaming the indices, the claim \eqref{eq:claimStep1} follows.
\\

\textbf{Step 2}. We prove that the sequences $(\mu_{0}^{j}), (\mu_{1}^{j})$ 
constructed in Step 1  satisfy:
 \begin{equation}\label{eq:claimStep2}
\limsup_{j\to \infty}  \Ent(\mu^{j}_{0}|\mm_{j}) \leq \Ent(\mu^{\infty}_{0}|\mm_{\infty}), \quad  \limsup_{j\to \infty}  \Ent(\mu^{j}_{1}|\mm_{j}) \leq \Ent(\mu^{\infty}_{1}|\mm_{\infty}).
  \end{equation}
We divide this step into two substeps. Recall the definition \eqref{eq:deftildepijn} of $\tilde \pi_{j,n}(\dd x_{1} \dd x_{2} \dd x_{3} \dd x_{4})$ and set $$\mu^{j,n}_{0}:= (P_{2})_{\sharp} \tilde \pi_{j,n}, \quad \mu^{j,n}_{1}:= (P_{4})_{\sharp} \tilde \pi_{j,n}.$$

\textbf{Step 2a}. We first prove  that:
 \begin{equation}\label{eq:claimStep2a}
 \Ent(\mu^{j,n}_{0}|\mm_{j}) \leq  \Ent( (P_{1})_{\sharp} \pi_{\infty,n}|\mm_{\infty}), \quad  \Ent(\mu^{j,n}_{1}|\mm_{j}) \leq   \Ent( (P_{2})_{\sharp} \pi_{\infty,n}|\mm_{\infty}), \quad \forall j,n\in \N.
  \end{equation}
We give the argument for the former in \eqref{eq:claimStep2a}, the latter being completely analogous.   The explicit expression  \eqref{eq:deftildepijn} of $\tilde \pi_{j,n}(\dd x_{1} \dd x_{2} \dd x_{3} \dd x_{4})$ combined with \eqref{eq:defpiinftyn} and with Fubini's Theorem permits to write
  \begin{align}
 &  (P_{1})_{\sharp} \pi_{\infty,n}= \rho^{\infty,n}_{0} \mm_{\infty}; \quad  \rho^{\infty,n}_{0}(x_{1})=\int_{X} \rho_{\infty,n}(x_{1},x_{3}) \, \mm_{\infty}(\dd x_{3}), \quad   (P_{1})_{\sharp} \pi_{\infty,n}\text{-a.e. } x_{1}\in X_{\infty}; \nonumber \\
&  \mu^{j,n}_{0}= \rho^{j,n}_{0} \mm_{j}; \quad  \rho^{j,n}_{0}(x_{2})=\int_{X}\left( \int_{X^{2}} \rho_{\infty,n}(x_{1},x_{3})\,  \ggamma_{j}(\dd x_{3} \dd x_{4})\right) (\ggamma_{j})_{x_{2}}(\dd x_{1}), \quad  \mu^{j,n}_{0}\text{-a.e. } x_{2}\in X_{j}; \label{eq:muj,n0=rhoj,n0}
  \end{align}
  where $\{(\ggamma_{j})_{x_{2}}\}$ is the disintegration of  $\ggamma_{j}$ with respect to $P_{2}$. Since $u(t)= t \log t$ is convex on $[0,\infty)$, Jensen's inequality gives:
  \begin{align*}
  \Ent(\mu^{j,n}_{0}|\mm_{j})&=\int_{X} u\big( \rho^{j,n}_{0}(x_{2})\big) \, \mm_{j}(\dd x_{2})\\
  &\leq \int_{X} \int_{X} u \left( \int_{X^{2}} \rho_{\infty,n}(x_{1},x_{3})\,  \ggamma_{j}(\dd x_{3} \dd x_{4})\right)  (\ggamma_{j})_{x_{2}}(\dd x_{1}) \, \mm_{j}(\dd x_{2}) \\
  &= \int_{X^{2}} u \big( \rho^{\infty,n}_{0}(x_{1}) \big)\,  \ggamma_{j}(\dd x_{1}\dd x_{2})=  \int_{X} u \big( \rho^{\infty,n}_{0}(x_{1}) \,  \mm_{\infty}(\dd x_{1})\\
  &=\Ent(  (P_{1})_{\sharp} \pi_{\infty,n}|\mm_{\infty}).
  \end{align*}

 \textbf{Step 2b}. We prove that the sequences 
 $(\mu^{j}_{0}),\,  (\mu^{j}_{1})$ constructed in Step 1 satisfy:
\begin{equation}\label{eq:claimStep2b}
\limsup_{j\to \infty} \Ent(\mu^{j}_{i}|\mm_{j})\leq 
\limsup_{k\to \infty}\Ent(\mu^{j_{k}, n_{k}}_{i}|\mm_{j_{k}}), \quad i=0,1.
\end{equation}
We give the argument for $i=0$, the case $i=1$ being completely analogous.   From the construction of $\mu^{k}_{0}$ in Step 1 (see \eqref{eq:deftildepijn}, \eqref{eq:convtauppij}, \eqref{eq:mu'ijn}, \eqref{eq:defpijn''}, \eqref{eq:defpij}, see also \eqref{eq:muj,n0=rhoj,n0}) it is not hard to check that $\mu^{k}_{0}=\rho^{k}_{0}\mm_{j_{k}}$,   where $\rho^{k}_{0}\in L^{\infty}(\mm_{j_{k}})$ satisfies
\begin{equation}\label{eq:rhok0rhojn}
0\leq \rho^{k}_{0}\leq c_{k} \, \rho^{j_{k},n_{k}}_{0} \leq c_{k} \|\rho_{\infty,n_{k}}\|_{L^{\infty}(\mm_{\infty}\otimes \mm_{\infty})} \quad \forall k\in \N, \qquad c_{k}\to 1 \text{ as } k\to \infty.
\end{equation}
The fact that  $u(t):=t \log t$ is convex on $[0,\infty)$ and $u(0)=0$, easily yields 
$$u(t+h)-u(t)\geq u(h), \quad \forall t,h\in [0,\infty).$$
Thus, \eqref{eq:rhok0rhojn} combined with Jensen's inequality gives
\begin{equation}\label{eq:uck-1}
\begin{split}
\int u( c_{k} \rho^{j_{k},n_{k}}_{0}) \mm_{j_{k}} -\int   u(\rho^{k}_{0}) \, \mm_{j_{k}}  
& \geq    \int u( c_{k} \rho^{j_{k},n_{k}}_{0}- \rho^{k}_{0}) \, \mm_{j_{k}}  \\
& \geq  u\left(\int (c_{k}  \rho^{j_{k},n_{k}}_{0}- \rho^{k}_{0}) \mm_{j_{k}}  \right) \\
&=  u\left(c_{k}-1\right)\to 0 \text{ as $k\to \infty$}.
\end{split}
\end{equation}
The claim \eqref{eq:claimStep2b} follows immediately from  \eqref{eq:uck-1}. The claim   \eqref{eq:claimStep2} is a straightforward consequence of \eqref{eq:claimStep2a} combined with \eqref{eq:defpiinftyn} and \eqref{eq:claimStep2b}.
  \\

  \textbf{Step 3}. Passing to the limit in the $\TCD$ condition.
\\  For simplicity of presentation we give the argument for the $\TCD^{e}_{p}(0,N)$ condition, the one for general $K\in \R$ being analogous just a bit more cumbersome due to the distortion coefficients. 
 Since for each $j\in \N$ the pair $(\mu_{0}^{j}, \mu_{1}^{j}) \in \big(\Dom(\Ent(\cdot|\mm_{j}))\big)^{2} \subset \Prob(X_{j})^{2}$ is  timelike $p$-dualisable, 
the assumption that $(X_{j},\sfd_{j}, \mm_{j}, \ll_{j}, \leq_{j}, \tau_{j})$ satisfies the $\TCD^{e}_{p}(0,N)$ condition yields the existence of  an  $\ell_{p}$-geodesic $(\mu^{j}_{t})_{t\in [0,1]}$ such that 
 \begin{equation}\label{eq:UnmujtConc}
U_{N}(\mu^{j}_{t}| \mm_{j}) \geq (1-t) \, U_{N}(\mu^{j}_{0}| \mm_{j}) + t \, U_{N}(\mu^{j}_{1}| \mm_{j}), \quad \forall t\in [0,1], \;  \forall j\in \N.
\end{equation} 

Since by construction  
$$\bigcup_{j\in \N, i\in\{0,1\}}\, \supp \, \mu^j_i  \subset E \Subset \bar{X}$$
and $\bar{X}$ is globally hyperbolic, Proposition \ref{prop:GH->KGH} yields that there exists a compact subset $\mathcal{K}\Subset \bar{X}$ such that
\begin{equation}\label{eq:mujtCompactCD}
\bigcup_{j\in \N, t\in [0,1]} \supp\, \mu^j_t \subset \mathcal{K} \Subset \bar{X}.
\end{equation}
Let $\eta^j\in \Prob(C([0,1], \mathcal{K})$ be the $\ell_p$-optimal dynamical plan representing the $\ell_p$-geodesic $(\mu^{j}_{t})_{t\in [0,1]}$.
 Prokhoroff Theorem combined with \eqref{eq:TGeoKcompact} and \eqref{eq:mujtCompactCD}  yields that the sequence $(\eta^j)_{j\in \N}\subset \Prob(C([0,1], \bar{X}))$ is pre-compact in narrow topology, i.e.\;there exists $\eta^\infty\subset \Prob(C([0,1], \bar{X}))$ such that, up to a subsequence,  $\eta^j\to \eta^\infty$ narrowly. 

We now claim that $\eta^\infty$ is an $\ell^p$-optimal dynamical plan from $\mu^\infty_0$ to $\mu^\infty_1$. Indeed, by the continuity of $\ee_t: C([0,1], \bar{X})\to \bar{X}$, it is easily seen that 
\begin{align}
\mu^j_t=(\ee_t)_\sharp \eta^j & \underset{j\to \infty}{\longrightarrow} (\ee_t)_\sharp \eta^\infty=:\mu^\infty_t \quad \text{narrowly, for every $t\in [0,1]$,} \label{eq:mujtomuinftyNarrowTCD}\\
(\ee_0, \ee_1)_\sharp \eta^j & \underset{j\to \infty}{\longrightarrow} (\ee_0, \ee_1)_\sharp \eta^\infty \quad \text{narrowly.} 
\end{align}
By the uniqueness of the narrow limit we infer that $(\ee_i)_\sharp \eta^\infty=\mu^\infty_i$, for $i=0,1$. The causal closedness of $\bar{X}$ ensures that 
$$(\ee_0, \ee_1)_\sharp \eta^\infty\subset \Pi_\leq (\mu^\infty_0, \mu^\infty_1).$$ 

We claim that such a coupling is $\ell_p$-optimal. Recalling that $\bar \tau$ is continuous and bounded (on $E$) and that the plans $(\ee_0, \ee_1)_\sharp \eta^j$ are $\ell_p$-optimal, we infer that
\begin{equation}\label{eq:limitellptaup}
 \ell_p(\mu_0^j, \mu_1^j)= \int_{\bar X^{2}} \bar \tau(x,y)^{p} \, (\ee_0, \ee_1)_\sharp \eta^j (\dd x\dd y) \underset{j\to \infty}{\longrightarrow}   \int_{\bar X^{2}} \bar \tau(x,y)^{p} \, (\ee_0, \ee_1)_\sharp \eta^\infty (\dd x\dd y). 
\end{equation}
The $\ell_p$-optimality of $(\ee_0, \ee_1)_\sharp \eta^\infty$ follows by the combination of \eqref{eq:claimStep1} and \eqref{eq:limitellptaup}.

We are left to show that the $\ell_p$-geodesic $(\mu_t^\infty)_{t\in [0,1]}$ defined in \eqref{eq:mujtomuinftyNarrowTCD} satisfies the $\TCD^e_p(0,N)$ condition.
The joint upper semicontinuity of $U_{N}$ under narrow convergence \eqref{eq:Entljointsc}  yields:
\begin{equation}\label{eq:Gammaliminfjt}
U_{N}(\mu^{\infty}_{t}|\mm_{\infty}) \geq \limsup_{j\in \N} U_{N}(\mu^{j}_{t}|\mm_{j}), \quad \forall t\in [0,1].
\end{equation}
The combination of \eqref{eq:claimStep2}, \eqref{eq:UnmujtConc} and  \eqref{eq:Gammaliminfjt} gives that 
$$
U_{N}(\mu^{\infty}_{t}| \mm_{\infty}) \geq (1-t) \, U_{N}(\mu^{\infty}_{0}| \mm_{\infty}) + t \, U_{N}(\mu^{\infty}_{1}, \mm_{\infty}), \quad \forall t\in [0,1], 
$$
as desired.
\end{proof}

In the proof of Theorem \ref{thm:StabTCD} we made use of the following approximation result.

\begin{lemma}\label{lem:approxpiinftylinfty} 
Let  $(X,\sfd, \mm, \ll, \leq, \tau)$ be a  globally hyperbolic Lorentzian geodesic space. 
\\Let $\mu, \nu\in \Prob_{c}(X), \, \mu,\nu\ll \mm$  such that there exists $\pi\in \Pi^{p\text{-opt}}_{\leq}(\mu,\nu)$ with $\pi(\{\tau>0\})=1$.

Then there exists a sequence $(\pi_{n})\subset \Pi_{\ll}(X^{2})$ with the following properties:
\begin{enumerate}
\item $\pi_{n}= \rho_{n}  \mm\otimes \mm \ll \mm\otimes \mm$ with $\rho_{n}\in L^{\infty}(\mm\otimes \mm)$;
\item $\pi_{n}\to \pi$ in the narrow convergence; 
\item If $(P_{1})_{\sharp} \pi_{n} = : \mu_{n} = \rho_{\mu_{n}}\mm$ 
and $(P_{2})_{\sharp} \pi_{n} =: \nu_{n}=\rho_{\nu_{n}} \mm$, it holds that 
$\rho_{\mu_{n}} \to  \rho_{\mu}$ and $\rho_{\nu_{n}} \to  \rho_{\nu}$ 
in $L^{1}(\mm)$. Moreover,
\begin{equation}\label{eq:narrowEntmunnun}
\lim_{n\to \infty} \Ent(\mu_{n}|\mm) = \Ent(\mu|\mm), \quad   
\lim_{n\to \infty} \Ent(\nu_{n}|\mm) = \Ent(\nu|\mm).
\end{equation}
\end{enumerate}
\end{lemma}

\begin{proof}
{\bf Step 1.} Basic approximation by product measures.
\\
First,   we  cover $\{\tau > 0\}$ with 
a countable family of products of open subsets $A_{i}\times B_{i} \subset \{\tau > 0\}$: 
\begin{equation}\label{eq:covertau>0}
\{\tau > 0\} = \bigcup_{i\in \N} A_{i}\times B_{i}, \quad  \text{with $A_{i},B_{i}\subset X$  open subsets.}
\end{equation}
Let 
$\bar \pi_{n} : = \pi\llcorner_{\cup_{i\leq n} A_{i}\times B_{i}}$ and define $\pi_{n}:= \bar \pi_{n} - \bar \pi_{n-1}, \, \pi_{0}=0$.  We have the following 
decomposition:
$$
\pi = \sum_{n\in \N} \pi_{n},\qquad \pi_{n}\perp \pi_{m},\qquad  \, 
\pi_{n} (\{\tau > 0\}\setminus  A_{n}\times B_{n}) = 0 .
$$
For $n\geq 1$, consider 
$$
\mu_{n} : = (P_{1})_{\sharp}\pi_{n}, \quad 
\nu_{n} : = (P_{2})_{\sharp}\pi_{n}, \quad
\eta_{n} : = \mu_{n}\otimes \nu_{n}/\pi_{n}(X^{2}).
$$
Observe that $\mu_{n}(X \setminus A_{n}) = \nu_{n}(X \setminus B_{n}) =  
\eta_{n} (X^{2} \setminus \{\tau > 0\}) = 0$ and,  by linearity of projections, 
$\mu = \sum_{n\in \N} \mu_{n}, \quad \nu = \sum_{n\in \N} \nu_{n}$.
Notice moreover that the factor $1/\pi_{n}(X^{2})$ in the definition of 
$\eta_{n}$ is necessary to obtain that 
$(P_{1})_{\sharp}\eta_{n} = \mu_{n}$ and $(P_{2})_{\sharp}\eta_{n} = \nu_{n}$. 
Finally,  set $\eta: = \sum_{n\in \N} \eta_{n}$ and note  that
$$
\eta \in \Pi_{\leq}(\mu,\nu), 
\quad \eta(X^{2}_{\ll})  = 1, \quad \eta \ll \mm\otimes \mm.
$$
Notice moreover that, writing $\eta = \rho\, \mm\otimes \mm$, then 
$$
\rho(x,y) = \rho_{\mu_{n}}(x) \, \rho_{\nu_{n}}(y) \leq   \rho_{\mu}(x) \, \rho_{\nu}(y), \quad \eta\text{-a.e. } (x,y)\in A_{n}\times B_{n},
$$
where 
$\rho_{\mu}$ (respectively $\rho_{\nu},\, \rho_{\mu_{n}},\, \rho_{\nu_{n}}$) is the density of  $\mu$ (resp. $\nu,\, \mu_{n},\, \nu_{n}$)) with respect to $\mm$.
\\

{\bf Step 2.}  We  iterate the construction taking finer coverings of the form \eqref{eq:covertau>0}  to obtain a sequence $(\eta_{m})$ converging in the narrow topology to $\pi$. 
\\Fix any $f,g \in C_{b}(X)$ and observe that
\begin{equation}\label{eq:fgpietapin}
\begin{split}
&\int_{X^{2}} f(x)g(y) \pi(\dd x\dd y)  - \int_{X^{2}} f(x)g(y) \eta(\dd x\dd y)  \\
&\qquad \quad = \sum_{n=1}^{\infty} 
\int_{X^{2}} f(x)g(y) \pi_{n}(\dd x\dd y)  - 
\int_{X} f(x) \mu_{n}(\dd x)\int_{X} g(y)\frac{\nu_{n}(\dd y)}{\pi_{n}(X^{2})}.
\end{split}
\end{equation}
Since $\mu,\nu$ have compact support,  we have that  
$f,g$ are uniformly continuous on $\supp \, \mu \, \cup \, \supp \, \nu \Subset X$. Given any $(x_{n},y_{n}) \in A_{n}\times B_{n}$, we estimate
\begin{equation}\label{eq:fgpinfxngyn}
\left| 
\int_{X^{2}} f(x)g(y) \pi_{n}(\dd x\dd y) - f(x_{n})g(y_{n}) \pi_{n}(X^{2})
\right| \leq \ve \pi_{n}(X^{2}) (\| f \|_{\infty}  + \| g\|_{\infty}),
\end{equation}
where $\ve$ is the modulus of continuity of both $f$ and $g$ over $A_{n}$ 
and $B_{n}$ respectively. Analogously,
\begin{align}
&~ \left| 
\int_{X} f(x)\mu_{n}(\dd x)\int_{X} g(y) \frac{\nu_{n}(\dd y)}{\pi_{n}(X^{2})} 
- f(x_{n})g(y_{n}) \pi_{n}(X^{2})
\right| \nonumber \\ 
&\qquad =~ 
\left| 
\int_{X} f(x)\mu_{n}(\dd x)\int_{X} g(y) \frac{\nu_{n}(\dd y)}{\pi_{n}(X^{2})}
- f(x_{n})\mu_{n}(X)g(y_{n})  \right| \nonumber \\
 &\qquad \leq ~ \left| \int_{X} (f(x) - f(x_{n}))\mu_{n}(\dd x)\int_{X} g(y) \frac{\nu_{n}(\dd y)}{\pi_{n}(X^{2})} \right| 
+ |f(x_{n})\mu_{n}(X)| \left|  \int_{X} g(y) \frac{\nu_{n}(\dd y)}{\pi_{n}(X^{2})} - g(y_{n})  \right| \nonumber \\
&\qquad \leq~ \ve \mu_{n}(X) \| g\|_{\infty}
+ \ve \mu_{n}(X)\| f \|_{\infty}  
= \ve \pi_{n}(X) (\| g\|_{\infty}+\| f \|_{\infty}  ). \label{eq:fmungnun}
\end{align}
Combining  \eqref{eq:fgpietapin}, \eqref{eq:fgpinfxngyn} and \eqref{eq:fmungnun}, we obtain 
$$
\left| \int_{X^{2}} f(x)g(y) \pi  - \int_{X^{2}} f(x)g(y) \eta \right| 
\leq 2 \ve (\| g\|_{\infty}+\| f \|_{\infty} ),
$$
where $\ve$ is the modulus of continuity of both $f$ and $g$  over subsets of  $\supp\,  \mu\, \cup \, \supp \, \nu\Subset X$
 with diameter at most $\sup_{n\in \N} \max\{\diam(A_{n}), \diam(B_{n})\}$. 
\\ Then, considering finer and finer open coverings  
  \begin{equation*}
\{\tau > 0\} = \bigcup_{i\in \N} A_{i}^{m}\times B_{i}^{m},   \text{ with $A_{i}^{m},B_{i}^{m}\subset X$  open sets, }  \lim_{m\to \infty}\sup_{i\in \N} \max\{\diam(A^{m}_{i}), \diam(B^{m}_{i})\}= 0, 
\end{equation*}
 and the corresponding measures $\eta_{m}$ constructed in Step 1, it holds
$$
\eta_{m} \in \Pi_{\leq}(\mu,\nu),\quad \eta_{m}(X_{\ll}^{2}) = 1, 
\quad \eta_{m} \ll \mm \otimes \mm, \quad \eta_{m} \to \pi \text{ narrowly}. 
$$

{\bf Step 3.} Conclusion by truncation and dominated convergence Theorem. \\
Let $\eta_{m}=\rho_{m}\mm\otimes \mm$ be the sequence constructed in Step 2. For any $C>0$ define
$$
\eta_{m}^{C} : = \alpha_{C,m} \min\{\rho_{m}, C \} \; \mm\otimes\mm, 
$$
where  $\alpha_{C,m}$ is the normalization constant.
It is standard  to check that 
$\eta_{m}^{C} \to \eta_{m}$ narrowly as $C \to \infty$.
By a diagonal argument we obtain a sequence 
$\eta_{m}^{C_{m}} \to \pi$ narrowly for some $C_{m} \to \infty$.
Define 
$$
\mu_{m} : = (P_{1})_{\sharp} \eta_{m}^{C_{m}}, 
\quad 
\nu_{m} : = (P_{2})_{\sharp} \eta_{m}^{C_{m}}.
$$
Writing $\mu_{m} = \rho_{\mu,m} \mm$, it holds 
\begin{equation}\label{eq:rhomumleqrhomu}
\rho_{\mu,m}(x) = \alpha_{C_{m},m}\int_{X} \min\{\rho_{m}(x,y), C_{m}\} \mm(\dd y)
\leq \alpha_{C_{m},m}\int_{X} \rho_{m}(x,y) \mm(\dd y) = 
 \alpha_{C_{m},m} \rho_{\mu}(x), 
\end{equation}
where the last identity follows from $(P_{1})_{\sharp}\eta_{m} = \mu$ for 
any $m\in \N$.
Hence, by dominated convergence Theorem,  $\rho_{\mu,m}(x) / \alpha_{C_{m},m}$ is converging to $\rho_{\mu}(x)$ 
in the stronger $L^{1}(\mm)$ norm, as $m\to \infty$. 

For the last claim \eqref{eq:narrowEntmunnun}, without loss of generality we can assume $\Ent(\mu|\mm) <\infty$ (otherwise it is trivial).
From dominated convergence Theorem and \eqref{eq:rhomumleqrhomu}, we deduce that 
$\Ent(\mu_{m}|\mm) \to \Ent(\mu|\mm)$. 
\\To conclude the proof, it is enough to repeat the last arguments  also for $\nu$. 
\end{proof}

\begin{remark}
Recalling from Theorem \ref{thm:CharLorRic} that smooth  globally hyperbolic spacetimes of dimension $\leq N$ with timelike Ricci curvature bounded below by $K\in \R$  satisfy $\TCD^{e}_{p}(K,N)$,
 Theorem \ref{thm:StabTCD}  yields that their  limit spaces (in the sense of Therem \ref{thm:StabTCD}) satisfy  $\mathsf{wTCD}^{e}_{p}(K,N)$. 
\end{remark}

\subsection{Optimal maps in timelike non-branching $\mathsf{TMCP}^{e}(K,N)$ spaces}
\label{Ss:Nonbranching}

In this section we prove some results about existence of optimal transport maps  in timelike non-branching $\mathsf{TMCP}^{e}(K,N)$ spaces, from which we will deduce the uniqueness of $\ell_{p}$-geodesics (this section should be compared with \cite{CM3} where the analogous results were obtained for metric-measure spaces satisfying $\mathsf{MCP}(K,N)$ and essentially non-branching). 

\begin{lemma}\label{lem:findelta}
Let  $(X,\sfd, \mm, \ll, \leq, \tau)$ be a timelike non-branching,  globally hyperbolic, Lorentzian geodesic space satisfying $\mathsf{TMCP}^{e}(K,N)$ for some $p\in (0,1), K\in \R, N\in (0,\infty)$.
Let $\mu_{0}\in \Prob_{c}(X)$, with $\mu_{0}\in \Dom(\Ent(\cdot|\mm))$ and  $\mu_{1}$ be a finite convex combination of Dirac masses, 
i.e. $\mu_{1}:=\sum_{j=1}^{n} \lambda_{j} \delta_{x_{j}}$ for some $\{x_{j}\}_{j=1,\ldots,n}\subset X$ with $x_{i}\neq x_{j}$ for $i\neq j$, 
and  $\{\lambda_{j}\}_{j=1,\ldots,n}\subset (0,1]$ with $\sum_{j=1}^{n} \lambda_{j}=1$. Assume that there exists $\pi\in \Pi^{p\text{-opt}}_{\leq}(\mu_{0},\mu_{1})$ such that $\supp \, \pi \Subset \{\tau>0\}$.

Then $\pi$ is the unique element in  $ \Pi^{p\text{-opt}}_{\leq}(\mu_{0},\mu_{1})= \{\pi\} $ with $\supp \, \pi \Subset \{\tau>0\}$. Moreover such a $\pi$ is induced by a map $T$, i.e. $\pi=(\id, T)_{\sharp} \mu_{0}$ and 
$$
\ell_{p}(\mu_{0}, \mu_{1} )^{p} =\int_{X} \tau(x,T(x))^{p} \, \mu_{0}(\dd x).
$$
\end{lemma}

\begin{proof}
We first show that $\pi$ is induced by a map, the uniqueness will follow.

Consider the set
\begin{equation}\label{eq:defS}
S:=\{x\in X\,:\, \exists x_{i}\neq x_{j} \ \text{ with } (x,x_{i}),(x,x_{j})\in \supp\, \pi  \} \subset \supp\, \mu_0,
\end{equation}
and, since $\supp \, \mu_{0}$ is compact, $S$ is easily seen to be a closed set and therefore compact.
It will be enough to prove the stronger statement $\mu_{0}(S) = 0$. 

Suppose by contradiction $\mu_{0}(S) > 0$.  Since $\mu_{1}$ is a finite sum of Dirac masses, up to taking a smaller $S$ and up to relabelling the points $x_{j}$,  
we can assume the existence of
$$
T_{1},T_{2} : S \to X, \qquad \text{graph}(T_{1}), \ \text{graph}(T_{2}) \subset \supp\, \pi,
$$
both $\mu_{0}$-measurable with $T_{1}(x)=x_{1}$ and $T_{2}(x)=x_{2}$ for all $x \in S$, with $x_{1} \neq x_{2}$. 

Possibly restricting to a subset of $S$, still of positive
$\mm$-measure, we also assume that if $1/C \leq \rho_{0} \leq C$
over $S$, where $\rho_{0}$ is the density of $\mu_{0}$ 
with respect to $\mm$.
Thanks to Lemma \ref{L:restriction} the couplings 
\begin{equation}\label{eq:deftrunctho0S}
\frac{\chi_{S\times \{x_{1} \}}}{\rho_{0}\mm(S)} \pi, \quad 
\frac{\chi_{S\times \{x_{2} \}}}{\rho_{0}\mm(S)} \pi,
\end{equation}
are optimal. Hence,  with no loss of generality, we can redefine $\mu_{0} : = \mm\llcorner_{S}/\mm(S)$ and consider 
 $\eta^{1} \in {\rm OptGeo}_{\ell_{p}}(\mu_{0},\delta_{x_{1}})$ and $\eta^{2} \in {\rm OptGeo}_{\ell_{p}}(\mu_{0},\delta_{x_{2}})$ given by Proposition \ref{prop:ellqgeodCS}.

Necessarily $\supp\, \eta^1 \cap \supp\, \eta^2=\emptyset $; indeed for $i = 1,2$ it holds $\eta^{i}(\{\gamma \colon \gamma_{1}=x_{i} \}) = 1$ and by construction $x_{1}\neq x_{2}$. 
Thus, again by  Proposition \ref{prop:ellqgeodCS},  it holds
\begin{equation}\label{eq:eteta1perpeteta2}
(\ee_{t})_{\sharp} \eta^1 \perp (\ee_{t})_{\sharp} \eta^2 , \quad \forall t\in (0,1].
\end{equation}
The $\mathsf{TMCP}^{e}(K,N)$ condition \eqref{eq:defTMCP(KN)}  gives that (see \eqref{eq:MCPEnt}),  for $i = 1,2$, 
\begin{equation} \label{eq:rhotitau1}
\int \rho_{t}^{i} \, \log(\rho_{t}^{i})  \, \mm \leq  -\log(\mm(S)) 
-N \log(\sigma_{K/N}^{(1-t)}(\|\tau \|_{L^{2}((\ee_{0},\ee_{1})_{\sharp} \eta_{i})}) ), \quad \forall t \in [0,1), \, i = 1,2,
\end{equation}
where we have written $(\ee_{t})_{\sharp} \eta^{i} = \rho_{t}^{i}\mm$. 
By  Jensen's inequality \eqref{E:jensenE} we have
\begin{eqnarray*}
\int_{X} \rho_{t}^{i}\log(\rho_{t}^{i})  \, \mm  \geq\, -\log(\mm(\{ \rho_{t}^{i} > 0\}))
\end{eqnarray*}
which, combined with \eqref{eq:rhotitau1}, gives

\begin{equation} \label{eq: mmrhot>01}
\liminf_{t\to 0} \mm \left(\{ \rho_{t}^{i} > 0 \} \right) \geq\mm \left(S\right) = \mm \left(\{ \rho_{0}^{i} > 0 \} \right). 
\end{equation}
Denote now
\begin{align*}
E&:= \bigcup_{t\in [0,1], i=1,2} \supp \, (\ee_{t})_{\sharp} \eta^{i}  \\
S^{\varepsilon}_{E} &:= \{ y\in E \,:\, \tau(x,y) \leq \varepsilon \text{ for some } x\in S\} 
\end{align*}
and notice that, by global hyperbolicity and Proposition \ref{prop:GH->KGH}(i), $E$ (and thus also $S^{\varepsilon}_{E}$) is a compact subset of $X$. Moreover, by Dominated Convergence Theorem, we have $\lim_{\varepsilon\to 0} \mm(S^{\varepsilon}_{E})=\mm(S)$. In particular there exists $\varepsilon_{0}>0$ such that 
\begin{equation}\label{eq:3/21}
\mm(S^{\varepsilon_{0}}_{E}) \leq \frac{3}{2} \mm(S).
\end{equation}
We now claim that there  exists a small   $t_{0} > 0$, such that 
\begin{equation}\label{eq:tCharTau11}
\mm \left( \{ \rho_{t_{0}}^{1} > 0 \} \cap \{ \rho_{t_{0}}^{2} > 0 \} \right) > 0.
\end{equation}
To this aim notice that, by construction, for  $(\ee_{t})_{\sharp} \eta^{i}$-a.e. $x \in X$ there exists a timelike geodesic $\gamma \in \TGeo (X)$ such that $x=\gamma_{t}$, $\gamma_{0}\in S$, $\gamma_{1}=x^{i}$, $i=1,2$; in particular, for $t \in [0, \varepsilon_{0}]$ the measure $(\ee_{t})_{\sharp} \eta^{i}$ is concentrated on $S^{\varepsilon_{0}}_{E}$.
But then the combination of  \eqref{eq: mmrhot>01} and  \eqref{eq:3/21} implies that there exists $t_{0}\in (0, \varepsilon_{0})$ satisfying the claim \eqref{eq:tCharTau11}.
\\Observing that \eqref{eq:tCharTau11} contradicts \eqref{eq:eteta1perpeteta2}, we conclude that $\pi$ is induced by a map.

We now show that there exists a unique element $\pi\in \Pi^{p\text{-opt}}_{\leq}(\mu_{0},\mu_{1})$ satisfying $\supp \, \pi \Subset \{\tau>0\}$. Assume by contradiction that there exist $\pi_{1}, \pi_{2}\in \Pi^{p\text{-opt}}_{\leq}(\mu_{0},\mu_{1})$ satisfying $\supp \, \pi_{1}, \supp \, \pi_{2}  \Subset \{\tau>0\}$ with $\pi_{1}\neq \pi_{2}$. By the first part of the proof, we know that there exist maps $T_{1}, T_{2}:X\to X$ such that $\pi_{i}=(\id,T_{i})_{\sharp} \mu_{0}$; in particular $T_{1}\neq T_{2}$ on a $\mu_{0}$-nonnegligible subset. It is straightforward to check that $\pi:=\frac{1}{2} (\pi_{1}+\pi_{2})$ satisfies $\pi\in \Pi^{p\text{-opt}}_{\leq}(\mu_{0},\mu_{1}),  \supp \, \pi  \Subset \{\tau>0\}$ and that $\pi$ cannot be induced by a map. This contradicts the first part of the proof.
\end{proof}

\begin{proposition}\label{prop:AbsContmut}
Let  $(X,\sfd, \mm, \ll, \leq, \tau)$ be a timelike non-branching, globally hyperbolic, Lorentzian geodesic space satisfying $\mathsf{TMCP}^{e}(K,N)$  for some $p\in (0,1), K\in \R, N\in (0,\infty)$.
Let $\mu_{0},\mu_{1}\in \Prob_{c}(X)$, with $\mu_{0}\in \Dom(\Ent(\cdot|\mm))$. Assume that there exists $\pi\in \Pi^{p\text{-opt}}_{\leq}(\mu_{0},\mu_{1})$ such that $\supp \, \pi \Subset \{\tau>0\}$.

Then there exist $\hat \pi \in  \Pi_{\leq}^{p{\text- opt}}(\mu_{0},\mu_{1})$ 
with $\hat \pi(\{\tau > 0 \}) = 1$
and  an $\ell_{p}$-geodesic $(\mu_{t})_{t\in [0,1]}$ 
from $\mu_{0}$ to $\mu_{1}$  satisfying
\begin{equation}\label{eq:UNmutt[0,1)}
U_{N}(\mu_{t}|\mm) \geq \sigma_{K/N}^{(1-t)}(\| \tau \|_{L^{2}(\hat \pi)})\, U_{N}(\mu_{0}|\mm), \quad \forall t\in [0,1).
\end{equation}
In particular $\mu_{t} \ll \mm$ for all $t \in [0,1)$. Moreover any optimal dynamical 
plan $\eta$ representing $(\mu_{t})_{t\in [0,1]}$ is such that 
$(\ee_{0},\ee_{1})_{\sharp}\eta (\{\tau > 0 \}) = 1$.
\end{proposition}
\begin{proof}

{\bf Step 1.} Additionally assume $\supp \,\mu_{0} \times \supp \,\mu_{1} \subset \{\tau > 0\}$. \\
If $\supp \, \mu_{1}$ is made of finitely many points, an easier variant of the following arguments give the result (more precisely it is enough to take $n$ to be the number of points in $\supp \, \mu_{1}$ and stop at the end of Step 2). Thus without loss of generality we can assume that $\supp \, \mu_{1}$ contains infinitely many points.
\\Let $B_{i}\subset \supp \, \mu_{1}, \, i=1,\ldots,n$ be a finite Borel partition of $\supp \, \mu_{1}$ with $\mu_{1}(B_{i})>0$ for each $i$. For every $i$ pick a point $x_{1}^{i} \in B_{i}$ and define 
$$\bar \mu_{1} : = \sum_{i=1}^{n} a_{i} \delta_{x_{1}^{i}},$$
 where  $a_{i} := \mu_{1} (B_{i}) $. 
Since $\supp \,\mu_{0} \times \supp \, \mu_{1} \subset \{\tau > 0\}$, there exists 
$\bar{\pi}\in \Pi^{p\text{-opt}}_{\leq}(\mu_{0},\bar{\mu}_{1})$ such that $\supp \, \bar{\pi} \Subset \{\tau>0\}$. 
 Let $T: X\to X$ be the $\ell_{p}(\mu_{0}, \bar{\mu}_{1})$-optimal map associated to $\bar{\pi}$  by Lemma \ref{lem:findelta} and define  $A_{i}:=T^{-1}(x_{1}^{i})$. Observe that the sets  $C_{i}=A_{i}\times \{x_{1}^{i}\}$ satisfy $C_{i}\Subset \{\tau>0\}$ and form a  finite Borel partition of  $\supp \,\bar{\pi}$.
Set  $\bar{\pi}^{i} : = \frac{1}{a_{i}} \bar{\pi} \llcorner_{C_{i}}$ and
$$
\bar{\mu}_{0}^{i} : = (P_{1})_{\sharp}\bar{\pi}^{i}, \qquad \bar{\mu}_{1}^{i} : = (P_{2})_{\sharp}\bar{\pi}^{i}=\delta_{x_{1}^{i}}.
$$
Note that, by construction, $\mu_{0}=\sum_{i} a_{i} \bar{\mu}_{0}^{i}$ and
\begin{equation}\label{eq:mu0iperpmu0j}
\bar{\mu}_{0}^{i} \perp \bar{\mu}_{0}^{j}, \quad \forall i\neq j. 
\end{equation}
Noting that $\supp \, \bar \pi^{i} \Subset \{\tau>0\}$, the  $\mathsf{TMCP}^{e}(K,N)$  condition ensures that there exists an $\ell_{p}$-geodesic $(\bar \mu_{t}^{i})_{t\in [0,1]}$ from  $\bar \mu_{0}^{i}$ to  $\bar \mu_{1}^{i}$
satisfying 
\begin{equation}\label{eq:UNbarmuti}
 U_{N}(\bar \mu_{t}^{i}|\mm) \geq \sigma_{K/N}^{(1-t)}(\| \tau(\cdot,x_{1}^{i}) \|_{L^{2}(\bar \mu_{0}^{i})})\, U_{N}( \bar{\mu}_{0}^{i}|\mm), \quad \forall i=1,\ldots,n, \; t\in [0,1).
\end{equation}
{\bf Step 2.} Taking the logarithm of  \eqref{eq:UNbarmuti} and summing over $i$ (recalling that $\sum_{i} a_{i}=1$), we obtain
\begin{equation}\label{E:entropy}
-\frac{1}{N}\sum_{i}a_{i}\Ent(\bar \mu_{t}^{i}|\mm) 
\geq \sum_{i}a_{i}\log \left( \sigma_{K/N}^{(1-t)}(\| \tau(\cdot,x_{1}^{i}) \|_{L^{2}(\bar \mu_{0}^{i})}) \right)
-\frac{1}{N}\sum_{i}a_{i}\Ent(\bar \mu_{0}^{i}|\mm).
\end{equation}
Call $\bar{\eta}^{i}\in  {\rm OptGeo}_{\ell_{p}}(\bar{\mu}_{0}^{i}, \bar{\mu}^{i}_{1})$ the $\ell_{p}$-optimal dynamical plan representing the $\ell_{p}$-geodesic $(\bar{\mu}_{t}^{i})_{t\in [0,1]}$.
Since by construction $ \bar{\mu}^{i}_{1}=\delta_{x^{i}_{1}}$ and $x^{i}_{1}\neq x^{j}_{1}$ for $i\neq j$, it follows that $\supp \, \bar{\eta}^{i} \cap  \supp \, \bar{\eta}^{j}=\emptyset$ for $i\neq j$. Proposition \ref{prop:ellqgeodCS}  implies that
\begin{equation}\label{eq:muitmujt}
 \bar{\mu}^{i}_{t} \perp  \bar{\mu}^{j}_{t}, \quad \forall t\in (0,1), \quad \forall i\neq j.
\end{equation}
Calling  $\bar \mu_{t} : = \sum_{i}a_{i}\bar \mu_{t}^{i}$ and using \eqref{eq:mu0iperpmu0j}, \eqref{eq:muitmujt} it follows that 
\begin{equation}\label{eq:entmutentmuti}
\Ent(\bar \mu_{t}|\mm) = \Ent\left(\sum_{i}a_{i} \bar \mu_{t}^{i}|\mm\right) = \sum_{i}a_{i}\Ent(\bar \mu_{t}^{i}|\mm) + \sum_{i}a_{i}\log(a_{i}), \quad \forall t\in [0,1).
\end{equation}
Hence adding $- \frac{1}{N}\sum_{i}a_{i}\log(a_{i})$ to both sides of \eqref{E:entropy}, using \eqref{eq:entmutentmuti}, and the convexity of the function $(-\infty,\pi^{2})\ni k \to \log \sigma_{k}^{(t)}(1)$ 
(recall that $\sigma_{k}^{(t)}(\vartheta) = \sigma_{k\vartheta^{2}}^{(t)}(1)$)
we obtain 
$$
U_{N}(\bar \mu_{t}|\mm) \geq 
\sigma_{K/N}^{(1-t)}(\| \tau \|_{L^{2}(\bar \pi)}) \,  U_{N}(\mu_{0}|\mm) , \quad \forall t\in [0,1).
$$

{\bf Step 3.} Taking finer partitions of $\supp \, \mu_{1}$ we can construct a sequence $\{\bar{\mu}_{1}^{k}\}_{k\in \N} \subset \Prob_{c}(X)$ such that each $\bar{\mu}_{1}^{k}$ is a finite convex combination of Dirac masses, $\supp \, \bar{\mu}_{1}^{k} \subset \supp \, \mu_{1}$ for each $k$, and $\bar{\mu}_{1}^{k} \to \mu_{1}$ narrowly.  
We then invoke Theorem \ref{T:stability} to obtain another sequence, that we still denote by 
$\bar{\mu}_{1}^{k}$, that is converging to $\mu_{1}$ narrowly and which is absolutely continuous with the previous 
$\bar{\mu}_{1}^{k}$, hence still obtained as a finite convex combination of Dirac deltas.

For each $k$ let $(\bar \mu_{t}^{k}=(\ee_{t})_{\sharp} \bar{\eta}^{k})_{t\in [0,1]}$ be the $\ell_{p}$-geodesic from $\mu_{0}$ to $\bar{\mu}_{1}^{k}$ and 
$\bar \pi_{k} \in \Pi^{p\text{-opt}}_{\leq}(\mu_{0}, \bar{\mu}_{1}^{k})$ the optimal coupling constructed in {Step 2} satisfying 
\begin{equation}\label{eq:MCPbarmutk}
U_{N}(\bar \mu_{t}^{k}|\mm) \geq \sigma_{K/N}^{(1-t)}(\| \tau \|_{L^{2}((\ee_{0},\ee_{1})_{\sharp}\bar \eta^{k})})\,  
U_{N}(\mu_{0}|\mm) , \quad \forall t\in [0,1).
\end{equation}
Notice indeed that by construction, 
$(\ee_{0},\ee_{1})_{\sharp}\bar \eta^{k} = \bar \pi_{k}$.

We aim to construct a limit $\ell_{p}$-geodesic $(\mu_{t})_{t\in [0,1]}$ from $\mu_{0}$ 
to $\mu_{1}$ satisfying  \eqref{eq:UNmutt[0,1)}.
First of all notice that by global hyperbolicity and Proposition \ref{prop:GH->KGH}(i), 
 $$ \bar{K}:=\bigcup_{t\in [0,1]} \fI(\supp \, \mu_{0}, \supp \, \mu_{1},t)\Subset X$$
 is a  compact subset, see \eqref{eq:defI(A,B,t)},\eqref{I(K1K2)}. It is easily  seen that  
\begin{equation}\label{eq:barmutkSuppComp}
\supp\, \bar{\mu}_{t}^{k}\subset  \fI(\supp \, \mu_{0}, \supp \, \mu_{1},t)\subset  \bar{K}, \quad \forall t\in [0,1], \,k\in \N.
\end{equation}
From \eqref{eq:LengthW1mut} we deduce that 
$$
\sup_{k\in \N} {\rm L}_{W_{1}} \left((\bar{\mu}^{k}_{t})_{t\in [0,1]}\right) \leq \bar{C}<\infty.
$$
By the metric Arzel\'a-Ascoli Theorem we deduce that there exists a limit continuous curve $(\mu_{t})_{t\in [0,1]} \subset (\Prob(\bar{K}), W_{1})$ such that (up to a sub-sequence)
$ W_{1}\left(\bar{\mu}^{k}_{t}, \mu_{t} \right)\to 0 $ and thus $\bar{\mu}^{k}_{t} \to \mu_{t}$ narrowly, as $n\to \infty$.
Recalling the assumption 
of $\pi\in \Pi^{p\text{-opt}}_{\leq}(\mu_{0},\mu_{1})$ with $\supp \, \pi \Subset \{\tau>0\}$, 
Theorem \ref{T:stability} and Lemma  \ref{L:limsup}  yield that 
\begin{equation}\label{eq:ellqnuepst}
t\ell_{p}(\mu_{0},\mu_{1})= t \;\lim_{k\to \infty}  \ell_{p}(\mu_{0},  \bar{\mu}^{k}_{1})= \lim_{k\to \infty}  \ell_{p}(\mu_{0}, \bar{\mu}^{k}_{t})\leq \ell_{p}(\mu_{0},\mu_{t}).
\end{equation}
Thus, by reverse triangle inequality, the curve   $(\mu_{t})_{t\in [0,1]}$ is an $\ell_{p}$-geodesic from $\mu_{0}$ to $\mu_{1}$ 
and any narrow limit $\hat \pi$ of $(\bar \pi_{k})$ is $\ell_{p}$-optimal.
The upper-semicontinuity  \eqref{eq:UNusc} of $U_{N}(\cdot|\mm)$ in narrow topology yields
\begin{equation}\label{eq:UNuscbarmukt}
\limsup_{j\to \infty} \, U_{N}(\bar{\mu}_{t}^{k_{j}}|\mm) \leq U_{N} ( \mu_{t}|\mm), \quad \forall t\in [0,1).
\end{equation}
Combining \eqref{eq:MCPbarmutk} and \eqref{eq:UNuscbarmukt}  gives  the desired  \eqref{eq:UNmutt[0,1)}.
By the additional assumption $\supp \,\mu_{0} \times \supp \,\mu_{1} \subset \{\tau > 0\}$, it follows that $\hat \pi (\{\tau >0 \}) = 1$ and the analogous  property 
for any optimal dynamical plan representing the geodesic.

\smallskip

{\bf Step 4.} Removing the assumption $\supp\, \mu_{0} \times \supp \, \mu_{1} \subset \{\tau > 0\}$.\\
By assumption there exists $\pi\in \Pi^{p\text{-opt}}_{\leq}(\mu_{0},\mu_{1})$ such 
that $\supp \, \pi \Subset \{\tau>0\}$. Since $\supp \, \pi$ is compact, we can find finitely many products of open subsets $A_{i}\times B_{i} \Subset \{\tau > 0\}$, $i=1,\ldots, n$, such that 
$\supp \, \pi \subset \bigcup_{i=1}^{n} A_{i}\times B_{i}$. 
Argueing by induction over $n\in \N$ noticing that
$$
\bigcup_{i=1}^{n} A_{i}\times B_{i}=\left( \left( A_{n}\setminus \bigcup_{i=1}^{n-1} A_{i} \right) \times B_{n} \right) \cup \left(  \bigcup_{i=1}^{n-1}   (A_{i}\cap A_{n})\times (B_{i}\cup B_{n}) \right)  \cup \left(  \bigcup_{i=1}^{n-1} (A_{i}\setminus A_{n}) \times B_{i}\right),
$$
it is easy to see that we can assume  with no loss in generality 
that $A_{i}\cap A_{j} = \emptyset$, provided we admit $A_{i}$ to be Borel. In this way we obtain that $\supp \, \pi \subset \bigcup_{i=1}^{n} A_{i}\times B_{i}$ and $\supp \, \mu_{0} \subset \bigcup_{i=1}^{n} A_{i}$ are both finite Borel pairwise disjoint unions with $A_{i}\times B_{i}\Subset \{\tau>0\}$ for every $i=1,\dots, n$.  Up to taking a subset of indices, we can  assume that $\pi(A_{i}\times B_{i})>0$, for all $i=1,\dots, n$. 

Setting $\bar \pi_{i} : = \pi\llcorner_{ A_{i}\times B_{i}}$, we obtain the following  decomposition:
$$
\pi = \sum_{i \leq n} \bar \pi_{i},\qquad \bar \pi_{i}\perp \bar \pi_{j}\; \text{ for $i\neq j$},\qquad  \, 
\bar \pi_{i} (\{\tau > 0\}\setminus  A_{i}\times B_{i}) = 0 .
$$
Finally, set $\pi_{i} : = \bar \pi_{i}/\bar \pi_{i}(X\times X)$ and  $\mu_{0,i} : = (P_{1})_{\sharp}\pi_{i}$, 
$\mu_{1,i} : = (P_{2})_{\sharp}\pi_{i}$. Clearly, it holds $\mu_{0,i} \perp \mu_{0,j}$ if $i \neq j$. By restriction property, $\pi_{i} \in \Pi_{\ll}^{p\text{-opt}}(\mu_{0,i},\mu_{1,i})$
and we can apply the previous part of the proof to the marginals $\mu_{0,i}, \mu_{1,i}$: 
there exists  $\hat \pi_{i} \in  \Pi_{\leq}^{p\text{-opt}}(\mu_{0,i},\mu_{1,i})$ 
and  an $\ell_{p}$-geodesic $(\mu_{t,i})_{t\in [0,1]}$ 
from $\mu_{0,i}$ to $\mu_{1,i}$  satisfying
$$
U_{N}(\mu_{t,i}|\mm) \geq \sigma_{K/N}^{(1-t)}(\| \tau \|_{L^{2}(\hat \pi_{i})})\, U_{N}(\mu_{0,i}|\mm), \quad \forall t\in [0,1).
$$
In particular $\mu_{t,i} \ll \mm$ for all $t \in [0,1)$.
We can then sum over $i$ the previous inequality 
and, reasoning like in Step 2 by using mutual orthogonality of $\mu_{0,i}$, 
we  have the claims.
\end{proof}

\begin{theorem}\label{T:1}
Let  $(X,\sfd, \mm, \ll, \leq, \tau)$ be a timelike non-branching, globally hyperbolic, Lorentzian geodesic space satisfying $\mathsf{TMCP}^{e}(K,N)$  for some $p\in (0,1), K\in \R, N\in (0,\infty)$.
\\Let $\mu_{0},\mu_{1}\in \Prob_{c}(X)$, with $\mu_{0}\in \Dom(\Ent(\cdot|\mm))$. Assume that there exists $\pi\in \Pi^{p\text{-opt}}_{\leq}(\mu_{0},\mu_{1})$ such that $ \pi \left( \{\tau>0\} \right)=1$.

Then there exists a unique optimal coupling $\pi\in \Pi^{p\text{-opt}}_{\leq}(\mu_{0},\mu_{1})$ such that $ \pi \left( \{\tau>0\} \right)=1$ and it is induced by a map $T$, i.e.  $\pi=(\id, T)_{\sharp} \mu_{0}$ and 
$$
\ell_{p}(\mu_{0}, \mu_{1} )^{p} =\int_{X} \tau(x,T(x))^{p} \, \mu_{0}(\dd x).
$$
\end{theorem}

\begin{proof}
The arguments are along the same lines of the proof of Lemma \ref{lem:findelta} but with some (non-completely trivial) modifications that we briefly discuss. 
\smallskip

{\bf Step 1.} Let $\Gamma \subset X^{2}_{\ll} $ be an $\ell_{p}$-monotone subset such that $\pi(\Gamma)=1$, given by Proposition \ref{P:OptiffMon}. Define
\begin{equation} \label{eq:defGammaGammaxThm}
 \Gamma(x) : =P_{2} \Big(\Gamma \cap ( \{x\}   \times X ) \Big), 
\end{equation}
and $S$ the set of those $x \in X$ such that $\Gamma(x)$ is not a singleton. Note that the set $S$ is Suslin. 
It will be enough to prove the stronger statement $\mu_{0}(S) = 0$. 

So suppose by contradiction $\mu_{0}(S) > 0$. By Von Neumann Selection Theorem, there exists
$$
T_{1},T_{2} : S \to X, \qquad \text{graph}(T_{1}), \ \text{graph}(T_{2}) \subset \Gamma,
$$
both $\mu_{0}$-measurable and $\sfd(T_{1}(x), T_{2}(x)) > 0$, for all $x \in S$. 
By Lusin Theorem, there exists a compact set $S_{1} \subset S$ such that the maps $T_{1}$ and $T_{2}$ are both continuous when restricted to $S_{1}$ 
and $\mu_{0}(S_{1}) > 0$. In particular 
$$
\inf_{x \in S_{1}} \sfd(T_{1}(x), T_{2}(x)) = \min_{x\in S_{1}} \sfd(T_{1}(x), T_{2}(x)) = 2r >0.
$$
Then one can deduce the existence of a couple of points $x_{1}, x_{2}\in X $, of a positive  $r>0$ and of a compact set $S_{2} \subset S_{1}$, again with $\mu_{0}(S_{2})>0$, such that
$$
\{ T_{1}(x) \colon x\in S_{2}  \}\subset B_{r}(x_{1}),\qquad    \{ T_{2}(x)  \colon x\in S_{2} \} \subset B_{r}(x_{2}), 
$$
with $\sfd(x_{1},x_{2}) > 2r$,  where $B_{r}(x_{i})$ is the open ball centred in $x_{i}$ and radius $r$, for $i = 1,2$ with respect to $\sfd$.  By the continuity of $\tau$, up to further reducing $r>0$, we can also suppose that $S_{2} \times \left(B_{r}(x_{1}) \cup  B_{r}(x_{2}) \right) \Subset \{\tau>0\}$. 
\smallskip

{\bf Step 2.} Following the arguments of the proof of Lemma \ref{lem:findelta} (see in particular \eqref{eq:deftrunctho0S}), we can invoke Lemma \ref{L:restriction}
and assume with no loss of generality $\mu_{0}$ to be restricted and renormalised to $S_{2}$. 
In particular we redefine $\mu_{0} : = \mm\llcorner_{S_{2}}/\mm(S_{2})$;  the following measures
 are well defined as well %
$$
\mu_{1}^{1} : = (T_{1})_{\sharp} \mu_{0}, \qquad  \mu_{1}^{2} : = (T_{2})_ {\sharp} \mu_{0};
$$
in particular $\mu_{1}^{1},\mu_{1}^{2}$ are Borel probability measures with $\supp\, \mu_{1}^{1} \cap \supp\, \mu_{1}^{2}=\emptyset $.
\\By Proposition \ref{prop:AbsContmut}  we know there exist  $\ell_{p}$-geodesics $(\mu_{t}^{i})_{t\in [0,1]}$ from $\mu_{0}$ to $\mu_{1}^{i}$, $i=1,2$,  satisfying
\begin{equation}\label{eq:UNmutt[0,1)iproof}
U_{N}(\mu_{t}^{i}|\mm) \geq \sigma_{K/N}^{(1-t)}(\| \tau \|_{L^{2}(\hat \pi^{i})})\, U_{N}(\mu_{0}|\mm), \quad \forall t\in [0,1), \, i=1,2.
\end{equation}
Using \eqref{eq:UNmutt[0,1)iproof}, one can now follow verbatim the proof of Lemma \ref{lem:findelta}  and conclude.
\end{proof}

\begin{theorem}\label{T:2}
Let  $(X,\sfd, \mm, \ll, \leq, \tau)$ be a timelike non-branching, globally hyperbolic, Lorentzian geodesic space satisfying $\mathsf{TMCP}^{e}(K,N)$   for some $p\in (0,1), K\in \R, N\in (0,\infty)$.
\\Let $\mu_{0},\mu_{1}\in \Prob_{c}(X)$, with $\mu_{0}\in \Dom(\Ent(\cdot|\mm))$. Assume that there exists $\pi\in \Pi^{p\text{-opt}}_{\leq}(\mu_{0},\mu_{1})$ such that $\pi (\{\tau>0\})=1$.

Then there exists a unique $\eta\in  {\rm OptGeo}_{\ell_{p}}(\mu_{0},\mu_1)$  with $ (\ee_{0}, \ee_{1})_{\sharp} \eta \, ( \{\tau>0\})=1$ and such $\eta$ is induced by a map, i.e. there exists ${\mathfrak T}:X\to \TGeo(X)$ such that $\eta={\mathfrak T}_{\sharp} \mu_{0}$.
\end{theorem}

\begin{proof}
As usual, it is sufficient to show that every $\eta\in  {\rm OptGeo}_{\ell_{p}}(\mu_{0},\mu_1)$  with $ (\ee_{0}, \ee_{1})_{\sharp} \eta \, ( \{\tau>0\})=1$  is induced by a map; indeed if  there exist $\eta_{1}\neq \eta_{2}  \in {\rm OptGeo}_{\ell_{p}}(\mu_{0},\mu_1)$  with $ (\ee_{0}, \ee_{1})_{\sharp} \eta_{i} \, ( \{\tau>0\})=1$ then also $\bar{\eta}:=\frac{1}{2}(\eta_{1}+\eta_{2})$ would be an element of  ${\rm OptGeo}_{\ell_{p}}(\mu_{0},\mu_1)$  with
$ (\ee_{0}, \ee_{1})_{\sharp} \bar \eta \, ( \{\tau>0\})=1$   but  $\bar{\eta}$ cannot be given by a map.

Assume by contradiction there exists $\eta\in  {\rm OptGeo}_{\ell_{p}}(\mu_{0},\mu_1)$  not induced by a map. In particular, given the disintegration of $\eta$ with respect to $\ee_{0}:\TGeo(X)\to X$
$$
\eta = \int_{X} \eta_{x} \, \mu_{0}(\dd x), 
$$
there exists a compact subset $D\subset \supp (\mu_{0})$ with $\mu_{0}(D)>0$ such that for $\mu_{0}$-a.e. $x \in D$ the probability measure $\eta_{x}$ is not a Dirac mass. 
Via a selection argument, for $\mu_{0}$-a.e. $x \in D$ we can also assume that $\eta_{x}$ is the sum of two Dirac masses.
Then for $\mu_{0}$-a.e. $x \in D$ there exist $t=t(x)\in (0,1)$ such that $(\ee_{t})_{\sharp} \eta_{x}$ is not a Dirac mass over $X$.
Then by continuity there exists an  open interval $I=I(x)\subset (0,1)$ containing $t(x)$ above such that  $(\ee_{s})_{\sharp} \eta_{x}$ is still not a Dirac mass over $X$, for every $s \in I(x)$.
\\It follows that we can find a subset  $\bar{D} \subset D \subset X$ still satisfying $\mu_{0}(\bar{D})>0$ with the following property: there exists $\bar{q}\in \Q\cap (0,1)$ such that $(\ee_{\bar q})_{\sharp} \eta_{x}$ is not a Dirac mass, for every  $x \in \bar{D}$.

 Indeed, since  $D=  \bigcup_{q \in 	\Q \cap (0,1)} D_{q}$ where 
$$D_{q}:=\left\{ x \in D \,: \, (\ee_{q})_{\sharp} \eta_{x} \text{ is not a Dirac mass} \right\} $$
and since $\mu_{0}(D)>0$, there must exist  $\bar{q}\in 	\Q \cap (0,1)$ with $\mu_{0}(D_{\bar{q}})>0$; we then set $\bar{D}:=D_{\bar{q}}$.
Set now
$$
\bar{\eta} = \frac{1}{\mu_{0}(\bar{D})} \int_{\bar{D}} \eta_{x} \, \mu_{0}(\dd x). 
$$
Note that $\bar{\eta}$  is an $\ell_{p}$-optimal dynamical plan satisfying $(\ee_{0}, \ee_{\bar{q}})_{\sharp} \bar{\eta} \, (\{\tau>0\})=1$. But $(\ee_{0}, \ee_{\bar{q}})_{\sharp} \bar{\eta}$ is an $\ell_{p}$-optimal coupling which is not given by a map, contradicting Theorem \ref{T:1}. 
\end{proof}

We then summarise the consequences of the previous results in the particular case of 
a uniform distribution as $\mu_{0}$. The main applications will be on the regularity of 
conditional measures of the disintegration of the reference measure $\mm$.

\begin{corollary}\label{C:MCP}
Let  $(X,\sfd, \mm, \ll, \leq, \tau)$ be a timelike non-branching, globally hyperbolic, Lorentzian geodesic space satisfying $\mathsf{TMCP}^{e}(K,N)$  for some $p\in (0,1), K\in \R, N\in [1,\infty)$. Let $\mu_{0},\mu_{1} \in \Prob_{c}(X)$ be two probability measures and 
assume that there exists $\pi\in \Pi^{p\text{-opt}}_{\leq}(\mu_{0},\mu_{1})$ such that $ \pi( \{\tau>0\})=1$.

Then $\pi$ is the unique element  
of $\Pi^{p\text{-opt}}_{\leq}(\mu_{0},\mu_{1})$ concentrated on $\{\tau > 0\}$.
Accordingly there exists a unique optimal dynamical plan $\eta$ such that 
$(\ee_{0},\ee_{1})_{\sharp}\eta = \pi$.  
Moreover, the  $\ell_{p}$-geodesic  
$\mu_{t} = (\ee_{t})_{\sharp} \eta$
satisfies
$\mu_{t} = \rho_{t} \mm \ll \mm$.  Finally, if 
$\mu_{0} = \mm\llcorner_{A_{0}}/\mm(A_{0})$ with  $A_{0} \subset X$ compact subset, then  
\begin{equation}\label{E:MCP}
\mm(\{\rho_{t} > 0 \}) \geq \sigma_{K/N}^{(1-t)}(\| \tau \|_{L^{2}( \pi)})^{N}\mm(A_{0}).
\end{equation}
In particular, calling $A_{1}=\supp \, \mu_{1}$ and using the notation of Proposition \ref{prop:BrunnMnk}, the following timelike half-Brunn-Minkowski  inequality holds:
$$
\mm(A_{t})^{1/N}\geq  \sigma_{K/N}^{(1-t)}(\Theta)\,\mm(A_{0})^{1/N}.
$$
\end{corollary}
\begin{proof}
From Theorem \ref{T:1} it follows the uniqueness of the elements in 
$\Pi_{\leq}^{p\text{-opt}}(\mu_{0},\mu_{1})$ that are also concentrated on 
$X^{2}_{\ll}$; hence $\pi$ is unique. 
From Theorem \ref{T:2} 
we deduce the existence of a unique 
$\eta\in  {\rm OptGeo}_{\ell_{p}}(\mu_{0},\mu_1)$ such that 
$(\ee_{0}, \ee_{1})_{\sharp} \eta \, ( \{\tau>0\})=1$. 
In particular there is only one optimal dynamical plan $\eta$ 
such that $(\ee_{0}, \ee_{1})_{\sharp} \eta = \pi$.

To deduce $\eqref{E:MCP}$ we invoke Proposition \ref{prop:AbsContmut}
together with the following
approximation procedure: denote by $T$ the optimal map associated to $\pi$, i.e. 
$(\id,T)_{\sharp} \mu_{0} = \pi$ (Theorem \ref{T:1}). 
Then by inner regularity, we can find compact sets $\{K_{n}\}_{n\in \N}$ such that 
$$
K_{n}\subset A, \quad \mm(A\setminus K_{n})\to 0, \quad 
\{ (x,T(x))\colon x \in K_{n} \} \subset \{ \tau > 0\},
$$
and $T:K_{n} \to X$ is continuous. In particular, the optimal transport plan 
$\pi_{n} : = (\id,T)_{\sharp} \mm\llcorner_{K_{n}}/\mm(K_{n})$ is such that 
$\spt \, \pi_{n}$ is a compact subset of $\{\tau >0\}$. 

Using Proposition \ref{prop:AbsContmut} with the uniqueness properties 
obtained just few lines above,
we deduce  that
$U_{N}(\mu^{n}_{t}|\mm) \geq \sigma_{K/N}^{(1-t)}(\| \tau \|_{L^{2}( \pi_{n})}) \mm(K_{n})^{1/N}$, 
where $\mu_{t}^{n} = (\ee_{t})_{\sharp} \pi_{n}$. 
Then by taking the limit as $n \to \infty$, from the upper semicontinuity of 
$U_{N}$ \eqref{eq:UNusc} and the convergence of $\| \tau \|_{L^{2}( \pi_{n})}$ 
to $\| \tau \|_{L^{2}( \pi)}$ by the boundedness of $\tau$, we deduce that 
$U_{N}(\mu_{t}|\mm) \geq \sigma_{K/N}^{(1-t)}(\| \tau \|_{L^{2}( \pi)}) 
\mm(A)^{1/N}$.

We  conclude applying twice Jensen's inequality as in \eqref{eq:twiceJensen}.
\end{proof}

\section{Localization of the timelike Measure Contraction Property}\label{Sec:LocTMCP}

\subsection{Transport relation and disintegration associated to a time separation function}
From now on, we make the standing assumptions that $(X,\sfd, \ll, \leq, \tau)$ is a globally hyperbolic Lorentzian geodesic  space and  $V\subset X$ is an achronal FTC  Borel subset (see Definition \ref{def:FTC}).
Recall that, associated to $V$, we have the signed  time-separation function $\tau_{V}:X\to [-\infty, +\infty]$ defined in  \eqref{eq:deftauV}.

\begin{lemma}\label{L:resume}
For each $x\in I^{+}(V)$  there exists a point $y_{x}\in V$ with $\tau_{V}(x)=\tau(y_{x},x)>0$. 
Moreover: 
\begin{equation}\label{eq:tauvzxtau}
\tau_{V}(z) - \tau_{V}(x) \geq \tau(y_{x},z)-\tau(y_{x},x)  \geq  \tau(x,z),  \quad \forall  x,z\in I^{+}(V)\cup V, \, x\leq z.
\end{equation}
\end{lemma}
\begin{proof}
The first claim follows directly from Lemma \ref{L:initialpoint}.
If $x\leq z$ and $\tau(y_{x},x) >0$, then  
also $y_{x}\leq z$ by transitivity.
By reverse triangle inequality \eqref{eq:deftau}, we deduce   
\eqref{eq:tauvzxtau}.
\end{proof}
\noindent
Notice that \eqref{eq:tauvzxtau} can be extended to the whole 
$X^{2}$ simply by replacing $\tau$ with $\ell$ (recall \eqref{eq:defell}): 
\begin{equation}\label{E:triang}
\tau_{V}(z) - \tau_{V}(x) \geq
\ell(x,z), \qquad \forall x,z \in (I^{+}(V)\cup V)^{2}.
\end{equation}
We can therefore naturally associate to $V$ the following  transport relation:
\begin{equation}\label{E:GammaV}
\Gamma_{V} : =  \{ (x,z) \in (I^{+}(V)\cup V)^{2} \cap X^{2}_{\leq} \, \colon \, 
 \tau_{V}(z) - \tau_{V}(x)  = \tau(x,z)>0 \}  \cup \{(x,x) \,:\, x\in I^{+}(V)\cup V\}.
\end{equation}
Recalling the Definition \ref{D:monotonicity}, 
inequality \eqref{E:triang} readily yields:
\begin{lemma}\label{lem:GammaVell1CM}
The set $\Gamma_{V}$ is $\ell$-cyclically monotone.
\end{lemma}
\begin{proof}
Take $(x_{1},z_{1}), \dots, (x_{n},z_{n}) \in \Gamma_{V}$ and sum 
$$
\sum_{i = 1}^{n} \ell(x_{i},z_{i}) = \sum_{i = 1}^{n} \tau(x_{i},z_{i}) = \sum_{i = 1}^{n} \tau_{V}(z_{i}) - \tau_{V}(x_{i}) \geq 
\sum_{i = 1}^{n} \ell(x_{i+1},z_{i}). 
$$
\end{proof}
\noindent
A consequence of $\ell$-cyclical monotonicity is the 
alignment along geodesics of the couples:

\begin{lemma}\label{L:geodesicandV}
Consider $(x,z) \in \Gamma_{V}$ with $x \neq z$, $x\notin V$. Then there exist 
$y \in V, \gamma \in \TGeo(y,z)$ and $t \in (0,1)$ such that 
$$
x = \gamma_{t}, \qquad \tau(y,\gamma_{s}) = \tau_{V}(\gamma_{s})  \quad \forall s\in [0,1], \qquad
(\gamma_{s},\gamma_{t}) \in \Gamma_{V}  \quad \forall s\in [0,t].
$$
\end{lemma}

\begin{proof}
From Lemma \ref{L:resume}, we have the existence of 
$y \in V$ such that $\tau_{V}(x) = \tau(y,x) > 0$. Moreover from 
$(x,z) \in \Gamma_{V}$ we get $(y,z) \in X_{\leq}^{2}$ and
$$
\tau(y,z) \leq \tau_{V}(z) = \tau_{V}(x) + \tau(x,z)   
= \tau(y,x)+\tau(x,z)  \leq \tau (y,z),
$$
yielding $0 < \tau(y,x) + \tau(x,z)= \tau(y,z) $ and $(y,z) \in \Gamma_{V}$. 
Hence we can concatenate a timelike geodesic from $y$ to $x$ with a timelike geodesic from $x$ to $z$ (whose existence is guaranteed by fact that $X$ is a Lorentzian geodesic space)  in order to obtain $\gamma \in \TGeo(y,z)$ and $t \in (0,1)$ such that 
$\gamma_{t} = x$, proving first claim.  In order to show the second claim, observe that for any $s \in [0,1]$ it holds:
$$
 \tau_{V}(\gamma_{s}) = 
 \tau_{V}(\gamma_{1})  
-\tau_{V}(\gamma_{1}) + \tau_{V}(\gamma_{s}) = \tau(y,z) 
-\tau_{V}(\gamma_{1}) + \tau_{V}(\gamma_{s}).
$$
From \eqref{eq:tauvzxtau} we know that  $\tau_{V}(\gamma_{1}) - \tau_{V}(\gamma_{s})\geq \tau(\gamma_{s},\gamma_{1})$ hence it follows that 
$$
\tau(y,\gamma_{s}) \leq \tau_{V}(\gamma_{s}) \leq \tau(y,z) - \tau(\gamma_{s},z) = \tau(y,\gamma_{s}),
$$
proving the second point.
For the last point, simply observe that
$$
\tau_{V}(\gamma_{t}) -\tau_{V}(\gamma_{s}) = \tau(y,\gamma_{t}) - \tau(y,\gamma_{s})
= \tau (\gamma_{s}, \gamma_{t}).
$$
\end{proof}

Next, we set $\Gamma_{V}^{-1}:=\{(x,y)\,:\, (y,x)\in \Gamma_{V}\}$ and we consider the \emph{transport relation} $R_{V}$ and 
the \emph{transport set with endpoints} $\T_{V}^{e}$:
\begin{equation}\label{E:transport}
R_{V} : = \Gamma_{V} \cup \Gamma_{V}^{-1}, \qquad 
\T_{V}^{e} : = P_{1}(R_{V}\setminus \{ x = y \}).
\end{equation}
The transport relation will be an equivalence relation on a specific subset of 
$\T_{V}^{e}$ that we will now construct.
Firstly we consider the following subsets of $\T_{V}^{e}$:
\begin{equation}\label{eq:defendpoints}
\begin{split}
\fa(\T_{V}^{e}) : =&~ \{ x \in \T_{V}^{e} \colon \nexists y \in \T_{V}^{e} \ s.t. \ (y,x) \in \Gamma_{V}, y\neq x \} \\
\fb(\T_{V}^{e}) : =&~ \{ x \in \T_{V}^{e} \colon \nexists y \in \T_{V}^{e} \ s.t. \ (x,y) \in \Gamma_{V}, y\neq x \},
\end{split}
\end{equation}
called the set of \emph{initial} and \emph{final points}, respectively.
Define the \emph{transport set without endpoints} 
\begin{equation}\label{E:nbtransport}
\T_{V} : = \T_{V}^{e} \setminus (\fa(\T_{V}^{e}) \cup \fb(\T_{V}^{e})).
\end{equation}

\begin{lemma}\label{lem:I+VTV}
It holds  $ I^{+}(V)  = (\mathcal{T}_{V} \cup \fb(\mathcal{T}_{V}^{e}))\setminus V$ and  $V \supset \fa(\T_{V}^{e})$.
\end{lemma}

\begin{proof}
By definition $R_{V} \subset  (I^{+}(V)\cup V)^{2}$ and  
since $V$ is achronal $I^{+}(V)\cap V=\emptyset$; hence
the inclusion $ I^{+}(V)  \supset (\mathcal{T}_{V} \cup \fb(\mathcal{T}_{V}^{e}))\setminus V$ is trivial.
To show the converse inclusion, for every $x\in I^{+}(V)$  
Lemma \ref{L:resume} ensures the existence of 
$y \in V$ such that $\tau_{V}(x) = \tau(y,x)>0$. Thus $(x,y)\in \Gamma_{V}^{-1}\subset R_{V}$, giving that $x\in \T_{V}^{e} \setminus \fa(\T_{V}^{e})= \T_{V} \cup \fb(\T_{V}^{e})$.
The argument for the second inclusion is trivial.
\end{proof}

\begin{proposition}\label{P:equivalence}
Assume in addition to the previous assumptions that $X$ is timelike (backward and forward) non-branching.
Then the transport relation $R_{V}$ 
is an equivalence relation over $\T_{V}$.
\end{proposition}
\begin{proof}
The reflexive property $(x,x) \in R_{V}$ for all $x \in \T_{V}$, as well as symmetry,  hold by the very definitions of $\Gamma_{V}$ and $R_{V}$. We are then left to show transitivity:
for every $(x,y),(y,z) \in R_{V}$ we next prove that $(x,z) \in R_{V}$. Clearly we can assume $x \neq y \neq z$, otherwise the claim is trivial.

\smallskip
\noindent
{\bf Case 1:} $(x,y), (y,z) \in \Gamma_{V}$. Using \eqref{eq:tauvzxtau}  and reverse triangle inequality we have 
\begin{align*}
  \tau_{V}(z) - \tau_{V}(x)  \geq \tau(x,z) 
	\geq \tau(x,y) + \tau(y,z) 	&~  =   \tau_{V}(y) - \tau_{V}(x)   + \tau_{V}(z) - \tau_{V}(y) =   \tau_{V}(z) - \tau_{V}(x).
\end{align*}
Hence $\tau(x,z) =  \tau_{V}(z)- \tau_{V}(x)$ 
and therefore 
$(x,z) \in \Gamma_{V} \subset R_{V}$.

\smallskip
\noindent
{\bf Case 2:} $(x,y), (y,z) \in \Gamma_{V}^{-1}$. Hence $(z,y), (y,x) \in \Gamma_{V}$ and therefore $(z,x) \in \Gamma_{V}$  from case 1.

\smallskip
\noindent
{\bf Case 3:} $(x,y) \in \Gamma_{V}$ and $(y,z) \in \Gamma_{V}^{-1}$. Hence $(x,y),(z,y) \in \Gamma_{V}$.
Since $y \notin \fb(\T_{V}^{e})$, there exists $w \in \T_{V}$ such that $(y,w) \in \Gamma_{V}$ and $y\neq w$. 
Hence from $(x,y), (z,y),(y,w) \in \Gamma_{V}$ 
we deduce like in case 1 that
$$
\tau(x,y) + \tau(y,w)  = \tau(x,w)>0, \qquad   \tau(z,y) + \tau(y,w)  = \tau(z,w)>0.
$$
Since by assumption $X$ is a Lorentzian geodesic space,  there exist  $\gamma^{1} \in \TGeo(x,w),\, \gamma^{2} \in \TGeo(z,w)$ 
with common intermediate point $y$. 
Then from the backward non-branching assumption, necessarily 
$\gamma^{1}_{[0,1]} \subset \gamma^{2}_{[0,1]}$ 
(or the other inclusion) holds true.  
Indeed if  the two maximizing geodesics $\gamma^1$ and $\gamma^2$ are distinct they cannot meet at the intermediate point $y$: otherwise one defines a new timelike curve $\eta$, say from $z$ to $w$,  defined by $\gamma^2$ from $z$ to $y$ and 
by $\gamma^1$ from $y$ to $w$. The curve $\eta$ will be a timelike geodesic from $z$ to $w$ by construction, it will coincide with $\gamma^2$ on a non-trivial interval, nonetheless it will not coicide with $\gamma^2$ giving a contradiction.

The last claim of Lemma \ref{L:geodesicandV} finally gives $(x,z) \in R_{V}$.

\smallskip
\noindent
{\bf Case 4:} $(x,y) \in \Gamma_{V}^{-1}$ and $(y,z) \in \Gamma_{V}$. The argument is analogous to
 case 3: since $y \notin \fa(\T_{v}^{e})$, 
there exists $w \in \T_{V}$ such that $(w,y) \in \Gamma_{V}$ and $w \neq y$. Then from the Lorentzian geodesic and (now forward) non-branching assumption, necessarily all the points $w,y,x,z$ lie on the same strictly timelike geodesics, giving that $(x,z) \in R_{V}$.
\end{proof}

\begin{lemma}\label{lem:XalphaI}
For each equivalence class $[x]$ of $(\T_{V,}R_{V})$ there exists a  convex set  $I\subset \R$ of the Real line  and a bijective map $F:I\to [x]$ satisfying:
\begin{equation}\label{eq:FIsometry}
\tau(F(t_{1}), F(t_{2})) = t_{2}-t_{1}, \quad \forall \,t_{1}\leq t_{2} \in I.
\end{equation}
Moreover, calling $\overline{\{z\in [x]\}}$ the topological closure of $\{z\in [x]\}\subset X$, it holds 
\begin{equation}\label{eq:closureVSendpoints}
\overline{\{z\in [x]\}}\setminus \{z\in [x]\} = \overline{\{z\in [x]\}}\setminus \T_{V}  \subset \fa(\T_{V}^{e}) \cup \fb(\T_{V}^{e}).
\end{equation}
\end{lemma}

\begin{proof}
For any $x \in \T_{V}$, denote with $[x]$ 
the associated equivalence class.
Consider the maps 
$$
F :  (0,\infty)\cap \Dom(F) \ni t \mapsto 
\{ y \colon (x,y) \in \Gamma_{V}, \,\tau(x,y) =  t \} 
$$
and 
$$
F: (-\infty,0)\cap \Dom(F) \ni t \mapsto \{ y \colon (y,x) \in \Gamma_{V},\, \tau(y,x) =- t \}  
$$
and $F(0) = x$. First observe that $F$ is surjective: for each $y \in [x], y\neq x$ with $(x,y)\in \Gamma_{V}$ (resp. $(x,y)\in \Gamma_{V}^{-1}$) it holds $\tau(x,y) \in (0, \infty)$, hence $\tau(x,y)\in \Dom(F)$ and $y \in F(\tau(x,y))$ (resp. $\tau(y,x)\in \Dom(F)$ and $y \in F(-\tau(y,x))$). \\
The fact that $F$ is injective follows readily from its definition.
\\We next show that $F$ is a single valued map. Assume by contradiction $y\neq z \in F(t)$ for some $t >0$ (resp. $t <0$);
since $x$ is not an initial (resp. final) point,
using the geodesic assumption like in the proof of Proposition \ref{P:equivalence} would produce a forward (resp. backward) branching time-like geodesic  giving a contradiction
with the non-branching assumption. 
\\Given $t \in \Dom(F)$, with a slight abuse of notation, we identify $F(t)$ with $\{F(t)\}$. 
\\For  $t_{1}< t_{2} \in \Dom(F)$, Lemma \ref{L:geodesicandV}  implies that the interval $[t_{1}, t_{2}]\subset F$ (i.e. $\Dom(F)\subset \R$ is a convex subset) and that \eqref{eq:FIsometry} holds. 

We now show \eqref{eq:closureVSendpoints}.
Let $(z_{n})\subset [x]$ be with $\inf_{n} \tau_{V}(z_{n})>0$ and $z_{n}\to \bar{z}$. It is easily seen that there exists $\bar x\in [x]$ such that $(\bar x, z_{n})\in \Gamma_{V}$ and $\bar{x}\neq \bar{z}$. Using the continuity of $\tau$ (by global hyperbolicity), the lower semicontinuity of $\tau_{V}$, \eqref{eq:tauvzxtau} and the causal closeness (see Proposition \ref{prop:GH->KGH}(ii)), it is easy to check that $(\bar x, \bar z)\in \Gamma_{V}\subset R_{V}$ and thus $\bar z\in \T^{e}_{V}$. Since by Proposition \eqref{P:equivalence} the equivalence classes of $R_{V}$ form a partition of $\T_{V}$, it follows that  if $\bar{z}\notin [x]$ then $\bar{z}\notin \T_{V}$; more precisely it is easily seen that $\bar z\in \fb(\T_{V}^{e})$.
\\Let now $(z_{n})\subset [x]$ be with $ \tau_{V}(z_{n})\to 0$ and $z_{n}\to \bar{z}$. By lower semicontinuity of $\tau_{V}$, it follows that $\tau_{V}(\bar z)=0$. Using the continuity of $\tau$ it is easy to check that $\bar{z}\in \T^{e}_{V}$ and  $(\bar{z}, x)\in R_{V}$. Arguing as above, it follows that  if $\bar{z}\notin [x]$ then $\bar{z}\notin \T_{V}$;  more precisely it is easily seen that $\bar z\in \fa(\T_{V}^{e})$.
 \end{proof}


\subsection{Disintegration of $\mm$ associated to $\tau_{V}$}\label{Ss:disintegration}

For ease of reading we recall same basic facts about analytic sets referring to Section 4 of \cite{Srivastava} for more details.

If $X$ is a general complete and separable metric space (that only here we abbreviate with Polish space),
the \emph{projective class $\Sigma^1_1(X)$} is the family of subsets $A\subset X$ for which there exists a Polish space $Y$  and $B \in \mathcal{B}(X \times Y)$ (where $\mathcal{B}$ is the Borel $\sigma$-algebra) such that $A = P_1(B)$. The \emph{coprojective class $\Pi^1_1(X)$} is the complement in $X$ of the class $\Sigma^1_1(X)$. The construction can be iterated but we will consider only $\Sigma^1_1$ that is also called the class \emph{of analytic sets}; accordingly $\Pi^1_1$ is the class of \emph{coanalytic sets}.
We list few properties of these family of sets:
$\Sigma^1_1$ and $\Pi^1_1$ are closed under countable unions and intersections;
$\Sigma^1_1$ is closed w.r.t. projections while $\Pi^1_1$ is closed w.r.t. coprojections;
if $A \in \Sigma^1_1$, then $X \setminus A \in \Pi^1_1$;
the intersection 
$\Sigma^1_1 \cap \Pi^1_1$ is the Borel $\sigma$-algebra $\mathcal{B}$. 
If $\mathcal{A}$ denotes the $\sigma$-algebra generated by $\Sigma^1_1$, then clearly $\mathcal{B}(X)\subset \mathcal{A}$.

Finally a subset of $X$ is \emph{universally measurable} if it belongs to all completed $\sigma$-algebras of all Borel measures on $X$:
it can be proved that every set in $\mathcal{A}$ is universally measurable.

We start with some measurability properties of the sets we have considered in this paper. 
For any $x \in X$ the set
$I^{+}(x) = \{ y \in X \colon \tau(x,y) >0 \}$
is open  by continuity  of $\tau$ (ensured by global hyperbolicity). Accordingly, 
$I^{+}(V)=\bigcup_{x\in V}  I^{+}(x)$ is an open subset of $X$.

By the very definition \eqref{eq:deftauV}, 
$\tau_{V}$ is $\sup$ of continuous functions thus it is lower semi-continuous.
It follows that the set $\Gamma_{V}$ is Borel measurable (see \eqref{E:GammaV}).
It follows that also  $R_{V}$ is  Borel measurable, yielding  that $\T_{V}^{e}$ defined in \eqref{E:transport} is an analytic set.
To conclude, we obtain measurability 
of the transport set $\T_{V}$ defined in \eqref{E:nbtransport}.

\begin{lemma}
The set $\T_{V}$ is analytic.
\end{lemma}
\begin{proof}
Just notice that $\T_{V}$ 
coincides with the following set
$$
 P_{2}\{ (x,y,z) \in I^{+}(V)\times I^{+}(V)\times
I^{+}(V)\colon (x,y) \in \Gamma_{V}, (y,z)
\in \Gamma_{V}, \,
\sfd(z,y)\neq 0, \, \sfd(x,y) \neq 0\}.
$$
Being the projection of a Borel set, the claim follows.
\end{proof}
One can also prove that $\fa(\T_{V}^{e})$ and 
$\fb(\T_{V}^{e})$ are coanalytic sets. 
We next  build an $\mm$-measurable quotient map $\QQ$ of the equivalence relation $R_{V}$ over $\T_{V}$.%

Before proceeding, we recall a result on the existence of a section for an equivalence relation
proved in \cite{biacava:streconv}; 
the terminology is borrowed from 
\cite{Srivastava}. 
For readers' convenience we will include its proof. 

\begin{theorem}[Corollary 2.7 of \cite{biacava:streconv}]\label{T:section}
Let $X$ be a Polish space and
$F \subset X \times X$ be 
$\mathcal{A}$-measurable such that $F_x$ is closed. 
Define the following equivalence relation: 
$x \sim y \iff F(x) = F(y)$. 
Then there exists an $\mathcal{A}$-measurable map
$f:P_{1}(F)\to X$ such that $(x,f(x))\in F$ and 
$f(x)= f(y)$ if $x\sim y$.
\end{theorem}

Few comments are in order. 
The set $F$ can be regarded also as a multivalued map with the following notation $F_x : = 
F \cap \{x\} \times X$ and 
$F(x) : = P_2(F_x)$. 
The assumption of $F$ being 
$\mathcal{A}$-measurable means that 
$F^{-1}(U) \in \mathcal{A}$ for any open set $U$, where 
$F^{-1}(U) : = \{ x\in X \colon 
F(x) \cap U \neq \emptyset \} = 
P_1(F \cap X \times U)$.

\begin{proof}
For all open sets $U \subset X$, consider the sets $F^{-1}(U)$; 
by assumption they will be in $\mathcal{A}$. Let $\mathcal{R}$ be the $\sigma$-algebra generated by all such $F^{-1}(U)$: by assumption $\mathcal{R} \subset \mathcal{A}$.

If $x \sim y$,  then
$$
x \in F^{-1}(U) \quad \Longleftrightarrow \quad y \in F^{-1}(U),
$$
so that each equivalence class is contained in an atom of $\mathcal{R}$. 
Moreover by construction 
the multivalued map $x \mapsto F(x)$ is $\mathcal{R}$-measurable.
We can thus apply \cite[Theorem 5.2.1]{Srivastava} ensuring the existence of an $\mathcal{R}$-measurable selection $f$ of $F$, that is a map $f : P_1(F) \to X$ 
such that $f(x) \in F(x)$. 

The property $f(x) \in F(x)$ 
simply means that $(x,f(x)) \in F$.
The $\mathcal{R}$-measurability condition implies that $f$ has to be costant on the atoms of $\mathcal{R}$. Since $\mathcal{R}$ 
does not separate the equivalence classes of $F$, this implies that if $x\sim y$ then $f(x) = f(y)$, proving all the claims.
\end{proof}

\begin{proposition}\label{prop:Qlevelset}
There exists an $\mathcal{A}$-measurable quotient map 
$\QQ:\T_{V}\to X$ of the equivalence relation $R_{V}$ over $\T_{V}$, 
  i.e. 
\begin{equation}\label{E:quotient}
\QQ : \mathcal{T}_{V} \to \mathcal{T}_{V}, 
\qquad (x,\QQ(x)) \in R_{V}, 
\qquad (x,y) \in R_{V} \Rightarrow \QQ(x) = \QQ(y).
\end{equation}
\end{proposition}

\begin{proof}
First consider the following saturated family of subsets of $\T_{V}$: 
$$
E_{n}: = \{y\in  \T_{V}\,:\, (x,y)\in R_{V} \text{ for some }  x \in \T_{V} \text{ with } \tau_{V}(x) > 1/n\}  , \quad \forall n\in \N, n\geq 1.
$$
By construction $E_{n}$ is analytic,  $E_{n} \subset E_{n+1}$ and $\T_{V} = \cup_{n} E_{n}$. 
Set $F_{n} : = E_{n}\setminus E_{n-1}$, 
with $F_{1} = E_{1}$, so that 
$F_n \in \mathcal{A}$.
Define the set $G \subset X \times X$ 
by
$$
G = \bigcup_{n} 
F_n \times \tau_V^{-1}([1/2n,1/n]) \cap R_V \cap \T_V\times \T_V.
$$
As multivalued map $G$ is $\mathcal{A}$-measurable: indeed for any $U$ open set, 
\begin{align*}
G^{-1}(U) = &~ \bigcup_n \left\{x \in F_n \colon R_V(x) \cap \tau_V^{-1}([1/2n,1/n]) \cap U \neq \emptyset
\right\} \\
= &~ 
\bigcup_n F_n \cap P_1 \left( 
R_V \cap (X\times (U \cap \tau_V^{-1}([1/2n,1/n]))
\right),
\end{align*} 
showing that $G^{-1}(U)$ is $\mathcal{A}$-measurable.
By construction $G_x$ is closed  and for 
$x,y \in \T_V$,  $G(x) = G(y)$ 
if and only if $(x,y) \in R_V$.
Hence we can apply Theorem \ref{T:section} and obtain a 
$\mathcal{A}$-measurable
map 
$\QQ : P_1(G) \to X$ such that 
$(x,\QQ(x)) \in G \subset R(V)$ and $\QQ(x) = \QQ(y)$ if $(x,y) \in R_V$. Since $P_1(G) = \T_V$, 
the claim follows.
\end{proof}

\textbf{Notation.} From now on we will denote $Q:=\QQ(\T_{V}) \subset X$ the quotient set (which is $\mathcal{A}$-measurable). The equivalence classes of 
$R_{V}$ inside $\mathcal T_{V}$ will be called \emph{rays} and denoted with $X_{\alpha}$, with $\alpha\in Q$.
\\

Applying the same trick used in \cite[Section 3.1]{CaMoLap},
Proposition \ref{prop:Qlevelset} allows to apply Disintegration Theorem \cite[Section 452]{Fremlin4} (see also \cite[Section 6.3]{CMi}), provided  the  measure $\mm$  is suitably modified into a finite measure.
To this aim, it will be useful the 
next elementary lemma (for its proof see \cite[Lemma 3.3]{CaMoLap}).

\begin{lemma}\label{lem:constrf}
There exists a Borel function  $f : X \to (0,\infty)$ satisfying 
\begin{equation}\label{eq:deff}
 \inf_{{\mathcal K}} f>0, \; \text{for any bounded subset ${\mathcal K}\subset X$}, \quad \int_{\T_{V}} f\, \mm = 1.
\end{equation}
\end{lemma}

\noindent
Then, given $f : X \to (0,\infty)$ satisfying \eqref{eq:deff},  set 
$\mu : = f \, \mm\llcorner_{\T_{V}}$,
and define the normalized quotient measure 
$\qq : = \QQ_{\sharp}\, \mu\in \Prob(X)$.
It is straightforward to check that 
$$
\QQ_{\sharp} (\mm\llcorner_{\T_{V}}) \ll \qq.
$$
Take indeed $E \subset Q$ with $\qq(E) = 0$; then by definition
$\int_{\QQ^{-1}(E)} f(x)\, \mm(\dd x) = 0$,
implying  $\mm(\QQ^{-1}(E)) = 0$, since $f > 0$.
From the Disintegration Theorem \cite[Section 452]{Fremlin4}, we deduce the existence 
of a map 
$$
Q \ni \alpha \longmapsto \mu_{\alpha} \in \mathcal{P}(X)
$$
verifying the following properties:
\begin{itemize}
\item[(1)] for any $\mu$-measurable set $B\subset X$, the map $\alpha \mapsto \mu_{\alpha}(B)$ is $\qq$-measurable; \smallskip
\item[(2)] for $\qq$-a.e. $\alpha \in Q$, $\mu_{\alpha}$ is concentrated on $\QQ^{-1}(\alpha)$; \smallskip
\item[(3)] for any $\mu$-measurable set $B\subset X$ and $\qq$-measurable set $C\subset Q$, the following disintegration formula holds: 
$$
\mu(B \cap \QQ^{-1}(C)) = \int_{C} \mu_{\alpha}(B) \, \qq(\dd \alpha).
$$
\end{itemize}
Finally the disintegration is $\qq$-essentially unique, i.e. if any other 
map $Q \ni \alpha \longmapsto \bar \mu_{\alpha} \in \mathcal{P}(X)$
satisfies the previous three points, then 
\begin{equation}\label{eq:UniqDisint}
\bar \mu_{\alpha} = \mu_{\alpha},\qquad \qq\text{-a.e.} \ \alpha \in Q. 
\end{equation}
Hence once $\qq$ is given (recall that $\qq$ depends on $f$ 
from Lemma \ref{lem:constrf}), the disintegration is unique up to a set of 
$\qq$-measure zero. 
In the case $\mm(X) < \infty$, the natural choice,  that we tacitly assume, 
is to take as $f$ the characteristic function of $\T_{V}$ normalised by $\mm(\T_{V})$ so that $\qq : = \QQ_{\sharp} (\mm\llcorner_{\T_{V}}/\mm(\T_{V})$).

All the previous properties will be summarized saying that 
$Q \ni \alpha \mapsto \mu_{\alpha}$ is a  \emph{disintegration of 
$\mu$ strongly consistent with respect to $\QQ$}.
It follows from \cite[Proposition 452F]{Fremlin4} that 
$$
\int_{X} g(x) \mu(\dd x) = \int_{Q} \int g(x) \mu_{\alpha}(\dd x) \,\qq(\dd \alpha),
$$
for every $g : X \to \R\cup\{\pm \infty\}$ such that $\int g \mu$ is 
well-defined in $\R\cup\{\pm \infty\}$. 
Hence picking $g = 1/f$ (where $f$ is the one used to define $\mu$),  we get that 
$$
\mm\llcorner_{\T_{V}} = \int_{Q} \frac{\mu_{\alpha}}{f} \, \qq(\dd \alpha),
$$
where the identity has to be understood in duality with test functions as the previous formula.

Defining $\mm_{\alpha} : = \mu_{\alpha}/f$, we 
obtain that $\mm_{\alpha}$ (called \emph{conditional measure}) is a Radon non-negative measure over $X$, verifying 
all the measurability properties (with respect to $\alpha\in Q$) of $\mu_{\alpha}$
and giving a disintegration of $\mm\llcorner_{\T_{V}}$ 
strongly consistent with respect to $\QQ$. Moreover, for every bounded subset $\mathcal K \subset X$, it holds
$$
\frac{1}{\sup_{\mathcal K} f } \mu_{\alpha}(K) \leq  \mm_{\alpha}( \mathcal K) =   \frac{\mu_{\alpha}}{f} (\mathcal K) \leq \frac{1}{\inf_{\mathcal K} f }, \quad \text{for $\qq$-a.e. $\alpha\in Q$.}
$$

In the next statement, we  summarize what obtained so far (cf. \cite{CaMoLap}).
We denote by  $\mathcal{M}_{+}(X)$ the 
space of non-negative Radon measures over $(X,\sfd)$.

\begin{theorem}\label{T:sigma-disint}
 Let $(X,\sfd, \ll, \leq, \tau)$ be a timelike non-braching globally hyperbolic Lorentzian geodesic  space, and $V\subset X$  a Borel achronal FTC subset.
 
Then  the 
measure $\mm$ restricted to the transport set without endpoints $\T_{V}$ 
admits the following disintegration formula: 
$$
\mm\llcorner_{\T_{V}} = \int_{Q} \mm_{\alpha} \, \qq(\dd \alpha),
$$
where $\qq$ is a Borel probability measure over $Q \subset X$ such that 
$\QQ_{\sharp}( \mm\llcorner_{\T_{V}} ) \ll \qq$ and the map 
$Q \ni \alpha \mapsto \mm_{\alpha} \in \mathcal{M}_{+}(X)$ satisfies the following properties:
\begin{itemize}
\item[(1)] for any $\mm$-measurable set $B$, the map $\alpha \mapsto \mm_{\alpha}(B)$ is $\qq$-measurable; \smallskip
\item[(2)] for $\qq$-a.e. $\alpha \in Q$, $\mm_{\alpha}$ is concentrated on $\QQ^{-1}(\alpha) = X_{\alpha}$ (strong consistency); \smallskip
\item[(3)] for any $\mm$-measurable set $B$ and $\qq$-measurable set $C$, the following disintegration formula holds: 
$$
\mm(B \cap \QQ^{-1}(C)) = \int_{C} \mm_{\alpha}(B) \, \qq(\dd \alpha);
$$
\item[(4)]  For every bounded subset $\mathcal K\subset X$ there exists a constant $C_{\mathcal K}\in (0,\infty)$ such that
$$
 \mm_{\alpha}(\mathcal K) \leq C_{\mathcal K}, \quad \text{for $\qq$-a.e. $\alpha\in Q$.}
$$
\end{itemize}
Moreover, fixed any $\qq$ as above such that $\QQ_{\sharp}( \mm\llcorner_{\T_{V}} ) \ll \qq$, the disintegration is $\qq$-essentially unique (in the sense of \eqref{eq:UniqDisint}).
\end{theorem}


\subsection{$\ell^{p}$-cyclically monotone subsets contained in the transport set ${\mathcal T}_{V}$}

We will now obtain two results permitting to include $\ell^{p}$-cyclically monotone sets inside $\ell$-cyclically monotone sets. 
This technique has been introduced in  \cite{cava:decomposition} and  pushed further in \cite{CM1,CM2} for  the metric setting, to generalize localization 
paradigm to metric measure spaces using the equivalence between 
optimality and cyclical monotonicity.
 
In the present setting, since the cost $\ell^{p}$ may take the value $-\infty$, $\ell^{p}$-cyclical monotonicity does not 
directly imply optimality. 
Nontheless using \cite{biacar:cmono} and its consequences included in Proposition \ref{P:OptiffMon},
we will use cyclically monotone sets to 
construct \emph{locally optimal} couplings and to deduce local estimates 
on the disintegration that will be then globalized.

There is a simple and natural way to construct Wasserstein geodesics with $0<p<1$: 
translate along transport rays by a constant ``distance''. 
Notice that $0<p<1$ plays a crucial role, as an 
analogous statement in the Riemannian setting does not hold true 
for $W_{2}$.

\begin{proposition}\label{P:translationellpMonot}
Consider $\Lambda \subset \Gamma_{V}$ with the following property: 
there exists $t>0$ such that for each $(x,y) \in \Lambda$, $\tau(x,y) = t$. 
Then for each $0<p<1$ the set $\Lambda$ is $\ell^{p}$-cyclically monotone.
\end{proposition}

\begin{proof}
Given $(x_{1},y_{1}),\dots, (x_{n},y_{n}) \in \Lambda$, we need to prove 
$$
\sum_{i=1}^{n} \ell(x_{i},y_{i})^{p} \geq \sum_{i=1}^{n} \ell(x_{i+1},y_{i})^{p}, 
$$
that can be rewritten as 
\begin{equation}\label{E:last}
t \geq \left( \frac{1}{n}\sum_{i=1}^{n} \ell(x_{i+1},y_{i})^{p} \right)^{1/p}. 
\end{equation}
From Lemma \ref{lem:GammaVell1CM} the corresponding inequality for $p =1$ 
is valid:
$$
n t =\sum_{i=1}^{n} \ell(x_{i},y_{i})\geq \sum_{i=1}^{n} \ell(x_{i+1},y_{i});
$$
we rewrite it as
\begin{equation}\label{E:last1}
t \geq \frac{1}{n}\sum_{i=1}^{n}\ell(x_{i+1},y_{i}). 
\end{equation}
Since by assumption  $0<p<1$, the concavity of the function $\R\ni s \mapsto s^{p}$ implies 
$$
\left(\frac{1}{n} \sum_{i=1}^{n} \ell(x_{i+1},y_{i}) \right)^{p} \geq
\frac{1}{n}\sum_{i=1}^{n} \ell(x_{i+1},y_{i})^{p},
$$
which, combined with \eqref{E:last1}, gives \eqref{E:last}.
\end{proof}

In the next proposition we give a  second way to construct $\ell^{p}$-cyclically monotone sets (cf. \cite{cava:decomposition}).

\begin{proposition}\label{P:cpgeod} 
Let $\Delta \subset  \Gamma_{V}$ 
be such that 
\begin{equation}\label{E:monotone}
(\tau_{V}(x_{0}) - \tau_{V}(x_{1}))(\tau_{V}(y_{0}) - \tau_{V}(y_{1})) \geq 0,  \quad \text{for all } (x_{0},y_{0}),(x_{1},y_{1}) \in \Delta.
\end{equation}
Then $\Delta$ is 
$\ell^{p}$-cyclically monotone for each $p\in (0,1)$.
\end{proposition}

\begin{proof}
Let $\{(x_{1},y_{1}), \dots, (x_{N},y_{N})\} \subset \Delta$ be an arbitrary finite subset of $\Delta$. Define $s_{i} : = \tau_{V}(x_{i})$, $t_{i} : = \tau_{V}(y_{i})$ and consider the auxiliary measures 
$$
\eta_{0} : = \frac{1}{N} \sum_{i =1}^{N} \delta_{s_{i}}, 
\qquad 
\eta_{1} : = \frac{1}{N} \sum_{i=1}^{N} \delta_{t_{i}}.
$$
Notice that the support $\eta_{0}$ and $\eta_{1}$ 
are confined inside a compact real interval, say $I$.   
Consider finally the map $F: I \to I$ defined  by 
\begin{equation}
F(s)=
\begin{cases} t_{i} \quad \text{if } s=s_{i}, \\
0 \quad \text{ elsewhere }.
\end{cases}
\end{equation}
Trivially $F_{\sharp} \eta_{0} = \eta_{1}$;
moreover,  by \eqref{E:monotone}, $F$ is monotone on $\supp \, \eta_{0}$. 
This implies that $\gr(F)$ is also $|\cdot|^{2}$-cyclically monotone on $\supp \,\eta_{0}$ and in particular 
$$
\int |x-F(x)|^{2} \eta_{0}(\dd x) = W_{2}(\eta_{0},\eta_{1}),
$$
where $W_{2}$ is intended to be defined over 
$\P_{2}(\R)$.
By \cite[Remark 2.19 (ii)]{Vil:topics}, $F$ is optimal for any cost 
$c(x,y) = h(|x-y|)$, with $h$ strictly convex and non-negative. 
For $p\in (0,1)$,  consider  the function $h(r) : = -r^{p} +a$,
where $a$ can be taken to be 
$$
a > 2 \sup_{s\in I} |s|^{p}.
$$
Thus $\bar c(s,t) := - |t-s|^{p} + a $ is non-negative and falls into the hypothesis of \cite[Remark 2.19 (ii)]{Vil:topics}.
Hence $\gr(F)$ restricted to $\supp \, \eta_{0}$ is also $\bar c$-cyclically monotone. 
We can now conclude as follows:
\begin{align*}
\sum_{i=1}^{N} \ell(x_{i},y_{i})^{p} 
=  &~ \sum_{i=1}^{N} (\tau_{V}(y_{i}) - \tau_{V}(x_{i}) )^{p} =   - \sum_{i=1}^{N} \bar c(s_{i},t_{i}) + N a\\
\geq &~  - \sum_{i=1}^{N} \bar c(s_{i},t_{i+1}) + N a= \sum_{i=1}^{N} (|\tau_{V}(y_{i+1}) - \tau_{V}(x_{i})| )^{p} \\
\geq &~ \sum_{i=1}^{N} \ell(x_{i},y_{i+1})^{p},
\end{align*}
where in the last inequality we used \eqref{E:triang}.
\end{proof}


\subsection{Regularity of the conditional measures}

Recall that by the Disintegration Theorem \ref{T:sigma-disint} we can write $\mm=\int_{Q} \mm_{\alpha} \qq(\dd \alpha)$, where $\mm_{\alpha}$ is a non-negative  Radon measure on $X$  concentrated on the ray $X_{\alpha}$, for $\qq$-a.e. $\alpha\in Q$. The goal of this section is to prove that the conditional measures  $\mm_{\alpha}$'s are absolutely continuous with respect to the Hausdorff measure $\cH^{1}$ restricted to the ray $X_{\alpha}$, for $\qq$-a.e. $\alpha\in Q$. 
Such a regularity of $\mm_{\alpha}$ can be inferred from 
the behavior of $\mm$ with respect to translation along the transport set ${\mathcal T}_{V}$ (cf.  \cite{biacava:streconv}). 
\\Let us set some notation. First recall the definition \eqref{E:transport} of transport set with endpoints  $\T_{V}^{e}$.   For any Borel set $A \subset \T_{V}^{e}$  and 
$t \in [0,+\infty)$ we can associate  its ``forward'' translation 
$$
A_{t} : = 
P_{2}\{ (x,y) \in (A \times \T_{V}^{e}) \cap \Gamma_{V} \colon \tau(x,y) = t \}.
$$
If $A$ is an analytic set, $A_{t}$ is analytic as well (recall that projections of analytic set is again analytic). 
In particular, for $A\subset \T_{V}^{e}$ having
$\mm(A) >0$ it makes sense to consider the set
$$
\{ t \in [0,+\infty) \colon \mm(A_{t}) > 0 \}.
$$
and to evaluate its Lebesgue  measure.

\begin{proposition}\label{P:translation}
Let  $(X,\sfd, \mm, \ll, \leq, \tau)$ be a timelike non-branching, globally hyperbolic, Lorentzian geodesic space satisfying $\mathsf{TMCP}^{e}(K,N)$  for some $p\in (0,1), K\in \R, N\in [1,\infty)$. For any analytic set $A \subset \T_{V}^{e}\setminus \fb(\T_{V}^{e})$ having $\mm(A)>0$ there exists $s > 0$ and a compact subset $B\subset A$ such that 
\begin{equation}\label{eq:PropBtransl}
\bigcup_{t\in [0,s]} B_{t} \Subset X, \qquad B_{t}\subset  \T_{V}^{e}\setminus \fb(\T_{V}^{e})\quad \text{and} \quad    \mm(B_{t}) > 0\quad \forall t \in [0,s).
\end{equation}
In particular, 
$|\{ t \in [0,+\infty) \colon \mm(A_{t}) > 0 \} | > 0$.
\end{proposition}

\begin{proof}
 Consider $A \subset \T_{V}^{e} \setminus \fb(\T_{v}^{e})$ with 
$\mm(A) > 0$. Take $s \in [0,+\infty) $
and consider the following subset of $\Gamma_{V}$:
$$
\Lambda_{s} : = 
\{ (x,y) \in (A \times \T_{V}^{e}) \cap \Gamma_{V}
\colon \tau(x,y) = s\}.
$$
From Proposition \ref{P:translationellpMonot} we deduce 
that $\Lambda_{s}$ is $\ell^{p}$-cyclically monotone, 
for each $s \in [0,+\infty) $.
We also observe that 
$$
0\leq s_{1} \leq s_{2} \quad\Longrightarrow \quad
P_{1}(\Lambda_{s_{1}}) \subset P_{1}(\Lambda_{s_{2}}) \subset A.
$$
Moreover, since $A \subset \T_{V}^{e} \setminus \fb(\T_{V}^{e})$, it follows that for each $x \in A$ there exist $s \in (0,+\infty) $ and $z \in \T_{V}$ 
such that $(x,z) \in \Lambda_{s}$, 
showing that 
$$
\bigcup_{s>0} P_{1}(\Lambda_{s}) = A.
$$
In particular, by monotone convergence,  we have  $\lim_{s\downarrow 0} \mm(P_{1}(\Lambda_{s})) 
= \mm(A) > 0$. Define then 
$B : = P_{1}(\Lambda_{s})$ for $s>0$ small enough so that 
$\mm(B) > 0$. 
We can also find a compact subset of $B$ of positive $\mm$-measure, that we still denote by $B$, 
and a measurable map $T : B \to \T_{V}$ such that 
$(x,T(x)) \in \Lambda_{s}$ for all $x \in B$.
We then consider the following measures 
$$
\mu_{0} : = \mm\llcorner_{B}/\mm(B), \qquad
\mu_{1} : = T_{\sharp} \mu_{0}.
$$
By construction, the coupling associated to $T$, i.e. 
$\pi_{T} = (\id, T)_{\sharp}\mu_{0}$ verifies the following two conditions:
$$
\int \tau(x,y)^{p} \pi_{T} (\dd x\dd y) = s^{p} \in (0, +\infty). 
$$
Since $\pi_{T}$ is $\ell^{p}$-cyclically monotone and $\pi_{T}(\{\tau>0\})=1$,
Proposition \ref{P:OptiffMon} ensures  it is an $\ell_{p}$-optimal coupling. 
Up to further restricting $\pi$, we can assume that $\supp \,\pi \Subset \{\tau >0\}$.
By Corollary \ref{C:MCP}, there is a unique $\ell_{p}$-geodesic $(\mu_{t})_{t\in [0,1]}$ between $\mu_{0}$ and $\mu_{1}$, and
$\mu_{t} \ll \mm$ for all $t \in [0,1)$. $\mathcal{K}$-global hyperbolicity implies that $\bigcup_{t\in [0,1]} \supp\,\mu_{t} \Subset X$.

Since $T$ is a translation of length $s$, 
it follows that $\mu_{t}$ is concentrated
inside $B_{ts}\subset A_{ts}$; being absolutely continuous, it implies that 
$$
\mm(A_{ts})> \mm(B_{ts}) > 0, \quad \forall t \in [0,1),
$$
proving the claim.
\end{proof}

\begin{corollary}\label{cor:maTv=0}
Under the same assumptions of Proposition \ref{P:translation}, it holds $\mm(\fa(\T_{V}^{e}))=0$.
\end{corollary}

\begin{proof}
Assume by contradiction $\mm(\fa(\T_{V}^{e}))>0$. Setting $A=\fa(\T_{V}^{e})$ in Proposition \ref{P:translation}, we obtain $B\subset A$ compact subset satisfying \eqref{eq:PropBtransl}.

\textbf{Step 1}.
With the same notation of Proposition \ref{P:translation}, we first claim that 
\begin{equation}\label{Atdisjoint}
B_{t_{0}}\cap B_{t_{1}}= \emptyset,  \quad\text{ for any }0<t_{0}< t_{1}<s.
\end{equation}
Indeed, if by  contradiction  there exists $y\in B_{t_{0}}\cap B_{t_{1}}$ then there exist $x,z \in  \fa(\T_{V}^{e})$ such that $\tau(x,y)=t_{0}$,  $\tau(z,y)=t_{1}$, $(x,y)\in \Gamma_{V}$ and $(z,y)\in \Gamma_{V}$. Since $y\notin \fb(\T_{V}^{e})$, we can repeat the argument in Case 3 of the proof of Proposition \ref{P:equivalence} and get that $(z,x)\in \Gamma_{V}$ contradicting that $x \in  \fa(\T_{V}^{e})$.
\\

\textbf{Step 2}. From Proposition \ref{P:translation} we have that there are uncountably many $t\in [0,s)$ satisfying $\mm(B_{t})>0$ and  \eqref{Atdisjoint}. Hence, on the one hand,
\begin{equation}\label{eq:mcupBtinfty}
\mm\left( \bigcup_{t\in (0,s)} B_{t} \right)=+\infty.
\end{equation} 
On the other hand, since by \eqref{eq:PropBtransl} $\bigcup_{t\in [0,s]} B_{t}$ is relatively compact and $\mm$ is by assumption a Radon measure, we have  $\mm\left( \bigcup_{t\in [0,s]} B_{t} \right)<\infty$ contradicting \eqref{eq:mcupBtinfty}.
\end{proof}

Of course, if we assume that $X$ endowed with the reversed causal structure satisfies the assumptions of Proposition \ref{P:translation}, then also $\mm(\fb(\T_{V}^{e}))=0$.

\begin{proposition}
Under the same assumptions of Proposition \ref{P:translation}, the conditional measure $\mm_{\alpha}$ (given in the Disintegration Theorem \ref{T:sigma-disint}) is absolutely continuous with respect to the Lebesgue  measure $\L^{1}\llcorner_{X_{\alpha}}$ along the ray $X_{\alpha}$, for $\qq$-a.e. $\alpha\in Q$.
\end{proposition}

\begin{proof}
Assume by contradiction there is a Borel subset $\hat{Q}\subset Q$ with $\qq(\hat{Q})>0$ such that $\mm_{\alpha} \not\ll \L^{1}\llcorner_{X_{\alpha}}$ for each $\alpha\in \hat{Q}$.
\\Let $\mm_{\alpha}=h_{\alpha} \L^{1}\llcorner_{X_{\alpha}} + \mm_{\alpha}^{\perp}$ be the Lebesgue decomposition of $\mm_{\alpha}$ with respect to  $\L^{1}\llcorner_{X_{\alpha}}$, with $\mm_{\alpha}^{\perp} \perp\L^{1}\llcorner_{X_{\alpha}}$. Then, for every $\alpha \in \hat{Q}$ there exists a Borel subset $A^{\alpha}\subset X_{\alpha}$ such that 
\begin{equation}\label{eq:PropAalpha}
\L^{1}(A^{\alpha})=0 \quad \text{ and } \quad \mm_{\alpha}^{\perp}=\mm_{\alpha}^{\perp}\llcorner_{A^{\alpha}}.
\end{equation}
 Define $A:=\bigcup_{\alpha\in \hat{Q}} A^{\alpha}\subset {\mathcal T}_{V}$ and observe that the Disintegration Theorem \ref{T:sigma-disint} gives
$$
\mm(A)=\int_{Q} \mm_{\alpha}(A) \, \qq(\dd \alpha)= \int_{\hat{Q}} \mm_{\alpha}^{+}(A^{\alpha}) \, \qq(\dd \alpha)>0.
$$
Proposition \ref{P:translation} implies
\begin{equation}\label{eq:intR+mAt>0}
0< \int_{\R^{+}} \mm(A_{t}) \, \dd t =  \int_{\R^{+}}  \left( \int_{Q} \mm_{\alpha}(A_{t}) \, \qq(\dd \alpha) \right)  \dd t =   \int_{Q}  \left( \int_{\R^{+}}  \mm_{\alpha}(A_{t}) \, \dd t  \right)\qq(\dd \alpha) ,
\end{equation}
where in the second equality we used the Disintegration Theorem  \ref{T:sigma-disint}, and the third equality follows by Fubini-Tonelli's Theorem. In order to simplify the notation, for the rest of the proof  we identify $X_{\alpha}$ with an interval in the Real line (see Lemma \ref{lem:XalphaI}). Observe that
\begin{align}
 \int_{\R^{+}}  \mm_{\alpha}(A_{t}) \, \dd t &= \L^{1} \otimes \mm_{\alpha} \left\{(t,x) \,:\, t>0, \,x\in X_{\alpha}, \, x-t\in A^{\alpha} \right\} \nonumber\\
 &=\int_{X_{\alpha}} \L^{1} (\left\{t>0\,:\,  x-t\in A^{\alpha} \right\}) \, \mm_{\alpha}(\dd x)=0, \label{intmAalpha=0}
\end{align}
where in the last equality we used that 
$$
\L^{1} (\left\{t>0\,:\,  x-t\in A^{\alpha} \right\})=\L^{1}(A_{\alpha})=0,
$$
by the invariance properties of the Lebesgue measure and \eqref{eq:PropAalpha}.
\\Plugging \eqref{intmAalpha=0} into \eqref{eq:intR+mAt>0} gives the contradiction $0<0$.
\end{proof}

We summarise the content of this subsection, combined with Lemma \ref{lem:I+VTV} and  the Disintegration Theorem  \ref{T:sigma-disint},  in the next statement.

\begin{theorem}\label{P:nointialpoints}
Let  $(X,\sfd, \mm, \ll, \leq, \tau)$ be a timelike non-branching,  globally hyperbolic, Lorentzian geodesic space satisfying $\mathsf{TMCP}^{e}(K,N)$  for some $p\in (0,1), K\in \R, N\in [1,\infty)$,  and assume that the causally-reversed structure satisfies the same conditions. 
Let $V\subset X$ be a Borel achronal FTC subset, $\T_{V}^{e}, \fa(\T_{V}^{e}), \fb(\T_{V}^{e})$ and $\T_{V}$ be defined in \eqref{E:transport}, \eqref{eq:defendpoints}, \eqref{E:nbtransport}.

Then $\mm(\fa(\T_{V}^{e}))=\mm(\fb(\T_{V}^{e})=0$ and the following disintegration formula holds true: 
\begin{equation}\label{E:disintegration}
\mm\llcorner_{I^{+}(V)}= \mm\llcorner_{\T^{e}_{V}} = 
\mm\llcorner_{\T_{V}} 
= \int_{Q} \mm_{\alpha}\, \qq(\dd \alpha)= \int_{Q} h(\alpha,\cdot) \, \L^{1}\llcorner_{X_{\alpha}}\, \qq(\dd \alpha),
\end{equation}
where
\begin{itemize}
\item $\qq$ is a probability measure over the Borel quotient set $Q \subset \T_{V}$;
\item  $h(\alpha,\cdot)\in L^{1}_{loc}(X_{\alpha}, \L^{1}\llcorner_{X_{\alpha}})$ for $\qq$-a.e. $\alpha\in Q$;
\item  the map 
$\alpha \mapsto \mm_{\alpha}(A)= h(\alpha,\cdot)\L^{1}\llcorner_{X_{\alpha}}(A)$ is 
$\qq$-measurable for every Borel set $A \subset \T_{V}$.
\end{itemize}
\end{theorem}


\subsection{Localization of $\TMCP^{e}(K,N)$}

In this section we localize the curvature condition $\TMCP^{e}(K,N)$ to the one dimensional metric measures spaces $(X_{\alpha},|\cdot|, \mm_{\alpha})$ decomposing $\T_{V}$, in the sense of the Disintegration Theorem 
\ref{P:nointialpoints} (cf. \cite{biacava:streconv, CaMoLap}).

\begin{theorem}\label{P:localKant}
Let $(X,\sfd, \mm, \ll, \leq, \tau)$ and $V\subset X$ be as in  Theorem \ref{P:nointialpoints} with $N\in (1,\infty)$, and recall the Disintegration formula \eqref{E:disintegration}.

Then,  for $\qq$-a.e. 
$\alpha \in Q$, the density $h(\alpha, \cdot)$ 
has an almost everywhere 
representative that is locally Lipschitz and strictly positive 
in the interior of $X_{\alpha}$, continuous on its closure, and satisfying 
\begin{equation}\label{E:MCP0N1d}
\left(\frac{\fs_{K/(N-1)}(b-\tau_{V}(x_{1}))}{\fs_{K/(N-1)}(b-\tau_{V}(x_{0}))}\right)^{N-1}
\leq \frac{h(\alpha, x_{1} ) }{h (\alpha, x_{0})} \leq 
\left( \frac{\fs_{K/(N-1)}(\tau_{V}(x_{1}) - a) }{\fs_{K/(N-1)}(\tau_{V}(x_{0}) -a)} \right)^{N-1},  
\end{equation}
for all $x_{0},x_{1}\in X_{\alpha}$, with $0\leq  a< \tau_{V}(x_{0})<\tau_{V}(x_{1})<b<\pi\sqrt{(N-1)/(K\vee 0)}$.  

In other words, for $\qq$-a.e. $\alpha\in Q$, the one-dimensional metric measure space 
$(X_{\alpha},|\cdot|, \mm_{\alpha})$ satisfies $\MCP(K,N)$.
\end{theorem}

\begin{proof}
For $x\in \T_{V}$ we will write $R(x)$ to denote its equivalence class in $(\T_{V}, R_{V})$, i.e. the ``ray passing through $x$'' (recall Proposition \ref{P:equivalence}). For a subset $B\subset \T_{V}$, we denote $R(B):=\bigcup_{x\in B} R(x)$.
\\Let $\bar Q \subset Q$ be an arbitrary compact subset of positive $\qq$-measure for which there
exist $\ve >0$ and $0<a_{0} < a_{1}$ such that   
\begin{equation*}
\begin{split}
\sup_{x,y \in X_{\alpha}} \tau(x,y) > \ve,
\qquad 
X_{\alpha} \cap \{\tau_{V}= a_{0}\} \neq \emptyset, 
\qquad X_{\alpha} \cap \{\tau_{V}= a_{1}\} \neq \emptyset \qquad \forall  \alpha \in \bar Q,\\
R(\bar{Q})\cap \tau_{V}^{-1}([a_{0}, a_{1}]) \Subset X, \quad \{(x,y)\in \Gamma_{V}\,:\, x,y\in R(\bar{Q}), \,  \tau_{V}(x)=a_{0}, \,  \tau_{V}(y)=a_{1} \} \Subset \{\tau>0\}.
\end{split}
\end{equation*}
For any $A_{0}\in (a_{0}, a_{1})$  
and $L_{0} > 0$ satisfying 
$A_{0} + L_{0} < a_{1}$,  consider the probability measure 
$$
\mu_{0} : = c_{\bar{Q},A_{0},L_{0}} \cdot \mm\llcorner_{\tau_{V}^{-1}(A_{0, }A_{0}+L_{0}) \cap R(\bar Q)}, 
$$
where $c_{\bar{Q},A_{0},L_{0}}$ is the normalization constant so that $\mu_{0}\in \Prob_{c}(X)$. 
\\Let $T_{a_{1}}:R(\bar{Q})\to R(\bar{Q})\cap \tau_{V}^{-1}(a_{1})$ be the ``ray-projection map''  defined by  $T_{a_{1}}(x) = \tau_{V}^{-1}(a_{1}) \cap R(x)$ and set
 $\mu_{1} : = (T_{a_{1}})_{\sharp} \mu_{0}$. Notice that $\{(x,T_{a_{1}}(x))\,:\, x\in \supp\,\mu_{0} \}\Subset \{\tau>0\}$.
 Moreover, Proposition \ref{P:cpgeod} implies that the associated coupling $\pi_{T_{a_{1}}}=(\id, T_{a_{1}})_{\sharp} \mu_{0}$ is $\ell^{p}$-cyclically monotone and thus, by  Proposition \ref{P:OptiffMon}, $\ell_{p}$-optimal. Analogously, setting $T^{t}(x):= \tau_{V}^{-1}((1-t)\tau_{V}(x)+ta_{1}) \cap R(x)$, it follows that 
  the curve of probability measures $\bar \mu_{t}=T^{t}_{\sharp} \mu_{0}$ is an $\ell_{p}$-geodesic.
Notice that  
\begin{equation}\label{eq:barmutconcentr}
\bar{\mu}_{t} \left(\tau_{V}^{-1}(A_{t, }A_{t} + L_{t}) \cap R(\bar Q) \right)=1,
\end{equation}
 where $A_{t} : = (1-t)A_{0} + t a_{1}$ and 
$L_{t}: = (1-t)L_{0} $.
\\Since by Corollary \ref{C:MCP}  there is a unique $\ell_{p}$-geodesic $(\mu_{t})_{t\in [0,1]}$ between $\mu_{0}$ and $\mu_{1}$, it must be $(\mu_{t})_{t\in [0,1]}=(\bar{\mu}_{t})_{t\in [0,1]}$. Thus,  combining \eqref{eq:barmutconcentr} with  \eqref{E:MCP}, we get 
$$
\mm(\tau_{V}^{-1}(A_{t}, A_{t} + L_{t}) \cap R(\bar Q))
\geq \sigma_{K/N}^{(1-t)}(\|\tau \|_{L^{2}(\pi_{T_{a_{1}}})})^{N} \mm(\tau_{V}^{-1}(A_{0} , A_{0} + L_{0}) \cap R(\bar Q)),
$$
that can be rewritten using the Disintegration formula \eqref{E:disintegration} as
\begin{equation*}
\int_{\bar{Q}} \mm_{\alpha}(\tau_{V}^{-1}(A_{t}, A_{t} + L_{t}) \, \qq(\dd \alpha)
\geq  \sigma_{K/N}^{(1-t)}(\|\tau \|_{L^{2}(\pi_{T_{a_{1}}})})^{N}
\int_{\bar{Q}}   \mm_{\alpha}(\tau_{V}^{-1}(A_{0}, A_{0} + L_{0})) \, \qq(\dd \alpha).
\end{equation*}
Recalling that $\mm_{\alpha}=h(\alpha, \cdot) \L^{1}$,  the arbitrariness of $\bar Q, a_{0}, a_{1}, A_{0}, L_{0}$ (letting $L_{0}\downarrow 0$) implies that
$$
(1-t)h_{\alpha}( (1-t)A_{0}+ t a_{1})\geq \sigma_{K/N}^{(1-t)}(a_{1} - A_{0})^{N} h_{\alpha}(A_{0})
$$
for $\qq$-a.e. $\alpha \in Q$, $\L^{1}$-a.e. $t\in (0,1)$,
that can be rewritten as 
\begin{equation}\label{E:firstMCP}
\frac{b-s}{b-a} h_{\alpha}( s)\geq \sigma_{K/N}^{(\frac{b-s}{b-a})}(b - a)^{N} h_{\alpha}(a),\quad \text{for $\qq$-a.e. $\alpha \in Q$, $\L^{1}$-a.e. $s\in (a,b)\subset X_{\alpha}$.}
\end{equation}
It is a standard trick to obtain the first inequality in \eqref{E:MCP0N1d} out of \eqref{E:firstMCP}.
We anyway include few details for the case $K>0$, the other one being completely analogous. 
Using the notation of \cite{sturm:II} and of \cite{BS10} 
we consider $\tau_{K,N}^{(t)}(\vartheta) : = t^{1/N}\sigma_{K/(N-1)}^{(t)}(\vartheta)^{\frac{N-1}{N}}$.
While $\tau_{K,N}^{(t)}(\vartheta)$ is always larger than $\sigma_{K/N}^{(t)}(\vartheta)$,
for $\vartheta\ll 1$ the two coefficients are almost identical: 
to be precise if $0< K'< \tilde K < K$ we can choose $\vartheta^{*}>0$ so that for all 
$0\leq \vartheta \leq \vartheta^{*}$ and all $t \in [0,1]$ the reverse inequality 
$\tau_{K',N}^{(t)}(\vartheta) \leq \sigma_{\tilde K/N}^{(t)}(\vartheta)$ is valid. 
Hence \eqref{E:firstMCP} becomes: 
$$
\frac{b-s}{b-a} h_{\alpha}( s)\geq \tau_{K',N}^{(\frac{b-s}{b-a})}(b-a)^{N} h_{\alpha}(a),
$$
provided $0<b-a <\vartheta^{*}$, that can be rewritten in the following form:
\begin{equation}\label{E:locMCP}
h_{\alpha}( s)\geq \sigma_{K'/(N-1)}^{(\frac{b-s}{b-a})}(b-a)^{N-1} h_{\alpha}(a), \qquad  \text{for $\qq$-a.e. $\alpha \in Q$, $\L^{1}$-a.e. $s\in (a,b)\subset X_{\alpha}$,   $ b-a <\vartheta^{*}$.}
\end{equation}
We have therefore proved that for each $K' <K$ the following is true: for any point $a$ there exists 
a neighborhood of $a$ where \eqref{E:locMCP} is valid. 
As shown for instance in \cite{BS10,cavasturm:MCP} this implies that the same inequality is valid
on the whole domain of $h_{\alpha}$ (local-to-global property).
Taking then the limit as $K' \to K$ from below we obtain the first inequality of 
\eqref{E:MCP0N1d}. 
\\Applying the analogous procedure to the causal-reversed structure we obtain the second inequality of \eqref{E:MCP0N1d}.
\end{proof}

\begin{remark} [The case $N=1$]\label{rem:MCPN01}
In case $N = 1$, under the same assumptions of  Theorem \ref{P:localKant} one can follow the proof up to \eqref{E:firstMCP} and obtain that 
$$
\frac{b-s}{b-a} h_{\alpha}( s)\geq \sigma_{K}^{(\frac{b-s}{b-a})}(b - a) 
h_{\alpha}(a),\quad \text{for $\qq$-a.e. $\alpha \in Q$, $\L^{1}$-a.e. $s\in (a,b)\subset X_{\alpha}$.}
$$
If $K \geq 0$, then $\sigma_{K}^{(\frac{b-s}{b-a})}(b - a) \geq \frac{b-s}{b-a}$ implying 
$h_{\alpha}( s) \geq h_{\alpha}(a)$; reversing the causal structure, it follows that 
$h_{\alpha}$ has to be constant.
For $K< 0$ we compute the Taylor expansion
$$
\sigma_{K}^{(t)}(\theta) = 
t\left[ \frac{1 + t^{2}\frac{\theta^{2}}{6} (-K) + o(\theta^{4}) }
{1 + \frac{\theta^{2}}{6} (-K) + o(\theta^{4})}\right] = t \left[ 1 - \frac{\theta^{2}}{6}(-K)(1-t) + o(\theta^{4}) \right].
$$
Hence we can conclude that $\liminf_{b\to a} (h_{\alpha}(b)-h_{\alpha}(a))/(b-a) \geq 0$. 
Again reversing the causal structure we obtain that $h_{\alpha}$ is locally Lipschitz and 
the reverse inequality holds, yielding $h_{\alpha}$ constant as well.
\end{remark}


\section{Applications}
\subsection{Synthetic mean curvature bounds for achronal FTC subsets}\label{SS:SyntMC}
In this section we will work under the standing assumptions of Theorem \ref{P:nointialpoints}.

Recall that, thanks to  Lemma \ref{L:initialpoint} and Lemma \ref{lem:I+VTV}, 
$\T^{e}_{V}\subset \{\tau_{V}>0\} \cup V$.
For each $t \geq 0$, we consider the map 
$f_{t}: \Dom(f_{t})\subset Q \to \T^{e}_{V}$, where 
\begin{equation}\label{eq:defft}
\Dom(f_{t}) : = \{\alpha \in Q \colon 
\bar{X}_{\alpha} \cap  \{\tau_{V}=t\}\setminus \fb(\T^{e}_{V})\neq \emptyset \},
\qquad 
f_{t}(\alpha) : = \bar{X}_{\alpha} \cap  
\{\tau_{V} = t\}\setminus \fb(\T^{e}_{V}),
\end{equation}
where $\bar{X}_{\alpha}$ denotes the closure of the ray $X_{\alpha} \subset X$. 
\\Proposition \ref{P:equivalence} ensures that  $f_{t}$ is single valued  for every $t\geq 0$ and injective for $t > 0$. Moreover  $f_{0}(\alpha)\in V$ for all $\alpha\in \dom(f_{0})$ (see Proposition \ref{eq:closureVSendpoints}). 
\\Thus, for each $\mm$-measurable subset $A\subset \T_{V}^{e}$ having $\mm(A) < \infty$ 
the next identities hold true:
\begin{equation}\label{eq:disintftsharp}
\mm(A) = \int_{Q} \int_{A\cap X_{\alpha}} h(\alpha,t)\dd t\,\qq(\dd \alpha) = \int_{[0,+\infty) } (f_{t})_{\sharp} (h(\cdot,t) \, \qq(\dd \alpha))(A)\,\dd t,
\end{equation}
where the first identity is the Disintegration formula \eqref{E:disintegration} and the second identity follows from Fubini-Tonelli's Theorem.
Define then
\begin{equation}\label{eq:defHt}
\H_{t} : = (f_{t})_{\sharp}\,h(\cdot,t) \qq, \quad \text{for all } t\geq 0.
\end{equation}
By definition, $\H_{t}$ is concentrated on the level set $\{\tau_{V} = t\}$. 
In particular $\H_{0}$ is concentrated on $V$.
An expert reader will recognise that $\H_{t}(\{\tau_{V} = t\})$ is a kind of $\tau$-Minkowski content  of the set $\{\tau_{V} = t\}$, with respect to $\mm$.
We summarize this construction in the following

\begin{proposition}\label{prop:coareaHt}
The following coarea-type formula holds true: 
$$
\mm\llcorner_{\T_{V}^{e}} = \int_{0}^{\infty} \H_{t} \,\dd t,
$$
meaning that for each measurable set $A \subset \T_{V}^{e}$ with 
$\mm(A) < \infty$, 
the map $[0, \infty) \ni t \mapsto \H_{t}(A)$ is measurable and 
$$
\mm(A) = \int_{0}^{\infty}\H_{t}(A)\dd t  = \int_{0}^{\infty}\H_{t}(A\cap\{\tau_{V} = t\})\dd t.
$$ 
\end{proposition}

We use the previous codimension-one measures to  propose the following weak notion 
of upper bound on the mean curvature of $V$.   Notice that, even if $\cH_0$ (as well as $\cH_t$ for every $t\geq 0$) is a well defined measure, in general it is not finite (even locally). Since from a geometric point of view the mean curvature is the first variation of the area, in order to speak of the former it is natural to assume that the latter is locally finite. In what follows, we will thus assume that $\cH_0$ is a non-negative Radon measure.

In the next definition we use the ``initial-point projection map'' 
$\fa: \T_{V}\to V,\, \fa:=f_{0} \circ \QQ$. 
It is not hard to check it is  
Borel measurable: notice indeed that 
$$
\gr (\fa) = \{ (x,y) \in \T_{V} \times V \colon \tau_{V}(x) = \tau(y,x) \},
$$
showing that $\gr (\fa)$ is analytic (recall that $\T_V$ is analytic). Then by \cite[Theorem 4.5.2]{Srivastava}
the map $\fa : \T_V \to V$ is Borel. 
\begin{definition}\label{D:MeanCurv}
The Borel achronal FTC subset $V\subset X$ has \emph{forward mean curvature bounded below by 
$H_{0} \in \R$} if $\cH_0$ is a non-negative Radon measure  with $(f_{0})_{\sharp} \qq\ll \cH_{0}$ and such that for any normal variation 
$$
V_{t,\phi} : = \{ x \in \T_{V} \colon 0 \leq \tau_{V}(x) \leq t \phi(\fa(x)) \},
$$
the following inequality holds true: 
$$
\limsup_{t \to 0} \frac{\mm(V_{t,\phi}) - t 
\int_{V} \phi \H_{0}}{t^{2}/2} \geq H_{0} \int_{V}\phi^{2} \H_{0}, 
$$
for any bounded Borel function $\phi:V\to [0,\infty)$ with compact support.
Analogously $V$ has \emph{forward mean curvature bounded above by  
$H_{0} \in \R$} if  $\cH_0$ is a non-negative Radon measure and for any normal variation 
$V_{t,\phi}$ as above 
the following inequality holds true: 
\begin{equation}\label{eq:2ndvarVol}
\liminf_{t \to 0} \frac{\mm(V_{t,\phi}) - t 
\int_{V} \phi \H_{0}}{t^{2}/2} \leq H_{0} \int_{V}\phi^{2} \H_{0}, 
\end{equation}
for any bounded  Borel  function $\phi:V\to [0,\infty)$ with compact support.
\end{definition}

\begin{remark}[The condition $(f_{0})_{\sharp} \qq\ll \cH_{0}$]
The condition 
\begin{equation}\label{eq:f0sharpqACH0}
(f_{0})_{\sharp} \qq\ll \cH_{0}\; \text{or, equivalently, } h(\alpha,0)>0 \; \text{for $\qq$-a.e. }\alpha\in Q,
\end{equation}
 should be read as a natural ``codimension one'' assumption on $V$, in the sense of $\tau$-Minkowski content. More precisely, \eqref{eq:f0sharpqACH0} is equivalent to require that 
 \begin{equation}\label{H0CE}
\forall E\subset V \text{ Borel  with $\qq(f_{0}^{-1}(E))>0$ it holds } \;  \liminf_{t\downarrow 0} \frac{1}{t} \; \mm(V_{t, \chi_{E}})>0.
\end{equation}
Indeed, from Theorem \ref{P:localKant}, it follows that there exists a continuous function $\omega_{K,N}:[0,\infty)\to [0,\infty)$ with $\omega_{K,N}(0)=0$ such that for $\qq$-a.e.  $\alpha\in Q$ it holds $|h(\alpha,s)- h(\alpha,0)|\leq \omega_{K,N}(s)$. Therefore
\begin{align*}
\mm(V_{t, \chi_{E}}) &= \int_{0}^{t} \int_{E} h(\alpha,s) \, \qq(\dd \alpha) \, \dd s= t \int_{E} h (\alpha,0) \, \qq(\dd \alpha)  +o(t) = t \cH_{0}(E)+o(t),
\end{align*}
and \eqref{H0CE} is easily seen to be equivalent to \eqref{eq:f0sharpqACH0}.
\\

\noindent
Let us also mention the following \emph{sufficient} condition implying \eqref{eq:f0sharpqACH0}: 
\begin{equation}\label{eq:extXalphaPastV}
\text{For $\qq$-a.e.} \; \alpha\in Q \; \, \exists x\in I^{-}(V), \, y\in X_{\alpha} \subset \T_{V} \text{ such that } \tau(x,y)=\tau(x, \fa(y))+\tau(\fa(y), y).
\end{equation}
Recalling that $\fa(y)\in V$, \eqref{eq:extXalphaPastV} amounts to ask that the ray $X_{\alpha}$ can be slightly extended past $V$, for $\qq$-a.e. $\alpha\in Q$; in this way, the point $\fa(y)=f_{0}(\alpha)\in V$ becomes an \emph{interior} point of the extended ray $(\tilde{X}_{\alpha}, |\cdot|, \tilde{\mm}_{\alpha})$, i.e.  $X_{\alpha}=\tilde{X}_{\alpha}\cap \{\tau_{V}>0\}, \quad \mm_{\alpha}= \tilde{\mm}_{\alpha}\llcorner  \{\tau_{V}>0\}$. 

We next briefly discuss why \eqref{eq:extXalphaPastV} implies \eqref{eq:f0sharpqACH0}. It is not hard to check that the family of (maximally) extended rays $\{(\tilde{X}_{\alpha}, |\cdot|, \tilde{\mm}_{\alpha})\}_{\alpha\in Q}$  corresponds to  (a part of) the disintegration of $\mm\llcorner (I^{+}(V)\cup I^{-}(V))$ associated to the signed time separation function $\tau_{V}$. Following verbatim the proof of Theorem \ref{P:localKant}, we get that $(\tilde{X}_{\alpha}, |\cdot|, \tilde{\mm}_{\alpha})$ satisfies $\MCP(K,N)$ for $\qq$-a.e. $\alpha\in Q$. In particular, \eqref{E:MCP0N1d} ensures that the density $h(\alpha, \cdot)$ does not vanish in the interior of $\tilde{X}_{\alpha}$ and thus $h(\alpha, 0)>0$ for $\qq$-a.e. $\alpha\in Q$.

\end{remark}

\begin{remark}[The disintegration formula, the measures $\cH_{t}$ and the mean curvature bounds in the smooth setting]
Let  $(M^{n}, g)$ be a $2\leq n$-dimensional smooth globally hyperbolic space-time and  $V\subset M$ be a smooth compact achronal spacelike hypersurface without boundary.  Then, the signed time-separation function $\tau_{V}$ from $V$  is smooth on a neighbourhood $U$ of $V$ and $\nabla \tau_{V}$ is the smooth timelike past-pointing unit normal vector field along $V$. More precisely, 
$$\nabla \tau_{V}(x) \perp T_{x} V, \quad g(\nabla \tau_{V}(x), \nabla \tau_{V}(x))=-1, \quad \forall x\in V.$$
Denote with $\vol_{g}$ the volume measure of $(M^{n},g)$ and with  $\vol_{V}$ the induced $(n-1)$-dimensional volume measure on $V$. By compactness of $V$, there exists $\delta>0$ such that the $g$-geodesic $[0, \delta]\ni t\mapsto \exp_{x}(-t \nabla \tau_{V}(x))$ is a future pointing maximal geodesic, for every $x\in V$. Define
$${\mathcal U}:=V\times[0, \delta]\subset V\times \R, \quad \Phi:{\mathcal U}\to M, \; \Phi(x,t):=\exp_{x}(-t \nabla \tau_{V}(x)).$$
For $\delta>0$ small enough it is a standard fact (tubular neighbourhood theorem) that $\Phi$ is a diffeomorphism onto its image and that the following integration formula holds true:
\begin{equation}\label{eq:intMvolV}
\int_{M} \varphi\, d\vol_{g}=\int_{V}\int_{0}^{\delta} \varphi \circ \Phi(x,t) \, \det D\Phi_{(x,t)}|_{T_{x}V} \, \dd t \, \vol_{V}(\dd x), \quad \forall \varphi\in C_{c}(\Phi({\mathcal U})).
\end{equation}
Consider also the map $\QQ:\Phi({\mathcal U})\to V$ given by $\QQ:=P_{1}\circ \Phi^{-1}$. Notice that, for every $x\in V$, it holds
$$\QQ^{-1}(V)=\T^{e}_{V}\cap \Phi({\mathcal U})= \T_{V}\cap \Phi({\mathcal U}), \quad \QQ^{-1}(x)=R(x)\cap \Phi({\mathcal U}) =[x]_{(\T_{V}, R_{V})}\cap \Phi({\mathcal U}),$$
i.e.  $\QQ^{-1}(x)$ is the transport ray associated to $\tau_{V}$ intersected with $\Phi({\mathcal U})$. Moreover,
$$
\qq:=\QQ_{\sharp} (\vol_{g}\llcorner_{\Phi({\mathcal U})})=  \psi \,\vol_{V} \ll \vol_{V}, \quad \text{where } \psi(x):= \left(\int_{0}^{\delta}   \det D\Phi_{(x,t)}|_{T_{x}V} \, \dd t \right), \, \forall x\in V.
$$
Hence, we can identify $Q$ with $V$, and the quotient measure $\qq$ with  $\psi \, \vol_{V}$. The integration formula \eqref{eq:intMvolV} can be thus rewritten as
\begin{equation}\label{eq:intMqq}
\int_{M} \varphi\, d\vol_{g}=\int_{V} \frac{1}{\psi(x)} \int_{0}^{\delta} \varphi \circ \Phi(x,t) \, \det D\Phi_{(x,t)}|_{T_{x}V} \, \dd t \, \qq(\dd x), \quad \forall \varphi\in C_{c}(\Phi({\mathcal U})).
\end{equation}
The uniqueness statement \eqref{eq:UniqDisint}  in the disintegration formula combined with \eqref{E:disintegration} and  \eqref{eq:intMqq} gives:
$$
h_{\alpha}(t)=\frac{1}{\psi(\alpha)}  \, \det D\Phi_{(\alpha,t)}|_{T_{\alpha}V}, \quad  h_{\alpha}(0)=\frac{1}{\psi(\alpha)}, \quad \forall \alpha\in V, \; \forall t\in [0,\delta].
$$
Moreover,  observing that $\Phi(\alpha,t)=f_{t}(\alpha)$ where the latter was defined in \eqref{eq:defft}, it follows that the measure $\H_{t}$ defined in \eqref{eq:defHt} can written as
$$\H_{t} : = (f_{t})_{\sharp}\,h(\cdot,t) \qq = \Phi(\cdot, t)_{\sharp} \left( \, \det D\Phi_{(\alpha,t)}|_{T_{\alpha}V}\,  \vol_{V}(\dd \alpha) \right), \quad \text{for all } t\geq 0, $$ 
in particular, $\H_{0}=\vol_{V}$, $\H_{t}$ is the $(n-1)$-volume measure on the hypersurface $\{\Phi(x,t):\, x\in V\}$ and Proposition \ref{prop:coareaHt} reduces to the standard co-area formula.
The definition \ref{D:MeanCurv} of mean curvature bounds also reduces to the classical notions. Indeed, for $\phi\in C^{\infty}(V;\R_{\geq 0})$, the region $V_{t,\phi}$ is the domain trapped between $V$ and the normal graph of $\phi$. The first variation of the volume is thus $\frac{\dd}{\dd t} \vol_{g}(V_{t,\phi})=\cH^{n-1}(\{\Phi(x, t\phi(x)):x\in V\})$, where  $\cH^{n-1}$ is the standard $(n-1)$-volume of the hypersurface $\{\Phi(x, t\phi(x)):x\in V\}$; in particular,  $\left.\frac{\dd}{\dd t}\right|_{t=0} \vol_{g}(V_{t,\phi})=\vol_{V}(V)=\cH_{0}(V)$.  The left hand side in \eqref{eq:2ndvarVol}, corresponding to the second variation of volume, is thus the first variation of the area which gives the mean curvature $\vec{H}_{V}$ of $V$:
\begin{align*}
\lim_{t \downarrow 0} \frac{\mm(V_{t,\phi}) - t 
\int_{V} \phi \H_{0}}{t^{2}/2}&= \left. \frac{\dd}{\dd t^{2}}\right|_{t=0}  \vol_{g}(V_{t,\phi})= \int_{V} \phi^{2}\, g(\vec{H}_{V}, \nabla \tau_{V})   \, \vol_{V}.
\end{align*}
\end{remark}

\begin{remark}[Example of a surface with a conical singularity]
The notion of   forward mean curvature bound should be compared with the recent related definition proposed by Ketterer \cite{Ket:HK}. In the notation of \cite{Ket:HK}, in order to have finite bound $H_0$ one needs that the rays $X_{\alpha}$ are extendable passing through $V$, which corresponds to have an interior \& exterior ball condition (equivalent, in the smooth setting, to a  local $L^{\infty}$ bound on the full second fundamental form), see \cite[Remark 5.9]{Ket:HK}. The notion proposed above in Definition \ref{D:MeanCurv} instead works well even if the set $V$ has corners or conical singularities. Indeed, for instance, it is not hard to see  that the set 
$$V=\{(x,t)\subset \R^{n,1}: t=\alpha|x|\}, \quad \alpha\in (0,1), $$
in the $(n+1)$-dimensional Minkowski space-time $\R^{n,1}$ is an achronal topological hypersurface, smooth outside the origin (where it is Lipschitz) and having forward mean curvature bounded above by  
$H_{0}=0$ in the sense of Definition \ref{D:MeanCurv}. Notice that for any compact subset, one could choose the upper bound on the mean curvature to be strictly negative, but such an upper bound approaches zero as $|x|\to \infty$.
\end{remark}

\subsection{Hawking Singularity theorem in a synthetic framework}

Let us define $D_{H_{0}, K,N}>0$ as follows:
\begin{equation}
D_{H_{0},K,N}:=
\begin{cases}
 \frac{\pi}{2} \sqrt{  \frac{N-1}{K}} & \quad  \text{if $K>0$, $N>1$, $H_{0}=0$}\\
 \sqrt{ \frac{N-1}{K}} \cot^{-1} \left(\frac{-H_{0}}{ \sqrt{K(N-1)}} \right) & \quad  \text{if $K>0$, $N>1$, $H_{0}\in \R\setminus\{0\}$}\\
-\frac{N-1}{H_{0}}  &\quad  \text{if $K=0$, $N>1$, $H_{0}<0$}\\
\sqrt{- \frac{N-1}{K}} \coth^{-1} \left(\frac{-H_{0}}{ \sqrt{-K(N-1)}} \right) &\quad  \text{if $K<0$, $N>1$, $H_{0}<-\sqrt{-K(N-1)}$.}
\end{cases}
\end{equation}

\begin{theorem}[Hawking Singularity Theorem for $\TMCP^{e}(K,N)$ spaces]\label{thm:HawkSingSint}
Let  $(X,\sfd, \mm, \ll, \leq, \tau)$ be a timelike non-branching,  globally hyperbolic, Lorentzian geodesic space satisfying $\mathsf{TMCP}^{e}(K,N)$  for some $p\in (0,1),\, K\in \R,\,  N\in [1,\infty)$  and assume that the causally-reversed structure satisfies the same conditions. 
\\ Let $V\subset X$ be a Borel achronal FTC subset having forward mean curvature bounded above by $H_{0}$ in the sense of Definition \ref{D:MeanCurv}.  If
\begin{enumerate}
\item $K>0$, $N>1$ and $H_{0}\in \R$, or 
\item $K=0$, $N>1$ and $H_{0}<0$, or
\item $K<0$, $N>1$ and $H_{0}<-\sqrt{-K(N-1)}<0$,
\end{enumerate}
\noindent
then for  every $x\in I^{+}(V)$ it holds $\tau_{V}(x)\leq D_{H_{0},K,N}$. 
In particular, for every timelike geodesic $\gamma\in \TGeo(X)$ with $\gamma_{0}\in V$, the maximal (on the right) domain of definition is contained in $\big[0, D_{H_{0},K,N}\big]$. In case $N=1, \, H_{0}<0$, it holds that $I^{+}(V)=\emptyset$.
\end{theorem}

\begin{proof}

\textbf{Step 1}: we show that  $\sup_{x\in I^{+}(V)}  \tau_{V}(x)\leq  D_{H_{0},K,N}$, case $N\in (1,\infty)$.
\\Recall that, from Lemma \ref{lem:I+VTV}, it holds  $I^{+}(V)=  \mathcal{T}^{e}_{V}\setminus V$. Moreover, from Theorem \ref{P:nointialpoints} we have the disintegration formula
\begin{equation}\label{eq:dinsintmhalphapf}
\mm\llcorner_{I^{+}(V)}= \mm\llcorner_{\T^{e}_{V}} =  \int_{Q} h(\alpha,\cdot) \, \L^{1}\llcorner_{X_{\alpha}}\, \qq(\dd \alpha),
\end{equation}
where the closure $\bar{X}_{\alpha}$ of each $X_{\alpha}$ is a timelike geodesic starting at a point $\fa_{\alpha}\in V$ and parametrized by arclength on a  (apriori possibly unbounded) closed Real interval $I_{\alpha}:=[0, d_{\alpha}]\subset [0,\infty)$ in terms of  $\tau_{V}(\cdot)=\tau(\fa_{\alpha}, \cdot)$, see  Lemma \ref{lem:XalphaI} . For simplicity of notation, in the rest of the proof we will identify  $\bar{X}_{\alpha}$ with the closed  Real  interval  $I_{\alpha}\subset [0,\infty)$.
\\From  Theorem \ref{P:localKant},   for $\qq$-a.e. 
$\alpha \in Q$, the density $h(\alpha, \cdot)$  in \eqref{eq:dinsintmhalphapf}, has an almost everywhere 
representative that is locally Lipschitz and strictly positive 
in the interior of $I_{\alpha}$  and continuous on $I_{\alpha}$ satisfying 
\begin{equation}\label{E:MCP0N1dpf}
h(\alpha, t)\geq h(\alpha, 0) \left(\frac{\fs_{K/(N-1)}(b_{\alpha}- t)}{\fs_{K/(N-1)}(b_{\alpha})}\right)^{N-1}  \text{for all } t\in [0,b_{\alpha}], \, b_{\alpha}\in I_{\alpha}.
\end{equation}
Recalling the notation of Definition \ref{D:MeanCurv} and using \eqref{eq:dinsintmhalphapf}-\eqref{E:MCP0N1dpf},  for every  bounded  Borel  function $\phi:V\to [0,\infty)$ with compact support satisfying $\phi(f_{0}(\alpha))\in I_{\alpha}$ for every $\alpha\in Q$, and for any $\qq$-measurable assignment $Q\ni \alpha\mapsto b_{\alpha}\in I_{\alpha}$ with $b_{\alpha}\geq \phi(f_{0}(\alpha)), \, b_{\alpha}>0$ for every $\alpha \in Q$ it holds: 
\begin{align*}
\mm(V_{t,\phi}) - t \int_{V}\phi \H_{0} 
&~= \int_{Q}
\left(\int_{[0,t\phi(f_{0}(\alpha))]} h(\alpha,x)\dd x \right)\, \qq(\dd \alpha) 
- t \int_{Q}\phi(f_0(\alpha)) h(\alpha,0)  \, \qq(\dd \alpha) 
\\
&~= \int_{Q} 
\left(\int_{[0,t]} h(\alpha,s\phi(f_0(\alpha))) \phi(f_0(\alpha))\dd s \right)\, \qq(\dd \alpha) 
- t \int_{Q}\phi(f_0(\alpha)) h(\alpha,0)  \, \qq(\dd \alpha) 
\\
&~ = \int_{Q} \left( \int_{[0,t]} 
(h(\alpha,s\phi(f_0(\alpha))) - h(\alpha,0) )\dd s
\right) \phi(f_0(\alpha))\,\qq(\dd \alpha)
\\
&~ \geq  
\int_{Q} \int_{[0,t]} 
\left( \left(\frac{\fs_{K/(N-1)}( b_{\alpha}-s\phi(f_0(\alpha)) )}{ \fs_{K/(N-1)}(b_{\alpha})}\right)^{N-1} -1\right) 
\dd s\,
 \phi(f_0(\alpha))\,h(\alpha,0)\, \qq(\dd \alpha)
\\
&~ 
=\int_{Q} \int_{[0,t]} 
\left( -   \sqrt{|K|(N-1)} \frac{\fc_{K/(N-1)}(b_{\alpha})} {\fs_{K/(N-1)} (b_{\alpha})} s\phi(f_0(\alpha)) + o(s) \right) 
\dd s\,
 \phi(f_0(\alpha))\,h(\alpha,0)\, \qq(\dd \alpha)
\\
&~ 
=\int_{Q}
\left( -   \sqrt{|K|(N-1)} \frac{\fc_{K/(N-1)}(b_{\alpha})} {\fs_{K/(N-1)} (b_{\alpha})}  \frac{t^{2}}{2} + o(t^{2}) \right) \phi(f_0(\alpha))^{2}\,h(\alpha,0)\qq(\dd \alpha), \quad \forall t\in (0,1).
\end{align*}
Taking $\liminf$ of both sides of the last inequality, using Fatou's Lemma and the assumption that the forward mean curvature of $V$ is bounded above by $H_{0}$ we deduce that
$$
H_{0} \int_{V}\phi^{2} \H_{0} 
\geq 
\liminf_{t \to 0} \frac{\mm(V_{t,\phi}) - t 
\int_{V} \phi \H_{0}}{t^{2}/2}   
\geq 
\int_{V}    -   \sqrt{|K|(N-1)} \frac{\fc_{K/(N-1)}(b_{\alpha})} {\fs_{K/(N-1)} (b_{\alpha})}     \phi^{2}\H_{0},
$$
implying
$$
 b_{\alpha} \leq D_{H_{0},K,N}  \quad \qq\text{-a.e. } \alpha\in Q.
$$
By the arbitrariness of the assignment  $Q\ni \alpha\mapsto b_{\alpha}\in I_{\alpha}=[0,d_{\alpha}]\subset [0,\infty)$, it follows that 
\begin{equation}\label{eq:ubdalpha}
 d_{\alpha} \leq D_{H_{0},K,N}   \quad \qq\text{-a.e. } \alpha\in Q.
\end{equation}
Since by construction $d_{\alpha}=\sup_{x\in X_{\alpha}} \tau_{V}(x)$, the combination of \eqref{eq:ubdalpha} and the disintegration  formula \eqref{eq:dinsintmhalphapf} yields
\begin{equation}\label{eq:ubtauV}
 \tau_{V}(x)  \leq D_{H_{0},K,N}, \quad \mm\text{-a.e. } x\in I^{+}(V).
\end{equation}
The lower semi-continuity of $\tau_{V}$ on $I^+(V)$ permits to promote \eqref{eq:ubtauV} to every $x\in I^{+}(V)$.
\smallskip

\textbf{Step 2.}
Consider any timelike geodesic $\gamma$ parametrized by arclength and defined on a maximal (on the right) interval $[0, a)\subset [0,\infty)$  such that  
$\gamma_{0} \in V$. We claim that $a\leq D_{H_{0}, K,N}$.
Indeed, if by contradiction for some $s_{0}\in [0,a)$
$$
\tau(\gamma_{0}, \gamma_{s_{0}}) =s_{0} >  D_{H_{0}, K,N},
$$
the very  definition \eqref{eq:deftauV} of $\tau_{V}$ would imply $\tau_{V} (\gamma_{s_{0}})>  D_{H_{0}, K,N}$ contradicting Step 1.
\smallskip

 \textbf{Step 3}: The case $N=1, H_0<0$.
\\Recalling Remark \ref{rem:MCPN01}, in case $N=1$ the density $h_{\alpha}(\cdot)$ is costant on $I_{\alpha}$. Thus, arguing along the lines of Step 1, we get that $H_{0} \int_{V}\phi^{2} \H_{0}\geq 0$ which gives a contradiction unless $I^{+}(V)=\emptyset$.
\end{proof}

\subsection{Timelike  Bishop-Gromov, Bonnet-Myers and Poincar\'e inequalities  for $\TMCP^{e}(K,N)$}\label{Subsec:TBGBM}

In order to state the next result we need to introduce a bit of notation. Given a Borel achronal FTC subset  $V\subset X$,  we say that a subset $E\subset I^{+}(V)\cup V$ is \emph{$(\tau_{V}, R_{0})$-conically shaped} if 
$$E=\{x\in  I^{+}(V)\cup V: \tau_{V}(x)\leq R_{0}, \, y_{x}\in E\text{ for all $y_{x}\in V$ with  $\tau(y_{x}, x)=\tau_{V}(x)$} \}.$$
Note that, for a closed subset $E$, the condition is equivalent to ask that for every $x\in E\cap \T_{V}$ the intersection $E\cap [x]_{(\T_{V}, R_{V})}$ corresponds to the interval $[0,R_{0}]$ via the map $F$ of Lemma \ref{lem:XalphaI}.

\begin{proposition}[A Bishop-Gromov type inequality for achronal FTC sets in $\TMCP^{e}(K,N)$ spaces]\label{thm:BGAchronal}
Let  $(X,\sfd, \mm, \ll, \leq, \tau)$ be a timelike non-branching, globally hyperbolic, Lorentzian geodesic space satisfying $\mathsf{TMCP}^{e}(N,N)$ for some $p\in (0,1), K\in \R,  N\in (1,\infty)$ and assume that the causally-reversed
structure satisfies the same conditions. 
\\ Let $V\subset X$ be a Borel achronal FTC subset.  Then, for every compact $(\tau_{V}, R_{0})$-conically shaped subset $E\subset I^{+}(V)\cup V$ it holds (recall the definition \eqref{eq:defHt} of $\cH_{t}$):
\begin{align}
\frac{\cH_{r}(\{\tau_{V}= r\} \cap E)}{\cH_{R}(\{\tau_{V}= R\} \cap E)} \geq  \left( \frac{  \fs_{K/(N-1)} (r)}{  \fs_{K/(N-1)} (R)} \right)^{N-1},& \quad \text{for all } 0\leq r\leq R\leq R_{0}  \label{eq:BGHnRr}   \\
\frac{\mm(\{\tau_{V}\leq r\} \cap E)}{\mm(\{\tau_{V}\leq R\} \cap E)} \geq \frac{ \int_{0}^{r} \left( \fs_{K/(N-1)} (t) \right)^{N-1}\, \dd t}{ \int_{0}^{R} \left( \fs_{K/(N-1)} (t) \right)^{N-1}\, \dd t},&  \quad \text{for all }   0\leq r\leq R\leq R_{0}. \label{eq:BGmmRr}
\end{align} 
\end{proposition}
\begin{proof}
In order to show \eqref{eq:BGHnRr} observe that the combination of \eqref{eq:disintftsharp}, \eqref{eq:defHt} and Theorem \ref{P:localKant}  gives
\begin{align*}
\cH_{r}(\{\tau_{V}= r\} \cap E)&=\int_{V\cap E} h(\alpha,r) \, \qq(\dd \alpha) \\
&\geq \left( \frac{  \fs_{K/(N-1)} (r)}{  \fs_{K/(N-1)} (R)} \right)^{N-1} \int_{V\cap E} h(\alpha,R) \, \qq(\dd \alpha) \\
&= \left( \frac{  \fs_{K/(N-1)} (r)}{  \fs_{K/(N-1)} (R)} \right)^{N-1} \cH_{r}(\{\tau_{V}= R\} \cap E).
\end{align*}
The claim \eqref{eq:BGmmRr} follows from  \eqref{eq:BGHnRr} by recalling \eqref{eq:disintftsharp}, \eqref{eq:defHt} and  the classical Gromov's Lemma (see for instance \cite[Lemma III.4.1]{Chav06}).
\end{proof}
Notice that, in particular, if $\{\tau_{V}\leq R_{0}\}\subset X$ is a compact subset then  \eqref{eq:BGHnRr} and \eqref{eq:BGmmRr} remain valid without capping with the cutoff set $E$ in the left hand side.

Let us introduce some notation for the next result.   For $u:X\to \R$ we will use the short-hand notation $u(\alpha,t)$ to denote $u(\bar{X}_{\alpha}\cap \{\tau_{V} =t\})$. Notice that if $u$ is Lipschitz then,  for every $\alpha\in Q$, the function $t\mapsto u(\alpha,t)$ is locally Lipschitz and thus $\L^{1}$-a.e. differentiable with derivative denoted as  $ \frac{\partial}{\partial t} u(\alpha, t)$. For $u$ with compact support, we will also use the notation  
$$u_{\alpha}:=\frac{1}{\mm_{\alpha}(\supp\, u)}  \int_{X_{\alpha}} u\,  \mm_{\alpha} \quad \text{if $\mm_{\alpha}(\supp\, u)\neq 0$,  and $u_{\alpha}:=0$ otherwise,}$$
to denote the average of $u$ on $(X_{\alpha}, \mm_{\alpha})$.

\begin{proposition}[A timelike Poincar\'e inequality  for $\TMCP^{e}(K,N)$]\label{prop:Poincare}
For every $(K,N,D)\in \R\times (1,\infty) \times (0,\infty)$,  there exists a constant $\lambda_{\mathsf{MCP}_{K,N,D}}>0$ with the following property.
\\Let  $(X,\sfd, \mm, \ll, \leq, \tau)$  and $V\subset X$ be as in Proposition  \ref{thm:BGAchronal}. Then, for every $u:X\to \R$ Lipschitz with compact support contained in $I^{+}(V)$ it holds
\begin{equation}\label{eq:TimelikePoinc}
\int_{X} \left| u- u_{\alpha} \right|^{2} \mm \leq \lambda_{\mathsf{MCP}_{K,N,D}} \int_{X} \left| \frac{\partial}{\partial t} u(\alpha, t)  \right|^{2} \mm,  
\end{equation}
where $D:=\sup_{\alpha\in Q}  \sup_{x,y\in X_{\alpha}\cap \supp \, u} \tau(x,y)\leq \sup_{x,y\in \supp \, u} \tau(x,y)<\infty$.
\end{proposition}

\begin{proof}
Since from Theorem \ref{P:localKant} each ray $(X_{\alpha}, \mm_{\alpha})$ is an $\MCP(K,N)$ space, from \cite{HM} we know that
$$
\int_{X_{\alpha}} \left| u(\alpha,t)- u_{\alpha} \right|^{2} \,  \mm_{\alpha}(\dd t) \leq \lambda_{\mathsf{MCP}_{K,N,D}} \int_{X_{\alpha}}  \left| \frac{\partial}{\partial t} u(\alpha, t)  \right|^{2}  \,  \mm_{\alpha}(\dd t).
$$
The claimed \eqref{eq:TimelikePoinc} follows then from the disintegration formula \eqref{E:disintegration}.
\end{proof}
It is possible to give quite   precise  estimates on the constan $\lambda_{\mathsf{MCP}_{K,N,D}}$, the interested reader is referred to \cite{HM}.

Finally we take advantage of the techniques developed in the second part of the paper to sharpen, for timelike non-branching spaces,  the timelike Bishop-Gromov  inequality obtained in Proposition  \ref{prop:BisGro} and the timelike Bonnet-Myers  inequality obtained in Proposition \ref{prop:BonMy}.

\begin{proposition}[A  timelike Bishop-Gromov inequality for timelike non-branching $\TMCP^{e}(K,N)$]\label{prop:BisGro2}
Let $(X,\sfd, \mm, \ll, \leq, \tau)$  be as in Proposition  \ref{thm:BGAchronal}. Then, for each $x_{0}\in X$, each compact subset  $E\subset I^{+}(x_{0})\cup\{x_{0}\}$ $\tau$-star-shaped with respect to $x_{0}$, and each $0<r<R\leq \pi \sqrt{(N-1)/(K \vee 0)}$, it holds:
\begin{equation}\label{eq:BS2}
\frac{s(E, r)}{s(E,R)} \geq \left(\frac{\fs_{K/(N-1)}(r)}  {\fs_{K/(N-1)}(R)} \right)^{N-1}, \quad  \frac{v(E, r)}{v(E,R)} \geq \frac{\int_{0}^{r} \fs_{K/(N-1)} (t)^{N-1} \dd t}   {\int_{0}^{R}\fs_{K/(N-1)} (t)^{N-1} \dd t }.
\end{equation}
\end{proposition}

\begin{proof}
Consider $\tau_{x_{0}}(\cdot):=\tau(x_{0},\cdot):I^{+}(x_{0})\to\R$. One can repeat verbatim (actually here it would be slightly easier) the constructions of Section \ref{Sec:LocTMCP} replacing $\tau_{V}$ by $\tau_{x_{0}}$ and obtain a partition (up to a set of $\mm$-measure zero) of $I^{+}(x_{0})$ into transport rays $\{X_{\alpha}\}_{\alpha\in Q}$ associated to $\tau_{x_{0}}$,  i.e. each $X_{\alpha}$ is a future pointing radial $\tau$-geodesic emanating from $x_{0}$.
One can disintegrate $\mm\llcorner_{I^{+}(x_{0})}$ accordingly as
$\mm\llcorner_{I^{+}(x_{0})} =\int_{Q} \mm_{\alpha} \, \qq(\dd \alpha)$ where each $\mm_{\alpha}$ is concentrated on  $X_{\alpha}$, and $(X_{\alpha}, |\cdot|, \mm_{\alpha})$ is a 1-dim. $\MCP(K,N)$ m.m.s..
One can now prove \eqref{eq:BS2}  along the same lines of the proof of Proposition \ref{thm:BGAchronal}.
\end{proof}

\begin{proposition}[A  timelike Bonnet-Myers  inequality for timelike non-branching $\TMCP^{e}(K,N)$]\label{prop:BonMy2}
Let $(X,\sfd, \mm, \ll, \leq, \tau)$  be as in Proposition  \ref{thm:BGAchronal}, with $K>0$.
Then
\begin{equation}\label{eq:BonMy2}
\sup_{x,y\in X} \tau(x,y) \leq \pi \sqrt{\frac{N-1}{K}}.
\end{equation}
\end{proposition}

\begin{proof}
Assume by contradiction that there exist $x_{0},x_{1}\in X$ with $\tau(x_{0},x_{1})\geq \pi \sqrt{(N-1)/K}+2\ve$, for some  $\ve>0$. Let $\delta>0$ be such that   
$$ \inf\{ \tau(x_{0},y):  y\in B^{\sfd}(x_{1},\delta)\} \geq \pi \sqrt{(N-1)/K}+\ve.$$
Consider the disintegration $\mm\llcorner_{I^{+}(x_{0})} =\int_{Q} \mm_{\alpha} \, \qq(\dd \alpha)$   associated to $\tau_{x_{0}}$, as outlined in the proof of Proposition \ref{prop:BisGro2}. Since $\mm(B^{\sfd}(x_{1},\delta))>0,$ it follows that $L_{\tau}(X_{\alpha}) \geq \pi \sqrt{(N-1)/K}+\ve$ for a $\qq$-non negligible subset of rays. But since every $(X_{\alpha}, |\cdot|, \mm_{\alpha})$ is a 1-dim. $\MCP(K,N)$ m.m.s. with full support, its diameter is at most  $\pi\sqrt{(N-1)/K}$ (as it's easily seen from \eqref{E:MCP0N1d}). Contradiction.
\end{proof}

\begin{remark}[Sharpness]\label{rem:sharpness}
The Lorentzian model spaces are: for $K<0$  (scaled) de Sitter space, $K=0$ Minkowski space, $K>0$ (scaled) anti-de Sitter space. Recall that the standard de Sitter space $(M^{n}, g_{dS})$ has constant sectional curvature equal to $1$, thus $\Ric_{g_{dS}}(v,v)=-(n-1)g_{dS}(v,v)$ for $v$ timelike, and hence it is the model space for $K=-(n-1)$. The Minkowski space has null sectional (and thus Ricci) curvatures, thus is the model space for $K=0$. The  anti de-Sitter space  $(M^{n}, g_{adS})$ has constant sectional curvature equal to $-1$, thus $\Ric_{g_{adS}}(v,v)=(n-1)g_{dS}(v,v)$ for $v$ timelike, and hence it is the model space for $K=n-1$. It is well known that any globally hyperbolic subset of $(M^{n}, g_{adS})$ has timelike diameter at most $\pi$, with sharp bound; this shows the sharpeness of  Proposition \ref{prop:BonMy2}.  Using that in the model spaces the sectional curvature is constant,  direct volume computations via Jacobi fields show that equality is achieved in \eqref{eq:BS2}; thus also Proposition \ref{prop:BisGro2} is sharp. Choosing $V$ to be a level set of the natural time-function in the model spaces, it is possible to check that equality is achieved  in \eqref{eq:BGHnRr} and  \eqref{eq:BGmmRr} as well.
\end{remark}

\subsection{The case of a spacetime with continuous metric}\label{SS:C0metric}
Next we specialise Theorem  \ref{thm:HawkSingSint} to the case of a spacetime with a continuous metric.  As observed in \cite{CG}, spacetimes with  continuous metrics may present pathological causal behaviour. For instance \cite{CG} (see also \cite[Section 5.1]{KS}) gives examples of spacetimes with H\"older-continous metrics where the null curves emanating from a point cover a set with non-empty interior, a phenomenon called ``bubbling''. In order to prevent such a pathological behavior, \cite{CG} proposed the notion of ``causally plain'' metric. Let us briefly recall it together with the needed notation.
\\

\noindent
\textbf{The notion of causally plain Lorentzian metric}. Let $\check g,g$ be two Lorentzian metrics. We write $\check g \prec g$ if $\{v\neq 0 : \check g(v,v)\leq 0\}\subset \{v : g(v,v)< 0\}$. For a neighbourhood $U$ of $x\in M$, set 
\begin{align*}
\check I^{+}_{g}(x,U):= \{y\in U: \, &\exists \text{ a smooth Lorentzian metric $\check g\prec g$ and a future pointing $\check g$-timelike curve }\\
&\text{$\gamma:[0,1]\to U$, with  $\gamma_{0}=x$, $\gamma_{1}=y$ and $\check g(\dot \gamma, \dot \gamma)<0$}\}.
\end{align*}
The set $\check I^{-}_{g}(x,U)$ is defined analogously.  It is clear that $\check I_{g}^{\pm}\subset I_{g}^{\pm}$ and equality holds for smooth metrics.
Let us also recall that a \emph{cylindrical neighbourhood} of a  point $x\in X$ with respect to $g$, is a relatively compact chart domain containing $x$ such that, in this chart, $g$ equals the Minkowski
metric at $x$ and the slopes of the light cones of $g$ stay close to 1 (for the precise definition see \cite[Def. 1.8]{CG}).

A spacetime $(M,g)$ is said to be \emph{causally plain} if every $x\in M$ admits a cylindrical neighbourhood $U$  such that $\partial \check I_{g}^{\pm}(x,U)=  \partial  J^{\pm}(x,U)$; otherwise $(M,g)$ is said to be \emph{bubbling} \cite[Def. 1.16]{CG}.
\\The rough idea is that $(M,g)$ is causally plain provided, for every $x\in M$,  the span of all null curves emanating from  $x$ has  empty interior. 
\\It was proved in \cite[Corollary 1.17]{CG} that a spacetime with locally Lipschitz continuous Lorentzian metric is causally plain. In the same paper \cite[Section 1.1]{CG}  (see also \cite[Section 5.1]{KS}) examples of H\"older-continuos bubbling Lorentzian metrics are discussed.
\\

Let $(M,g)$ be a spacetime with a $C^{0}$-Lorentzian metric. Recall that any Cauchy hypersurface is causally complete. 
\\It is then clear that Proposition \ref{prop:BrunnMnk},   Proposition \ref{prop:BisGro}, Proposition \ref{prop:BonMy} and Remark \ref{rem:GeomPropTMCP} give the following:

\begin{corollary}[Geometric properties of a  globally hyperbolic causally plain spacetime with $C^{0}$ metric, with synthetic timelike Ricci bounded below]\label{cor:GeomPropTMCPC0}
Let $(M,g)$ be a $2\leq n$-dimensional globally hyperbolic, causally plain spacetime with a  $C^{0}$-Lorentzian metric. 
\begin{itemize}
\item \emph{Timelike Bishop-Gromov:}  If $(M,g)$ satisfies $\mathsf{TMCP}^{e}(K,N)$ for some $p\in (0,1), K\in \R,  N\in [1,\infty)$, then for each $x_{0}\in M$, each compact subset  $E\subset I^{+}(x_{0})\cup\{x_{0}\}$ $\tau$-star-shaped with respect to $x_{0}$, and each $0<r<R\leq \pi \sqrt{N/(K \vee 0)}$,  inequality \eqref{eq:BS} holds.
\item \emph{Timelike Bonnet-Myers:}  If $(M,g)$ satisfies $\mathsf{TMCP}^{e}(K,N)$ for some $p\in (0,1), K>0,  N\in [1,\infty)$, then  \eqref{eq:BonMy} holds. In particular $(M,g)$ is not timelike geodesically complete.
\item \emph{Timelike Brunn-Minkowski:} If $(M,g)$ satisfies  $\mathsf{wTCD}^{e}_{p}(K,N)$ for some $p\in (0,1), K\in \R,  N\in [1,\infty)$, then \eqref{eq:BrunnMink} holds.
\end{itemize}
\end{corollary}

In Corollary \ref{Cor:GeomPropTMCPNomBranch} below, taking advantage  of the techniques developed in Section \ref{Ss:Nonbranching}  and  Section \ref{Sec:LocTMCP},     the results of Corollary \ref{cor:GeomPropTMCPC0} will be improved to sharp forms in case of  timelike non-branching structures. 
\\

It is clear that Theorem \ref{thm:HawkSingSint} implies the following result for  a spacetime with $C^{0}$-Lorentzian metric.

\begin{corollary}[Hawking Singularity Theorem for a spacetime with a $C^{0}$-Lorentzian metric]\label{cor:HawSingLCont}
Let $(M,g)$ be a $2\leq n$-dimensional timelike non-branching, globally hyperbolic, causally plain spacetime with a   $C^{0}$-Lorentzian metric satisfying $\mathsf{TMCP}^{e}(K,N)$ for some $p\in (0,1), K\in \R, N\in (1,\infty)$ and assume that the causally-reversed structure satisfies the same conditions. 
\\ Let $V\subset M$ be a Borel achronal FTC subset (or, more strongly, let $V$ be a Cauchy hypersurface) having forward mean curvature bounded above by $H_{0}<0$ in the sense of Definition \ref{D:MeanCurv}. 

Then for  every $x\in I^{+}(V)$ it holds $\tau_{V}(x)\leq D_{H_{0},K,N}$, provided $H_0, K,N$ fall in the range specified in Theorem \ref{thm:HawkSingSint} . 
In particular, for every timelike geodesic $\gamma\in \TGeo(M)$ with $\gamma_{0}\in V$, the maximal (on the right) domain of definition is contained in $\big[0, D_{H_{0},K,N}\big]$; in particular  
 $(M,g)$ is not timelike geodesically complete.
\end{corollary}

\begin{remark}\label{rem:litHS}[Literature about Hawking singularity Theorem]
Hawking singularity Theorem was proved in \cite[Theorem 4, p. 272]{HawEll} for smooth space-times (the proof works for $C^{2}$ metrics) assuming that $V$ is a \emph{compact} spacelike slice. The result was extended to  $C^{1,1}$ metrics in \cite{Kunz:Haw}  and to $C^1$ metrics in \cite{Graf}, by approximating the metric of low regularity with smoother metrics. The extension to non-compact future causally complete $V$ was established in \cite[Theorem 3.1]{Ga} (see also \cite{GrTr}) in the smooth setting, and extended to $C^{1,1}$ metrics in \cite{Graf1}. Theorem  \ref{thm:HawkSingSint} and Corollary \ref{cor:HawSingLCont}, already in the smooth setting, relax the future causal completeness with the weaker future timelike completeness (in addition to extend the results to a  synthetic framework, including $C^{0}$ metrics with timelike non-branching geodesics).

Hawking (as well as Penrose and Hawking-Penrose) singularity Theorem was also extended to (smooth) closed cone structures \cite{Min} and smooth weighted Lorentz-Finsler manifolds \cite{LMO}.
Let us  mention that a first synthetic singularity theorem was recently shown in \cite{AGKS} under the stronger assumptions that the space is a synthetic warped product with lower bound on sectional curvature in the sense of comparison triangles (\'a la Alexandrov).

Few weeks after we announced the present work, we learnt of \cite{BKMW}, proving a \emph{Riemannian} version of Hawking's singularity theorem in the framework of \emph{metric measure spaces} with Ricci curvature bounded below in a synthetic sense via optimal transport.
\end{remark}

\begin{remark}[On the timelike non-branching assumption]
\begin{itemize}
\item The validity of the non-branching property for timelike geodesics (assumed in Theorem  \ref{thm:HawkSingSint} and Corollary \ref{cor:HawSingLCont}) is verified for $C^{1,1}_{\rm{loc}}$-Lorentzian metrics by standard Cauchy-Lipschitz theory of ODEs.
\item In the metric theory, it was recently proved \cite{QDeng} that infinitesimally Hilbertian $\CD(K,N)$ spaces are non-branching. It is natural to expect that an analogous result holds also in the Lorentzian setting, namely that infinitesimally Minkowskian (to be properly defined) $\TCD^e_p(K,N)$ spaces are timelike non-branching.
\end{itemize}
\end{remark}

Specialising  Proposition \ref{thm:BGAchronal}, Proposition \ref{prop:Poincare}, Proposition \ref{prop:BisGro2}, Proposition \ref{prop:BonMy2} to the case of a  spacetime with a $C^{0}$-Lorentzian metric  give:

\begin{corollary}[Timelike Bishop-Gromov, Bonnet-Myers and Poincare' inequalities]\label{Cor:GeomPropTMCPNomBranch}
Let $(M,g)$ be a $2\leq n$-dimensional timelike non-branching, globally hyperbolic, causally plain spacetime with a   $C^{0}$-Lorentzian metric satisfying $\mathsf{TMCP}^{e}(K,N)$ for some $p\in (0,1), K\in\R,  N\in (1,\infty)$ and assume that the causally-reversed structure satisfies the same conditions.  \\ Let $V\subset M$ be a Borel achronal FTC subset (or, more strongly, let $V$ be a Cauchy hypersurface). Then:
\begin{itemize}
\item \emph{Timelike Bishop-Gromov I:}  For every compact $(\tau_{V}, R_{0})$-conically shaped subset $E\subset I^{+}(V)\cup V$ the inequalities \eqref{eq:BGHnRr} and  \eqref{eq:BGmmRr} hold.
In particular, if $\{\tau_{V}\leq R_{0}\}\subset X$ is a compact subset then  \eqref{eq:BGHnRr} and \eqref{eq:BGmmRr} remain valid without capping with the cutoff set $E$ in the left hand side. 
\item \emph{Timelike Bishop-Gromov II:}  For each $x_{0}\in M$, each compact subset  $E\subset I^{+}(x_{0})\cup\{x_{0}\}$ $\tau$-star-shaped with respect to $x_{0}$, and each $0<r<R\leq \pi \sqrt{(N-1)/(K \vee 0)}$, the inequalities \eqref{eq:BS2} hold.
\item \emph{Timelike Poincar\'e:} For every $u:M\to \R$ Lipschitz with compact support contained in $I^{+}(V)$, the inequality \eqref{eq:TimelikePoinc} holds.
\item \emph{Timelike Bonnet-Myers:} If $K>0$, then  inequality \eqref{eq:BonMy2} holds.
\end{itemize}
\end{corollary}


\renewcommand{\thesection}{A}
\section{Appendix - $\TMCP^{e}(K,N)$ on smooth Lorentzian manifolds}\label{sec:AppTMCPsmooth}

\setcounter{lemma}{0}
\setcounter{equation}{0}  
\setcounter{subsection}{0}

\begin{theorem}\label{thm:TMCPsmooth}
Let $(M^{n}, g)$ be a  globally hyperbolic smooth spacetime of dimension $n\geq 2$ without boundary. Then the associated Lorentzian geodesic space satisfies $\TMCP^{e}(K,n)$ if only if $\Ric_{g}(v,v)\geq -K g(v,v)$ for every timelike vector $v\in TM$.
\end{theorem}

\begin{proof} 
{\bf Step 1}: ``If'' implication. \\ 
From Theorem \ref{thm:CharLorRic}, the Lorentzian geodesic space associated to $(M^{n}, g)$ satisfies the $\TCD^{e}_{p}(K,n)$ condition, which in turn implies $\TMCP^{e}(K,n)$ by Proposition \ref{prop:CD->MCP} (see also Remark \ref{rem:C0metrics}). 

\smallskip
{\bf Step 2} : ``Only if'' implication. \\
Fix $x\in M$ and $v\in T_{p}M$ future pointing with $g(v,v)=-1$.
Let  $U$ be a compact subset  of $\{w\in T_{x}M: w \text{ is future pointing with } g(w,w)<0 \}$,  star-shaped with respect to $0$,  such that $rv\in U$ for $r>0$ small enough and such that  the exponential map $\exp_{x}^{g}:U\to M$ of $g$ based at $x$ is a diffeomorphism onto its image when restricted to $U$.
Calling $d\vol_{g}$ the volume density on $M$ associated to $g$, recall that it can be represented as
\begin{equation}\label{eq:dvolgA}
d\vol_{g}(y)= (\exp_{x}^{g})_{\sharp} \left(  {\rm A}_{x}(r,\xi) dr d\xi\right), \quad \text{ for all } y=\exp_{x}^{g}(r\xi)\in \exp_{x}^{g}(U),
\end{equation}
where ${\rm A}_{x}(r,\xi)$ denotes the volume density on $\{r\xi\in U: g(\xi,\xi)=-1\}$ induced by $g$.
  
 Fix  a $g$-orthonormal basis $e_{1}, e_{2},\ldots, e_{n}$ of $T_{p}M$ with $e_{1}=v$ and 
 denote by $\kappa_{i}$ the sectional curvature of the plane spanned by $e_{1}$ and $e_{i}$, 
 for $i=2,\dots, n$. 
 Recalling the definitions  \eqref{eq:deffsfc}, \eqref{eq:sigmakappa}   of  $\fs_{\kappa}(\vartheta)$ and $
\sigma_{\kappa}^{(t)}(\vartheta)$ respectively,
it is easy to check that for small $r>0$ it holds:
\begin{equation}\label{eq:expsigmakappar}
\sigma_{\kappa}^{(1/2)} (2r)=\frac{\fs_{\kappa}( r)}{\fs_{\kappa}(2 r)}=\frac{1}{2}\left(1+\frac{\kappa}{2} r^{2} + O(r^{4})\right).
\end{equation}
Standard Jacobi-fields computations (see for instance \cite{EhSa98} for the Lorentzian setting) give that
\begin{equation}\label{eq:ArvA2rv}
\frac{ {\rm A}_{x}(r,v)}{ {\rm A}_{x}(2r,v)}= \prod_{i=2}^{n} \frac{\fs_{\kappa_{i}}( r)}{\fs_{\kappa_{i}}(2 r)}+O(r^{3}).
\end{equation}
Plugging  \eqref{eq:expsigmakappar} into \eqref{eq:ArvA2rv} yields
\begin{align}
\frac{ {\rm A}_{x}(r,v)}{ {\rm A}_{x}(2r,v)}&= \frac{1}{2^{n-1}} \prod_{i=2}^{n}  \left(1+\frac{\kappa_{i}}{2} r^{2}\right) +O(r^{3}) =  \frac{1}{2^{n-1}} \left(1+\sum_{i=2}^{n} \kappa_{i} r^{2}\right) +O(r^{3}) \nonumber\\
&= \frac{1}{2^{n-1}} \left(1+\Ric_{g}(v,v) r^{2} \right)+O(r^{3}).\label{eq:ArvA2rvRic}
\end{align}
We next relate $ {\rm A}_{x}(r,v)/ {\rm A}_{x}(2r,v)$ with the $\TMCP^{e}(K,n)$ condition via localisation. 

Consider  $\tau_{x}(\cdot):=\tau(x,\cdot):I^{+}(x)\supset \exp_{x}^{g}(U)\to\R$. By the very definitions, we have $\tau_{x} \left(\exp_{x}^{g}(r\xi)\right)=r$ for every $r\xi\in U$, $g(\xi,\xi)=-1$.
In other terms, the partition of $\exp_{x}^{g}(U)\setminus \{x\}$ by future pointing $g$-geodesics emanating from $x$  coincides with the partition by transport rays induced by $\tau_{x}$.

Under this identification, the disintegration of $d\vol_{g}$ induced by $\tau_{x}$ is nothing but \eqref{eq:dvolgA}.  Theorem \ref{P:localKant}  then gives that $r\mapsto {\rm A}_{x}(r,v)$ is an $\MCP(K,n)$ density on an interval $(0, \ve_{v})$ (see for instance the proof of \cite[Theorem 3.2]{Ohta1}): it thus satisfies
\begin{equation}\label{eq:ArvMCP}
\frac{ {\rm A}_{x}(r,v)}{ {\rm A}_{x}(2r,v)}\geq \left( \frac{\fs_{K/(n-1)}(r)}{\fs_{K/(n-1)}(2r)}\right)^{n-1}.
\end{equation}
Combining \eqref{eq:ArvA2rvRic} with  \eqref{eq:ArvMCP}, we obtain
\begin{align*}
\Ric_{g}(v,v) r^{2}&\geq   \left( \frac{2\fs_{K/(n-1)}(r)}{\fs_{K/(n-1)}(2r)}\right)^{n-1} -1 +O(r^{3}) =   \left( 1+ \frac{K}{n-1} r^{2} \right)^{n-1} -1 +O(r^{3}) \\
&= K r^{2}+O(r^{3}).
\end{align*}
Dividing  both sides by $r^{2}$ and sending $r\downarrow 0$, we thus obtain $\Ric_{g}(v,v)\geq K=-K g(v,v)$. 
\\By the arbitrariness of $x$ and $v$, the proof is complete.
\end{proof}

\begin{corollary}\label{Cor:CompMCPsmooth}
Let $(M^{n}, g)$ be a  globally hyperbolic smooth spacetime of dimension $n\geq 2$ without boundary. 
\begin{enumerate}
\item If $\Ric_{g}(v,v)\geq -K g(v,v)$ for every timelike vector $v\in TM$, then the  associated Lorentzian geodesic space satisfies $\TMCP^{e}(K',N')$ for every $K'\leq K$ and $N'\geq N$.
\item If the  Lorentzian geodesic space associated to $(M^{n}, g)$ satisfies $\TMCP^{e}(K,N)$, then $n\leq N$.
\end{enumerate}
\end{corollary}

\begin{proof}
The first statement follows from Theorem \ref{thm:TMCPsmooth} and  
Lemma \ref{lem:TCDscaling} (or from  Theorem \ref{thm:CharLorRic} and  Proposition \ref{prop:CD->MCP}).
\\We now prove the second statement. We will build on the proof of Theorem \ref{thm:TMCPsmooth}. Fix $x\in M$ and let
 $$U\subset \{w\in T_{x}M: w \text{ is future pointing with } g(w,w)<0 \}$$
 be compact  star-shaped with respect to $0$, with non-empty interior,  such that  the exponential map $\exp_{x}^{g}:U\to M$ of $g$ based at $x$ is a diffeomorphism onto its image when restricted to $U$. 
Calling
$$B_{r}^{g}(x,U):=\exp_{x}^{g} (\{w\in U: |g(w,w)|\leq r\}), \quad r>0,$$
it is easy to see that there exists $c=c_{U}>0$ such that
\begin{equation}\label{eq:volgsmallr}
\vol_{g}(B_{r}^{g}(x,U))=c \, r^{n}+O(r^{n+1}), \quad \text{for small } r>0.
\end{equation}
On the other hand, using that $r\mapsto {\rm A}_{x}(r,v)$ is an $\MCP(K,N)$ density (see the discussion before \eqref{eq:ArvMCP}) and recalling \eqref{eq:dvolgA}, we obtain via the classical Gromov's Lemma (see for instance \cite[Lemma III.4.1]{Chav06}) that
\begin{equation}\label{eq:TimelikeBG}
(0,\ve)\ni r\mapsto \frac{\vol_{g}(B_{r}^{g}(x,U))} {\int_{0}^{r} \left[\fs_{K/(N-1)}(t) \right]^{N-1} \, \dd t} \quad \text{is monotone non-increasing}.
\end{equation}
Since $\fs_{K/(N-1)}(t)=O(t)$ for small $t>0$, it is easy to see  that the combination of \eqref{eq:volgsmallr} and  \eqref{eq:TimelikeBG} yields $n\leq N$.
\end{proof}

\begin{remark}
In general, $\TMCP^{e}(K,N)$ on a globally hyperbolic smooth spacetime \emph{does not} imply that  $\Ric_{g}(v,v)\geq -K g(v,v)$ for every timelike vector $v\in TM$. It follows  that $\TMCP^{e}(K,N)$ is  a \emph{strictly weaker}  condition than  $\mathsf{wTCD}^{e}_{p}(K,N)$. More precisely, the following holds: For each $N>1$ there exists a constant $c_{N}>0$ such that each  globally hyperbolic smooth spacetime
with  timelike Ricci curvature $\geq 0$, dimension $\leq N-1$ and $\tau$-diameter $\leq L$  satisfies $\TMCP^{e}(K,N)$ for each positive $K\leq c_{N}/L^{2}$ (compare with \cite[Remark 5.6]{sturm:II} for the Riemannian setting).  
\end{remark}

\begin{proof}
From Theorem \ref{thm:TMCPsmooth} we know that $\TMCP^{e}(0,N-1)$ holds.
Recalling that the $\TMCP^{e}(K,N)$ condition  is equivalent to \eqref{eq:MCPEnt}, it is sufficient to show that 
\begin{equation*}
t^{N-1}\geq \left(\sigma_{K/N}^{(t)}(\vartheta) \right)^{N}, \quad \forall t\in [0,1], \, \vartheta\in [0,L].
\end{equation*}
Now, for sufficiently small $c_{N}\in (0,1)$ and all $K\vartheta^{2}\leq c_{N}$, the right-hand side can be estimated from above by $t^{N}(1+(1-t^{2}) K\vartheta^{2})$. But clearly $
t^{N-1}\geq t^{N}\left( 1+(1-t^{2}) K\vartheta^{2} \right)$,  for all  $K\vartheta^{2}\leq c_{N}$.
\end{proof}

\end{document}